\documentclass[b5paper]{book}
\usepackage{amsmath,amsthm,amssymb}

\usepackage[utf8]{inputenc}
\usepackage[T1]{fontenc}
\usepackage{microtype}
\usepackage[charter]{mathdesign}

\usepackage[dutch,american]{babel}
\usepackage[hidelinks]{hyperref}
\hypersetup{
plainpages=false,
pdfpagelabels=true,
bookmarksnumbered=false,
bookmarks=true,
pdfpagemode=UseNone,
pdfstartview=FitH,
pdfdisplaydoctitle=true,
pdflang=en-US,
pdftitle={Ergodic theorems for polynomials in nilpotent groups},
pdfauthor={Pavel Zorin-Kranich},
pdfsubject={Mathematics Subject Classification (2010): Primary 37A30}
}

\usepackage[safeinputenc=true,backref=true,style=alphabetic,firstinits,maxnames=4,useprefix=true,url=false]{biblatex}

\usepackage{makeidx}
\makeindex
\usepackage{enumitem}

\newtheorem{theorem}[equation]{Theorem}
\newtheorem{multvNtheorem}[equation]{Multiplicator von Neumann Theorem}
\newtheorem{structuretheorem}[equation]{Structure Theorem}
\newtheorem{inversetheorem}[equation]{Inverse Theorem}
\newtheorem{lemma}[equation]{Lemma}
\newtheorem{proposition}[equation]{Proposition}
\newtheorem{corollary}[equation]{Corollary}
\theoremstyle{definition}
\newtheorem{definition}[equation]{Definition}
\theoremstyle{remark}
\newtheorem{remark}[equation]{Remark}

\newtheorem{example}[equation]{Example}

\newcommand{\veps}{\varepsilon}
\newcommand*{\Gb}[1][]{G_{\bullet #1}}
\newcommand*{\Hb}[1][]{H_{\bullet #1}}
\newcommand*{\tHb}[1][]{\tilde H_{\bullet #1}}
\newcommand{\PolyDomain}{\Gamma} % Domain of polynomials
\newcommand*{\poly}[1][\Gb]{P(\PolyDomain,#1)}
\newcommand*{\polyn}[1][\Gb]{P_{0}(\PolyDomain,#1)}
\newcommand*{\VIP}[1][\Gb]{\mathrm{VIP}(#1)}
\usepackage{xstring}
\usepackage{ifthen}
\newcommand{\EncloseInParentheses}[1]{
\StrLen{#1}[\arglen]
\ifthenelse{1 < \arglen}{(#1)}{#1}
}
\newcommand*{\PE}[2]{{#1}^{\otimes \EncloseInParentheses{#2}}}
\newcommand*{\HKZ}{\mathcal{Z}}
\newcommand*{\roots}[2][]{\sqrt[#1]{#2}}
\newcommand{\nsubg}{\trianglelefteq}
\newcommand{\subg}{\leq}
\newcommand*{\N}{\mathbb{N}}
\newcommand*{\Z}{\mathbb{Z}}
\newcommand*{\R}{\mathbb{R}}
\newcommand*{\dif}{\mathrm{d}}
\newcommand*{\inv}{^{-1}}
\newcommand*{\dom}{\mathrm{dom}\,}
\newcommand*{\E}{\mathbb{E}}
\newcommand*{\from}{\colon}
\def\<{\left\langle}
\def\>{\right\rangle}
\newcommand*{\Fin}{\mathcal{F}}
\newcommand*{\Fine}{\Fin_{\emptyset}}
\newcommand*{\gb}{g_{\bullet}}
\newcommand*{\hb}{h_{\bullet}}
\newcommand*{\Fb}[1][]{F_{\bullet #1}}
\newcommand*{\Wb}[1][]{W_{\bullet #1}}

\newcommand*{\FE}{\mathrm{FE}}
\newcommand*{\AP}{\mathrm{AP}}
\DeclareMathOperator*{\IPlim}{IP-lim}
\DeclareMathOperator*{\wIPlim}{w-IP-lim}
\DeclareMathOperator{\im}{im}
\DeclareMathOperator{\fix}{fix}
\DeclareMathOperator{\rank}{rank}
\DeclareMathOperator{\lin}{lin}
\DeclareMathOperator{\Ad}{Ad}
\newcommand*{\sD}{\tilde D} % symmetric derivative
\newcommand*{\rD}{\hat D}
\newcommand*{\nint}[1]{\lfloor #1 \rceil}
\newcommand*{\dint}[1]{\| #1 \|}
\newcommand*{\floor}[1]{\lfloor #1 \rfloor}
\newcommand*{\C}{\mathbb{C}}
\newcommand*{\T}{\mathbb{T}}
\newcommand*{\bfg}{\mathbf{g}}
\newcommand*{\bfh}{\mathbf{h}}
\newcommand*{\gf}{\bfg f}
\def\<{\left\langle}
\def\>{\right\rangle}

\DeclareMathOperator{\complexity}{cplx}
\newcommand*{\cplx}{\mathbf{c}}
\newcommand*{\vd}{d}
\newcommand*{\Av}[2][\bfg]{\mathcal{A}^{#1}_{#2}}
\newcommand*{\Q}{\mathbb{Q}}
\newcommand*{\m}{\mathrm{m}}
\newcommand*{\id}{\mathrm{id}}

\newcommand{\aveN}{\frac{1}{N}\sum_{n=1}^N}

\newcommand{\Fo}{\Phi}
\newcommand{\ave}[2]{\E_{#1 \in #2}}
\newcommand{\aveFN}{\ave{n}{\Fo_N}}
\newcommand{\aveFMn}{\ave{n}{\Fo_{M}}}

\newcommand{\aveFn}{\ave{n}{\Fo_N}}
\newcommand*{\RTT}[1]{$\mathrm{RTT}(#1)$}
\newcommand{\ul}[1]{\underline{#1}}
\newcommand{\ol}[1]{\overline{#1}}
\newcommand{\ud}{\overline{d}}
\newcommand{\ld}{\underline{d}}
\newcommand{\AG}{\mathcal{G}} % Amenable group
\newcommand{\lag}{\mathfrak{g}}
\newcommand{\comm}{\mathrm{comm}}
\newcommand*{\cube}[2]{#1^{\square}_{#2}}
\newcommand*{\cuben}[2]{#1^{\tilde\square}_{#2}}

\def\clap#1{\hbox to 0pt{\hss#1\hss}}

\def\mathclap{\mathpalette\mathclapinternal}

\def\mathclapinternal#1#2{%
  \clap{$\mathsurround=0pt#1{#2}$}}

\makeatletter
\def\cleardoublepage{\clearpage\if@twoside \ifodd\c@page\else
    \hbox{}
    \thispagestyle{empty}
    \newpage
    \if@twocolumn\hbox{}\newpage\fi\fi\fi}
\makeatother \clearpage{\pagestyle{plain}\cleardoublepage}

\begin{document}
\frontmatter
\begin{titlepage}
\begin{center}
\textbf{\huge Ergodic theorems for polynomials\\ in nilpotent groups}

\vspace{0.1\textheight}
\selectlanguage{dutch}
\textsc{academisch proefschrift}
\vspace{0.1\textheight}

ter verkrijging van de graad van doctor

aan de Universiteit van Amsterdam

op gezag van de Rector Magnificus

prof.\ dr.\ D.C.\ van den Boom

ten overstaan van een door het college voor promoties ingestelde

commissie, in het openbaar te verdedigen in de Agnietenkapel

op donderdag 12 september 2013, te 14:00 uur

\vspace{0.1\textheight}
door
\vspace{0.1\textheight}

{\large Pavel Zorin-Kranich}

\vspace{0.1\textheight}
geboren te St.\ Petersburg, Rusland
\selectlanguage{american}
\end{center}
\end{titlepage}
\clearpage
\thispagestyle{empty}
\selectlanguage{dutch}
\subsubsection*{Promotiecommissie}
\begin{tabular}{ll}
Promotor: & prof.\ dr.\ A.J.\ Homburg\\[1ex]
Copromotor: & dr.\ T.\ Eisner\\[1ex]
Overige leden:
& prof.\ dr.\ E.\ Lesigne\\
& prof.\ dr.\ R.\ Nagel\\
& dr.\ H.\ Peters\\
& prof.\ dr.\ C.\ Thiele\\
& prof.\ dr.\ J.J.O.O.\ Wiegerinck
\end{tabular}

\vspace{1ex}\noindent
Faculteit der Natuurwetenschappen, Wiskunde en Informatica
\selectlanguage{american}

\cleardoublepage
\phantomsection
\addcontentsline{toc}{chapter}{Contents}
\tableofcontents
\chapter{Introduction}
Furstenberg's groundbreaking ergodic theoretic proof \cite{MR0498471} of Szemer\'edi's theorem on arithmetic progressions in dense subsets of integers suggested at least two possible directions for generalization.
One is connected with earlier work of Furstenberg and consists in investigating actions of groups other than $\Z$.
The other looks at polynomial, rather than linear, sequences.
Indeed, in the same article Furstenberg \cite[Theorem 1.2]{MR0498471} gave a short qualitative proof of Sárközy's theorem \cite{MR0466059} on squares in difference sets.

Furstenberg's proof of the multiple recurrence theorem involves three main steps: a structure theorem for measure-preserving systems that exhibits dichotomy between (relative) almost periodicity and (relative) weak mixing, a coloring argument that deals with the almost periodic part of the structure, and a multiple weak mixing argument dealing with the weakly mixing part.
It was the structure theorem whose generalization to $\Z^{d}$-actions required most of the additional work in the multiple recurrence theorem for commuting transformations due to Furstenberg and Katznelson \cite{MR531279}.
The coloring argument carried through using Gallai's multidimensional version of van der Waerden's theorem.
Also the multiple weak mixing argument worked similarly to the case of $\Z$-actions, namely by induction on the number of terms in the multiple ergodic average.

However, in the polynomial case, induction on the number of terms does not seem to work.
This difficulty has been resolved by Bergelson \cite{MR912373} who has found an appropriate induction scheme, called \emph{PET induction}.
Later, jointly with Leibman \cite{MR1325795}, he completed his program by proving a polynomial multiple recurrence theorem by Furstenberg's method, using a polynomial van der Waerden theorem as the main new ingredient.

This is when polynomials in nilpotent groups appeared on the stage.
Both the coloring and the multiple weak mixing steps in Furstenberg's framework involve polynomial maps and PET induction when carried out for nilpotent groups, even if one is ultimately interested in linear sequences, see \cite{MR1650102}.
A similar phenomenon occurs in Walsh's recent proof of norm convergence of nilpotent multiple ergodic averages \cite{MR2912715} (which we extended to arbitrary amenable groups in \cite{arxiv:1111.7292}).
Thus polynomials in nilpotent groups, with Leibman's axiomatization \cite{MR1910931}, seem indispensable for understanding multiple recurrence for nilpotent group actions.

While the work mentioned above concerns Ces\`aro averages of multicorrelation sequences, there are by now at least two alternative approaches to the study of asymptotic behavior of dynamical systems: using IP-limits or using limits along idempotent ultrafilters.
It is the former approach on which we concentrate.
The proof of the IP multiple recurrence theorem due to Furstenberg and Katznelson \cite{MR833409} parallels Furstenberg's earlier averaging arguments, although rigidity replaces almost periodicity, mild mixing weak mixing, and the Hales-Jewett theorem the van der Waerden theorem.
A direct continuation of their work in the polynomial direction has been carried out by Bergelson and McCutcheon \cite{MR1692634}, mixing the techniques outlined above and a polynomial extension of the Hales-Jewett theorem proved earlier by Bergelson and Leibman \cite{MR1715320}.

One of our main results is a nilpotent extension of the IP multiple recurrence theorem \cite{arxiv:1206.0287}.
We take a somewhat novel approach to the structure theorem, obtaining dichotomy between compactness and mixing not on the level of the acting group, but on the level of the group of polynomials with values in the acting group.
We have also found it necessary to prove a new coloring result, sharpening the nilpotent Hales-Jewett theorem due to Bergelson and Leibman \cite{MR1972634}.
On the other hand, the mixing part is handled in essentially the usual way using PET induction.
As remarked earlier, this method compels us to deal with polynomial mappings.
Our arguments heavily rely on an efficient axiomatization of IP-polynomials along the lines of Leibman's work, but incorporating some more recent ideas.

Another reason to study polynomial, rather than linear, sequences in nilpotent groups comes from quantitative equidistribution theory on nilmanifolds (that is, compact homogeneous spaces of nilpotent Lie groups).
Nilmanifolds play an important role in the theory of multiple ergodic averages \cite{MR2150389}, where one is interested in linear orbits of the form $(g^{n}x)$, where $x$ is a point in the nilmanifold and $g$ an element of the structure group.
An obstacle for establishing results that are uniform in all such orbits is the fact that $g$ is drawn from the possibly non-compact structure group.
This can be circumvented by considering polynomial orbits, since every linear orbit on a nilmanifold can be represented as a polynomial orbit with coefficients that come from fixed compact subsets of the structure group.
This is an important ingredient in the proof of the quantitative equidistribution theorem of Green and Tao \cite{MR2877065}.

We review this circle of ideas in order to motivate our proof of the uniform extension of the Wiener-Wintner theorem for nilsequences.
This is a result about universally good weights for the pointwise ergodic theorem, that is, sequences $(a_{n})$ such that for every ergodic measure-preserving system $(X,T)$ and every $f\in L^{\infty}(X)$ the averages
\[
\aveN a_{n} T^{n}f
\]
converge as $N\to\infty$ pointwise almost everywhere.
The Wiener-Wintner theorem for nilsequences \cite[Theorem 2.22]{MR2544760} states that nilsequences are universally good weights, the full measure set on which convergence holds being independent of the nilsequence.
We show that convergence in this result is in fact uniform over a class of nilsequences of bounded complexity provided that $f$ is orthogonal to a certain nilfactor of $(X,T)$ (this is joint work with T.~Eisner \cite{arxiv:1208.3977}).
The explicit description of a full measure set on which the above averages converge also allows us to deduce a version of the Wiener-Wintner theorem for non-ergodic systems (note that an appeal to the ergodic decomposition does not suffice for this purpose).

An opposite extreme to nilsequences in the class of universally good weights for the pointwise ergodic theorem are the random weights provided by Bourgain's return times theorem \cite{MR1557098}.
This result has been generalized to certain multiple ergodic averages by Rudolph \cite{MR1489899} using the machinery of joinings.
In a different direction, a Wiener-Wintner type extension of the return times theorem has been obtained by Assani, Lesigne, and Rudolph \cite{MR1357765} using the Conze-Lesigne algebra.
This suggested to attack the multiple term return times theorem using Host-Kra structure theory, which we do in the last chapter.
This leads us to a joint extension \cite{arxiv:1210.5202} of all aforementioned weighted pointwise ergodic theorems, in which we also identify characteristic factors.
The proof involves a version of Bourgain's orthogonality criterion valid for arbitrary tempered F\o{}lner sequences \cite{arxiv:1301.1884}.

\section*{Acknowledgments}
\addcontentsline{toc}{section}{Acknowledgments}
I owe a large part of my mathematical upbringing to Rainer Nagel and Fulvio Ricci and take this opportunity to express my gratitude.
In connection with this thesis, I thank Vitaly Bergelson, Tanja Eisner, and Nikos Frantzikinakis for asking questions that initiated the research reflected here and for patient explanations that provided me with helpful tools.

This thesis was partially funded by The Netherlands Organisation for Scientific Research
\selectlanguage{dutch}(Nederlandse Organisatie voor Wetenschappelijk Onderzoek).
\selectlanguage{american}

\section*{Summary}
\addcontentsline{toc}{section}{Summary}
In Chapter~\ref{chap:groups} we extend Leibman's theory of polynomials in nilpotent groups \cite{MR1910931} to IP-polynomials.

In Chapter~\ref{chap:mean} we extend Walsh's nilpotent multiple mean ergodic theorem \cite{MR2912715} to polynomial actions of arbitrary amenable groups.

In Chapter~\ref{chap:recurrence} we sharpen the topological nilpotent IP multiple recurrence theorem of Bergelson and Leibman \cite{MR1972634}, prove a nilpotent analog of an IP polynomial ergodic theorem of Bergelson, H\aa{}land Knutson, and McCutcheon \cite{MR2246589}, and use these results to prove a joint extension of the IP polynomial multiple recurrence theorem of Bergelson and McCutcheon \cites{MR1692634,MR2151599} and Leibman's nilpotent polynomial multiple recurrence theorem \cite{MR1650102}.

In Chapter~\ref{chap:fourier} we review the proof of Leibman's orbit closure theorem \cite{MR2122919} due to Green and Tao \cite{MR2877065} and use their ideas to prove a uniform extension of the Wiener-Wintner theorem for nilsequences \cite{MR2544760}.

In Chapter~\ref{chap:RTT} we extend Bourgain's return times theorem \cite{MR1557098} to arbitrary locally compact second countable amenable groups and prove a joint extension of Rudolph's multiple term return times theorem \cite{MR1489899}, the Wiener-Wintner return times theorem \cite{MR1357765}, and the Wiener-Wintner theorem for nilsequences.

\selectlanguage{dutch}
\section*{Samenvatting}
\addcontentsline{toc}{section}{Samenvatting}
Furstenbergs baanbrekende ergodisch-theoretische bewijsvoering van Szemer\'edi's stelling over rekenkundige rijen in grote deelverzamelingen van gehele getallen suggereert tenminste twee mogelijkheden tot generalisatie.
E\'en daarvan hangt samen met eerder werk van Furstenberg en beschouwt werkingen van andere groepen dan $\Z$.
De andere richting beschouwt polynomiale, in plaats van lineaire, rijen.

Het bewijs van Furstenberg kent drie stappen: een structuurstelling voor maatbewarende afbeeldingen die een dichotomie tussen bijna-periodiciteit en zwak-mixing geeft, een kleuring-argument voor bijna-periodiciteit, en een meervoudig-zwak-mixing-argument dat zwak-mixing behandelt.
Het was de structuurstelling waarvoor de generalisatie naar $\Z^d$-acties het meeste extra werk vergde in de meervoudige terugkeerstelling voor commuterende transformaties.
Het kleuring-argument werd gegeneraliseerd met Gallai's meerdimensionale versie van van der Waerden's stelling.
Het meervoudig-zwak-mixing argument werkt analoog voor het geval van $\Z$-acties, namelijk met een inductie op het aantal termen in de meervoudige ergodische gemiddelden.

Voor polynomiale rijen lijkt zo'n inductie op het aantal termen niet te werken. 
Bergelson vond een oplossing met een geschikt inductieschema dat PET-inductie wordt genoemd.
Samen met Leibman voltooide hij het programma voor polynomiale rijen met de methode van Furstenberg, met een polynomiale van der Waerdenstelling als nieuw ingredi\"ent.

Deze ontwikkelingen gaven aanleiding tot de studie van polynomen in nilpotente groepen. 
Polynomen in nilpotente groepen zijn onontbeerlijk voor de studie van meervoudige terugkeerstellingen voor werkingen van nilpotente groepen.

Een alternatieve benadering voor de studie van asymptotisch gedrag van dynamische systemen maakt gebruik van IP-limieten. 
Het bewijs van de IP-meervoudige terugkeerstelling door Furstenberg en Katznelson loopt parallel
aan Furstenbergs eerdere argumenten.

De resultaten in dit proefschrift sluiten aan bij deze cirkel aan idee\"en.  
In Hoofdstuk~\ref{chap:groups} wordt er een analogon van Leibman's theorie van polynomen in nilpotente groepen voor IP-polynomen opgezet.
In Hoofdstuk~\ref{chap:mean} wordt Walsh's meervoudige ergodische stelling tot polynomiale werkingen van middelbare groepen uitgebreid.
In Hoofdstuk~\ref{chap:recurrence} wordt een gemeenschappelijke uitbreiding van zowel de IP-polynomiale als ook de nilpotente meervoudige terugkeerstelling bewezen.
In Hoofdstuk~\ref{chap:fourier} wordt een uniforme versie van de stelling van Wiener--Wintner voor nilrijen bewezen.
In Hoofdstuk~\ref{chap:RTT} wordt een versie van de terugkeertijdenstelling bewezen die zowel de meervoudige, de Wiener--Wintner, als ook de nilrij-uitbreiding omvat.
\selectlanguage{american}

\mainmatter
\chapter{Nilpotent groups}
\label{chap:groups}
\section{General facts}
Here we present in a self-contained way everything that we will need to know about nilpotent groups.
\subsection{Commutators and filtrations}
We use the convention $[a,b]=a\inv b\inv ab$ for commutators and $a^{b}=b\inv a b$ for conjugation.
The following identities, which hold in arbitrary groups and are due to Hall~\cite{0007.29102} (see also \cite[p.~107]{MR0088496}), are fundamental for dealing with commutators efficiently.
\begin{equation}
\label{eq:a-bc}
[a,bc]=[a,c][a,b][[a,b],c]
\end{equation}
\begin{equation}
\label{eq:ab-c}
[ab,c]=[a,c][[a,c],b][b,c]
\end{equation}
\begin{equation}
\label{eq:group-jacobi-identity}
[[a,b],c^{a}] [[c,a],b^{c}] [[b,c],a^{b}] = \id
\end{equation}
Note also for future use the identity
\begin{align}
[xy,uv]
&=
[x,u] [x,v] [[x,v],[x,u]] [[x,u],v] \notag\\
&\quad\cdot [[x,v] [x,u] [[x,u],v],y] \label{eq:ab-cd}\\
&\quad\cdot [y,v] [y,u] [[y,u],v]. \notag
\end{align}
Given subsets $A,B$ of a group $G$ we denote by $[A,B]$ the subgroup generated by the elements $[a,b]$, where $a\in A$, $b\in B$, by $AB$ the set of elements of the form $ab$, $a\in A$, $b\in B$, and by $\<A\>$ the subgroup generated by the elements of $A$.
The subgroup relation is denoted by ``$\subg$'' and the normal subgroup relation by ``$\nsubg$''.
Note that if $A,B \subg G$ are subgroups and one of them is normal, then we have $AB=\<AB\>$.
\begin{theorem}[see e.g.\ {\cite[Theorem 5.2]{MR0207802}}]
\label{thm:group-jacobi}
Let $G$ be a group and $A,B,C \nsubg G$ be normal subgroups.
Then
\[
[[A,B],C] \subg [[C,A],B] [[B,C],A].
\]
\end{theorem}
\begin{proof}
Notice the following version of \eqref{eq:ab-c}:
\begin{equation}
\label{eq:ab-c-2}
[ab,c] = [a^{b},c^{b}] [b,c]
\end{equation}
By this identity it suffices to show that for every $a\in A$, $b\in B$, and $c\in C$ the commutator $[[a,b],c]$ is contained in the group on the right.
Since $C^{a} = C$, this follows from \eqref{eq:group-jacobi-identity}.
\end{proof}
\begin{definition}
Let $G$ be a group.
The \emph{lower central series} of $G$ is the sequence of subgroups $G_{i}$, $i\in\N$, defined by $G_{0}=G_{1}:=G$ and $G_{i+1}:=[G_{i},G]$ for $i\geq 1$.
The group $G$ is called \emph{nilpotent (of nilpotency class $d$)} if $G_{d+1}=\{\id\}$.

A \emph{prefiltration} $\Gb$ is a sequence of nested groups
\index{prefiltration}
\begin{equation}
\label{eq:prefiltration}
G_{0}\geq G_{1} \geq G_{2} \geq \dots
\quad\text{such that}\quad
[G_{i},G_{j}]\subset G_{i+j}
\quad\text{for any }i,j\in\N.
\end{equation}
A \emph{filtration} (on a group $G$) is a prefiltration in which $G_{0}=G_{1}$ (and $G_{0}=G$).
\index{filtration}
\end{definition}
We will frequently write $G$ instead of $G_{0}$.
Conversely, most groups $G$ that we consider are endowed with a prefiltration $\Gb$ such that $G_{0}=G$.
A group may admit several prefiltrations, and we usually fix one of them even if we do not refer to it explicitly.

A prefiltration is said to have \emph{length} $d\in\N$ if $G_{d+1}$ is the trivial group and length $-\infty$ if $G_{0}$ is the trivial group.
Arithmetic for lengths is defined in the same way as conventionally done for degrees of polynomials, i.e.\ $d-t=-\infty$ if $d<t$.

\begin{lemma}[see e.g.~{\cite[Theorem 5.3]{MR0207802}}]
\label{lem:lcs-filtration}
Let $G$ be a group.
Then the lower central series $\Gb$ is a filtration.
\end{lemma}
\begin{proof}
The fact that
\[
[G_{0},G_{i}] = [G_{i},G_{0}] \subset G_{i}
\]
is equivalent to $G_{i}$ being normal in $G$, and this is quickly established by induction on $i$.
This also shows that $G_{i+1} \subseteq G_{i}$ for all $i$.

It remains to show that
\[
[G_{i},G_{j}] \subseteq G_{i+j}
\quad \text{for } i,j\geq 1.
\]
To this end use induction on $j$.
For $j=1$ this follows by definition of $G_{i+1}$, so suppose that the above statement is known for $j$.
Then we have
\begin{multline*}
[G_{i},G_{j+1}]
= [G_{i},[G_{j},G_{1}]]
\subset [[G_{1},G_{i}],G_{j}] [[G_{j},G_{i}],G_{i}]\\
= [G_{i+1},G_{j}] [G_{i+j},G_{1}]
\subset G_{i+1+j}
\end{multline*}
by Theorem~\ref{thm:group-jacobi} and two applications of the inductive hypothesis.
\end{proof}
Let $G$ be a group.
A \emph{simple $n$-fold commutator (on $G$)} is an element of the form
\index{simple $n$-fold commutator}
\[
[[\dots [[g_{1},g_{2}],g_{3}]\dots,g_{n-1}],g_{n}],
\quad g_{1},\dots,g_{n}\in G.
\]
For brevity, we denote simple $n$-fold commutators by $[g_{1},\dots,g_{n}]$.
\begin{lemma}[see {\cite[Problem 5.3.3]{MR0207802}}]
\label{lem:gen-by-simple-comm}
Let $G$ be a group and $\Gb$ the lower central series on $G$.
Then, for every $n\geq 1$, the group $G_{n}$ is generated by the simple $n$-fold commutators on $G$.
\end{lemma}
\begin{proof}
Use induction on $n$.
For $n=1$ the conclusion is trivial, so suppose that the conclusion is known for $n$.

The group $G_{n+1}$ is generated by commutators of the form $[a,b]$ with $a\in G_{n}$ and $b\in G$.
By the inductive hypothesis we have $a = \prod_{i=1}^{M} c_{i}^{\sigma_{i}}$, where $c_{i}$ are simple $n$-fold commutators on $G$ and $\sigma_{i}\in\{\pm 1\}$.

Using \eqref{eq:ab-c-2} and induction on $M$ we see that $[a,b]$ can be written as a product of elements of the form $[c^{\sigma},b]$, where $c$ is a simple $n$-fold commutator and $\sigma\in\{\pm 1\}$.
The commutator $[c,b]$ is clearly a simple $n+1$-fold iterated commutator, and $[c\inv,b] = [c,b^{c}]\inv$ is the inverse of a simple $n+1$-fold iterated commutator.
\end{proof}

\begin{lemma}[{\cite[Lemma 2.6]{MR2122920}}]
\label{lem:abelianization-generates}
Let $G$ be a nilpotent group and $H\leq G$ a subgroup such that $H[G,G]=G$.
Then $H=G$.
\end{lemma}
\begin{proof}
Use induction on the nilpotency class $d$ of $G$.
If $d=1$, then $[G,G]=\{\id\}$ and the conclusion holds trivially.
In the inductive step apply the induction hypothesis to $HG_{d}/G_{d}\leq G/G_{d}$.
This yields $HG_{d}=G$.
By Lemma~\ref{lem:gen-by-simple-comm}, the group $G_{d}$ is generated by the simple $d$-fold commutators on $G$.
Since $G_{d}$ is central and $G=HG_{d}$, every such commutator equals a simple $d$-fold commutator on $H$, so that $G_{d}\leq H$, and the conclusion follows.
\end{proof}

\subsection{Commensurable subgroups}
On filtered groups, simple iterated commutators behave like multilinear forms modulo higher order error terms.
\begin{lemma}
\label{lem:simple-comm-of-prod}
Let $G$ be a group and $\Gb$ be the lower central series.
Then we have
\[
[\prod_{i=1}^{m_{1}}g_{1,i},\dots,\prod_{i=1}^{m_{n}}g_{n,i}]
\equiv
\prod_{i_{1}=1}^{m_{1}}\dots \prod_{i_{n}=1}^{m_{n}}[g_{1,i_{1}},\dots,g_{n,i_{n}}] \mod G_{l_{1}+\dots+l_{n}+1}
\]
for any $n\geq 1$, $l_{k}\in\N$, $m_{k}\in\N$, and $g_{k,i}\in G_{l_{k}}$.
\end{lemma}
\begin{proof}
The case $n=1$ is trivial, and the case $n=2$ follows by induction on $m_{1}$ and $m_{2}$ using \eqref{eq:ab-c}, \eqref{eq:a-bc}, and Lemma~\ref{lem:lcs-filtration}.

Assume that the conclusion holds for $n=2$ and for some other value of $n$, then the conclusion for $n+1$ follows since
\[
[\prod_{i=1}^{m_{1}}g_{1,i},\dots,\prod_{i=1}^{m_{n+1}}g_{n+1,i}]
=
[\prod_{i_{1}=1}^{m_{1}}\dots \prod_{i_{n}=1}^{m_{n}}[g_{1,i_{1}},\dots,g_{n,i_{n}}] c, \prod_{i=1}^{m_{n+1}}g_{n+1,i}]
\]
for some $c\in G_{l_{1}+\dots+l_{n}+1}$ by the induction hypothesis, by \eqref{eq:ab-c}, and using the case $n=2$.
\end{proof}

We will use the above result to obtain some useful facts about finite index subgroups of nilpotent groups.
\begin{definition}
Let $G$ be a group and $H\subset G$.
We write
\[
\roots[r]{H} := \{g\in G : g^{r}\in H\}
\quad\text{and}\quad
\roots{H} := \bigcup_{r\in\N_{>0}} \roots[r]{H}.
\]
\end{definition}
The set $\roots{H}$ is called the \emph{closure of $H$} in \cite{MR1881925}.

Clearly, $G\subset\roots{H}$ is a necessary local condition for $H$ to be a finite index subgroup of $G$.
More interestingly, this condition is also locally sufficient, as will follow from the next result.
\begin{lemma}
Let $G$ be a nilpotent group with a finite generating set $F$.
Let also $H\leq G$ be a subgroup and assume $F\subset\roots{H}$.
Then $[G:H]<\infty$.
\end{lemma}
\begin{proof}
Since $F$ is finite, we have in fact $F\subset\roots[r]{H}$ for some $r$.

We use induction on the nilpotency class $d$ of $G$.
If $d=0$, then the conclusion holds trivially.
So assume that the conclusion holds for $d-1$.

Let $\Gb$ be the lower central series of $G$.
Without loss of generality we may assume that $F$ is symmetric, that is, $F\inv = F$.
By Lemma~\ref{lem:gen-by-simple-comm}, the group $G_{d}$ is generated by the simple $d$-fold commutators on $G$.
By Lemma~\ref{lem:simple-comm-of-prod} and since $G_{d+1}$ is trivial, we may take these simple $d$-fold commutators on $F$.
By Lemma~\ref{lem:simple-comm-of-prod} again, we have
\[
[f_{1},\dots,f_{d}]^{r^{d}}
=
[f_{1}^{r},\dots,f_{d}^{r}]
\quad\text{for } f_{1},\dots,f_{d}\in F,
\]
and this is an element of $H$ by the assumption.

Since there are only finitely many simple $d$-fold commutators on $F$, and since the subgroup $G_{d}$ is central in $G$, this readily implies that $H$ has finite index in $HG_{d}$.
Hence, without loss of generality, we may replace $H$ by $HG_{d}$.
In particular, we may assume $G_{d}\subset H$.
The conclusion follows from the identity
\[
[G:H] = [G/G_{d}, H/G_{d}]
\]
and the induction hypothesis.
\end{proof}
The conclusion of this lemma need not hold for solvable groups.
Consider the semidirect product $G=\Z_{2} \ltimes \Z$ that is associated to the inversion action $\pi:\Z_{2} \curvearrowright \Z$ given by $\pi(\bar a)(b)=(-1)^{a}b$.
Consider the generating set $F=\{(\bar 1,0),(\bar 1,1)\}$.
Then $F$ consists of elements of order $2$, so the hypothesis of the lemma holds with $H$ being the trivial subgroup.
On the other hand, $[G:H]=\infty$.

Recall that two subgroups $H,H'\leq G$ are called \emph{commensurable} if $H\cap H'$ has finite index both in $H$ and $H'$.
\index{commensurable subgroups}
The \emph{commensurator} $\comm_{G}(H)$ of a subgroup $H\leq G$ is the set of all $g\in G$ such that $H$ and $gHg\inv$ are commensurable.
\index{commensurator}

\begin{corollary}
\label{cor:finite-ext}
Let $G$ be a nilpotent group and $H\leq G$ a finitely generated subgroup.
Then $\roots{H}$ is a subgroup of $G$ and $H$ has finite index in every finitely generated subgroup of $\roots{H}$ that contains $H$.
In particular, $\roots{H} \leq \comm_{G}(H)$.
\end{corollary}
\index{surgroup}
A succinct way of formulating this corollary would be that $\roots H$ is a ``locally finite index surgroup of $H$''.
Note that a property holds \emph{locally} in a group if it holds for every finitely generated subgroup.
For a subgroup $A\leq B$ there seems to be no standard name for the relation of $B$ to $A$: the words ``extension'' and ``supergroup'' are reserved for other purposes.
We will use the word ``surgroup'' in this situation.

If $H\nsubg G$ is normal, then this result reduces to the well-known fact that the torsion elements of a nilpotent group form a subgroup.

Recall that a group is called \emph{Noetherian} if every ascending chain of subgroups is eventually constant.
\index{Noetherian group}
It is well-known that if $K,Q$ are Noetherian groups and we have a short exact sequence
\[
0 \to K \to G \to Q \to 0,
\]
then $G$ is also a Noetherian group.
From this it follows that every finitely generated nilpotent group is Noetherian.
This can be seen by induction on the nilpotency class $l$.
For commutative groups this follows from the structure theorem for finitely generated abelian groups.
Assume that the conclusion is known for groups of nilpotency class $l-1$.
By Lemma~\ref{lem:gen-by-simple-comm} and Lemma~\ref{lem:simple-comm-of-prod}, the commutative group $G_{l}$ is finitely generated, and we can apply the induction hypothesis in the short exact sequence $0\to G_{l} \to G \to G/G_{l} \to 0$.

\subsection{Hirsch length}
We use Hirsch length of a group as a substitute for the concept of the rank of a free $\Z$-module.
Recall that a subnormal series in a group is called \emph{polycyclic} if the quotients of consecutive subgroups in this series are cyclic and a group is called \emph{polycyclic} if it admits a polycyclic series.
\begin{definition}
The \emph{Hirsch length} $h(G)$ of a polycyclic group $G$ is the number of infinite quotients of consecutive subgroups in a polycyclic series of $G$.
\index{Hirsch length}
\end{definition}
Recall that the Hirsch length is well-defined by the Schreier refinement theorem, see e.g.\ \cite[Theorem 5.11]{MR1307623}.
This is due to Hirsch \cite[Theorem 1.42]{0018.14505}.
For a finitely generated  nilpotent group $G$ with a filtration $\Gb$ one has
\[
h(G) = \sum_{i} \rank G_{i}/G_{i+1}.
\]
\begin{lemma}
\label{lem:fin-ind-hirsch}
Let $G$ be a finitely generated nilpotent group.
Then for every subgroup $V\leq G$ we have that $h(V)=h(G)$ if and only if $[G:V]<\infty$.
\end{lemma}
\begin{proof}
If $[G:V]<\infty$, then we can find a finite index subgroup $W\leq V$ that is normal in $G$, and the equality $h(G)=h(W)=h(V)$ follows from the Schreier refinement theorem.

Let now $\Gb$ be the lower central series of $G$.
Let $V\leq G$ be a subgroup with $h(V)=h(G)$ and assume in addition that $G_{i}\leq V$ for some $i=1,\dots,d+1$.
We show that $[G:V]<\infty$ by induction on $i$.
For $i=1$ the claim is trivial and for $i=d+1$ it provides the desired equivalence.

Assume that the claim holds for some $i$.
Let $V_{i} := V\cap G_{i}$ be the filtration on $V$ induced by $\Gb$ and assume $V_{i+1}=G_{i+1}$.
By the assumption we have
\[
\sum_{j=1}^{d} \rank G_{j}/G_{j+1} = h(G) = h(V) = \sum_{j=1}^{d} \rank V_{j}/V_{j+1},
\]
and, since $V_{j}/V_{j+1}\cong V_{j}G_{j+1}/G_{j+1} \leq G_{j}/G_{j+1}$ for every $j$, this implies that $V_{i}/G_{i+1} \leq G_{i}/G_{i+1}$ is a finite index subgroup.
Let $K\subset G_{i}$ be a finite set such that $KV_{i}/G_{i+1}=G_{i}/G_{i+1}$.
Then $KV\leq G$ is a subgroup and a finite index surgroup of $V$.
Moreover, we have $KV\supseteq G_{i}$, and by the first part of the lemma we obtain $h(KV)=h(V)$.

By the induction hypothesis $KV$ has finite index in $G$, so the index of $V$ is also finite.
\end{proof}

\begin{lemma}
\label{lem:fin-ext}
Let $G$ be a finitely generated nilpotent group with a filtration $\Gb$ of length $d$ and let $V\leq G$ a subgroup.
Then for every $j=1,\dots,d+1$ and every $g\in G$ there exist at most finitely many finite index surgroups of $V$ of the form $\<V,gc\>$ with $c\in G_{j}$.
\end{lemma}
\begin{proof}
We use descending induction on $j$.
The case $j=d+1$ is clear, so assume that the conclusion is known for $j+1$ and consider some $g\in G$.

Let $c_{a}$, $a=0,1$ be elements of $G_{j}$ such that $\<V,gc_{a}\>$ are finite index surgroups of $V$.
Then also $\<VG_{j+1},gc_{a}\>/G_{j+1}$ is a finite index surgroup of $VG_{j+1}/G_{j+1}$, so that there exists an $m>0$ such that $(gc_{a}G_{j+1})^{m}\in VG_{j+1}/G_{j+1}$ for $a=0,1$.

Since the elements $c_{a}G_{j+1}$ are central in $G/G_{j+1}$ this implies $(c_{0}\inv c_{1})^{m}G_{j+1}\in (VG_{j+1}\cap G_{j})/G_{j+1}$.
But the latter group is a subgroup of the finitely generated abelian group $G_{j}/G_{j+1}$, so that $c_{0}\inv c_{1} \in K(VG_{j+1}\cap G_{j})$ for some finite set $K\subset G_{j}$ that does not depend on $c_{0},c_{1}$.

Multiplying $c_{1}$ with an element of $V$ we may assume that $c_{1} \in c_{0}KG_{j+1}$.
By the induction hypothesis for each $g'\in gc_{0}K$ there exist at most finitely many finite index surgroups of the form $\<V,g'c'\>$ with $c'\in G_{j+1}$, so we have only finitely many surgroups of the form $\<V,gc_{1}\>$ as required.
\end{proof}
\begin{corollary}
\label{cor:fin-ext}
Let $G$ be a finitely generated nilpotent group and $V$ be a subgroup.
Then there exist at most finitely many finite index surgroups of $V$ of the form $\<V,c\>$.
\end{corollary}
\begin{proof}
Consider any filtration $\Gb$ and apply Lemma~\ref{lem:fin-ext} with $j=1$ and $g=1_{G}$.
\end{proof}
The following example shows that Corollary~\ref{cor:fin-ext} may fail for virtually nilpotent groups.
Consider the semidirect product $G=\Z_{2} \ltimes \Z$ that is associated to the inversion action $\pi:\Z_{2} \curvearrowright \Z$ given by $\pi(\bar a)(b)=(-1)^{a}b$.
Then $G_{2}=[G,G]=2\Z$ is an abelian subgroup of index $4$ and $G_{i+1}=[G,G_{i}]=2^{i}\Z$ for all $i\in\N$, so $G$ is not nilpotent.
Let $V = \{0\} \leq G$ be the trivial subgroup.
Since we have $(\bar{1}a)^{2}=0 \in V$ for any $a\in\Z$, each group of the form $\<V,\bar{1}a\>$ is a surgroup of $V$ with index $2$.
On the other hand, for every value of $a$ we obtain a different surgroup.

\section{Polynomial mappings}
\label{sec:poly}
In this section we set up the algebraic framework for dealing with polynomials with values in a nilpotent group.
We begin with a generalization of Leibman's result that polynomial mappings into a nilpotent group form a group under pointwise operations \cite[Proposition 3.7]{MR1910931}.
Following an idea from the proof of that result by Green and Tao \cite[Proposition 6.5]{MR2877065}, we encode the information that is contained in Leibman's notion of vector degree in a prefiltration indexed by $\N=\{0,1,\dots\}$ (see \cite[Appendix B]{MR2950773} for related results regarding prefiltrations indexed by more general partially ordered semigroups).
The treatment below first appeared in \cite{arxiv:1206.0287}.

Let $\Gb$ be a prefiltration of length $d$ and let $t\in\N$ be arbitrary.
We denote by $\Gb[+t]$ the prefiltration of length $d-t$ given by $(\Gb[+t])_{i}=G_{i+t}$ and by $\Gb[/t]$ the prefiltration of length $\min(d,t-1)$ given by $G_{i/t}=G_{i}/G_{t}$ (this is understood to be the trivial group for $i\geq t$; note that $G_{t}$ is normal in each $G_{i}$ for $i\leq t$ by \eqref{eq:prefiltration}).
These two operations on prefiltrations can be combined: we denote by $\Gb[/t+s]$ the prefiltration given by $G_{i/t+s}=G_{i+s}/G_{t}$, it can be obtained applying first the operation $/t$ and then the operation $+s$ (hence the notation).

If $\Gb$ is a prefiltration and $\bar d = (d_{i})_{i\in\N} \subset\N$ is a superadditive sequence (i.e.\ $d_{i+j}\geq d_{i}+d_{j}$ for all $i,j\in\N$; by convention $d_{-1}=-\infty$) then $\Gb^{\bar d}$, defined by
\begin{equation}
\label{eq:Gbard}
G^{\bar d}_{i} = G_{j} \quad \text{whenever} \quad d_{j-1} < i \leq d_{j},
\end{equation}
is again a prefiltration.

We define $\Gb$-polynomial maps by induction on the length of the prefiltration.
\begin{definition}
\label{def:polynomial}
Let $\PolyDomain$ be any set and $\mathcal{T}$ be a set of partially defined maps $T:\PolyDomain \supset \dom(T) \to \PolyDomain$.
Let $\Gb$ be a prefiltration of length $d\in \{-\infty\}\cup\N$.
A map $g\from\PolyDomain\to G_{0}$ is called \emph{$\Gb$-polynomial} (with respect to $\mathcal{T}$) if either $d=-\infty$ (so that $g$ identically equals the identity) or for every $T\in\mathcal{T}$ there exists a $\Gb[+1]$-polynomial map $D_{T}g$ such that
\index{polynomial map}
\begin{equation}
D_{T}g = g\inv Tg := g\inv (g\circ T)
\quad\text{on}\quad
\dom T.
\end{equation}
We write $\poly$ for the set of $\Gb$-polynomial maps, usually suppressing any reference to the set of maps $\mathcal{T}$.
\end{definition}
Informally, a map $g:\PolyDomain\to G_{0}$ is polynomial if every discrete derivative $D_{T}g$ is polynomial ``of lower degree'' (the ``degree'' of a $\Gb$-polynomial map would be the length of the prefiltration $\Gb$, but we prefer not to use this notion since it is necessary to keep track of the prefiltration $\Gb$ anyway).
The connection with Leibman's notion of vector degree is provided by \eqref{eq:Gbard}: a map has vector degree $\bar d$ with respect to a prefiltration $\Gb$ if and only if it is $\Gb^{\bar d}$-polynomial.

Note that if a map $g$ is $\Gb$-polynomial then the map $gG_{t}$ is $\Gb[/t]$-polynomial for any $t\in\N$ (but not conversely).
We abuse the notation by saying that $g$ is $\Gb[/t]$-polynomial if $gG_{t}$ is $\Gb[/t]$-polynomial.
In assertions that hold for all $T\in\mathcal{T}$ we omit the subscript in $D_{T}$.

The next theorem is the basic result about $\Gb$-polynomials.
\begin{theorem}
\label{thm:poly-group}
For every prefiltration $\Gb$ of length $d\in\{-\infty\}\cup\N$ the following holds.
\begin{enumerate}
\item\label{poly-group:commutator}
Let $t_{i}\in\N$ and $g_{i} \from \PolyDomain \to G$ be maps such that $g_{i}$ is $\Gb[/(d+1-t_{1-i})+t_{i}]$-polynomial for $i=0,1$.
Then the commutator $[g_{0},g_{1}]$ is $\Gb[+t_{0}+t_{1}]$-polynomial.
\item\label{poly-group:product}
Let $g_{0},g_{1} \from \PolyDomain \to G$ be $\Gb$-polynomial maps.
Then the product $g_{0}g_{1}$ is also $\Gb$-polynomial.
\item\label{poly-group:inverse}
Let $g \from \PolyDomain \to G$ be a $\Gb$-polynomial map.
Then its pointwise inverse $g\inv$ is also $\Gb$-polynomial.
\end{enumerate}
\end{theorem}
\begin{proof}
We use induction on $d$.
If $d=-\infty$, then the group $G_{0}$ is trivial and the conclusion hold trivially.
Let $d\geq 0$ and assume that the conclusion holds for all smaller values of $d$.

We prove part \eqref{poly-group:commutator} using descending induction on $t=t_{0}+t_{1}$.
We clearly have $[g_{0},g_{1}]\subset G_{t}$.
If $t\geq d+1$, there is nothing left to show.
Otherwise it remains to show that $D[g_{0},g_{1}]$ is $\Gb[+t+1]$-polynomial.
To this end we use the commutator identity
\begin{multline}
\label{eq:derivative-of-commutator}
D[g_0,g_1]
=
[g_0,D g_1] \cdot [[g_0,D g_1],[g_0,g_1]] \\
\cdot [[g_0,g_1], Dg_1]
\cdot [[g_{0},g_{1}Dg_{1}],Dg_{0}] \cdot [Dg_{0},g_{1}Dg_{1}].
\end{multline}
We will show that the second to last term is $\Gb[+t+1]$-polynomial, the argument for the other terms is similar.
Note that $Dg_{0}$ is $\Gb[/(d+1-t_{1})+t_{0}+1]$-polynomial.
By the inner induction hypothesis it suffices to show that $[g_{0},g_{1}Dg_{1}]$ is $\Gb[/(d-t_{0})+t_{1}]$-polynomial.
But the prefiltration $\Gb[/(d-t_{0})]$ has smaller length than $\Gb$, and by the outer induction hypothesis we can conclude that $g_{1}Dg_{1}$ is $\Gb[/(d-t_{0})+t_{1}]$-polynomial.
Moreover, $g_{0}$ is clearly $\Gb[/(d-t_{0}-t_{1})]$-polynomial, and by the outer induction hypothesis its commutator with $g_{1}Dg_{1}$ is $\Gb[/(d-t_{0})+t_{1}]$-polynomial as required.

Provided that each multiplicand in \eqref{eq:derivative-of-commutator} is $\Gb[+t+1]$-polynomial, we can conclude that $D[g_{0},g_{1}]$ is $\Gb[+t+1]$-polynomial by the outer induction hypothesis.

Part \eqref{poly-group:product} follows immediately by the Leibniz rule
\begin{equation}
\label{eq:leibniz}
D(g_{0}g_{1}) = Dg_{0} [Dg_{0},g_{1}] Dg_{1}
\end{equation}
from \eqref{poly-group:commutator} with $t_{0}=1$, $t_{1}=0$ and the induction hypothesis.

To prove part \eqref{poly-group:inverse} notice that
\begin{equation}
D (g\inv)
=
g (D g)\inv g\inv
=
[g\inv,Dg] (Dg)\inv.
\end{equation}
By the induction hypothesis the map $g\inv$ is $\Gb[/d]$-polynomial, the map $Dg$ is $\Gb[+1]$-polynomial, and the map $(Dg)\inv$ is $\Gb[+1]$-polynomial.
Thus also $D(g\inv)$ is $\Gb[+1]$-polynomial by \eqref{poly-group:commutator} and the induction hypothesis.
\end{proof}
Discarding some technical information that was necessary for the inductive proof we can write the above theorem succinctly as follows.
\begin{corollary}
\label{cor:poly-group}
Let $\Gb$ be a prefiltration of length $d$.
Then the set $\poly$ of $\Gb$-polynomials on $\PolyDomain$ is a group under pointwise operations and admits a canonical prefiltration of length $d$ given by
\[
\poly \geq \poly[{\Gb[+1]}] \geq \dots \geq \poly[{\Gb[+d+1]}].
\]
\end{corollary}
Clearly, every subgroup $F\leq\poly$ admits a canonical prefiltration $\Fb$ given by
\begin{equation}
\label{eq:F-prefiltration}
F_{i} := F \cap \poly[{\Gb[+i]}].
\end{equation}
\begin{remark}
If $\PolyDomain$ is a group, then we recover \cite[Proposition 3.7]{MR1910931} setting
\begin{equation}
\label{eq:T-right}
\mathcal{T} = \{T_{b}:n\mapsto nb,\dom(T_{b})=\PolyDomain, \text{ where } b\in\PolyDomain\}.
\end{equation}
\end{remark}
\begin{remark}
Polynomial mappings defined on $\PolyDomain=\Z$ with translation maps \eqref{eq:T-right} are called \emph{polynomial sequences}.
\index{polynomial sequence}
The polynomial sequences fail to form a group if $G$ is replaced by the dihedral group $D_{3}$ that is the smallest non-nilpotent group.
Indeed, let $\delta$ be a rotation and $\sigma$ a reflection in $D_{3}$, then $\delta^{3}=\sigma^{2}=(\sigma\delta)^{2}=1$.
The sequences $(\dots,\sigma,1,\sigma,1,\dots)$ and $(\dots,\sigma\delta,1,\sigma\delta,1,\dots)$ are polynomial (they vanish after any two discrete differentiations), but their pointwise product $(\dots,\delta,1,\delta,1,\dots)$ is not.
\end{remark}

If $\PolyDomain=\Z^{r}$ or $\PolyDomain=\R^{r}$, then examples of polynomial mappings are readily obtained considering $g(n)=T_{1}^{p_{1}(n)}\cdot\dots\cdot T_{l}^{p_{l}(n)}$, where $p_{i}\from\Z^{r}\to\Z$ (resp. $\R^{r}\to\R$) are conventional polynomials and $T_{i} : \Z \to G$ (resp. $\R\to G$) are one-parameter subgroups.

For noncommutative groups $\PolyDomain$ it is not evident that there exist any non-trivial polynomial functions $g\from\PolyDomain\to G$ to some nilpotent group.
However, group homomorphisms are always polynomial.
\begin{example}
\label{ex:homo-poly}
If $\PolyDomain$ is a group and
\begin{equation}
\label{eq:T}
\mathcal{T} = \{T_{a,b}:\PolyDomain\to\PolyDomain, n\mapsto anb,\text{ where } a,b\in\PolyDomain\},
\end{equation}
then every group homomorphism $g\from\PolyDomain\to G_{1}$ is polynomial.
In particular, every homomorphism to a nilpotent group is polynomial with respect to the lower central series.

This can be seen by induction on the length $d$ of the prefiltration $\Gb$ as follows.
If $d=-\infty$, then there is nothing to show.
Otherwise write
\begin{equation}
D_{T_{a,b}}g(n)=g(n)\inv g(anb)=[g(n),g(a)\inv] g(ab).
\end{equation}
By the induction hypothesis $gG_{d}$ is $\Gb[/d]$-polynomial and the constant maps $g(a)\inv$, $g(ab)$ are $\Gb[+1]$-polynomial since they take values in $G_{1}$.
Hence $D_{T_{a,b}}g$ is $\Gb[+1]$-polynomial by Theorem~\ref{thm:poly-group}.
\end{example}
We will encounter further concrete examples of polynomials in Proposition~\ref{prop:mono-poly} and Lemma~\ref{lem:poly-fvip}.

\section{IP-polynomials}
\renewcommand*{\PolyDomain}{\Fine}
In this section we will work with polynomials defined on the partial semigroup\footnote{A partial semigroup \cite{MR1262304} is a set $\Gamma$ together with a partially defined operation $*\from\Gamma\times\Gamma\to\Gamma$ that is associative in the sense that $(a*b)*c=a*(b*c)$ whenever both sides are defined.} $\Fine$ of finite subsets of $\N$ with the operation $\alpha*\beta=\alpha\cup\beta$ that is only defined if $\alpha$ and $\beta$ are disjoint.
\index{partial semigroup}
It is partially ordered by the relation
\[
\alpha<\beta :\iff \max\alpha < \min\beta.
\]
Note that in particular $\emptyset<\alpha$ and $\alpha<\emptyset$ for any $\alpha\in\Fine$.

The set $\mathcal{T}$ is then given by
\begin{equation}
\mathcal{T} = \{T_{\alpha}:\beta\mapsto \alpha*\beta,\dom(T_{\alpha})=\{\beta:\alpha\cap\beta=\emptyset\}, \text{ where } \alpha\in\Gamma\}.
\end{equation}
If $T=T_{\alpha}$, then we also write $D_{\alpha}$ instead of $D_{T_{\alpha}}$.
We write $\VIP \leq \poly$ for the subgroup of polynomials that vanish at $\emptyset$ and call its members \emph{VIP systems}.
\index{VIP system}
For every $g\in\VIP$ and $\beta\in\Fine$ we have
\begin{equation}
\label{eq:polyn-g1}
g(\beta) = g(\emptyset) D_{\beta}g(\emptyset) \in G_{1}.
\end{equation}
Therefore the symmetric derivative $\sD$, defined by
\begin{equation}
\label{eq:symm-der}
\sD_{\beta}g(\alpha) := D_{\beta}g(\alpha)g(\beta)\inv = g(\alpha)\inv g(\alpha\cup\beta) g(\beta)\inv,
\end{equation}
maps $\VIP$ into $\VIP[{\Gb[+1]}]$.
Moreover, $\VIP$ admits the canonical prefiltration of length $d-1$ given by
\[
\VIP \geq \VIP[{\Gb[+1]}] \geq \dots \geq \VIP[{\Gb[+d]}].
\]
There is clearly no need to keep track of values of VIP systems at $\emptyset$, so we consider them as functions on $\Fin:=\Fine\setminus\{\emptyset\}$.

The group $\VIP$ can be alternatively characterized by $\VIP=\{1_{G}\}$ for prefiltrations $\Gb$ of length $d=-\infty,0$ and
\[
g\in\VIP \iff g:\Fin\to G_{1} \text{ and } \forall\beta\in\Fin\, \sD_{\beta}g\in\VIP[{\Gb[+1]}].
\]
This characterization shows that if $G$ is an abelian group with the standard filtration $G_{0}=G_{1}=G$, $G_{2}=\{1_{G}\}$, then $\VIP$ is just the set of IP systems in $G$.

\subsection{IP-polynomials in several variables}
The inductive procedure that has been so far utilized in all polynomial extensions of Szemer\'edi's theorem inherently relies on polynomials in several variables.
We find it more convenient to define polynomials in $m$ variables not on $\Fin^{m}$, but rather on the subset $\Fin^{m}_{<} \subset \Fin^{m}$ that consists of \emph{ordered} tuples, that is,
\[
\Fin^{m}_{<} = \{ (\alpha_{1},\dots,\alpha_{m}) \in \Fin^{m} : \alpha_{1}<\dots<\alpha_{m} \}.
\]
Analogously, $\Fin^{\omega}_{<}$ is the set of infinite increasing sequences in $\Fin$.
We will frequently denote elements of $\Fin^{m}_{<}$ or $\Fin^{\omega}_{<}$ by $\vec\alpha=(\alpha_{1},\alpha_{2},\dots)$.
\begin{definition}
Let $\Gb$ be a prefiltration and $F\leq\VIP$ a subgroup.
We define the set $\PE{F}{m}$ of \emph{polynomial expressions} in $m$ variables by induction on $m$ as follows.
\index{polynomial expression}
We set $\PE{F}{0}=\{1_{G}\}$ and we let $\PE{F}{m+1}$ be the set of functions $g \from \Fin^{m+1}_{<} \to G_{0}$ such that
\[
g(\alpha_{1},\dots,\alpha_{m+1}) = W^{\alpha_{1},\dots,\alpha_{m}}(\alpha_{m+1})S(\alpha_{1},\dots,\alpha_{m}),
\]
where $S\in\PE{F}{m}$ and $W^{\alpha_{1},\dots,\alpha_{m}}\in F$ for every $\alpha_{1}<\dots<\alpha_{m}$.
\end{definition}
Note that $\PE{F}{1}=F$.
Polynomial expressions also behave well with respect to filtrations.
\begin{lemma}
Suppose that $F$ is invariant under conjugation by constant functions.
Then, for every $m$, the set $\PE{F}{m}$ is a group under pointwise operations and admits a canonical prefiltration given by $(\PE{F}{m})_{i}=\PE{(F_{i})}{m}$.

If $K\leq F$ is a subgroup that is invariant under conjugation by constant functions, then $\PE{K}{m} \leq \PE{F}{m}$ is also a subgroup.
\end{lemma}
\begin{proof}
We use induction on $m$.
For $m=0$ there is nothing to show.
Let
\[
R_{j}\in \PE{(F_{t_{j}})}{m+1}:(\alpha_{1},\dots,\alpha_{m+1})\mapsto W_{j}^{\alpha_{1},\dots,\alpha_{m}}(\alpha_{m+1})S_{j}(\alpha_{1},\dots,\alpha_{m}),
\quad j=0,1
\]
be polynomial expressions in $m+1$ variables.
Suppressing the variables $\alpha_{1},\dots,\alpha_{m}$, we have
\[
R_{0}R_{1}\inv(\alpha_{m+1})
=
W_{0}(\alpha_{m+1})
\underbrace{\left(S_{0}S_{1}\inv W_{1}\inv S_{1} S_{0}\inv\right)}_{\in F}(\alpha_{m+1})
S_{0}S_{1}\inv,
\]
so that $R_{0}R_{1}\inv \in \PE{F}{m+1}$.
Hence $\PE{F}{m+1}$ is a group.

In order to show that $\PE{(\Fb)}{m+1}$ is indeed a prefiltration we have to verify that
\begin{align*}
[R_{0},R_{1}]
&=
[W_{0}S_{0},W_{1}S_{1}]
\in (F_{t_{0}+t_{1}})^{m+1}.
\end{align*}
This follows from the identity \eqref{eq:ab-cd}.
It is clear that $\PE{K}{m}\leq \PE{F}{m}$ is a subgroup provided that both sets are groups.
\end{proof}
For every $m\in\N$ there is a canonical embedding $\PE{F}{m}\leq \PE{F}{m+1}$ that forgets the last variable.
Thus we can talk about
\[
\PE{F}{\omega}:=\injlim_{m\in\N} \PE{F}{m} = \bigcup_{m\in\N} \PE{F}{m},
\]
this is a group of maps defined on $\Fin^{\omega}_{<}$ with prefiltration $(\PE{F}{\omega})_{i}=\PE{(F_{i})}{\omega}$.

\subsection{Polynomial-valued polynomials}
It will be beneficial to consider IP-polynomials with values in a group of IP-polynomials, which should in turn have a sufficiently rich structure.
\begin{definition}
Let $\Gb$ be a filtration of length $d$.
A \emph{VIP group} is a subgroup $F\leq\VIP$ that is closed under conjugation by constant functions and under $\sD$ in the sense that for every $g\in F$ and $\alpha\in\Fin$ the symmetric derivative $\sD_{\alpha}g$ lies in $F_{1}$ (defined in \eqref{eq:F-prefiltration}).
\index{VIP group}
\end{definition}
In particular, the group $\VIP$ itself is VIP.
The definition of a VIP group is tailored to the following construction.
\begin{proposition}
\label{prop:substitution-poly}
Let $F\leq \VIP$ be a VIP group.
Then for every $g\in \PE{F}{m}$ the substitution map
\begin{equation}
\label{eq:substitution}
h:\vec\beta=(\beta_{1},\dots,\beta_{m}) \in \Fin^{m}_{<}
\mapsto
(g[\vec\beta] : \vec\alpha\in\Fin^{\omega}_{<} \mapsto g(\cup_{i\in\beta_{1}}\alpha_{i},\dots,\cup_{i\in\beta_{m}}\alpha_{i}))
\end{equation}
lies in $\PE{\VIP[\PE{F}{\omega}]}{m}$.
\end{proposition}
\begin{proof}
We proceed by induction on $m$.
In case $m=0$ there is nothing to show, so suppose that the assertion is known for $m$ and consider $g\in\PE{F}{m+1}$.
By definition we have
\[
g(\alpha_{1},\dots,\alpha_{m+1}) = W^{\alpha_{1},\dots,\alpha_{m}}(\alpha_{m+1})S(\alpha_{1},\dots,\alpha_{m})
\]
and
\[
h(\beta_{1},\dots,\beta_{m+1})(\vec\alpha) = W^{\cup_{i\in\beta_{1}}\alpha_{i},\dots,\cup_{i\in\beta_{m}}\alpha_{i}}[\beta_{m+1}](\vec\alpha) S[\beta_{1},\dots,\beta_{m}](\vec\alpha).
\]
In view of the induction hypothesis it remains to verify that the map
\[
\tilde h : \beta\mapsto(\vec\alpha \mapsto W^{\cup_{i\in\beta_{1}}\alpha_{i},\dots,\cup_{i\in\beta_{m}}\alpha_{i}}[\beta](\vec\alpha)),
\quad \beta>\beta_{m}>\dots>\beta_{1},
\]
is in $\VIP[\PE{F}{\omega}]$.
The fact that $\tilde h(\beta)\in\PE{F}{\omega}$ for all $\beta$ follows by induction on $|\beta|$ using the identity
\begin{multline*}
W^{\cup_{i\in\beta_{1}}\alpha_{i},\dots,\cup_{i\in\beta_{m}}\alpha_{i}}[\beta\cup\{b\}](\vec\alpha)
=\\
W^{\cup_{i\in\beta_{1}}\alpha_{i},\dots,\cup_{i\in\beta_{m}}\alpha_{i}}(\alpha_{b}) \sD_{\cup_{i\in\beta} \alpha_{i}}W^{\cup_{i\in\beta_{1}}\alpha_{i},\dots,\cup_{i\in\beta_{m}}\alpha_{i}}(\alpha_{b}) W^{\cup_{i\in\beta_{1}}\alpha_{i},\dots,\cup_{i\in\beta_{m}}\alpha_{i}}[\beta](\vec\alpha)
\end{multline*}
that holds whenever $b>\beta>\beta_{m}>\dots>\beta_{1}$.
In order to see that $\tilde h$ is polynomial in $\beta$ observe that
\[
\sD_{\gamma}\tilde h(\beta) : \vec\alpha \mapsto \sD_{\cup_{i\in\gamma}\alpha_{i}}W^{\cup_{i\in\beta_{1}}\alpha_{i},\dots,\cup_{i\in\beta_{m}}\alpha_{i}}(\cup_{i\in\beta}\alpha_{i}),
\quad
\beta>\gamma>\beta_{m}>\dots>\beta_{1}.
\qedhere
\]
\end{proof}

\subsection{Monomial mappings}
In this section we verify that monomial mappings into nilpotent groups in the sense of Bergelson and Leibman \cite[\textsection 1.3]{MR1972634} are polynomial in the sense of Definition~\ref{def:polynomial}.

For a sequence of finite sets $R=(R_{0},R_{1},\dots)$ only finitely many of which are non-empty and a set $\alpha$ write
\[
R[\alpha] := \alpha^{0}\times R_{0} \uplus \alpha^{1}\times R_{1} \uplus \dots
\]
Here the symbol $\uplus$ denotes disjoint union and $\alpha^{i}$ are powers of the set $\alpha$ (note that $\alpha^{0}$ consists of one element, the empty tuple).
\begin{proposition}
\label{prop:mono-poly}
Let $\Gb$ be a prefiltration of length $d$ and $N\subset \N$ any subset.
Let $\gb \from R[N] \to G$, $x\mapsto g_{x}$ be a mapping such that $\gb(N^{i}\times R_{i}) \subset G_{i}$ for every $i\in\N$ and $\prec$ be any linear ordering on $R[N]$.
Then the map
\[
g \from \Fin(N) \to G,
\quad
\alpha \mapsto \prod_{j\in R[\alpha]}^{\prec} g_{j}
\]
is $\Gb$-polynomial on the partial semigroup $\Fin(N)\subset\Fin$ that consists of finite subsets of $N$ (here the symbol $\prec$ on top of $\prod$ indicates the order of factors in the product).
\end{proposition}
\begin{proof}
We induct on the length of the prefiltration $\Gb$.
If $d=-\infty$, then there is nothing to prove.
Otherwise let $\beta\in\Fin(N)$.
We have to show that $D_{\beta}g$ is $\Gb[+1]$-polynomial.

Let $B \subset R[N]$ be a finite set and $A\subset B$.
By induction on the length of an initial segment of $A$ (that proceeds by pulling the terms $g_{j}$, $j\in A$, out of the double product one by one, leaving commutators behind) we see that
\begin{equation}
\label{eq:der}
\prod_{j\in B}^{\prec} g_{j}
= \prod_{j\in A}^{\prec} g_{j}
\prod_{j\in B\setminus A}^{\prec} \prod_{k\in A^{\leq d}}^{\succ-\mathrm{lexicographic}}g_{j,k},
\end{equation}
where $A^{\leq d}$ is the set of all tuples of elements of $A$ with at most $d$ coordinates in $N$ and
\[
g_{j,\emptyset}=g_{j},
\quad
g_{j,(k_{0},\dots,k_{i})} =
\begin{cases}
[g_{j,(k_{0},\dots,k_{i-1})},g_{k_{i}}] & \text{if } j\prec k_{0} \prec \dots \prec k_{i},\\
1 & \text{otherwise}.
\end{cases}
\]

Let $\alpha\in\Fin(N)$ be disjoint from $\beta$.
Applying \eqref{eq:der} with $A:=R[\alpha]$ and $B:=R[\alpha\cup\beta]$ we obtain
\[
D_{\beta}g(\alpha) = \prod_{j\in R[\alpha\cup\beta]\setminus R[\alpha]}^{\prec} \prod_{k\in R[\alpha]^{\leq d}}^{\succ-\mathrm{lexicographic}}g_{j,k},
\]
where $g_{j,(k_{0},\dots,k_{i})} \in G_{l+l_{0}+\dots+l_{i}}$ provided that $j\in \alpha^{l}\times R_{l}$ and $k_{0}\in \alpha^{l_{0}}\times R_{l_{0}}, \dots, k_{i}\in \alpha^{l_{i}}\times R_{l_{i}}$.

The double product can be rewritten as $\prod_{l\in S[\alpha]}^{\prec'}h_{l}$ for some sequence of finite sets $S$, an ordering $\prec'$ on $S[N']$, where $N'=N\setminus\beta$, and $\hb : S[N']\to G$.
The sequence of sets $S$ is obtained by the requirement
\[
(R[\alpha\cup\beta]\setminus R[\alpha]) \times R[\alpha]^{\leq d} = S[\alpha]
\]
for every $\alpha\subset N'$.
The lexicographic ordering on $(R[N]\setminus R[N']) \times R[N']^{\leq d}$ induces an ordering $\prec'$ on $S[N']$.
Define $h_{z}=g_{j,k}$ if $(j,k)$ corresponds to $z\in S[N']$.

By construction we have $\hb((N')^{i}\times S_{i}) \subset G_{i+1}$ since each element of $R[N]\setminus R[N']$ has at least one coordinate in $N$ but not $N'$.
Thus $D_{\beta}g$ is $\Gb[+1]$-polynomial by the induction hypothesis.
\end{proof}
\begin{corollary}
\label{cor:set-monomials}
Let $G$ be a nilpotent group with lower central series
\[
G=G_{0}=G_{1}\geq\dots\geq G_{s}\geq G_{s+1}=\{1_{G}\},
\]
let $\gb : N^{d}\to G$ be an arbitrary mapping, and let $\prec$ be any linear ordering on $N^{d}$.
Then the map
\[
g\from\Fin(N)\to G,
\quad
\alpha \mapsto \prod_{j\in \alpha^{d}}^{\prec} g_{j}
\]
is polynomial on the partial semigroup $\Fin(N)$ with respect to the filtration
\begin{equation}
\label{eq:scalar-poly-filtration-2}
G_{0} \geq
\underbrace{G_{1} \geq \dots \geq G_{1}}_{d \text{ times}} \geq \dots \geq
\underbrace{G_{s} \geq \dots \geq G_{s}}_{d \text{ times}} \geq G_{s+1}.
\end{equation}
\end{corollary}

%%% Local Variables:
%%% mode: latex
%%% TeX-master: "phd-thesis.tex"
%%% End: 
\chapter{Mean convergence}
\label{chap:mean}
\renewcommand{\PolyDomain}{\AG}
\newcommand{\indx}{\alpha}
\newcommand{\iset}{A}
The problem of mean convergence of mean ergodic averages has been recently resolved by Walsh \cite{MR2912715}.
In this chapter we discuss modifications to his arguments that provide convergence in the unifom Ces\`aro sense along F\o{}lner nets in arbitrary amenable groups.
This material appeared in \cite{arxiv:1111.7292}.
Let us start by recalling the relevant definitions.

We denote by $\AG$ a locally compact (not necessarily second countable) amenable group with a left Haar measure $|\cdot|$.
We fix a probability space $X$ and a filtered nilpotent group $\Gb$ of unitary operators on $L^{2}(X)$ that act as isometric algebra homomorphisms on $L^{\infty}(X)$ (thus the group $G$ comes from a group of measure-preserving transformations).
Polynomial maps $\AG\to G$ are defined with respect to the translation maps \eqref{eq:T}.
\begin{definition}
\label{def:folner-net}
A net $(\Fo_{\alpha})_{\alpha\in A}$ of nonempty compact subsets of $\AG$ is called a \emph{Følner net} if
\index{F\o{}lner net}
for every compact set $K \subset \AG$ one has
\[
\lim_{\alpha} \sup_{l\in K} |l\Fo_{\alpha} \Delta \Fo_{\alpha}|/|\Fo_{\alpha}| = 0.
\]
\end{definition}
Note that $(\Fo_{\alpha}b_{\alpha})_{\alpha\in A}$ is a Følner net for any shifts $b_{\alpha}\in\AG$ whenever $(\Fo_{\alpha})_{\alpha}$ is Følner.
It is well-known that every amenable group admits a F\o{}lner net, which can be chosen to be a sequence if the group is $\sigma$-compact \cite[Theorem 4.16]{MR961261}.

We use the expectation notation $\E_{n\in I} f(n) = \frac{1}{|I|} \int_{n\in I} f(n)$ for finite measure subsets $I \subset\AG$, where the integral is taken with respect to the left Haar measure.
The convergence theorem takes the following form.
\begin{theorem}
\label{thm:polynomial-norm-convergence}
Let $g_{1}, \dots, g_{j} \in\poly$ be measurable and $f_{0},\dots,f_{j} \in L^{\infty}(X)$ be arbitrary bounded functions.
Then for every Følner net $(\Fo_{\indx})_{\indx\in A}$ in $\AG$ and any choice of $(a_{\indx})_{\indx\in A} \subset \AG$ the limit
\begin{equation}
\label{eq:polynomial-average}
\lim_{\indx} \ave{m}{a_{\indx}\Fo_{\indx}} f_{0} g_{1}(m) f_{1} \cdot\dots\cdot g_{j}(m) f_{j}
\end{equation}
exists in $L^{2}(X)$ and is independent of the Følner net $(\Fo_{\indx})_{\indx\in A}$ and the shifts $(a_{\indx})_{\indx\in A}$.
\end{theorem}
In view of Example~\ref{ex:homo-poly} this result applies for instance if $\AG=G$ is the discrete Heisenberg group and $g_{k}(m)=m^{k}$.

We have to address two additional issues not arising in the discrete setting $\AG=\Z^{r}$.
The first is that the family of sets $\{a \Fo_{\indx}b\}_{a,b\in\AG,\indx\in A}$ need not be directed by inclusion.
However, it is directed by \emph{approximate} (up to a small proportion) inclusion.
This turns out to be sufficient for our purposes and gives uniform convergence over two-sided shifts $a\Fo_\indx b$ as a byproduct.

The second issue is that in general a function from a directed set to itself, unlike a sequence of natural numbers, cannot be majorized by a monotone function.
Thus we have to avoid to pass to monotone functions.

We remark that Theorem~\ref{thm:polynomial-norm-convergence} provides convergence of the averages in \cite[Theorem 1.2]{arXiv:1105.5612} on the joining (and not only of their expectations on the first factor) but fails to produce the invariance.
An analog of Theorem~\ref{thm:polynomial-norm-convergence} cannot hold for solvable groups of exponential growth in view of counterexamples due to Bergelson and Leibman \cite{MR2041260}.

Walsh's argument is based on Kreisel's \emph{no-counterexample interpretation of convergence}.
In order to explain this and some other ideas involved in this technique, we begin with a proof of a (known) quantitative version of the von Neumann mean ergodic theorem for multiplicators on the unit circle $\T$.

\section{A close look at the von Neumann mean ergodic theorem}
\label{sec:vn}
Throughout this section $\mu$ denotes a Borel measure on $\T$ and $Uf(\lambda) = \lambda f(\lambda)$ is a multiplicator on $L^2(\T,\mu)$.
The von Neumann mean ergodic theorem in its simplest form reads as follows.
\begin{multvNtheorem}
\label{thm:vn-mult}
Let $\mu$ and $U$ be as above.
Then the ergodic averages $a_{N} = \E_{n\leq N} U^n f$ converge in $L^2(\mu)$.
\end{multvNtheorem}
\begin{proof}
The averages $a_{N}$ are dominated by $|f|$ and converge pointwise, namely to $f(1)$ at $1$ and to zero elsewhere.
\end{proof}

It is well-known that no uniform bound on the rate of convergence of the ergodic averages can be given even if $U$ is similar to the Koopman operator of a measure-preserving transformation \cite{MR510630}.
However, there does exist a uniform bound on the \emph{rate of metastability} of the ergodic averages.
Let us recall the concept of \emph{metastability}.
\index{metastability}
The sequence $(a_N)$ converges if and only if it is Cauchy, i.e.
\[
\forall\epsilon>0 \, \exists M \, \forall N,N' (M \leq N,N' \implies \|a_N - a_{N'}\|_2 < \epsilon).
\]
The negation of this statement, i.e. ``$(a_N)$ is not Cauchy'' reads
\[
\exists\epsilon>0 \, \forall M \, \exists N,N' \colon M \leq N, N' ,\, \|a_N - a_{N'}\|_2 \geq \epsilon.
\]
Choosing witnesses $N(M)$, $N'(M)$ for each $M$ and defining
\[
F(M) = \max\{N(M), N'(M)\}
\]
we see that this is equivalent to
\[
\exists\epsilon>0 \, \exists F \from\N\to\N \, \forall M \, \exists N,N' \colon M \leq N,N' \leq F(M) ,\, \|a_N - a_{N'}\|_2 \geq \epsilon.
\]
Negating this we obtain that $(a_N)$ is Cauchy if and only if
\[
\forall\epsilon>0 \, \forall F \from\N\to\N \, \exists M \, \forall N,N'
(M \leq N,N' \leq F(M) \implies \|a_N - a_{N'}\|_2 < \epsilon ).
\]
This kind of condition, namely that the oscillation of a function is small on a finite interval is called \emph{metastability}.
A bound on the \emph{rate of metastability} is a bound on $M$ that may depend on $\epsilon$ and $F$ but not the sequence $(a_{N})_{N}$.

The appropriate reformulation of the von Neumann mean ergodic theorem for the operator $U$ in terms of metastability reads as follows.
\begin{multvNtheorem}[finitary version]
\label{thm:vn-fin}
Let $\mu$ and $U$ be as above. Then for every $\epsilon>0$, every function $F\from\N\to\N$ and every $f \in L^{2}(\mu)$ there exists a number $M$ such that for every $M \leq N, N' \leq F(M)$ we have
\begin{equation}
\label{eq:vn}
\Big\| \E_{n \leq N} U^n f - \E_{n \leq N'} U^n f \Big\|_2 < \epsilon.
\end{equation}
\end{multvNtheorem}
Although Theorem~\ref{thm:vn-fin} is equivalent to Theorem~\ref{thm:vn-mult} by the above considerations, we now attempt to prove it as stated.

\begin{proof}[Proof of Theorem~\ref{thm:vn-fin}]
It clearly suffices to consider strictly monotonically increasing functions $F$.
Let us assume $\|f\|_{2}=1$, take an arbitrary $M$ and see what can be said about the averages in \eqref{eq:vn}.

Suppose first that $f$ is supported near $1$, say on the disc $A_{M}$ with radius $\frac{\epsilon}{6 F(M)}$ and center $1$.
Then $U^n f$ is independent of $n$ up to a relative error of $\frac{\epsilon}{6}$ provided that $n \leq F(M)$, hence both averages are nearly equal.

Suppose now that the support of $f$ is bounded away from $1$, say $f$ is supported on the complement $B_{M}$ of the disc with radius $\frac{12}{\epsilon M}$ and center $1$.
Then the exponential sums $\frac1N \sum_{1 \leq n \leq N} \lambda^n$ are bounded by $\frac{\epsilon}{6}$ for all $\lambda$ in the support of $f$ provided that $N \geq M$, hence both averages are small.

However, there is an annulus whose intersection with the unit circle $E_M = \T \setminus ( A_M \cup B_M )$ does not fall in any of the two cases.
The key insight is that the regions $E_{M_i}$ can be made pairwise disjoint if one chooses a sufficiently rapidly growing sequence $(M_i)_i$, for instance it suffices to ensure $\frac{12}{\epsilon M_{i+1}} < \frac{\epsilon}{6 F(M_{i})}$.

Given $f$ with $\|f\|_2 \leq 1$, we can by the pigeonhole principle find an $i$ such that $\|f E_{M_i}\|_{2} < \epsilon/6$ (here we identify sets with their characteristic functions).
Thus we can split
\begin{equation}
\label{eq:decomposition-walsh}
f = \sigma + u + v,
\end{equation}
where $\sigma = f A_{M_i}$ is ``structured'', $u = f B_{M_i}$ is ``pseudorandom'', and $v = f E_{M_i}$ is $L^2$-small.
By the above considerations we obtain \eqref{eq:vn} for all $M_i \leq N, N' \leq F(M_i)$.
\end{proof}

Observe that the sequence $(M_i)_i$ in the foregoing proof does not depend on the measure $\mu$.
Moreover, a finite number of disjoint regions $E_{M_i}$ suffices to ensure that $f E_{M_i}$ is small for some $i$.
This yields the following strengthening of the von Neumann theorem.
\begin{multvNtheorem}[quantitative version]
\label{thm:vn-quan}
For every $\epsilon>0$ and every function $F\from\N\to\N$ there exist natural numbers $M_1, \dots, M_K$ such that for every $\mu$ and every $f \in L^{2}(\mu)$ with $\|f\|_{2}\leq 1$ there exists an $i$ such that for every $M_i \leq N, N' \leq F(M_i)$ we have
\[
\Big\| \E_{n \leq N} U^n f - \E_{n \leq N'} U^n f \Big\|_2 < \epsilon,
\]
where $Uf(\lambda)=\lambda f(\lambda)$ is a multiplicator as above.
\end{multvNtheorem}

The spectral theorem or the Herglotz-Bochner theorem can be used to deduce a similar result for any unitary operator.
The argument of Avigad, Gerhardy, and Towsner \cite[Theorem 2.16]{MR2550151} gives a similar result for arbitrary contractions on Hilbert spaces.
An even more precise result regarding contractions on uniformly convex spaces has been recently obtained by Avigad and Rute \cite{arxiv:1203.4124v1}.

Quantitative statements similar to Theorem~\ref{thm:vn-quan} with uniform bounds that do not depend on the particular measure-preserving system allow us to use a certain induction argument that breaks down if this uniformity is disregarded.
A decomposition of the form \eqref{eq:decomposition-walsh}, albeit a much more elaborate one (Structure Theorem~\ref{thm:structure}), will also play a prominent role.

\section{Complexity}
\label{sec:complexity}
In this section we give a streamlined treatment of the notion of \emph{complexity} due to Walsh \cite[\textsection 4]{MR2912715}.
It serves as the induction parameter in the proof of Theorem~\ref{thm:polynomial-norm-convergence}.

An ordered tuple $\bfg = (g_{0},\dots, g_{j})$ of measurable mappings from $\AG$ to $G$ in which $g_{0} \equiv 1_{G}$ is called a \emph{system} (it is not strictly necessary to include the constant mapping $g_{0}$ in the definition, but it comes in handy in inductive arguments).

The complexity of the trivial system $\bfg = (1_{G})$ is by definition at most zero, in symbols $\complexity\bfg \leq 0$.
A system has finite complexity if it can be reduced to the trivial system in finitely many steps by means of two operations, \emph{reduction} (used in Proposition~\ref{prop:metastability}) and \emph{cheating} (used in Theorem~\ref{thm:metastability}).

For $a,b \in \AG$ we define the \emph{$(a,b)$-reduction} of mappings $g,h \from\AG\to G$ to be the mapping
\index{reduction}
\[
\<g|h\>_{a,b}(n)
= D_{a,b}(g\inv)(n) T_{a,b}h(n)
= g(n)g(anb)\inv h(anb)
\]
and the $(a,b)$-reduction of a system $\bfg = (g_{0},\dots, g_{j})$ to be the system
\[
\bfg_{a,b}^{*}
:=
\bfg' \uplus \<g_{j}|\bfg'\>_{a,b}
=
\left( g_{0},\dots, g_{j-1}, \<g_{j}|g_{0}\>_{a,b},\dots,\<g_{j}|g_{j-1}\>_{a,b} \right),
\]
where we use the shorthand notation $\bfg' = (g_{0},\dots,g_{j-1})$ and $\<g_{j}|(g_{0},\dots,g_{j-1})\> = (\<g_{j}|g_{0}\>,\dots,\<g_{j}|g_{j-1}\>)$, and where the symbol ``$\uplus$'' denotes concatenation.
If the reduction $\bfg_{a,b}^{*}$ has complexity at most $\cplx-1$ for every $a,b \in \AG$, then the system $\bfg$ is defined to have complexity at most $\cplx$.

Furthermore, if $\bfg$ is a system of complexity at most $\cplx$ and the system $\tilde\bfg$ consists of functions of the form $gc$, where $g\in\bfg$ and $c\in G$, then we cheat and set $\complexity\tilde\bfg \leq \cplx$.
\index{cheating}
This definition tells that striking out constants and multiple occurrences of the same mapping in a system as well as rearranging mappings will not change the complexity, and adding new mappings can only increase the complexity, for example
\[
\complexity(1_{G},g_{2},g_{1}c,g_{1},c')
= \complexity(1_{G},g_{1},g_{2})
\leq  \complexity(1_{G},g_{1},g_{2},g_{3}).
\]
Note that cheating is transitive in the sense that if one can go from system $\bfg$ to system $\tilde\bfg$ in finitely many cheating steps, then one can also go from $\bfg$ to $\tilde\bfg$ in one cheating step.

In general a system need not have finite complexity.
We record here a streamlined proof of Walsh's result that that every polynomial system does have finite complexity.
We say that a system $(g_{0},\dots,g_{j})$ is $\Gb$-polynomial for a prefiltration $\Gb$ if $g_{i} \in \poly$ for every $i$.
For brevity we will denote discrete derivatives by
\[
D_{a,b}g(n) := g\inv(n) T_{a,b}g(n),
\text{ where }
T_{a,b}g(n) = g(anb).
\]
Note that for every $\Gb$-polynomial $g$ and $a,b\in\AG$ the translate $T_{a,b}g$ is also $\Gb$-polynomial (since $Tg=gD_{T}g$).
We will omit the indices $a,b$ in statements that hold for all $a,b$.
\begin{theorem}
\label{thm:finite-complexity}
The complexity of every $\Gb$-polynomial system $\tilde\bfg = (g_{0},\dots,g_{j})$ is bounded by a constant $\cplx(\vd,j)$ that only depends on the length $\vd$ of the prefiltration $\Gb$ and the size $j$ of the system.
\end{theorem}
The proof is by induction on $d$.
For induction purposes we need the formally stronger statement below.
We use the convenient shorthand notation $g(h_{0},\dots,h_{k})=(gh_{0},\dots,gh_{k})$.

\begin{proposition}
\label{prop:finite-complexity}
Let $\tilde\bfg = (g_{0},\dots,g_{j})$ be a $\Gb$-polynomial system.
Let also $\bfh_{0},\dots,\bfh_{j}$ be $\Gb[+1]$-polynomial systems and assume $\complexity\bfh_{j}\leq\cplx_{j}$.
Then the complexity of the system $\bfg = g_{0}\bfh_{0} \uplus\dots\uplus g_{j}\bfh_{j}$ is bounded by a constant $\cplx'=\cplx'(\vd,j,|\bfh_{0}|,\dots,|\bfh_{j-1}|,\cplx_{j})$, where $\vd$ is the length of $\Gb$.
\end{proposition}
The induction scheme is as follows.
Theorem~\ref{thm:finite-complexity} with length $\vd-1$ is used to prove Proposition~\ref{prop:finite-complexity} with length $\vd$, that in turn immediately implies Theorem~\ref{thm:finite-complexity} with length $\vd$.
The base case, namely Theorem~\ref{thm:finite-complexity} with $\vd=-\infty$, is trivial and $\cplx(-\infty,j)=0$.
\begin{proof}[Proof of Prop.~\ref{prop:finite-complexity} assuming Thm.~\ref{thm:finite-complexity} for length $\vd-1$]
It suffices to obtain a uniform bound on the complexity of $\bfg^{*}$ for every reduction $\bfg^{*}=\bfg^{*}_{a,b}$, possibly cheating first.
Splitting $\bfh_{j} = \bfh_{j}' \uplus (h)$ (where $\bfh_{j}'$ might be empty) we obtain
\begin{equation}
\label{eq:g-star}
\bfg^{*}
=
g_{0}\bfh_{0} \uplus\dots\uplus g_{j-1}\bfh_{j-1} \uplus g_{j}\bfh_{j}'
\uplus
\< g_{j} h | g_{0}\bfh_{0} \uplus\dots\uplus g_{j-1}\bfh_{j-1} \uplus g_{j}\bfh_{j}' \>.
\end{equation}
Note that for every $\Gb[+1]$-polynomial $h'$ we have
\begin{equation}
\label{eq:red-gjgj}
\< g_{j}h | g_{j}h' \> = g_{j}h (Tg_{j} Th)\inv Tg_{j} Th' = g_{j} h (Th)\inv Th' = g_{j} \<h|h'\>
\end{equation}
and
\begin{multline}
\label{eq:red-gjgi}
\< g_{j} h | g_{i}h' \> = D(h\inv g_{j}\inv) Tg_{i} Th' = D(h\inv g_{j}\inv) g_{i} Dg_{i} Th'\\
= g_{i} D(h\inv g_{j}\inv) [D(h\inv g_{j}\inv), g_{i}] Dg_{i} Th' = g_{i} \tilde h,
\end{multline}
where $\tilde h$ is a $\Gb[+1]$-polynomial by Theorem~\ref{thm:poly-group}.
By cheating we can rearrange the terms on the right-hand side of \eqref{eq:g-star}, obtaining
\begin{equation}
\label{eq:g-star-2}
\complexity\bfg^{*}
\leq
\complexity\left(
g_{0}\tilde\bfh_{0} \uplus\dots\uplus g_{j-1}\tilde\bfh_{j-1}
\uplus
g_{j} \left( \bfh_{j}' \uplus \< h | \bfh_{j}' \> \right)
\right)
\end{equation}
for some $\Gb[+1]$-polynomial systems $\tilde\bfh_{0},\dots,\tilde\bfh_{j-1}$ with cardinality $2|\bfh_{0}|,\dots,2|\bfh_{j-1}|$, respectively.

We use nested induction on $j$ and $\cplx_{j}$.
In the base case $j=0$ we have $\bfg=\bfh_{0}$ and we obtain the conclusion with
\[
\cplx'(\vd,0,\cplx_{0})=\cplx_{0}.
\]
Suppose that $j>0$ and the conclusion holds for $j-1$.
If $\cplx_{j}=0$, then by cheating we may assume $\bfh_{j}=(1_{G})$.
Moreover, \eqref{eq:g-star-2} becomes
\[
\complexity\bfg^{*}
\leq
\complexity\left(
g_{0}\tilde\bfh_{0} \uplus\dots\uplus g_{j-1}\tilde\bfh_{j-1}
\right).
\]
The induction hypothesis on $j$ and Theorem~\ref{thm:finite-complexity} applied to $\tilde\bfh_{j-1}$ yield the conclusion with the bound
\begin{multline*}
\cplx'(\vd,j,|\bfh_{0}|,\dots,|\bfh_{j-1}|,0)
=
\cplx'(\vd,j-1,2|\bfh_{0}|,\dots,2|\bfh_{j-2}|,\cplx(\vd-1,2|\bfh_{j-1}|))+1.
\end{multline*}
If $\cplx_{j}>0$, then by cheating we may assume $\bfh_{j}\neq (1_{G})$ and $\complexity\bfh_{j}^{*} \leq \cplx_{j} - 1$, and \eqref{eq:g-star-2} becomes
\[
\complexity\bfg^{*}
\leq
\complexity\left(
g_{0}\tilde\bfh_{0} \uplus\dots\uplus g_{j-1}\tilde\bfh_{j-1}
\uplus g_{j}\bfh_{j}^{*}
\right).
\]
The induction hypothesis on $\cplx_{j}$ now yields the conclusion with the bound
\[
\cplx'(\vd,j,|\bfh_{0}|,\dots,|\bfh_{j-1}|,\cplx_{j})
=
\cplx'(\vd,j,2|\bfh_{0}|,\dots,2|\bfh_{j-1}|,\cplx_{j}-1)+1.
\qedhere
\]
\end{proof}
\begin{proof}[Proof of Thm.~\ref{thm:finite-complexity} assuming Prop.~\ref{prop:finite-complexity} for length $\vd$]
Use Proposition~\ref{prop:finite-complexity} with system $\tilde\bfg$ as in the hypothesis and systems $\bfh_{0},\dots,\bfh_{j}$ being the trivial system $(1_{G})$.
This yields the bound
\[
\cplx(\vd,j) = \cplx'(\vd,j,1,\dots,1,0).
\qedhere
\]
\end{proof}

\section{The structure theorem}
\label{sec:structure}
The idea to prove a structure theorem for elements of a Hilbert space via the Hahn-Banach theorem is due to Gowers \cite[Proposition 3.7]{MR2669681}.
The insight of Walsh \cite[Proposition 2.3]{MR2912715} was to allow the ``structured'' and the ``pseudorandom'' part in the decomposition to take values in varying spaces that satisfy a monotonicity condition.

His assumption that these spaces are described by norms that are equivalent to the original Hilbert space norm can be removed.
In fact the structure theorem continues to hold for spaces described by extended seminorms\footnote{An \emph{extended seminorm} $\|\cdot\|$ on a vector space $H$ is a function with \emph{extended} real values $[0,+\infty]$ that is subadditive, homogeneous (i.e.\ $\|\lambda u\| = |\lambda| \|u\|$ if $\lambda\neq0$) and takes the value $0$ at $0$.}
\index{extended seminorm}
that are easier to construct in practice as we will see in Lemma~\ref{lem:Sigma-ext-seminorm}.

We caution the reader about the assignment of symbols: elements of $\AG$ are denoted in this chapter by $a,b,n,m$, indices (elements of $\iset$) by $\alpha,\beta,N,M$, real numbers by $\epsilon,\delta,C$, integers by $i,j,k,t,K,\cplx$, and bounded functions on $X$ by $f,\sigma$.
Without loss of generality we work with real-valued functions on $X$.

The Hahn-Banach theorem is used in the following form.
\begin{lemma}
\label{lem:sep}
Let $V_{i}$, $i=1,\dots,k$, be convex subsets of a Hilbert space $H$, at least one of which is open, and each of which contains $0$.
Let $V:=c_{1}V_{1}+\dots+c_{k}V_{k}$ with $c_{i}>0$ and take $f\not\in V$.
Then there exists a vector $\phi\in H$ such that $\<f,\phi\> \geq 1$ and $\<v,\phi\> < c_{i}\inv$ for every $v\in V_{i}$ and every $i$.
\end{lemma}
\begin{proof}
By the assumption the set $V$ is open, convex and does not contain $f$.
By the Hahn-Banach theorem there exists a $\phi\in H$ such that $\<f,\phi\> \geq 1$ and $\<v,\phi\> < 1$ for every $v\in V$.
The claim follows.
\end{proof}
The next result somewhat resembles Tao's structure theorem \cite{MR2274314}, though Tao's result gives additional information (positivity and boundedness of the structured part).
Gowers \cite{MR2669681} described tricks that allow to extract this kind of information from a proof via the Hahn-Banach theorem.
\begin{structuretheorem}
\label{thm:structure}
For every $\delta>0$,
any functions $\omega,\psi\from\iset\to\iset$,
and every $M_{\bullet}\in\iset$ there exists an increasing sequence of indices
\begin{equation}
\label{eq:decomposition-seq}
M_{\bullet} \leq M_1 \leq \dots \leq M_{\lceil 2\delta^{-2} \rceil}
\end{equation}
for which the following holds.
Let $\eta \from\R_+\to\R_+$ be any function and \mbox{$(\|\cdot\|_\indx)_{\indx\in\iset}$} be a net of extended seminorms on a Hilbert space $H$ such that the net of dual extended seminorms $(\|\cdot\|_\indx^*)_{\indx\in\iset}$ decreases monotonically.
Then for every $f\in H$ with $\|f\| \leq 1$ there exists a decomposition
\begin{equation}
\label{eq:decomposition}
f = \sigma + u + v
\end{equation}
and an $1\leq i\leq \lceil 2\delta^{-2} \rceil$ such that
\begin{equation}
\|\sigma\|_\beta < C^{\delta,\eta}_i, \quad
\|u\|_\alpha^* < \eta(C^{\delta,\eta}_i), \quad\text{and}\quad
\|v\| < \delta,
\end{equation}
where the indices $\alpha$ and $\beta$ satisfy $\omega(\alpha) \leq M_i$ and $\psi(M_i) \leq \beta$,
and where the constant $C^{\delta,\eta}_{i}$ belongs to a decreasing sequence that only depends on $\delta$ and $\eta$ and is defined inductively starting with
\begin{equation}
\label{eq:C}
C^{\delta,\eta}_{\lceil 2\delta^{-2} \rceil} = 1
\quad\text{by}\quad
C^{\delta,\eta}_{i-1} = \max \Big\{ C^{\delta,\eta}_{i}, \frac{2}{\eta(C^{\delta,\eta}_{i})} \Big\}.
\end{equation}
\end{structuretheorem}
In the sequel we will only use Theorem~\ref{thm:structure} with the identity function $\omega(\alpha):=\alpha$, in which case we can choose $\alpha=M_i$, and with $\delta$ and $\eta$ as in \eqref{eq:eta}.
\begin{proof}
It suffices to consider functions such that $\omega(N) \geq N$ and $\psi(N) \geq N$ for all $N$ (in typical applications $\psi$ grows rapidly).

The sequence $(M_{i})$ and auxiliary sequences $(\alpha_{i})$, $(\beta_{i})$ are defined inductively starting with $\alpha_{1}:=M_{\bullet}$ by
\[
M_{i} := \omega(\alpha_{i}),
\quad \beta_{i}:=\psi(M_{i}),
\quad \alpha_{i+1}:=\beta_{i},
\]
so that all three sequences increase monotonically.
Let $r$ be chosen later and assume that there is no $i \in \{1,\dots,r\}$ for which a decomposition of the form \eqref{eq:decomposition} with $\alpha=\alpha_{i}$, $\beta=\beta_{i}$ exists.

For every $i \in \{1,\dots,r\}$ we apply Lemma~\ref{lem:sep} with $V_{1},V_{2},V_{3}$ being the open unit balls of $\|\cdot\|_{\beta_{i}}$, $\|\cdot\|_{\alpha_{i}}^{*}$ and $\|\cdot\|$, respectively, and with $c_{1}=C_{i}$, $c_{2}=\eta(C_{i})$, $c_{3}=\delta$.
Note that $V_{3}$ is open in $H$.
We obtain vectors $\phi_{i} \in H$ such that
\[
\<\phi_{i},f\> \geq 1, \quad \|\phi_{i}\|_{\beta_{i}}^{*} \leq (C_{i})\inv, \quad \|\phi_{i}\|_{\alpha_{i}}^{**} \leq \eta(C_{i})\inv, \quad \|\phi_{i}\| \leq \delta\inv.
\]
Take $i<j$, then $\beta_{i}\leq \alpha_{j}$, and by \eqref{eq:C} we have
\begin{multline*}
|\<\phi_{i},\phi_{j}\>|
\leq \|\phi_{i}\|_{\alpha_{j}}^{*} \|\phi_{j}\|_{\alpha_{j}}^{**}
\leq \|\phi_{i}\|_{\beta_{i}}^{*} \|\phi_{j}\|_{\alpha_{j}}^{**}\\
\leq (C_{i})\inv \eta(C_{j})\inv
\leq (C_{j-1})\inv \eta(C_{j})\inv
\leq (2 \eta(C_{j})\inv)\inv \eta(C_{j})\inv
= \frac12,
\end{multline*}
so that
\[
r^{2} \leq \<\phi_{1}+\dots+\phi_{r},f\>^{2}
\leq \|\phi_{1}+\dots+\phi_{r}\|^{2}
\leq r \delta^{-2} + \frac{r^{2}-r}{2},
\]
which is a contradiction if $r \geq 2 \delta^{-2}$.
\end{proof}

\section{Reducible and structured functions}
\label{sec:reducible}
In this section we adapt Walsh's notion of a structured function and a corresponding inverse theorem to our context.
Informally, a function is reducible with respect to a system if its shifts can be approximated by shifts arising from reductions of this system, uniformly over Følner sets that are not too large.
A function is structured if it is a linear combination of reducible functions.

In order to formulate the relevant properties concisely we introduce two pieces of notation.
Given a Følner net $(\Fo_{\alpha})_{\alpha\in A}$, we call sets of the form $a \Fo_\alpha b$, $a,b\in\AG$, $\alpha\in A$, \emph{Følner sets}.
Such sets are usually denoted by the letter $I$.
For a Følner set $I$ we write $\lfloor I \rfloor = \alpha$ if $I = a \Fo_\alpha b$ for some $a,b\in\AG$.

By the Følner property for every $\gamma>0$ there exists a function $\varphi_{\gamma}\from A\to A$ such that
\begin{equation}
\label{eq:varphi}
\sup_{l\in \Fo_{\alpha}}|l \Fo_\beta \Delta \Fo_\beta| / |\Fo_\beta| < \gamma
\text{ for every } \beta \geq \varphi_{\gamma}(\alpha).
\end{equation}
\begin{definition}
\label{def:uniformly-N-reducible}
Let $\bfg=(g_{0},\dots,g_{j})$ be a system, $\gamma>0$ and $N \in \iset$.
A function $\sigma$ bounded by one is called \emph{uniformly $(\bfg, \gamma, N)$-reducible}
\index{uniformly reducible function}
(in symbols $\sigma\in\Sigma_{\bfg, \gamma, N}$)
if for every Følner set $I$ with $\varphi_{\gamma}(\lfloor I \rfloor) \leq N$ there exist
functions $b_0, \dots, b_{j-1}$ bounded by one, an arbitrary finite measure set $J \subset \AG$ and some $a\in\AG$ such that for every $l \in I$
\begin{equation}
\label{ineq:unif-red}
\Big\| g_j(l) \sigma - \E_{m \in J} \prod_{i=0}^{j-1} \<g_j | g_i\>_{a,m}(l) b_i \Big\|_\infty
< \gamma.
\end{equation}
\end{definition}
Walsh's definition of $L$-reducibility with parameter $\epsilon$ corresponds to uniform $(\bfg, \gamma, N)$-reducibility with $N=\varphi_{\gamma}(L)$ and a certain $\gamma=\gamma(\epsilon)$ that will now be defined along with other parameters used in the proof of the main result.

Given $\epsilon>0$ we fix
\begin{equation}
\label{eq:eta}
\delta = \frac{\epsilon}{2^{2}\cdot 3^{2}}
\quad \text{and}
\quad \eta(x) = \frac{\epsilon^{2}}{2^{3}\cdot 3^{3} x}
% so that the Inverse Theorem~\ref{thm:inverse} works
\end{equation}
and define the decreasing sequence
$C_{1}^{\delta,\eta} \geq \dots \geq C_{\lceil 2\delta^{-2} \rceil}^{\delta,\eta}$
as in \eqref{eq:C}.
It is in turn used to define the function
\begin{equation}
\label{eq:gamma}
\gamma = \gamma^{1}(\epsilon) = \frac{\epsilon}{3\cdot 8 C^{*}},
\quad\text{where}\quad C^{*} = C_{1}^{\delta,\eta},
\end{equation}
and its iterates $\gamma^{\cplx+1}(\epsilon) = \gamma^{\cplx}(\gamma)$.

The ergodic average corresponding to a system $\bfg = (g_{0},\dots,g_{j})$, bounded functions $f_{0},\dots, f_{j}$ and a finite measure set $I \subset \AG$ is denoted by
\[
\Av{I}[f_0, \dots, f_j]
:=
\E_{n\in I} \prod_{i=0}^j g_i(n) f_i.
\]
The inverse theorem below tells that any function that gives rise to a large ergodic average correlates with a reducible function.
\begin{inversetheorem}
\label{thm:inverse}
Let $\epsilon > 0$.
Suppose that $\|u\|_\infty \leq 3C$, the functions $f_0, \dots, f_{j-1}$ are bounded by one, and $\| \Av{I}[f_0, \dots, f_{j-1}, u] \|_2 > \epsilon/6$ for some Følner set $I = a \Fo_N b$.
Then there exists a uniformly $(\bfg, \gamma, N)$-reducible function $\sigma$ such that $\<u,\sigma\> > 2\eta(C)$.
\end{inversetheorem}
\begin{proof}
Set $b_0 := \Av{I}[f_0, \dots, f_{j-1}, u] f_{0} / \|u\|_\infty$, so that $\|b_0\|_\infty \leq 1$, and $b_{1}:=f_{1},\dots,b_{j-1}:=f_{j-1}$.
Recall $g_{0}\equiv 1_{G}$ and note that
\begin{align*}
2\eta(C)
&<
\|u\|_\infty\inv \left\| \Av{I} [f_0, \dots, f_{j-1}, u] \right\|_2^2\\
&=
\< \E_{n\in I} \prod_{i=0}^{j-1} g_i(n) f_i \cdot g_j(n) u, \frac{\Av{I}[f_0, \dots, f_{j-1}, u]}{\|u\|_\infty} \>\\
&=
\E_{n\in I} \< g_j(n) u, \prod_{i=0}^{j-1} g_i(n) b_i \>\\
&=
\< u, \underbrace{\E_{m\in \Fo_{N}} \prod_{i=0}^{j-1} g_j(amb)\inv g_i(amb) b_i}_{=: \sigma} \>.
\end{align*}
We claim that $\sigma$ is uniformly $(\bfg, \gamma, N)$-reducible.

Consider a Følner set $\tilde a \Fo_L \tilde b$ with $\varphi_{\gamma}(L) \leq N$.
We have to show \eqref{ineq:unif-red} for some $J\subset\AG$ and every $l\in \Fo_L$.
By definition \eqref{eq:varphi} of $\varphi_{\gamma}$ we obtain
\[
\Big\| \sigma - \E_{m\in \Fo_N} \prod_{i=0}^{j-1} g_j(almb)\inv g_i(almb) b_i \Big\|_\infty
\leq \frac{|l \Fo_N \Delta \Fo_N|}{|\Fo_N|} < \gamma.
\]
Since $g_j(\tilde al\tilde b)$ is an isometric algebra homomorphism, we see that $g_j(\tilde al\tilde b) \sigma$ is uniformly approximated by
\[
\E_{m\in \Fo_N} \prod_{i=0}^{j-1} g_j(\tilde al\tilde b) g_j(almb)\inv g_i(almb) b_i.
\]
Splitting $almb = a \tilde a\inv \cdot \tilde al\tilde b \cdot \tilde b\inv mb$ we can write this function as
\[
\E_{m\in \Fo_N} \prod_{i=0}^{j-1} \<g_j,g_i\>_{a \tilde a\inv,\tilde b\inv mb}(\tilde al\tilde b) b_i
=
\E_{m\in \tilde b\inv \Fo_N b} \prod_{i=0}^{j-1} \<g_j,g_i\>_{a \tilde a\inv,m}(\tilde al\tilde b) b_i,
\]
which gives \eqref{ineq:unif-red} with $J=\tilde b\inv \Fo_N b$ for the Følner set $I = \tilde a \Fo_L \tilde b$.
\end{proof}
Structure will be measured by extended seminorms associated to sets $\Sigma$ of reducible functions by the following easy lemma.
\begin{lemma}[cf. {\cite[Corollary 3.5]{MR2669681}}]
\label{lem:Sigma-ext-seminorm}
Let $H$ be an inner product space and $\Sigma\subset H$.
Then the formula
\begin{equation}
\|f\|_{\Sigma} := \inf\Big\{ \sum_{t=0}^{k-1}|\lambda_{t}| : f=\sum_{t=0}^{k-1}\lambda_{t}\sigma_{t}, \sigma_{t} \in \Sigma \Big\},
\end{equation}
where empty sums are allowed and the infimum of an empty set is by convention $+\infty$, defines an extended seminorm on $H$ whose dual extended seminorm is given by
\begin{equation}
\|f\|_{\Sigma}^{*} := \sup_{\phi\in H: \|\phi\|_{\Sigma}\leq 1} |\<f,\phi\>| = \sup_{\sigma\in\Sigma}|\<f,\sigma\>|.
\end{equation}
\end{lemma}
Heuristically, a function with small dual seminorm is pseudorandom since it does not correlate much with structured functions.

\section{Metastability of averages for finite complexity systems}
\label{sec:metastability}
We come to the proof of the norm convergence result.
Instead of Theorem~\ref{thm:polynomial-norm-convergence} we consider a quantitative statement that is strictly stronger in the same way as the quantitative von Neumann Theorem~\ref{thm:vn-quan} is strictly stronger than the finitary von Neumann Theorem~\ref{thm:vn-fin}.
We use the notation
\[
\Av{I,I'}[f_0, \dots, f_j]
:=
\Av{I}[f_0, \dots, f_j] - \Av{I'}[f_0, \dots, f_j].
\]
for the difference between two multiple averages.
We will now quantify the statement from the introduction that F\o{}lner sets are approximately ordered by inclusion.

A finite measure set $K$ is said to be \emph{$\gamma$-approximately included} in a measurable set $I$, in symbols $K \lesssim_{\gamma} I$, if $|K \setminus I|/|K| < \gamma$.
The next lemma states that the family of Følner sets is directed by $\gamma$-approximate inclusion.
\begin{lemma}
\label{lem:ceil}
For every $\gamma>0$ and any compact sets $I$ and $I'$ with positive measure
there exists an index $\lceil I, I' \rceil_{\gamma} \in A$ with the property that
for every $\beta \geq \lceil I, I' \rceil_{\gamma}$ there exists some $b\in\AG$ such that
$I \lesssim_{\gamma} \Fo_\beta b$ and $I' \lesssim_{\gamma} \Fo_\beta b$.
\end{lemma}
Note that the expectation satisfies $\E_{n\in aIb} f(n) = \E_{n\in I} f(anb)$ for any $a,b\in\AG$.
\begin{proof}
Let $K \subset \AG$ be compact and $c>0$ to be chosen later.
By the Følner property there exists an index $\alpha_{0}\in A$ such that for every $\alpha \geq \alpha_{0}$ we have $|l\Fo_\alpha \cap \Fo_\alpha|/|\Fo_\alpha| > 1-c$ for all $l\in K$.
Integrating over $K$ and using Fubini's theorem we obtain
\begin{align*}
1-c
&< \E_{l\in K} \E_{\tilde b\in \Fo_\alpha} 1_{\Fo_\alpha}(l \tilde b)\\
&= \E_{\tilde b\in \Fo_\alpha} \E_{l\in K} 1_{\Fo_\alpha \tilde b\inv}(l)\\
&= \E_{\tilde b\in \Fo_\alpha} |K \cap \Fo_\alpha \tilde b\inv|/|K|.
\end{align*}
Therefore there exists a $b\in\AG$ (that may depend on $\alpha \geq \alpha_0$) such that $|K \cap \Fo_\alpha b|/|K| > 1-c$, so $|K \setminus \Fo_\alpha b|/|K| < c$.

We apply this with $K = I \cup I'$ and $c = \gamma \frac{\min\{|I|,|I'|\}}{|K|}$.
Let $\lceil I, I' \rceil_{\gamma} := \alpha_{0}$ as above and $\alpha \geq \lceil I, I' \rceil_{\gamma}$.
Then for an appropriate $b\in\AG$ we have
\[
|I \setminus \Fo_\alpha b|/|I| \leq |K \setminus \Fo_\alpha b|/|I| < |K|\beta/|I| \leq \gamma,
\]
and analogously for $I'$.
\end{proof}
With this notation in place, we can formulate the main metastability result.
\begin{theorem}
\label{thm:metastability}
For every complexity $\cplx\in\N$ and every $\epsilon>0$ there exists $K_{\cplx,\epsilon}\in\N$ such that
for every function $F \from \iset \to \iset$ and every $M \in \iset$ there exists a tuple of indices
\begin{equation}
\label{eq:main-thm-seq}
M\leq M^{\cplx,\epsilon,F}_1 , \dots , M^{\cplx,\epsilon,F}_{K_{\cplx,\epsilon}} \in \iset
\end{equation}
of size $K_{\cplx,\epsilon}$
such that for every system $\bfg$ with complexity at most $\cplx$ and every choice of functions $f_0, \dots, f_j \in L^{\infty}(X)$ bounded by one there exists $1 \leq i \leq K_{\cplx,\epsilon}$ such that
for all Følner sets $I,I'$ with $M^{\cplx,\epsilon,F}_i \leq \lfloor I\rfloor,\lfloor I'\rfloor$ and $\lceil I, I'\rceil_{\gamma^{\cplx}(\epsilon)} \leq F(M^{\cplx,\epsilon,F}_i)$ we have
\begin{equation}
\label{eq:main-thm-est}
\| \Av{I,I'}[f_0, \dots, f_j] \|_2 < \epsilon.
\end{equation}
\end{theorem}
Recall that $\lceil I,I' \rceil_{\gamma^{\cplx}(\epsilon)}$ was defined in Lemma~\ref{lem:ceil}.
Theorem~\ref{thm:metastability} will be proved by induction on the complexity $\cplx$.
As an intermediate step we need the following.
\begin{proposition}
\label{prop:metastability}
For every complexity $\cplx\in\N$ and every $\epsilon>0$ there exists $\tilde K_{\cplx,\epsilon}\in\N$ such that
for every function $F \from \iset \to \iset$ and every $\tilde M \in \iset$ there exists a tuple of indices
\begin{equation}
\label{eq:main-prop-seq}
\tilde M \leq \tilde M^{\cplx,\epsilon,F}_1 , \dots , \tilde M^{\cplx,\epsilon,F}_{\tilde K_{\cplx,\epsilon}} \in \iset
\end{equation}
of size $\tilde K_{\cplx,\epsilon}$
as well as an index $N = N_{\cplx,\epsilon,F}(\tilde M)$ such that the following holds.
For every system $\bfg$ such that every reduction $\bfg^*_{a,b}$ ($a,b\in\AG$) has complexity at most $\cplx$,
every choice of functions $f_0, \dots, f_{j-1} \in L^{\infty}(X)$ bounded by one,
and every finite linear combination $\sum_{t} \lambda_t \sigma_t$ of uniformly $(\bfg, \gamma, N)$-reducible functions $\sigma_{t}$
there exists $1 \leq \tilde i \leq \tilde K_{\cplx,\epsilon}$ such that
for all Følner sets $I,I'$ with $\tilde M^{\cplx,\epsilon,F}_{\tilde i} \leq \lfloor I\rfloor,\lfloor I'\rfloor$ and $\lceil I, I'\rceil_{\gamma^{\cplx+1}(\epsilon)} \leq F(\tilde M^{\cplx,\epsilon,F}_{\tilde i})$ we have
\begin{equation}
\Big\| \Av{I,I'}[f_0, \dots, f_{j-1},\sum_{t} \lambda_t \sigma_t] \Big\|_2 < 8 \gamma \sum_t |\lambda_t|.
\end{equation}
\end{proposition}
The induction procedure is as follows.
Theorem~\ref{thm:metastability} for complexity $\cplx$ is used to deduce Proposition~\ref{prop:metastability} for complexity $\cplx$, which is in turn used to show Theorem~\ref{thm:metastability} for complexity $\cplx+1$.
The base case (Theorem~\ref{thm:metastability} with $\cplx=0$) is trivial, take $K_{0,\epsilon}=1$ and $M_{1}^{0,\epsilon,F}=M$.

\begin{proof}[Proof of Prop.~\ref{prop:metastability} assuming Thm.~\ref{thm:metastability} for complexity $\cplx$]
The tuple \eqref{eq:main-prop-seq} and the index $N$ will be chosen later.
For the moment assume that $I, I' \lesssim_{\gamma} I_{0}$ for some Følner set $I_{0}$ with $\varphi_{\gamma}(\lfloor I_{0}\rfloor) \leq N$.
Consider the functions $b_{0}^{t},\dots,b_{j-1}^{t}$ bounded by one, the set $J^{t} \subset \AG$ and the element $a^t \in\AG$ from the definition of uniform $(\bfg, \gamma, N)$-reducibility of $\sigma_{t}$ over $I_0$ (Definition~\ref{def:uniformly-N-reducible}).
Write $O(x)$ for an error term bounded by $x$ in $L^\infty(X)$.
By \eqref{ineq:unif-red} we have
\begin{multline*}
\Av{I}[f_0, \dots, f_{j-1}, \sigma_t]
= \frac{1}{|I|} \int_{n\in I} \prod_{i=0}^{j-1} g_i(n) f_i \cdot g_j(n)\sigma_t\\
= \frac{1}{|I|} \int_{n\in I\cap I_0} \prod_{i=0}^{j-1} g_i(n) f_i
\left( \E_{m\in J^t} \prod_{i=0}^{j-1} \<g_j|g_i\>_{a^t,m}(n) b_i^t + O(\gamma) \right)
+ \frac{|I\setminus I_0|}{|I|} O(1).
\end{multline*}
The first error term accounts for the $L^{\infty}$ error in the definition of uniform reducibility
and the second for the fraction of $I$ that is not contained in $I_{0}$.
This can in turn be approximated by
\begin{multline*}
\frac{1}{|I|} \int_{n\in I} \E_{m\in J^t} \prod_{i=0}^{j-1} g_i(n) f_i
\prod_{i=0}^{j-1} \<g_j|g_i\>_{a^t,m}(n) b_i^t
+ \frac{|I\cap I_0|}{|I|} O(\gamma)
+ \frac{|I\setminus I_0|}{|I|} O(2)\\
= \E_{m\in J^{t}} \Av[\bfg^{*}_{a^t,m}]{I}[f_0, \dots, f_{j-1}, b_0^t, \dots, b_{j-1}^t] + O(3\gamma).
\end{multline*}
Using the analogous approximation for $I'$ and summing over $t$ we obtain
\begin{multline}
\label{eq:est-prop}
\| \Av{I,I'}[f_0, \dots, f_{j-1},\sum_{t} \lambda_t \sigma_t] \|_{2}\\
\leq
\sum_{t} |\lambda_{t}| \E_{m\in J^{t}} \| \Av[\bfg^{*}_{a^t,m}]{I,I'}[f_0, \dots, f_{j-1},b_{0}^{t},\dots,b_{j-1}^{t}] \|_{2}
+6 \gamma \sum_{t} |\lambda_{t}|.
\end{multline}
If $G$ is commutative and $\bfg$ consists of affine mappings, then the maps that constitute systems $\bfg^{*}_{a^t,m}$ differ at most by constants, and in this case one can bound the first summand by a norm of a difference of averages associated to certain functions on $X \times \uplus_{t} J^{t}$ similarly to the reduction in \cite[\textsection 5]{MR2408398}.
In general we need (a version of) the more sophisticated argument of Walsh that crucially utilizes the uniformity in the induction hypothesis.
The argument provides a bound on most (with respect to the weights $|\lambda_{t}|/|J^{t}|$) of the norms that occur in the first summand.

Let $r = r(\cplx,\epsilon)$ be chosen later.
We use the operation $M \mapsto M^{\cplx,\gamma,F_{s}}_i$ and the constant $K=K_{\cplx,\gamma}$ from Theorem~\ref{thm:metastability} (with $\gamma$ in place of $\epsilon$) to inductively define functions $F_{r},\dots,F_{1} \from\iset\to\iset$ by
\[
F_{r}=F, \quad F_{s-1}(M) := \sup_{1 \leq i \leq K} F_{s}(M^{\cplx,\gamma,F_{s}}_i).
\]
This depends on a choice of a supremum function for the directed set $\iset$ that can be made independently of all constructions performed here.
Using the same notation define inductively for $1\leq i_1,\dots,i_r \leq K$ the indices
\[
\tilde M^{()} := \tilde M, \quad
\tilde M^{(i_{1},\dots,i_{s-1},i_{s})} :=
(\tilde M^{(i_{1},\dots,i_{s-1})})^{\cplx,\gamma,F_{s}}_{i_{s}}.
\]
The theorem tells that for every $t$, $m$ and $1 \leq i_{1},\dots,i_{s-1} \leq K$
there exists some $1 \leq i_{s} \leq K$ such that
\begin{equation}
\label{eq:est-red}
\| \Av[\bfg^{*}_{a^t,m}]{I,I'}[f_0, \dots, f_{j-1},b_{0}^{t},\dots,b_{j-1}^{t}] \|_{2} <\gamma
\end{equation}
holds provided
\begin{equation}
\label{eq:main-ind-cond}
\tilde M^{(i_{1},\dots,i_{s})} = (\tilde M^{(i_{1},\dots,i_{s-1})})^{\cplx,\gamma,F_{s}}_{i_{s}} \leq \lfloor I\rfloor, \lfloor I'\rfloor
\text{ and }
\lceil I, I'\rceil_{\gamma^{\cplx}(\gamma)} \leq F_{s}(\tilde M^{(i_{1},\dots,i_{s})}).
\end{equation}

Start with $s=1$.
By the pigeonhole principle there exists an $i_1$ such that \eqref{eq:est-red} holds for at least the fraction $1/K$ of the pairs $(t,m)$ with respect to the weights $|\lambda_{t}|/|J^{t}|$ (provided \eqref{eq:main-ind-cond} with $s=1$).

Using the pigeonhole principle repeatedly on the remaining pairs $(t,m)$ with weights $|\lambda_{t}|/|J^{t}|$ we can find a sequence $i_{1},\dots, i_{r}$ such that for all pairs but the fraction $(\frac{K-1}{K})^{r}$ the estimate \eqref{eq:est-red} holds provided that the conditions \eqref{eq:main-ind-cond} are satisfied for all $s$.

By definition we have
$\tilde M \leq \tilde M^{(i_{1})} \leq \tilde M^{(i_{1},i_{2})} \leq \dots \leq \tilde M^{(i_{1},\dots,i_{r})}$
and
\begin{multline*}
F_{1}(\tilde M^{(i_{1})})
\geq F_{2}((\tilde M^{(i_{1})})^{\cplx,\gamma,F_{2}}_{i_{2}})
= F_{2}(\tilde M^{(i_{1},i_{2})})
\geq \dots\\
\geq F_{r}(\tilde M^{(i_{1},\dots,i_{r})})
= F(\tilde M^{(i_{1},\dots,i_{r})})
\end{multline*}
for any choice of $i_{1},\dots,i_{r}$.
Therefore the conditions \eqref{eq:main-ind-cond} become stronger as $s$ increases.
Recall from \eqref{eq:gamma} that $\gamma^{\cplx}(\gamma) = \gamma^{\cplx+1}(\epsilon)$, thus we only need to ensure
\begin{equation}
\label{eq:main-ind-cond2}
\tilde M^{(i_{1},\dots,i_{r})} \leq \lfloor I\rfloor,\lfloor I'\rfloor
\text{ and }
\lceil I, I'\rceil_{\gamma^{\cplx+1}(\epsilon)} \leq F(\tilde M^{(i_{1},\dots,i_{r})}).
\end{equation}
This is given by the hypothesis if we define the tuple \eqref{eq:main-prop-seq} to consist of all numbers $\tilde M^{(i_{1},\dots,i_{r})}$ where $i_{1},\dots,i_{r} \in \{1,\dots,K\}$, so $\tilde K_{\cplx,\epsilon} = (K_{\cplx,\gamma})^{r}$.

We now choose $r$ to be large enough that $(\frac{K-1}{K})^{r} < \gamma$.
% It suffices to have $r > -K\ln\gamma = -\ln\gamma / (1/K) > -\ln\gamma / (\ln K - \ln(K-1))$
Then the sum at the right-hand side of \eqref{eq:est-prop} splits into a main term that can be estimated by $\gamma\sum_t |\lambda_t|$ using \eqref{eq:est-red} and an error term that can also be estimated by $\gamma\sum_t |\lambda_t|$ using the trivial bound
\[
\| \Av[\bfg^{*}_{a^t,m}]{I,I'}[f_0, \dots, f_{j-1},b_{0}^{t},\dots,b_{j-1}^{t}] \|_{2} \leq 1.
\]
Finally, the second condition in \eqref{eq:main-ind-cond2} by definition means that there exists a Følner set $I_0$ such that $\lfloor I_0 \rfloor = F(\tilde M^{(i_{1},\dots,i_{r})})$ and $I,I' \lesssim_{\gamma^{\cplx+1}(\epsilon)} I_0$. In particular we have $I,I' \lesssim_{\gamma} I_0$ since $\gamma^{\cplx+1}(\epsilon) \leq \gamma^{1}(\epsilon) = \gamma$.
Taking
\[
N := \sup_{1\leq i_{1},\dots,i_{r} \leq K} \varphi_{\gamma}(F(\tilde M^{(i_{1},\dots,i_{r})}))
\]
guarantees $\varphi_{\gamma}(\lfloor I_0 \rfloor) \leq N$.
\end{proof}

\begin{proof}[Proof of Thm.~\ref{thm:metastability} assuming Prop.~\ref{prop:metastability} for complexity $\cplx-1$]
Let $\cplx$, $\epsilon$, $F$ and a system $\bfg$ with complexity at most $\cplx$ be given.
By cheating we may assume that every reduction $\bfg^*_{a,b}$ ($a,b\in\AG$) has complexity at most $\cplx-1$.

We apply the Structure Theorem~\ref{thm:structure} with the following data.
The extended seminorms $\|\cdot\|_{N} := \|\cdot\|_{\Sigma_{\bfg, \gamma, N}}$, $N\in\iset$, are given by Lemma~\ref{lem:Sigma-ext-seminorm}; the dual extended seminorms $\|\cdot\|_{N}^{*} = \|\cdot\|_{\Sigma_{\bfg, \gamma, N}}^{*}$ decrease monotonically since $\Sigma_{\bfg, \gamma, N'} \subset \Sigma_{\bfg, \gamma, N}$ whenever $N' \geq N$.
The function $\psi(\tilde M) := N_{\cplx,\epsilon,F}(\tilde M)$ is given by Proposition~\ref{prop:metastability} with $\cplx$, $\epsilon$, $F$ as in the hypothesis of this theorem.
Finally, $\omega(\alpha):=\alpha$ and $M_{\bullet}:=M$.
The structure theorem provides a decomposition
\begin{equation}
\label{eq:dec-fj}
f_{j} = \sum_{t}\lambda_{t}\sigma_{t} + u + v,
\end{equation}
where $\sum_{t}|\lambda_{t}| < C_{i}^{\delta,\eta} =: C_{i} \leq C^{*}$, $\sigma_{t} \in \Sigma_{\bfg, \gamma, B}$, $\|u\|_{M_i}^{*} < \eta(C_{i})$ and $\|v\|_{2} < \delta$.
Here $\psi(M_{i}) \leq B$, and the index $M_{i} \geq M_{\bullet}=M$ comes from the sequence \eqref{eq:decomposition-seq} that depends only on $\psi$, $M$ and $\epsilon$, and whose length $\lceil 2\delta^{-2} \rceil$ depends only on $\epsilon$.
Note that $\psi$ in turn depends only on $\cplx$, $\epsilon$ and $F$.

We will need an $L^{\infty}$ bound on $u$ in order to use the Inverse Theorem~\ref{thm:inverse}.
To this end let $S = \{ |v| \leq C_{i}\} \subset X$, then
\[
|u| 1_{S} \leq 1_{S} + \sum_{t}|\lambda_{t}|1_{S} + |v| 1_{S} \leq 3 C_{i}.
\]
Moreover, the restriction of $u$ to $S^{\complement}$ is bounded by
\[
|u| 1_{S^{\complement}} \leq 1_{S^{\complement}} + \sum_{t}|\lambda_{t}| 1_{S^{\complement}} +  |v| 1_{S^{\complement}} \leq 3 |v| 1_{S^{\complement}},
\]
so it can be absorbed in the error term $v$.
It remains to check that $\|u1_{S}\|_{M_i}^{*}$ is small.
By Chebyshev's inequality we have $C_{i}^{2} \mu(S^{\complement}) \leq \|v\|_{2}^{2}$, so that $\mu(S^{\complement})^{1/2} \leq \delta/C_{i}$.
Let now $\sigma \in \Sigma_{M_i}$ be arbitrary and estimate
\begin{multline*}
|\<u1_{S},\sigma\>|
\leq
|\<u,\sigma\>| + |\<u1_{S^{\complement}},\sigma\>|
\leq
\|u\|_{M_i}^{*} + \|u1_{S^{\complement}}\|_{2}\|\sigma 1_{S^{\complement}}\|_{2}\\
<
\eta(C_{i}) + 3 \|v\|_{2} \mu(S^{\complement})^{1/2}
\leq
\eta(C_{i}) + 3 \delta \cdot \delta/C_{i}
< 2\eta(C_{i}).
\end{multline*}
Thus (replacing $u$ by $u 1_{S}$ and $v$ by $v+u1_{S^{\complement}}$ if necessary) we may assume $\|u\|_{\infty} \leq 3 C_{i}$
at the cost of having only $\|u\|_{M_i}^{*} < 2\eta(C_{i})$ and $\|v\|_{2} \leq 4\delta < \epsilon/6$.

Now we estimate the contributions of the individual summands in \eqref{eq:dec-fj} to \eqref{eq:main-thm-est}.
The bounds
\[
\| \Av{I}[f_0, \dots, f_{j-1}, v] \|_2 \leq \frac\epsilon6
\quad\text{and}\quad
\| \Av{I'}[f_0, \dots, f_{j-1}, v] \|_2 \leq \frac\epsilon6
\]
are immediate.
Proposition~\ref{prop:metastability} for complexity $\cplx-1$ with $\tilde M = M_{i}$ (applicable since the functions $\sigma_{t}$ are uniformly $(\bfg, \gamma, \psi(M_{i}))$-reducible) shows that
\[
\Big\| \Av{I,I'}[f_0, \dots, f_{j-1}, \sum_{t}\lambda_{t}\sigma_{t}] \Big\|_2
< 8 \gamma \sum_t |\lambda_t| < \frac\epsilon3,
\]
provided that the Følner sets $I,I'$ satisfy
\[
\tilde M^{\cplx-1,\epsilon,F}_{\tilde i} \leq \lfloor I\rfloor,\lfloor I'\rfloor
\text{ and }
\lceil I, I'\rceil_{\gamma^{\cplx}(\epsilon)} \leq F(\tilde M^{\cplx-1,\epsilon,F}_{\tilde i})
\]
for some $\tilde M^{\cplx-1,\epsilon,F}_{\tilde i}$ that belongs to the tuple \eqref{eq:main-prop-seq} given by the same proposition.
The former condition implies in particular $M_{i} \leq \tilde M^{\cplx-1,\epsilon,F}_{\tilde i} \leq \lfloor I\rfloor, \lfloor I'\rfloor$,
and in this case the Inverse Theorem~\ref{thm:inverse} shows that
\[
\| \Av{I}[f_0, \dots, f_{j-1}, u] \|_2 \leq \frac\epsilon6
\quad\text{and}\quad
\| \Av{I'}[f_0, \dots, f_{j-1}, u] \|_2 \leq \frac\epsilon6,
\]
since otherwise there exists a uniformly $(\bfg, \gamma, M_i)$-reducible function $\sigma$ such that $\<u,\sigma\> > 2\eta(C_i)$.

We obtain the conclusion of the theorem with the tuple \eqref{eq:main-thm-seq} being the concatenation of the tuples \eqref{eq:main-prop-seq} provided by Proposition~\ref{prop:metastability} with $\tilde M=M_{i} \geq M$ for $1 \leq i \leq \lceil 2\delta^{-2} \rceil$.
In particular, $K_{\epsilon,\cplx} = \lceil 2\delta^{-2} \rceil \tilde K_{\epsilon,\cplx-1}$.
\end{proof}
This completes the induction and thus the proof of Proposition~\ref{prop:metastability} and Theorem~\ref{thm:metastability}.
The latter theorem implies the following convergence result whose proof has been already outlined in the discussion of the von Neumann mean ergodic theorem.
The proof is nevertheless included for completeness.
\begin{corollary}
\label{cor:main}
Let $\bfg = (g_{0},\dots,g_{j})$ be a system with finite complexity and $f_{0},\dots,f_{j} \in L^{\infty}(X)$ be bounded functions.
Then for every Følner net $(\Fo_{\indx})_{\indx\in\iset}$ in $\AG$ the limit
\begin{equation}
\label{eq:average}
\lim_{\lfloor I \rfloor \in\iset} \Av{I}[f_0, \dots, f_j]
\end{equation}
exists in $L^{2}(X)$ and is independent of the Følner sequence $(\Fo_{N})_{N}$.
\end{corollary}
\begin{proof}
We may assume that the functions $f_0,\dots,f_j$ are bounded by one and $\complexity\bfg \leq \cplx$ for some $\cplx<\infty$.
We use the abbreviations
\[
\gf(m) := g_{0}(m) f_{0} \cdot\dots\cdot g_{j}(m) f_{j}, \quad
\gf(I) := \ave{m}{I}\gf(m) = \Av{I}[f_0, \dots, f_j].
\]
Assume that the functions $\gf(I)$ do not converge in $L^{2}(X)$ along $\lfloor I \rfloor\in\iset$.
Then there exists an $\epsilon>0$ such that for every $M\in\iset$ there exist Følner sets $I,I'$ such that $M \leq \lfloor I \rfloor, \lfloor I' \rfloor$ and
\[
\|\gf(I)-\gf(I')\|_{2} = \| \Av{I,I'}[f_0, \dots, f_j] \|_2 > \epsilon.
\]
This contradicts Theorem~\ref{thm:metastability} with $F(M) := \lceil I, I' \rceil_{\gamma^{\cplx}(\epsilon)}$.
Therefore the limit
\[
\gf(\AG) := \lim_{\lfloor I \rfloor \in\iset} \gf(I)
\]
exists.
The uniqueness is clear for limits along Følner sequences since any two such sequences are subsequences of some other Følner sequence.
The advantage of the averaging argument below is that it also works for nets.

Let $(\Fo_{\indx'}')_{\indx'\in\iset'}$ be another Følner sequence in $\AG$.
Let $\epsilon>0$ be given and take $N\in\iset$ so large that
\[
\left\| \gf(I) - \gf(\AG) \right\|_2 < \epsilon
\]
whenever $\lfloor I \rfloor \geq N$.

Let $N'\in\iset'$ be so large that $|n \Fo_{N'}' \Delta \Fo_{N'}'| < \epsilon |\Fo_{N'}'|$ for every $n \in \Fo_{N}$.
Then
\[
\left\| \ave{m}{\Fo_{N'}'} \gf(m) - \ave{m}{n \Fo_{N'}'} \gf(m) \right\|_\infty < \epsilon
\]
for every $n\in \Fo_{N}$. Writing $O(\epsilon)$ for an error term that is bounded by $\epsilon$ in $L^{2}(X)$ we obtain
\begin{align*}
\ave{m}{\Fo_{N'}'} \gf(m)
&=
\ave{n}{\Fo_N} \ave{m}{\Fo_{N'}'} \gf(m)\\
&=
\ave{n}{\Fo_N} \ave{m}{\Fo_{N'}'} \gf(n m) + O(\epsilon)\\
&=
\ave{m}{\Fo_{N'}'} \ave{n}{\Fo_N} \gf(n m) + O(\epsilon)\\
&=
\ave{m}{\Fo_{N'}'} \ave{n}{\Fo_N m} \gf(n) + O(\epsilon)\\
&=
\ave{m}{\Fo_{N'}'} \gf(\AG) + O(2\epsilon)\\
&=
\gf(\AG) + O(2\epsilon).
\end{align*}
Since $\epsilon$ is arbitrary, the averages $\gf(\Fo_{N'}')$ converge to $\gf(\AG)$ for $N'\in\iset'$.
\end{proof}
Theorem~\ref{thm:polynomial-norm-convergence} follows immediately from Theorem~\ref{thm:finite-complexity} applied to the polynomial system $\bfg = (g_{0},\dots,g_{j})$, where $g_{0} \equiv 1_{G}$, and Corollary~\ref{cor:main}.

\section{Right polynomials and commuting group actions}
\label{app:comm-act}
An inspection reveals that every occurrence of $a\in\AG$ and related objects $a',a_{i},a^{t},a_{N},a_{N}'\in\AG$ in this chapter could be replaced by $1_{\AG}$ (in fact we could restrict their values to any subgroup of $\AG$, but we do not use this).
This leads to the notion of right translation and right derivative of a mapping $g\from\AG\to G$ that are defined by
\[
T_{b}g(n)=g(nb)
\quad\text{and}\quad
D_{b}g(n)=g(n)\inv g(nb),
\]
respectively.
Right polynomials, right reduction $\<\cdot,\cdot\>_{b}$ and $\cdot^{*}_{b}$ and right complexity of systems are defined similarly to polynomials, reduction and complexity, respectively, with right derivatives in place of derivatives.
Right Følner sets are sets of the form $\Fo_{N}b$.

With these definitions we obtain an analog of Corollary~\ref{cor:main} with complexity replaced by right complexity and Følner sets replaced by right Følner sets and an analog of Theorem~\ref{thm:finite-complexity} with polynomials replaced by right polynomials.
Together they immediately imply the following analog of Theorem~\ref{thm:polynomial-norm-convergence} for right polynomials.
\begin{theorem}
\label{thm:right-polynomial-norm-convergence}
Let $g_{1}, \dots, g_{j} \from\AG\to G$ be measurable right polynomial mappings and $f_{1},\dots,f_{j} \in L^{\infty}(X)$ be arbitrary bounded functions.
Then for every Følner net $(\Fo_{\indx})_{\indx\in A}$ in $\AG$ the limit
\begin{equation}
\label{eq:right-polynomial-average}
\lim_{\indx\in\iset} \ave{m}{\Fo_{\indx}} g_{1}(m) f_{1} \cdot\dots\cdot g_{j}(m) f_{j}
\end{equation}
exists in $L^{2}(X)$ and is independent of the Følner net $(\Fo_{\indx})_{\indx\in A}$.
\end{theorem}

A second application of the analog of Corollary~\ref{cor:main} described above deals with commuting actions of $\AG$ without any further assumptions on the group $G$ generated by the corresponding unitary operators.
Note that a group action $\tau_{i}$ gives rise to an \emph{antihomomorphism} $g_{i}\from\AG\to G$, $g_{i}(m)f = f\circ \tau_{i}(m)$, that is, a mapping such that $g_{i}(nb)=g_{i}(b)g_{i}(n)$ for every $n,b\in\AG$.

\begin{proposition}
\label{prop:comm-antihom}
Let $g_{0}\equiv 1_{G}$ and $g_{1},\dots,g_{j} \from\AG\to G$ be antihomomorphisms that commute pairwise in the sense that $g_{i}(n)g_{k}(b)=g_{k}(b)g_{i}(n)$ for every $n,b\in\AG$ provided $i\neq k$.
Then the system $(g_{0},g_{0}g_{1},\dots,g_{0}\dots g_{j})$ has right complexity at most $j$.
\end{proposition}
\begin{proof}
Every antihomomorphism $g_{i}\from\AG\to G$ satisfies
\[
D_{b}(g_{i}\inv)(n) = g_{i}(n) g_{i}(nb)\inv = g_{i}(n) (g_{i}(b)g_{i}(n))\inv = g_{i}(b)\inv
\]
and
\[
T_{b}g_{i}(n) = g_{i}(nb) = g_{i}(b)g_{i}(n).
\]
Thus for every $i<j$ we have
\begin{align*}
\< g_{0}\dots g_{j}, g_{0}\dots g_{i} \>_{b}
&=
D_{b} ((g_{0}\dots g_{j})\inv) T_{b}(g_{0}\dots g_{i})\\
&=
(g_{0}(b)\dots g_{j}(b))\inv g_{0}(b)g_{0}\dots g_{i}(b)g_{i}\\
&=
g_{0}\dots g_{i} g_{i+1}(b)\inv \dots g_{j}(b)\inv.
\end{align*}
Since $g_{i+1}(b)\inv \dots g_{j}(b)\inv \in G$ is a constant, we obtain
\[
\complexity (g_{0},g_{0}g_{1},\dots,g_{0}\dots g_{j})^{*}_{b}
=
\complexity (g_{0},g_{0}g_{1},\dots,g_{1}\dots g_{j-1})
\]
by cheating.
We can conclude by induction on $j$.
\end{proof}
Proposition~\ref{prop:comm-antihom} and the analog of Corollary~\ref{cor:main} for right complexity have the following immediate consequence.
\begin{theorem}
\label{thm:commuting-actions-norm-convergence}
Let $\tau_{1}, \dots, \tau_{j}$ be measure-preserving actions of $\AG$ on $X$ that commute pairwise in the sense that
\[
\tau_{i}(m)\tau_{k}(n) = \tau_{k}(n)\tau_{i}(m)
\quad\text{whenever}\quad
m,n\in\AG, \quad i\neq k
\]
and $f_{0},\dots,f_{j} \in L^{\infty}(X)$ be arbitrary bounded functions.
Then for every Følner net $(\Fo_{\indx})_{\indx\in A}$ in $\AG$ the limit
\begin{equation}
\label{eq:commuting-actions-average}
\lim_{\indx\in\iset} \ave{m}{\Fo_{\indx}} f_{0}(x) f_{1}(\tau_{1}(m)x) \cdot\dots\cdot f_{j}(\tau_{1}(m)\dots \tau_{j}(m)x)
\end{equation}
exists in $L^{2}(X)$ and is independent of the Følner net $(\Fo_{\indx})_{\indx\in A}$.
\end{theorem}
This result generalizes the double ergodic theorem for commuting actions of an amenable group due to Bergelson, McCutcheon, and Zhang \cite[Theorem 4.8]{MR1481813}.

%%% Local Variables: 
%%% mode: latex
%%% TeX-master: "phd-thesis.tex"
%%% End:
\chapter{Recurrence}
\label{chap:recurrence}
Furstenberg's ergodic theoretic proof \cite{MR0498471} of Szemer\'edi's theorem on arithmetic progressions \cite{MR0369312} has led to various generalizations of the latter.
Recall that Furstenberg's original multiple recurrence theorem provides a syndetic set of return times.
The \emph{IP} recurrence theorem of Furstenberg and Katznelson \cite{MR833409}, among other things, improves this to an IP* set.
The idea to consider the limit behavior of a multicorrelation sequence not along a F\o{}lner sequence but along an IP-ring has proved to be very fruitful and allowed them to obtain the density Hales-Jewett theorem \cite{MR1191743}.

In a different direction, Bergelson and Leibman \cite{MR1325795} have proved a \emph{polynomial} multiple recurrence theorem.
The set of return times in this theorem was shown to be syndetic by Bergelson and McCutcheon \cite{MR1411223}.
That result has been extended from commutative to nilpotent groups of transformations by Leibman \cite{MR1650102}.
Many of the additional difficulties involved in the nilpotent extension were algebraic in nature and have led Leibman to develop a general theory of polynomial mappings into nilpotent groups \cite{MR1910931}.
An important aspect of the proofs of these polynomial recurrence theorems, being present in all later extensions including the present one, is that the induction process involves ``multiparameter'' recurrence even if one is ultimately only interested in the ``one-parameter'' case.

More recently an effort has been undertaken to combine these two directions.
Building on their earlier joint work with Furstenberg \cite{MR1417769}, Bergelson and McCutcheon \cite{MR1692634} have shown the set of return times in the polynomial multiple recurrence theorem is IP*.
Joint extensions of their result and the IP recurrence theorem of Furstenberg and Katznelson have been obtained by Bergelson, H{\aa}land Knutson and McCutcheon for single recurrence \cite{MR2246589} and McCutcheon for multiple recurrence \cite{MR2151599}.
The results of the last two papers also provide multiple recurrence along admissible generalized polynomials (Definition~\ref{def:generalized-poly}), and, more generally, along FVIP-systems (Definition~\ref{def:fvip}).

In \cite{arxiv:1206.0287} we continue this line of investigation.
Our Theorem~\ref{thm:SZ-general} generalizes McCutcheon's IP polynomial multiple recurrence theorem to the nilpotent setting.
Its content is best illustrated by the following generalization of Leibman's nilpotent multiple recurrence theorem (here and throughout this chapter group actions on topological spaces and measure spaces are on the right and on function spaces on the left.).
\begin{theorem}
\label{thm:SZ-intro}
Let $T_{1},\dots,T_{t}$ be invertible measure-preserving transformations on a probability space $(X,\mathcal{A},\mu)$ that generate a nilpotent group.
Then for every $A\in\mathcal{A}$ with $\mu(A)>0$, every $m\in\N$, and any admissible generalized polynomials $p_{i,j}:\Z^{m}\to\Z$, $i=1,\dots,t$, $j=1,\dots,s$, the set
\begin{equation}
\label{eq:SZ-intro}
\Big\{ \vec n\in\Z^{m} : \mu\big( \bigcap_{j=1}^{s} A \big(\prod_{i=1}^{t}T_{i}^{p_{i,j}(\vec n)} \big)\inv \big) > 0 \Big\}
\end{equation}
is FVIP* in $\Z^{m}$, that is, it has nontrivial intersection with every FVIP-system in $\Z^{m}$.
\end{theorem}
In particular, the set \eqref{eq:SZ-intro} is IP*, so that it is syndetic \cite[Lemma 9.2]{MR603625}.
The class of admissible generalized polynomials contains ordinary integer polynomials that vanish at zero, for further examples see e.g.\ \eqref{eq:generalized-poly-examples}.
By the Furstenberg correspondence principle we obtain the following combinatorial corollary.
\begin{corollary}
\label{cor:SZ-combinatorial}
Let $G$ be a finitely generated nilpotent group, $T_{1},\dots,T_{t}\in G$, and $p_{i,j}:\Z^{m}\to\Z$, $i=1,\dots,t$, $j=1,\dots,s$, be admissible generalized polynomials.
Then for every subset $E\subset G$ with positive upper Banach density the set
\[
\Big\{ \vec n\in\Z^{m} : \exists g\in G : g\prod_{i=1}^{t}T_{i}^{p_{i,j}(\vec n)}\in E, j=1,\dots,s \Big\}
\]
is FVIP* in $\Z^{m}$.
\end{corollary}

\section{Topological multiple recurrence}
\label{sec:hales-jewett}
In this section we refine the nilpotent Hales-Jewett theorem due to Bergelson and Leibman \cite[Theorem 0.19]{MR1972634} using the induction scheme from \cite[Theorem 3.4]{MR1715320}.
This allows us to deduce a \emph{multiparameter} nilpotent Hales-Jewett theorem that will be ultimately applied to polynomial-valued polynomials mappings.

\subsection{PET induction}
First we describe the PET (polynomial exhaustion technique) induction scheme \cite{MR912373}.
\index{PET induction}
For a polynomial $g\in\VIP$ define its \emph{level} $l(g)$ as the greatest integer $l$ such that $g\in\VIP[{\Gb[+l]}]$.
We define an equivalence relation on the set of non-zero $\Gb$-polynomials by $g\sim h$ if and only if $l(g)=l(h)<l(g\inv h)$.
Transitivity and symmetry of $\sim$ follow from Theorem~\ref{thm:poly-group}.

\begin{definition}
A \emph{system} is a finite subset $A\subset\VIP$.
The \emph{weight vector} of a system $A$ is the function
\index{weight vector}
\[
l\mapsto\text{the number of equivalence classes modulo }\sim\text{ of level }l\text{ in }A.
\]
The lexicographic ordering is a well-ordering on the set of weight vectors and the \emph{PET induction} is induction with respect to this ordering.
\end{definition}
\begin{proposition}
\label{prop:PET}
Let $A$ be a system, $h\in A$ be a mapping of maximal level and $B\subset G_{1}$, $M\subset\Fin$ be finite sets.
Then the weight vector of the system
\[
A'' = \{ b h\inv g \sD_{\alpha}g b\inv, \quad g\in A, \alpha\in M, b\in B \} \setminus \{1_{G}\}
\]
precedes the weight vector of $A$.
\end{proposition}
\begin{proof}
We claim first that the weight vector of the system
\[
A' = \{ h\inv g \sD_{\alpha}g, \quad \alpha\in M, g\in A \} \setminus \{1_{G}\}
\]
precedes the weight vector of $A$.
Indeed, if $l(g)<l(h)$, then $g\sim h\inv g \sD_{\alpha}g$.
If $l(g)=l(\tilde g)=l(h)$ and $g\sim \tilde g \not\sim h$, then $h\inv g \sD_{\alpha}g \sim h\inv \tilde g \sD_{\tilde\alpha}\tilde g$.
Finally, if $g\sim h$, then $l(h\inv g \sD_{\alpha}g)>l(h)$.
Thus the weight vector of $A'$ does not differ from the weight of vector of $A$ before the $l(h)$-th position and is strictly smaller at the $l(h)$-th position, as required.

We now claim that the weight vector of the system
\[
A'' = \{ b gb\inv, \quad g\in A', b\in B \}
\]
coincides with the weight vector of the system $A'$.
Indeed, this follows directly from
\[
bgb\inv = g[g,b\inv] \sim g.
\qedhere
\]
\end{proof}

\subsection{Nilpotent Hales-Jewett theorem}
The following refined version of the nilpotent IP polynomial topological mutiple recurrence theorem due to Bergelson and Leibman \cite[Theorem 0.19]{MR1972634} does not only guarantee the existence of a ``recurrent'' point, but also allows one to choose it from a finite subset $xS$ of any given orbit.
\begin{theorem}[Nilpotent Hales-Jewett]
\label{thm:nil-hj}
\index{theorem!nilpotent Hales-Jewett}
Assume that $G$ acts on the right on a compact metric space $(X,\rho)$ by homeomorphisms.
For every system $A$, every $\epsilon>0$ and every $H\in\Fin$ there exists $N\in\Fin$, $N>H$, and a finite set $S\subset G$ such that for every $x\in X$ there exist a non-empty $\alpha\subset N$ and $s\in S$ such that $\rho(xsg(\alpha),xs)<\epsilon$ for every $g\in A$.
\end{theorem}
Here we follow Bergelson and Leibman and use ``Hales-Jewett'' as a shorthand for ``IP topological multiple recurrence'', although Theorem~\ref{thm:nil-hj} does not imply the classical Hales-Jewett theorem on monochrome combinatorial lines.
The fact that Theorem~\ref{thm:nil-hj} does indeed generalize \cite[Theorem 0.19]{MR1972634} follows from Corollary~\ref{cor:set-monomials} that substitutes \cite[\textsection 1--2]{MR1972634}.

The reason that Theorem~\ref{thm:nil-hj} does not imply the classical Hales-Jewett theorem is that it does not apply to semigroups.
However, it is stronger than van der Waerden-type topological recurrence results, since it makes no finite generation assumptions.
We refer to \cite[\textsection 5.5]{MR1972634} and \cite[\textsection 3.3]{MR1715320} for a discussion of these issues.
It would be interesting to extend Theorem~\ref{thm:nil-hj} to nilpotent semigroups (note that nilpotency of a group can be characterized purely in terms of semigroup relations).
\begin{proof}
We use PET induction on the weight vector $w(A)$.
If $w(A)$ vanishes identically, then $A$ is the empty system and there is nothing to show.
Assume that the conclusion is known for every system whose weight vector precedes $w(A)$.
Let $h\in A$ be an element of maximal level, without loss of generality we may assume $h\not\equiv 1_{G}$.
Let $k$ be such that every $k$-tuple of elements of $X$ contains a pair of elements at distance $<\epsilon/2$.

We define finite sets $H_{i}\in\Fin$, finite sets $B_{i},\tilde B_{i}\subset G$, systems $A_{i}$ whose weight vector precedes $w(A)$, positive numbers $\epsilon_{i}$, and finite sets $N_{i}\in\Fin$ by induction on $i$ as follows.
Begin with $H_{0}:=H$ and $B_{0}=\tilde B_{0}=\{1_{G}\}$.
The weight vector $w(A_{i})$ of the system
\[
A_{i} := \{b h\inv g \sD_{m}g b\inv,
\quad g\in A,m\subset N_{0}\cup\dots\cup N_{i-1},b\in B_{i}\}
\]
precedes $w(A)$ by Proposition~\ref{prop:PET}.
By uniform continuity we can choose $\epsilon_{i}$ such that
\[
\rho(x,y)< \epsilon_{i} \implies \forall \tilde b\in \tilde B_{i} \quad \rho(x \tilde b,y \tilde b)<\frac{\epsilon}{2k}.
\]
By the induction hypothesis there exists $N_{i}\in\Fin$, $N_{i}>H_{i}$, and a finite set $S_{i}\subset G$ such that
\begin{equation}
\label{eq:phj-ind}
\forall x\in X \quad \exists n_{i}\subset N_{i}, s_{i}\in S_{i} \quad \forall g\in A_{i}
\quad \rho(xs_{i}g(n_{i}),xs_{i}) < \epsilon_{i}.
\end{equation}
Finally, let $H_{i+1}:=H_{i} \cup N_{i}$ and
\[
B_{i+1} := \{ s_{i}b h(\alpha_{i})\inv,
\quad \alpha_{i}\subset N_{i},s_{i}\in S_{i},b\in B_{i}\} \subset G,
\]
\[
\tilde B_{i+1} := \{ bg(m),
\quad g\in A,m \subset N_{0}\cup\dots\cup N_{i},b\in B_{i+1}\} \subset G.
\]
This completes the inductive definition.
Now fix $x\in X$.
We define a sequence of points $y_{i}$ by descending induction on $i$.
Begin with $y_{k}:=x$.
Assume that $y_{i}$ has been chosen and choose $n_{i}\subset N_{i}$ and $s_{i}\in S_{i}$ as in \eqref{eq:phj-ind}, then set $y_{i-1}:=y_{i}s_{i}$.

Finally, let $x_{0}:=y_{0}s_{0}h(n_{0})\inv$ and $x_{i+1}:=x_{i}h(n_{i+1})\inv$.
We claim that for every $g\in A$ and any $0\leq i\leq j\leq k$ we have
\begin{equation}
\label{eq:phj-inner-ind}
\rho(x_{j}g(n_{i+1}\cup\dots\cup n_{j}),x_{i}) < \frac{\epsilon}{2k}(j-i).
\end{equation}
This can be seen by ascending induction on $j$.
Let $i$ be fixed, the claim is trivially true for $j=i$.
Assume that the claim holds for $j-1$ and let $g$ be given.
Consider
\[
b:=s_{j-1}\dots s_{0}h(n_{0})\inv \dots h(n_{j-1})\inv \in B_{j}
\quad \text{and}
\]
\[
\tilde b:=b g(n_{i+1}\cup\dots\cup n_{j-1}) \in \tilde B_{j}.
\]
By choice of $n_{j}$ and $s_{j}$ we have
\[
\rho(y_{j}s_{j}b h(n_{j})\inv g(n_{i+1}\cup\dots\cup n_{j}) g(n_{i+1}\cup\dots\cup n_{j-1})\inv b\inv, y_{j}s_{j} ) < \epsilon_{j}.
\]
By definition of $\epsilon_{j}$ this implies
\[
\rho(y_{j}s_{j}b h(n_{j})\inv g(n_{i+1}\cup\dots\cup n_{j}) g(n_{i+1}\cup\dots\cup n_{j-1})\inv b\inv \tilde b, y_{j}s_{j} \tilde b ) < \frac{\epsilon}{2k}.
\]
Plugging in the definitions we obtain
\[
\rho(x_{j}g(n_{i+1}\cup\dots\cup n_{j}), x_{j-1}g(n_{i+1}\cup\dots\cup n_{j-1}) ) < \frac{\epsilon}{2k}.
\]
The induction hypothesis then yields
\begin{multline*}
\rho(x_{j}g(n_{i+1}\cup\dots\cup n_{j}),x_{i})\\
\leq
\rho(x_{j}g(n_{i+1}\cup\dots\cup n_{j}), x_{j-1}g(n_{i+1}\cup\dots\cup n_{j-1}))
+
\rho(x_{j-1}g(n_{i+1}\cup\dots\cup n_{j-1}),x_{i})\\
<
\frac{\epsilon}{2k}
+
\frac{\epsilon}{2k}((j-1)-i)
=
\frac{\epsilon}{2k}(j-i)
\end{multline*}
as required.

Recall now that by definition of $k$ there exist $0\leq i<j\leq k$ such that $\rho(x_{i},x_{j})<\frac{\epsilon}{2}$.
By \eqref{eq:phj-inner-ind} we have
\[
\rho(x_{j}g(n_{i+1}\cup\dots\cup n_{j}),x_{j})
\leq
\rho(x_{j}g(n_{i+1}\cup\dots\cup n_{j}),x_{i})+\rho(x_{i},x_{j})
<
\epsilon
\]
for every $g\in A$.
But $x_{j}=xs$ for some
\[
s \in S := S_{k}\dots S_{0}h(\Fin(N_{0}))\inv \dots h(\Fin(N_{k}))\inv,
\]
and we obtain the conclusion with $N=N_{0}\cup\dots\cup N_{k}$ and $S$ as above.
\end{proof}
We remark that \cite[Theorem 3.4]{MR1715320} provides a slightly different set $S$ that can be recovered substituting $y_{k}:=xh(N_{k})$ and $y_{i-1}:=y_{i}s_{i}h(N_{i-1})$ in the above proof and making the corresponding adjustments to the choices of $B_{i}$, $b$ and $S$.

\subsection{Multiparameter nilpotent Hales-Jewett theorem}
We will now prove a version of the nilpotent Hales-Jewett theorem in which the polynomial configurations may depend on multiple parameters $\alpha_{1},\dots,\alpha_{m}$.
\begin{theorem}[Multiparameter nilpotent Hales-Jewett]
\label{thm:multi-nil-hj}
\index{theorem!multiparameter nilpotent Hales-Jewett}
Assume that $G$ acts on the right on a compact metric space $(X,\rho)$ by homeomorphisms and let $m\in\N$.
For every finite set $A\subset\PE{\VIP}{m}$, every $\epsilon>0$ and every $H\in\Fin$ there exists a finite set $N\in\Fin$, $N>H$, and a finite set $S\subset G$ such that for every $x\in X$ there exists $s\in S$ and non-empty subsets $\alpha_{1}<\dots<\alpha_{m}\subset N$ such that $\rho(xsg(\alpha_{1},\dots,\alpha_{m}),xs)<\epsilon$ for every $g\in A$.
\end{theorem}
\begin{proof}
We use induction on $m$.
The base case $m=0$ is trivial.
Assume that the conclusion is known for some $m$, we prove it for $m+1$.

Let $A\subset\PE{\VIP}{m+1}$ and $H$ be given.
For convenience we write $\vec\alpha=(\alpha_{1},\dots,\alpha_{m})$ and $\alpha=\alpha_{m+1}$.
By definition each $g\in A$ can be written in the form
\[
g(\alpha_{1},\dots,\alpha_{m+1}) = g_{2}^{\vec\alpha}(\alpha) g_{1}(\vec\alpha)
\]
with $g_{1}\in\PE{\VIP}{m}$ and $g_{2}^{\vec\alpha} \in \VIP$.

We apply the induction hypothesis with the system $\{g_{1}, g\in A\}$ and $\epsilon/2$, thereby obtaining a finite set $N\in\Fin$, $N>H$, and a finite set $S\subset G$.
We write ``$\vec\alpha\subset N$'' instead of ``$\alpha_{1}<\dots<\alpha_{m}\subset N$''.

By uniform continuity there exists $\epsilon'$ such that
\[
\rho(x,y)<\epsilon' \implies
\forall s\in S, \vec\alpha\subset N,g\in A \quad
\rho(xsg_{1}(\vec\alpha),ysg_{1}(\vec\alpha))<\epsilon/2.
\]

We invoke Theorem~\ref{thm:nil-hj} with the system $\{ s g_{2}^{\vec\alpha} s\inv, s\in S, \vec\alpha\subset N, g\in A\}$ and $\epsilon'$, this gives us a finite set $N'\in\Fin$, $N'>N$, and a finite set $S'\subset G$ with the following property:
for every $x\in X$ there exist $s'\in S'$ and $\alpha\subset N'$ such that
\[
\forall s\in S, \vec\alpha\subset N, g\in A \quad
\rho(xs's g_{2}^{\vec\alpha}(\alpha)s\inv,xs') < \epsilon'.
\]
By choice of $\epsilon'$ this implies
\[
\forall s\in S, \vec\alpha\subset N, g\in A \quad
\rho(xs'sg_{2}^{\vec\alpha}(\alpha) g_{1}(\vec\alpha), xs'sg_{1}(\vec\alpha)) < \epsilon/2.
\]
By choice of $N$ and $S$, considering the point $xs'$, we can find $\vec\alpha\subset N$ and $s\in S$ such that
\[
\forall g\in A \quad \rho(xs'sg_{1}(\vec\alpha),xs's) < \epsilon/2.
\]
Combining the last two inequalities we obtain
\[
\forall g\in A \quad \rho(xs'sg(\vec\alpha,\alpha),xs's) < \epsilon.
\]
This yields the conclusion with finite sets $N\cup N'$ and $S'S$.
\end{proof}

The combinatorial version is derived using the product space construction of Furstenberg and Weiss \cite{MR531271}.
\begin{corollary}
\label{cor:color-multi-nil-hj}
Let $\Gb$ be a filtration on a countable nilpotent group $G$, $m\in\N$, $A\subset\PE{\VIP}{m}$ a finite set, and $l\in\N_{>0}$.
Then there exists $N\in\N$ and finite sets $S,T\subset G$ such that for every $l$-coloring of $T$ there exist $\alpha_{1}<\dots<\alpha_{m}\subset N$ and $s\in S$ such that the set $\{sg(\vec\alpha), g\in A\}$ is monochrome (and in particular contained in $T$).
\end{corollary}
\begin{proof}
Let $X:=l^{G}$ be the compact metrizable space of all $l$-colorings of $G$ with the right $G$-action $xg(h)=x(gh)$.
We apply Theorem~\ref{thm:multi-nil-hj} to this space, the system $A$, the set $H=\emptyset$, and an $\epsilon>0$ that is sufficiently small to ensure that $\rho(x,x')<\epsilon$ implies $x(e_{G})=x'(e_{G})$.

This yields certain $N\in\N$ and $S\subset G$ that enjoy the following property:
for every coloring $x\in X$ there exist $\alpha_{1}<\dots<\alpha_{m}\subset N$ and $s\in S$ such that $\{sg(\vec\alpha),g\in A\}$ is monochrome.
Observe that this property only involves a finite subset $T=\cup_{g\in A}Sg(\Fin(N)^{m}_{<})\subset G$.
\end{proof}
In the proof of our measurable recurrence result we will apply this combinatorial result to polynomial-valued polynomial mappings.
We encode all the required information in the next corollary.
\begin{corollary}
\label{cor:pair-color}
Let $m\in\N$, $K\leq F\leq \VIP$ be VIP groups, and $\FE\leq\PE{\VIP}{\omega}$ be a countable subgroup that is closed under substitutions $g\mapsto g[\vec\beta]$ (recall \eqref{eq:substitution}).

Then for any finite subsets $(R_{i})_{i=0}^{t}\subset\PE{K}{m}\cap\FE$ and $(W_{k})_{k=0}^{v-1}\subset\PE{F}{m}\cap\FE$ there exist $N,w\in\N$ and $(L_{i},M_{i})_{i=1}^{w}\subset(\PE{K}{N}\cap\FE)\times(\PE{F}{N}\cap\FE)$ such that for every $l$-coloring of the latter set there exists an index $a$ and sets $\beta_{1}<\dots<\beta_{m}\subset N$ such that the set $(L_{a}R_{i}[\vec\beta],M_{a}W_{k}[\vec\beta]L_{a}\inv)_{i,k}$ is monochrome (and in particular contained in the set $(L_{i},M_{i})_{i=1}^{w}$).
We may assume $L_{1}\equiv 1_{G}$.
\end{corollary}
\begin{proof}
By Proposition~\ref{prop:substitution-poly} the maps $\vec\beta \mapsto (R_{i}[\vec\beta],W_{k}[\vec\beta]R_{i}[\vec\beta])$ are polynomial expressions with values in $\PE{K}{\omega}\times\PE{F}{\omega}$.
By the assumption they also take values in $\FE\times\FE$.
Given an $l$-coloring $\chi$ of $(\PE{K}{\omega}\cap\FE)\times(\PE{F}{\omega}\cap\FE)$ we pass to the $l$-coloring $\tilde\chi(g,h)=\chi(g,hg\inv)$.
Corollary~\ref{cor:color-multi-nil-hj} then provides the desired $N$ and $(L_{i},M_{i})_{i=1}^{w}=T\cup S$.
\end{proof}

\section{FVIP groups}
\label{sec:fvip}
For reasons that will become clear shortly, our ergodic multiple recurrence result is restricted to a certain class of VIP systems with the following finite generation property.
\begin{definition}
\label{def:fvip}
An \emph{FVIP group} is a finitely generated VIP group.
\index{FVIP group}
An \emph{FVIP system} is a member of some FVIP group.
\index{FVIP system}
\end{definition}
The main result about FVIP groups is the following nilpotent version of \cite[Theorem 1.8]{MR1417769} and \cite[Theorem 1.9]{MR2246589} that will be used to construct ``primitive extensions'' (we will recall the definitions of a primitive extension and an IP-limit in due time).
\begin{theorem}
\label{thm:fvip-proj}
Let $\Gb$ be a prefiltration of finite length and $F\leq\VIP$ be an FVIP group.
Suppose that $G_{0}$ acts on a Hilbert space $H$ by unitary operators and that for each $(g_{\alpha})_{\alpha}\in F$ the weak limit $P_{g}=\wIPlim_{\alpha\in\Fin} g_{\alpha}$ exists.
Then
\begin{enumerate}
\item each $P_{g}$ is an orthogonal projection and
\item these projections commute pairwise.
\end{enumerate}
\end{theorem}
The finite generation assumption cannot be omitted in view of a counterexample in \cite{MR1417769}.

\subsection{Partition theorems for IP-rings}
An \emph{IP-ring} is a subset of $\Fin$ that consists of all finite unions of a given strictly increasing chain $\alpha_{0}<\alpha_{1}<\dots$ of elements of $\Fin$ \cite[Definition 1.1]{MR833409}.
\index{IP-ring}
In particular, $\Fin$ is itself an IP-ring (associated to the chain $\{0\}<\{1\}<\dots$).
Polynomials are generally assumed to be defined on $\Fin$ even if we manipulate them only on some sub-IP-ring of $\Fin$.

Since we will be dealing a lot with assertions about sub-IP-rings we find it convenient to introduce a shorthand notation.
If some statement holds for a certain sub-IP-ring $\Fin'\subset\Fin$ then we say that it holds \emph{without loss of generality} (wlog).
In this case we reuse the symbol $\Fin$ to denote the sub-IP-ring on which the statement holds (in particular this IP-ring may change from use to use).
With this convention the basic Ramsey-type theorem about IP-rings reads as follows.
\begin{theorem}[Hindman \cite{MR0349574}]
\label{thm:hindman}
\index{theorem!Hindman's}
Every finite coloring of $\Fin$ is wlog monochrome.
\end{theorem}
This is not the same as the assertion ``wlog every finite coloring of $\Fin$ is monochrome'', since the latter would mean that there exists a sub-IP-ring on which \emph{every} coloring is monochrome.

As a consequence of Hindman's theorem~\ref{thm:hindman}, a map from $\Fin$ to a compact metric space for every $\epsilon>0$ wlog has values in an $\epsilon$-ball.
As the next lemma shows, for polynomial maps into compact metric groups the ball can actually be chosen to be centered at the identity.
In a metric group we denote the distance to the identity by $\dint{\cdot}$.
\begin{lemma}
\label{lem:near-identity}
Let $\Gb$ be a prefiltration in the category of compact metric groups and $P \in \VIP$.
Then for every $\epsilon>0$ we have wlog $\dint{P} < \epsilon$.
\end{lemma}
\begin{proof}
We use induction on the length of the prefiltration $\Gb$.
If the prefiltration is trivial, then there is nothing to show, so assume that the conclusion is known for $\Gb[+1]$.

Let $\delta,\delta'>0$ be chosen later.
By compactness and Hindman's theorem~\ref{thm:hindman} we may wlog assume that the image $P(\Fin)$ is contained in some ball $B(g,\delta)$ with radius $\delta$ in $G_{1}$.
By uniform continuity of the group operation we have $\sD_{\beta}P(\alpha) \in B(g\inv,\delta')$ for any $\alpha>\beta\in\Fin$ provided that $\delta$ is small enough depending on $\delta'$.
On the other hand, for a fixed $\beta$, by the induction hypothesis we have wlog $\dint{\sD_{\beta}P} < \delta'$, so that $\dint{g\inv}<2\delta'$.
By continuity of inversion this implies $\dint{g}<\epsilon/2$ provided that $\delta'$ is small enough.
This implies $\dint{P}<\epsilon$ provided that $\delta$ is small enough.
\end{proof}

\begin{corollary}[{\cite[Proposition 1.1]{MR2246589}}]
\label{cor:poly-subg}
Let $\Wb$ be a prefiltration, $A\subset \VIP[\Wb]$ be finite, and $V\leq W$ be a finite index subgroup.
Then wlog for every $g\in A$ we have $g(\Fin) \subset V$.
\end{corollary}
\begin{proof}
Let $g\in A$.
Passing to a subgroup we may assume that $V$ is normal.
Taking the quotient by $V$, we may assume that $W$ is finite and $V=\{1_{W}\}$.
By Lemma~\ref{lem:near-identity} with an arbitrary discrete metric we may wlog assume that $g\equiv 1_{W}$.
\end{proof}

In course of proof of Theorem~\ref{thm:fvip-proj} it will be more convenient to use a convention for the symmetric derivative that differs from \eqref{eq:symm-der}, namely
\[
\rD_{\alpha}g(\beta) = g(\alpha)\inv D_{\alpha}g(\beta).
\]
Clearly a VIP group is also closed under $\rD$.
\begin{lemma}
\label{lem:deri-poly}
Let $F$ be a VIP group, $W\leq F$ be a subgroup and $V\leq W$ be a finite index subgroup.
Suppose that $g\in F$ is such that the symmetric derivative $\rD_{\alpha}g\in W$ for all $\alpha$.
Then wlog for every $\alpha$ the symmetric derivative $\rD_{\alpha}g$ coincides with an element of $V$ on some sub-IP-ring of the form $\{\beta\in\Fin : \beta>\beta_{0}\}$.
\end{lemma}
\begin{proof}
Since $V$ has finite index and by Hindman's theorem~\ref{thm:hindman} we can wlog assume that $\rD_{\alpha}g\in w\inv V$ for some $w\in W$ and all $\alpha$.
Assume that $w\not\in V$.
Let
\[
h(\alpha) := w \rD_{\alpha}g =
\begin{cases}
w, & \alpha=\emptyset\\
v_{\alpha}\in V & \text{otherwise}.
\end{cases}
\]
Let $\alpha_{1}<\dots<\alpha_{d}$ be non-empty, by induction on $d$ we see that $D_{\alpha_{d}}\dots D_{\alpha_{1}}h(\alpha) \in V$ for all $\alpha\neq\emptyset$ and $D_{\alpha_{d}}\dots D_{\alpha_{1}}h(\emptyset) \in V w^{(-1)^{d}}V$.

On the other hand the map $\alpha\mapsto h(\alpha)(\beta)$ is $\Gb[+1]$-polynomial on $\{\alpha : \alpha\cap\beta=\emptyset\}$ for fixed $\beta$.
Therefore $D_{\alpha_{d}}\dots D_{\alpha_{1}}h(\emptyset)$ vanishes at all $\beta > \alpha_{d}$, that is, $w$ coincides with an element of $V$ on $\{\beta : \beta>\alpha_{d}\}$.
\end{proof}
It is possible to see Lemma~\ref{lem:deri-poly} (and Lemma~\ref{lem:notinK} later on) as a special case of Corollary~\ref{cor:poly-subg} by considering the quotient of $\VIP$ by the equivalence relation of equality on IP-rings of the form $\{\alpha:\alpha>\alpha_{0}\}$, but we prefer not to set up additional machinery.

In order to apply the above results we need a tool that provides us with finite index subgroups.
To this end recall the following multiparameter version of Hindman's theorem~\ref{thm:hindman}.
\begin{theorem}[Milliken \cite{MR0373906}, Taylor \cite{MR0424571}]
\label{thm:milliken-taylor}
\index{theorem!Milliken-Taylor}
Every finite coloring of $\Fin_{<}^{k}$ is wlog monochrome.
\end{theorem}

The next lemma is a substitute for \cite[Lemma 1.6]{MR1417769} in the non-commutative case.
This is the place where the concept of Hirsch length is utilized.
\begin{lemma}
\label{lem:sfi}
Let $G$ be a finitely generated nilpotent group and $g \from\Fin\to G$ be any map.
Then wlog there exist a natural number $l>0$ and a subgroup $W\leq G$ such that for any $\alpha_{1}<\dots<\alpha_{l} \in \Fin$ the elements $g_{\alpha_{1}},\dots,g_{\alpha_{l}}$ generate a finite index subgroup of $W$.
\end{lemma}
\begin{proof}
By the Milliken-Taylor theorem~\ref{thm:milliken-taylor} we may wlog assume that for each $l\leq h(G)+1$ the Hirsch length $h(\<g_{\alpha_{1}},\dots,g_{\alpha_{l}}\>)$ does not depend on $(\alpha_{1},\dots,\alpha_{l})\in\Fin^{l}_{<}$.
Call this value $h_{l}$.
It is an increasing function of $l$ that is bounded by $h(G)$, hence there exists an $l$ such that $h_{l}=h_{l+1}$.
Fix some $(\alpha_{1},\dots,\alpha_{l})\in\Fin^{l}_{<}$ and let $V:=\<g_{\alpha_{1}},\dots,g_{\alpha_{l}}\>$.

Since $h_{l+1}=h_{l}$ and by Lemma~\ref{lem:fin-ind-hirsch}, we see that $\<V,g_{\alpha}\>$ is a finite index extension of $V$ for each $\alpha > \alpha_{l}$.
By Corollary~\ref{cor:fin-ext} and Hindman's Theorem~\ref{thm:hindman} we may wlog assume that each $g_{\alpha}$ lies in one such extension $W$.
By definition of $h_{l}$ this implies that wlog for every $(\alpha_{1},\dots,\alpha_{l})\in\Fin^{l}_{<}$ the Hirsch length of the group $\<g_{\alpha_{1}},\dots,g_{\alpha_{l}}\> \leq W$ is $h(W)$.
Hence each $\<g_{\alpha_{1}},\dots,g_{\alpha_{l}}\> \leq W$ is a finite index subgroup by Lemma~\ref{lem:fin-ind-hirsch}.
\end{proof}

\subsection{IP-limits}
Let $X$ be a topological space, $m\in\N$ and $g:\Fin_{<}^{m}\to X$ be a map.
We call $x\in X$ an \emph{IP-limit} of $g$, in symbols $\IPlim_{\vec\alpha}g_{\vec\alpha}=x$, if for every neighborhood $U$ of $x$ there exists $\alpha_{0}$ such that for all $\vec\alpha\in\Fin_{<}^{m}$, $\vec\alpha>\alpha_{0}$, one has $g_{\vec\alpha}\in U$.
\index{IP-limit}

By the Milliken-Taylor theorem~\ref{thm:milliken-taylor} and a diagonal argument, cf.\ \cite[Lemma 1.4]{MR833409}, we may wlog assume the existence of an IP-limit (even of countably many IP-limits) if $X$ is a compact metric space, see \cite[Theorem 1.5]{MR833409}.

If $X$ is a Hilbert space with the weak topology, then we write $\wIPlim$ instead of $\IPlim$ to stress the topology.

Following a tradition, we write arguments of maps defined on $\Fin$ as subscripts in this section.
We also use the notation and assumptions of Theorem~\ref{thm:fvip-proj}.

The next lemma follows from the equivalence of the weak and the strong topology on the unit sphere of $H$ and is stated for convenience.
\begin{lemma}
\label{lem:norm}
Assume that $f\in\fix P_{g}$, that is, that $\wIPlim_{\alpha}g_{\alpha}f=f$.
Then also $\IPlim_{\alpha}g_{\alpha}f=f$ (in norm).
\end{lemma}
For any subgroup $V\leq F$ we write $P_{V}$ for the orthogonal projection onto the space $\bigcap_{g\in V} \fix P_{g}$.
\begin{lemma}
\label{lem:gp}
Assume that $V=\<g_{1},\dots,g_{s}\>$ is a finitely generated group and that $P_{g_{1}},\dots,P_{g_{s}}$ are commuting projections.
Then $P_{V}=\prod_{i=1}^{s} P_{g_{i}}$.
\end{lemma}
\begin{proof}
Clearly we have $P_{V} \leq \prod_{i=1}^{s} P_{g_{i}}$, so we only need to prove that each $f$ that is fixed by $P_{g_{1}},\dots,P_{g_{s}}$ is also fixed by $P_{g}$ for any other $g\in V$.

To this end it suffices to show that if $f$ is fixed by $P_{g}$ and $P_{h}$ for some $g,h\in V$, then it is also fixed by $P_{gh\inv}$.
Lemma~\ref{lem:norm} shows that $\IPlim_{\alpha} g_{\alpha}f=f$ and $\IPlim_{\alpha} h_{\alpha}f=f$.
Since each $h_{\alpha}$ is unitary we obtain $\IPlim_{\alpha} h_{\alpha}\inv f=f$.
Since each $g_{\alpha}$ is isometric, this implies
\[
\wIPlim g_{\alpha}h_{\alpha}\inv f = \IPlim_{\alpha} g_{\alpha}h_{\alpha}\inv f = f
\]
as required.
\end{proof}

The next lemma is the main tool to ensure IP-convergence to zero.
\begin{lemma}[{\cite[Lemma 1.7]{MR1417769}}]
\label{lem:comm-proj}
\index{IP van der Corput lemma}
Let $(P_{\alpha})_{\alpha\in\Fin}$ be a family of commuting orthogonal projections on a Hilbert space $H$ and $f\in H$.
Suppose that, whenever $\alpha_{1}<\dots<\alpha_{l}$, one has $\prod_{i=1}^{l}P_{\alpha_{i}}f=0$.
Then $\IPlim_{\alpha} \|P_{\alpha}f\| = 0$.
\end{lemma}

Finally, we also need a van der Corput-type estimate.
\begin{lemma}[{\cite[Lemma 5.3]{MR833409}}]
\label{lem:vdC}
Let $(x_{\alpha})_{\alpha\in\Fin}$ be a bounded family in a Hilbert space $H$.
Suppose that
\[
\IPlim_{\beta} \IPlim_{\alpha} \<x_{\alpha},x_{\alpha\cup\beta}\>=0.
\]
Then wlog we have
\[
\wIPlim_{\alpha} x_{\alpha} = 0.
\]
\end{lemma}

\begin{proof}[Proof of Theorem~\ref{thm:fvip-proj}]
We proceed by induction on the length of the prefiltration $\Gb$.
If $\Gb$ is trivial there is nothing to prove.
Assume that the conclusion is known for $\Gb[+1]$.

First, we prove that $P_{g}$ is an orthogonal projection for any $g\in F$ (that we now fix).
Since $P_{g}$ is clearly contractive it suffices to show that it is a projection.

By Lemma~\ref{lem:sfi} we may assume that, for some $l>0$ and any $\alpha_{1}<\dots<\alpha_{l}$, the derivatives $\rD_{\alpha_{1}}g,\dots,\rD_{\alpha_{l}}g$ generate a finite index subgroup of some $W\leq F_{1}$ (recall that $F_{1}=F\cap \VIP[{\Gb[+1]}]$).
We split
\begin{equation}
\label{eq:splitting}
H = \bigcap_{V\leq W}\ker P_{V} \oplus \overline{\lin}\Big(\bigcup_{V\leq W}\im P_{V}\Big) =: H_{0} \cup H_{1},
\end{equation}
where $V$ runs over the finite index subgroups of $W$.
It suffices to show $P_{g}f=P_{g}^{2}f$ for each $f$ in one of these subspaces.

\paragraph{Case 0}
Let $f\in H_{0}$ and $\alpha_{1}<\dots<\alpha_{l}$.
By choice of $W$ we know that
\[
V:=\<\rD_{\alpha_{1}}g,\dots,\rD_{\alpha_{l}}g\> \leq W
\]
is a finite index subgroup.
Since the projections $P_{\rD_{\alpha_{i}}g}$ commute by the inductive hypothesis, their product equals $P_{V}$ (Lemma~\ref{lem:gp}), and we have $P_{V}f=0$ by the assumption.

By Lemma~\ref{lem:comm-proj} this implies $\IPlim_{\alpha} \| P_{\rD_{\alpha}g} f \|=0$.
Therefore
\begin{multline*}
\IPlim_{\alpha} \Big| \IPlim_{\beta} \< (\rD_{\alpha}g)_{\beta}f, g_{\alpha}\inv f \> \Big|
\leq
\IPlim_{\alpha} \| \wIPlim_{\beta} (\rD_{\alpha}g)_{\beta} f \|\\
=
\IPlim_{\alpha} \| P_{\rD_{\alpha}g} f \|
= 0,
\end{multline*}
so that
\[
\IPlim_{\alpha} \IPlim_{\beta} \< g_{\alpha\cup\beta}f, g_{\beta}f \> = 0.
\]
By Lemma~\ref{lem:vdC} this implies $P_{g}f=0$ (initially only wlog, but we have assumed that the limit exists on the original IP-ring).

\paragraph{Case 1}
Let $V\leq W$ and $f=P_{V}f$, by linearity we may assume $\|f\|=1$.
Let $\rho$ be a metric for the weak topology on the unit ball of $H$ with $\rho(x,y)\leq\|x-y\|$.
Let $\epsilon>0$.
By definition of IP-convergence and by uniform continuity of $P_{g}$ there exists $\alpha_{0}$ such that
\[
\forall\alpha>\alpha_{0} \quad
\rho(g_{\alpha}f,P_{g}f) < \epsilon
\text{ and }
\rho(P_{g}g_{\alpha}f,P_{g}^{2}f) < \epsilon.
\]
By Lemma~\ref{lem:deri-poly} we can choose $\alpha>\alpha_{0}$ such that $\rD_{\alpha}g$ coincides with an element of $V$ on some sub-IP-ring, so that in particular $P_{\rD_{\alpha}g}f=f$.
By Lemma~\ref{lem:norm} there exists $\beta_{0}>\alpha$ such that
\[
\forall\beta>\beta_{0}
\quad
\| (\rD_{\alpha}g)_{\beta}f-f \| < \epsilon.
\]
Applying $g_{\beta}g_{\alpha}$ to the difference on the left-hand side we obtain
\[
\| g_{\alpha\cup\beta}f - g_{\beta}g_{\alpha}f \| < \epsilon,
\text{ so that }
\rho(g_{\alpha\cup\beta}f, g_{\beta}g_{\alpha}f) < \epsilon.
\]
Observe that $\alpha\cup\beta > \alpha_{0}$, so that
\[
\rho(P_{g} f, g_{\beta}g_{\alpha}f) < 2\epsilon.
\]
Taking IP-limit along $\beta$ we obtain
\[
\rho(P_{g} f, P_{g}g_{\alpha}f) \leq 2\epsilon.
\]
A further application of the triangle inequality gives
\[
\rho(P_{g} f, P_{g}^{2}f) < 3\epsilon,
\]
and, since $\epsilon>0$ was arbitrary, we obtain $P_{g}f=P_{g}^{2}f$.

\paragraph{Commutativity of projections}
Let us now prove the second conclusion, namely that $P_{g}$ and $P_{g'}$ commute for any $g,g'\in F$.
Observe that the function $\alpha\mapsto g_{\alpha}$ can be seen as a polynomial-valued function in $\VIP[{\poly}]$ whose values are constant polynomials.
Moreover we can consider the constant function in $\poly[{\Fb}]$ whose value is $g'$.
Taking their commutator we see that
\[
\alpha\mapsto [g_{\alpha},g']
\quad \in \VIP[{\poly}],
\]
and, since $F$ is a VIP group, this map in fact lies in $\VIP[{\Fb}]$.
By \eqref{eq:polyn-g1} it takes values in $F_{1}$.
By Lemma~\ref{lem:sfi} we may assume that for any $\alpha_{1}<\dots<\alpha_{l}$ the maps $[g_{\alpha_{1}},g'],\dots,[g_{\alpha_{l}},g']$ generate a finite index subgroup of some $W\leq F_{1}$.
Interchanging $g$ and $g'$ and repeating this argument we may also wlog assume that for any $\alpha_{1}<\dots<\alpha_{l'}$ the maps $[g'_{\alpha_{1}},g],\dots,[g'_{\alpha_{l'}},g]$ generate a finite index subgroup of some $W'\leq F_{1}$.
Consider the splitting
\begin{equation}
\label{eq:splitting2}
H = \Big(\bigcap_{V\leq W}\ker P_{V} \cap \bigcap_{V'\leq W'}\ker P_{V'}\Big)
\oplus \overline{\lin}\Big(\bigcup_{V\leq W}\im P_{V} \cup \bigcup_{V'\leq W'}\im P_{V'}\Big)
=: H_{0} \cup H_{1}.
\end{equation}

\paragraph{Case 0}
Let $f\in H_{0}$.
As above we have $\IPlim_{\alpha} \|P_{[g_{\alpha},g']}f\| = 0$, and in particular
\begin{multline*}
0
= \IPlim_{\alpha} \< \wIPlim_{\beta} [g_{\alpha},g'_{\beta}]f, g_{\alpha}\inv P_{g'}f\>\\
= \IPlim_{\alpha} \IPlim_{\beta} \<g_{\alpha}g'_{\beta}f, g'_{\beta} P_{g'}f\>
= \IPlim_{\alpha} \<g_{\alpha} P_{g'}f, P_{g'}f\>,
\end{multline*}
since $\IPlim_{\beta}g'_{\beta} P_{g'}f = P_{g'}f$ by Lemma~\ref{lem:norm}.
Hence $P_{g}P_{g'}f \perp P_{g'}f$, which implies $P_{g}P_{g'}f=0$ since $P_{g}$ is an orthogonal projection.

Interchanging the roles of $g$ and $g'$, we also obtain $P_{g'}P_{g}f=0$.

\paragraph{Case 1}
Let $V\leq W$ and $f=P_{V}f$.
By Corollary~\ref{cor:poly-subg} we may wlog assume that $[g_{\alpha},g']\in V$ for all $\alpha$.
Let $\alpha$ be arbitrary, by Lemma~\ref{lem:norm} the limit
\[
\IPlim_{\beta} [g_{\alpha},g'_{\beta}] f = f
\]
also exists in norm.
Therefore
\[
g_{\alpha} P_{g'} f
= \wIPlim_{\beta} g_{\alpha} g'_{\beta} f
= \wIPlim_{\beta} g'_{\beta} g_{\alpha} [g_{\alpha},g'_{\beta}] f
= \wIPlim_{\beta} g'_{\beta} g_{\alpha} f
= P_{g'} g_{\alpha} f.
\]
Taking IP-limits on both sides we obtain
\[
P_{g} P_{g'} f = P_{g'} P_{g} f.
\]
The case $V'\leq W'$ and $f=P_{V'}f$ can be handled in the same way.
\end{proof}
If the group $G$ acts by measure-preserving transformations then the Hilbert space projections identified in Theorem~\ref{thm:fvip-proj} are in fact conditional expectations as the following folklore lemma shows.
\begin{lemma}
Let $X$ be a probability space and $(T_{\alpha})_{\alpha}$ be a net of operators on $L^{2}(X)$ induced by measure-preserving transformations.
Assume that $T_{\alpha}\to P$ weakly for some projection $P$.
Then $P$ is a conditional expectation.
\end{lemma}
\begin{proof}
Note that $\im P \cap L^{\infty}(X)$ is dense in $\im P$.

Let $f,g\in\im P \cap L^{\infty}(X)$.
Since the weak and the norm topology coincide on the unit sphere of $L^{2}(X)$, we have $\|T_{\alpha}f-f\|_{2} \to 0$ and $\|T_{\alpha}g-g\|_{2} \to 0$.
Therefore
\begin{multline*}
\|P(fg)-fg\|_{2}
\leq
\limsup_{\alpha} \|T_{\alpha}(fg)-fg\|_{2}\\
=
\limsup_{\alpha} \|(T_{\alpha}f-f)T_{\alpha}g+f(T_{\alpha}g-g)\|_{2}\\
\leq
\limsup_{\alpha} \|T_{\alpha}f-f\|_{2} \|T_{\alpha}g\|_{\infty}+\|f\|_{\infty} \|T_{\alpha}g-g\|_{2}
=
0.
\end{multline*}
This shows that $\im P\cap L^{\infty}(X)$ is an algebra, and the assertion follows.
\end{proof}

\subsection{Generalized polynomials and FVIP groups}
In order to obtain some tangible combinatorial applications of our results we will need non-trivial examples of FVIP groups.
The first example somewhat parallels Proposition~\ref{prop:mono-poly}.
\begin{lemma}[\cite{MR2246589}]
\label{lem:fvip-monomial}
Let $(n^{1}_{i})_{i\in\N},\dots,(n^{a}_{i})_{i\in\N} \subset \Z$ be any sequences, $(G,+)$ be a commutative group, $(y_{i})_{i\in\N}\subset G$ be any sequence, and $d\in\N$.
Then the maps of the form
\begin{equation}
\label{eq:fvip-monomial}
v(\alpha) = \sum_{i_{1}<\dots<i_{e}\in\alpha} n^{j_{1}}_{i_{1}}\cdots n^{j_{e-1}}_{i_{e-1}} y_{i_{e}},
\quad e\leq d, 1\leq j_{1},\dots,j_{e-1} \leq a,
\end{equation}
generate an FVIP subgroup $F\leq\VIP$, where the prefiltration $\Gb$ is given by $G_{0}=\dots=G_{d}=G$, $G_{d+1}=\{1_{G}\}$.
\end{lemma}
Maps of the form \eqref{eq:fvip-monomial} were originally studied in connection with admissible generalized polynomials (Definition~\ref{def:generalized-poly}).
We will not return to them in the sequel and a proof of the above lemma is included for completeness.
\begin{proof}
The group $F$ is by definition finitely generated and closed under conjugation by constants since $G$ is commutative.
It remains to check that the maps of the form \eqref{eq:fvip-monomial} are polynomial and that the group $F$ is closed under symmetric derivatives.

To this end we use induction on $d$.
The cases $d=0,1$ are clear (in the latter case the maps \eqref{eq:fvip-monomial} are IP-systems), so let $d>1$ and consider a map $v$ as in \eqref{eq:fvip-monomial} with $e=d$.
For $\beta<\alpha$ we have
\begin{multline*}
\sD_{\beta} v(\alpha)
=
-v(\alpha)+v(\beta\cup\alpha)-v(\beta)\\
=
\sum_{k=1}^{d-1} \sum_{i_{1}<\dots<i_{k}\in\beta, i_{k+1}<\dots<i_{d}\in\alpha} n^{j_{1}}_{i_{1}}\cdots n^{j_{e-1}}_{i_{d-1}} y_{i_{d}}\\
=
\sum_{k=1}^{d-1} \sum_{i_{1}<\dots<i_{k}\in\beta} n^{j_{1}}_{i_{1}}\cdots n^{j_{k}}_{i_{k}} \underline{\sum_{i_{k+1}<\dots<i_{d}\in\alpha} n^{j_{k+1}}_{i_{k+1}}\cdots n^{j_{e-1}}_{i_{d-1}} y_{i_{d}}}.
\end{multline*}
The underlined expression is $\Gb[+1]$-polynomial by the induction hypothesis and lies in $F$ by definition.
Since this holds for every $\beta$, the map $v$ is $\Gb$-polynomial.
Since the derivatives are in $F$ for every map $v$, the group $F$ is FVIP.
\end{proof}

The following basic property of FVIP groups will be used repeatedly.
\begin{lemma}
\label{lem:fvip-plus-fvip}
Let $F,F' \leq\VIP$ be FVIP groups.
Then the group $F\vee F'$ is also FVIP.
\end{lemma}
\begin{proof}
The group $F\vee F'$ is clearly finitely generated and invariant under conjugation by constants.
Closedness under $\sD$ follows from the identity
\begin{equation}
\label{eq:sD-Leibniz}
\sD_{m}(gh) = h\inv \sD_{m}g g(m) h \sD_{m}h g(m)\inv.
\qedhere
\end{equation}
\end{proof}

We will now elaborate on the example that motivated Bergelson, H\aa{}land Knutson and McCutcheon to study FVIP systems in the first place \cite{MR2246589}.
They have shown that ranges of generalized polynomials from a certain class necessarily contain FVIP systems.

We begin by recalling the definition of the appropriate class.
We denote the integer part function by $\floor{\cdot}$, the nearest integer function by $\nint{\cdot} = \floor{\cdot+1/2}$ and the distance to nearest integer by $\dint{a} = |a-\nint{a}|$ (this is consistent with the notation for general metric groups applied to $\R/\Z$).
\begin{definition}
\label{def:generalized-poly}
The set $\mathcal{G}$ of \emph{generalized polynomials} (in $l$ variables) is the smallest $\Z$-algebra of functions $\Z^{l}\to\Z$ that contains $\Z[x_{1},\dots,x_{l}]$ such that for every $p_{1},\dots,p_{t}\in\mathcal{G}$ and $c_{1},\dots,c_{t}\in\R$ the map $\floor{\sum_{i=1}^{t}c_{i}p_{i}}$ is in $\mathcal{G}$.
\index{generalized polynomial}
The notion of degree is extended from $\Z[x_{1},\dots,x_{l}]$ to $\mathcal{G}$ inductively by requiring $\deg p_{0}p_{1}\leq \deg p_{0}+\deg p_{1}$, $\deg (p_{0}+p_{1}) \leq \max(\deg p_{0},\deg p_{1})$, and $\deg \floor{\sum_{i=1}^{t}c_{i}p_{i}} \leq \max_{i}\deg p_{i}$, the degree of each generalized polynomial being the largest number with these properties.

The set of $\mathcal{G}_{a}$ of \emph{admissible} generalized polynomials is the smallest ideal of $\mathcal{G}$ that contains the maps $x_{1},\dots,x_{l}$ and is such that for every $p_{1},\dots,p_{t}\in\mathcal{G}_{a}$, $c_{1},\dots,c_{t}\in\R$, and $0<k<1$ the map $\floor{\sum_{i=1}^{t}c_{i}p_{i}+k}$ is in $\mathcal{G}_{a}$.
\index{generalized polynomial!admissible}
\end{definition}
Some examples of generalized polynomials are
\begin{multline}
\label{eq:generalized-poly-examples}
n^{3}+n,
\quad [\pi n+1/2],
\quad [\pi n+1/2][\pi n],
\quad [\pi n^{2} [e n] + 1/e],\\
\quad [\sqrt{2}n^{3}+1/e][\log 3 \cdot [\sqrt{3} n^{2}+n]^{2} + n],
\quad n^{5}-n+1,
\quad [\pi n],
\end{multline}
of which all but the last two are admissible.

The construction of FVIP systems in the range of an admissible generalized polynomial in \cite{MR2246589} proceeds by induction on the polynomial and utilizes Lemma~\ref{lem:fvip-monomial} at the end.
We give a softer argument that gives a weaker result in the sense that it does not necessarily yield an FVIP system of the form \eqref{eq:fvip-monomial}, but requires less computation.

For a ring $R$ (with not necessarily commutative multiplication, although we will only consider $R=\Z$ and $R=\R$ in the sequel) and $d\in\N$ we denote by $R_{\bullet}^{d}$ the prefiltration (with respect to the additive group structure) given by $R_{0}=\dots=R_{d}=R$ and $R_{d+1}=\{0_{R}\}$.
\begin{lemma}
\label{lem:fvip-ring-product}
Let $F_{i} \leq \VIP[R_{\bullet}^{d_{i}}]$, $i=0,1$, be FVIP groups.
Then the pointwise products of maps from $F_{0}$ and $F_{1}$ generate an FVIP subgroup of $\VIP[R_{\bullet}^{d_{0}+d_{1}}]$.
\end{lemma}
\begin{proof}
This follows by induction on $d_{0}+d_{1}$ using the identity
\[
\sD_{\beta} vw = (v+\sD_{\beta}v+v(\beta)) \sD_{\beta} w + (\sD_{\beta} v+v(\beta))w + (\sD_{\beta}v+v)w(\beta)
\]
for the symmetric derivative of a pointwise product.
\end{proof}

Applying Lemma~\ref{lem:near-identity} to $\R/\Z$ we obtain the following.
\begin{corollary}
\label{cor:near-integer}
Let $P$ be an FVIP system in $\R$.
Then for every $\epsilon>0$ wlog $\dint{P}<\epsilon$.
\end{corollary}
This allows us to show that we can obtain $\Z$-valued FVIP systems from $\R$-valued FVIP-systems by rounding.
\begin{lemma}
\label{lem:nearest-integer-FVIP}
Let $P \in \VIP[\R_{\bullet}^{d}]$ be an FVIP system.
Then wlog $\nint{P} \in \VIP[\Z_{\bullet}^{d}]$ and $\nint{P}$ is an FVIP system.
\end{lemma}
\begin{proof}
We use induction on $d$.
For $d=0$ there is nothing to show, so assume that $d>0$.
By the assumption every symmetric derivative of $(P_{\alpha})$ lies in an FVIP group of polynomials of degree $<d$ that is generated by $q_{1},\dots,q_{a}$, say.
By the induction hypothesis we know that wlog each $\nint{q_{i}}$ is again an FVIP system and by Lemma~\ref{lem:fvip-plus-fvip} they lie in some FVIP group $F$.
By Corollary~\ref{cor:near-integer} we may assume wlog that $\dint{P} < 1/12$.
Let now $\beta$ be given, by the hypothesis we have
\[
\sD_{\beta} P(\alpha) = \sum_{i}c_{i}q_{i}(\alpha)
\quad \text{for } \alpha>\beta
\]
with some $c_{i}\in \Z$.
By Corollary~\ref{cor:near-integer} we may wlog assume that $|c_{i}|\cdot \dint{q_{i}}(\alpha) < 1/(4\cdot 2^{i})$ for all $\alpha>\beta$.
This implies
\[
|\sD_{\beta} \nint{P}(\alpha) - \sum_{i}c_{i}\nint{q_{i}}(\alpha)| < 1/2
\quad \text{for } \alpha>\beta,
\]
so that
\[
\sD_{\beta} \nint{P}(\alpha) = \sum_{i}c_{i}\nint{q_{i}}(\alpha)
\quad \text{for } \alpha>\beta,
\]
since both sides are integer-values functions.
In fact we can do this for all $\beta$ with fixed $\max\beta$ simultaneously.
By a diagonal argument, cf.\ \cite[Lemma 1.4]{MR833409}, we may then assume that for every $\beta$ we have
\[
\sD_{\beta} \nint{P}(\alpha) = \sum_{i}c_{i}\nint{q_{i}}(\alpha)
\quad \text{for } \alpha>\beta
\]
with some $c_{i}\in\Z$.
Hence $F \vee \<\nint{P}\> \leq \VIP[\Z_{\bullet}^{d}]$ is an FVIP group.
\end{proof}
Recall that an \emph{IP-system} in $\Z^{l}$ is a family $(n_{\alpha})_{\alpha\in\Fin} \subset \Z^{l}$ such that $n_{\alpha\cup\beta}=n_{\alpha}+n_{\beta}$ whenever $\alpha,\beta\in\Fin$ are disjoint.
\begin{theorem}[{\cite[Theorem 2.8]{MR2246589}}]
\label{thm:gen-poly-FVIP}
For every generalized polynomial $p:\Z^{l}\to\Z$ and every FVIP system $(n_{\alpha})_{\alpha}$ in $\Z^{l}$ of degree at most $d$ there exists $n\in\Z$ such that the IP-sequence $(p(n_{\alpha})-n)_{\alpha\in\Fin}$ is wlog FVIP of degree at most $d\deg p$.
If $p$ is admissible, then we may assume $n=0$.
\end{theorem}
For inadmissible polynomials it may not be possible to obtain the above result with $n=0$.
Indeed, consider the example $p(n)=[\pi n]$.
Since $\pi n$ is equidistributed modulo $2$, we can find a sequence $(n_{k})$ such that $\pi n_{k} \in (-2^{-k},0) \mod 2$ for each $k$.
Consider the IP system $n_{\alpha} = \sum_{k\in\alpha} n_{k}$.
Then $\pi n_{\alpha} \in (-1,0) \mod 2$ for each $\alpha\in\Fin$, so that $p(n_{\alpha})$ is odd.
Lemma~\ref{lem:near-identity} applied to $\Z/2\Z$ now shows that no sub-IP-sequence of $(p(n_{\alpha}))_{\alpha\in\Fin}$ can be a VIP system.
\begin{proof}
We begin with the first part and use induction on $p$.
The class of maps that satisfy the conclusion is closed under $\Z$-linear combinations by Lemma~\ref{lem:fvip-plus-fvip} and under multiplication by Lemma~\ref{lem:fvip-ring-product}.
This class clearly contains the polynomials $1,x_{1},\dots,x_{l}$.
Thus it remains to show that, whenever $p_{1},\dots,p_{t}\in\mathcal{G}$ satisfy the conclusion and $c_{1},\dots,c_{t}\in\R$, the map $\floor{P}$ with $P=\sum_{i=1}^{t}c_{i}p_{i}$ also satisfies the conclusion.

By the assumption we have wlog that $(P(n_{\alpha})-C)_{\alpha}$ is an $\R$-valued FVIP system for some $C\in\R$.
By Hindman's theorem~\ref{thm:hindman} we may wlog assume that $\floor{P(n_{\alpha})}=\nint{P(n_{\alpha})-C}+n$ for some integer $n$ with $|n-C|<2$ and all $\alpha\in\Fin$.
The conclusion follows from Lemma~\ref{lem:nearest-integer-FVIP}.

Now we consider admissible generalized polynomials $p$ and use induction on $p$ again.
The conclusion clearly holds for $x_{1},\dots,x_{l}$, passes to linear combinations and passes to products with arbitrary generalized polynomials by Lemma~\ref{lem:fvip-ring-product} and the first part of the statement.
Assume now that $p_{1},\dots,p_{t}\in\mathcal{G}_{a}$ satisfy the conclusion and $c_{1},\dots,c_{t}\in\R$, $0<k<1$.
Then $(P(n_{\alpha}))_{\alpha}$ with $P:=\sum_{i=1}^{t}c_{i}p_{i}$ is wlog an $\R$-valued FVIP system, and by Corollary~\ref{cor:near-integer} we have wlog $\dint{P} < \min(k,1-k)$.
This implies $\floor{P(n_{\alpha})+k}=\nint{P(n_{\alpha})}$ and this is wlog an FVIP system by Lemma~\ref{lem:nearest-integer-FVIP}.
\end{proof}
As an aside, consider the set of real-valued generalized polynomials $\mathcal{RG}$ \cite[Definition 3.1]{MR2747062} that is defined similarly to $\mathcal{G}$, except that it is required to be an $\R$-algebra.
Following the proof of Theorem~\ref{thm:gen-poly-FVIP} we see that for every $p\in\mathcal{RG}$ and every FVIP system $(n_{\alpha})_{\alpha}\subset\Z^{l}$ wlog there exists a constant $C\in\R$ such that $(p(n_{\alpha})-C)_{\alpha}$ is an FVIP system.
Clearly, if $p$ is of the form $\floor{q}$ then $C\in\Z$ and if $p\in\R[x_{1},\dots,x_{l}]$ with zero constant term then $C=0$.
This, together with Corollary~\ref{cor:near-integer}, implies (an FVIP* version of) \cite[Theorem D]{MR2318563}.

Our main example (that also leads to Theorem~\ref{thm:SZ-intro}) are maps induced by admissible generalized polynomial sequences in finitely generated nilpotent groups.
\begin{lemma}
\label{lem:poly-fvip}
Let $G$ be a finitely generated nilpotent group with a filtration $\Gb$.
Let $p:\Z^{l}\to\Z$ be an admissible generalized polynomial, $(n_{\alpha})_{\alpha} \subset \Z^{l}$ be an FVIP system of degree at most $d$ and $g\in G_{d\deg p}$.
Then wlog $(g^{p(n_{\alpha})})_{\alpha}$ is an element of $\VIP$ and an FVIP system.
\end{lemma}
\begin{proof}
By Theorem~\ref{thm:gen-poly-FVIP} we can wlog assume that $(p(n_{\alpha}))_{\alpha}$ is a $\Z$-valued FVIP system of degree $\leq d\deg p$.
Using the (family of) homomorphism(s) $\Z_{\bullet}^{d\deg p}\to\Gb$, $1\mapsto g$, we see that $(g^{p(n_{\alpha})})_{\alpha}$ is contained in a finitely generated subgroup $F_{0}\leq\VIP$ that is closed under $\sD$.

Let $A\subset F_{0}$ and $B\subset G$ be finite generating sets.
Then the group generated by $bab\inv$, $a\in A$, $b\in B$, is FVIP in view of the identity \eqref{eq:sD-Leibniz}.
\end{proof}

\section{Measurable multiple recurrence}
\label{sec:primitive}
Following the general scheme of Furstenberg's proof, we will obtain our multiple recurrence theorem by (in general transfinite) induction on a suitable chain of factors of the given measure-preserving system.
For the whole section we fix a nilpotent group $G$ with a filtration $\Gb$ and an FVIP group $F\leq\VIP$.

In the induction step we pass from a factor to a ``primitive extension'' that enjoys a dichotomy: each element of $F$ acts on it either relatively compactly or relatively mixingly.
Since the reasoning largely parallels the commutative case here, we are able to refer to the article of Bergelson and McCutcheon \cite{MR1692634} for many proofs.
The parts of the argument that do require substantial changes are given in full detail.

Whenever we talk about measure spaces $(X,\mathcal{A},\mu)$, $(Y,\mathcal{B},\nu)$, or $(Z,\mathcal{C},\gamma)$ we suppose that they are regular and that $G$ acts on them on the right by measure-preserving transformations.
This induces a left action on the corresponding $L^{2}$ spaces.
Recall that to every factor map $(Z,\mathcal{C},\gamma)\to (Y,\mathcal{B},\nu)$ there is associated an essentially unique measure disintegration
\[
\gamma = \int_{y\in Y} \gamma_{y} \dif\nu(y),
\]
see \cite[\textsection 5.4]{MR603625}.
We write $\|\cdot\|_{y}$ for the norm on $L^{2}(Z,\gamma_{y})$.
Recall also that the fiber product $Z\times_{Y}Z$ is the space $Z\times Z$ with the measure $\int_{y\in Y} \gamma_{y} \otimes \gamma_{y} \dif\nu(y)$.

\subsection{Compact extensions}
We begin with the appropriate notion of relative compactness.
Heuristically, an extension is relatively compact if it is generated by the image of a relatively Hilbert-Schmidt operator.
\begin{definition}[{\cite[Definition 3.4]{MR1692634}}]
Let $Z\to Y$ be a factor.
A \emph{$Z|Y$-kernel} is a function $H \in L^{\infty}(Z \times_{Y} Z)$ such that
\index{kernel@($Z\mid Y$-)kernel}
\[
\int H(z_{1},z_{2}) \dif\gamma_{z_{2}}(z_{1}) = 0
\]
for a.e. $z_{2}\in Z$.
If $H$ is a $Z|Y$-kernel and $\phi\in L^{2}(Z)$ then
\[
H*\phi(z_{1}) := \int H(z_{1},z_{2}) \phi(z_{2}) \dif\gamma_{z_{1}}(z_{2}).
\]
\end{definition}
The map $\phi\mapsto H*\phi$ is a Hilbert-Schmidt operator on almost every fiber over $Y$ with uniformly bounded Hilbert-Schmidt norm.
These operators are self-adjoint provided that $H(z_{1},z_{2})=\overline{H(z_{2},z_{1})}$ a.e.

\begin{definition}[{\cite[Definition 3.6]{MR1692634}}]
Suppose that $X \to Z \to Y$ is a chain of factors, $K\leq F$ is a subgroup and $H$ is a non-trivial self-adjoint $X|Y$-kernel that is $K$-invariant in the sense that
\[
\IPlim_{\alpha} g(\alpha)H = H
\]
for every $g\in K$.
The extension $Z\to Y$ is called \emph{$K$-compact} if it is generated by functions of the form $H*\phi$, $\phi\in L^{2}(X)$.
\index{extension!compact@($K$-)compact}
\end{definition}

\begin{lemma}[{\cite[Remark 3.7(i)]{MR1692634}}]
Let $X \to Z \to Y$ be a chain of factors in which $Z\to Y$ is a $K$-compact extension generated by a $X|Y$-kernel $H$.
Then $H$ is in fact a $Z|Y$-kernel and $Z$ is generated by functions of the form $H*\phi$, $\phi\in L^{2}(Z)$.
\end{lemma}
\begin{proof}
Call the projection maps $\pi:X\to Z$, $\theta:X\to Y$.
Let $\phi\in L^{2}(X)$.
Since $H*\phi$ is $Z$-measurable we have
\begin{align*}
H*\phi(x)
&= \int H*\phi(x_{1}) \dif\mu_{\pi(x)}(x_{1})\\
&= \int \int H(x_{1},x_{2})\phi(w_{2}) \dif\mu_{\theta(x_{1})}(x_{2}) \dif\mu_{\pi(x)}(x_{1})\\
&= \int \int H(x_{1},x_{2})\phi(x_{2}) \dif\mu_{\theta(x)}(x_{2}) \dif\mu_{\pi(x)}(x_{1})\\
&= \int \int H(x_{1},x_{2}) \dif\mu_{\pi(x)}(x_{1}) \phi(x_{2}) \dif\mu_{\theta(x)}(x_{2})\\
&= \E(H|Z\times_{Y} X)*\phi(w).
\end{align*}
Since this holds for all $\phi$ we obtain $H=\E(H|Z\times_{Y} X)$.
Since $H$ is self-adjoint this implies that $H$ is $Z\times_{Y}Z$-measurable.
This in turn implies that $H*\phi = H*\E(\phi|Z)$ for all $\phi\in L^{2}(X)$.
\end{proof}
In view of this lemma the reference to the ambient space $X$ is not necessary in the definition of a $K$-compact extension.
Just like in the commutative case, compactness is preserved upon taking fiber products (this is only used in the part of the proof of Theorem~\ref{thm:wm} that we do not write out).
\begin{lemma}[{\cite[Remark 3.7(ii)]{MR1692634}}]
Let $Z\to Y$ be a $K$-compact extension.
Then $Z\times_{Y}Z \to Y$ is also a $K$-compact extension.
\end{lemma}

\subsection{Mixing and primitive extensions}
Now we define what we mean by relative mixing and the dichotomy between relative compactness and relative mixing.
\begin{definition}[{\cite[Definition 3.5]{MR1692634}}]
Let $Z\to Y$ be an extension.
A map $g\in F$ is called \emph{mixing on $Z$ relatively to $Y$} if for every $H\in L^{2}(Z\times_{Y} Z)$ with $\E(H|Y)=0$ one has $\wIPlim_{\alpha}g(\alpha) H = 0$.
An extension $Z\to Y$ is called \emph{$K$-primitive} if it is $K$-compact and each $g\in F\setminus K$ is mixing on $Z$ relative to $Y$.
\index{extension!primitive@($K$-)primitive}
\end{definition}

The above notion of mixing might be more appropriately called ``mild mixing'', but we choose a shorter name since there will be no danger of confusion.

The next lemma is used in the suppressed part of the proof of Theorem~\ref{thm:wm}.
\begin{lemma}[{\cite[Proposition 3.8]{MR1692634}}]
\label{lem:fiber-prod-prim}
Let $Z\to Y$ be a $K$-primitive extension.
Then $Z\times_{Y}Z \to Y$ is also a $K$-primitive extension.
\end{lemma}

Like in the commutative setting \cite[Lemma 2.8]{MR2151599} the compact part of a primitive extension is wlog closed under taking derivatives, but there is also a new aspect, namely that it is also closed under conjugation by constants.
\begin{lemma}
\label{lem:K-conj-inv}
Let $Z\to Y$ be a $K$-primitive extension.
Then $K$ is closed under conjugation by constant functions.
Moreover wlog $K$ is an FVIP group.
\end{lemma}
\begin{proof}
Let $g\in F\setminus K$, $h\in G$ and $H\in L^{2}(Z\times_{Y}Z)$ be such that $\E(H|Y)=0$.
Then
\[
\wIPlim_{\alpha} (h\inv gh)(\alpha) H
=
h\inv \wIPlim_{\alpha} g(\alpha) (hH)
=
0,
\]
so that $F\setminus K$ is closed under conjugation by constant functions, so that $K$ is also closed under conjugation by constant functions.

Since $F$ is Noetherian, the subgroup $K$ is finitely generated as a semigroup.
Fix a finite set of generators for $K$.
By Hindman's Theorem~\ref{thm:hindman} we may wlog assume that for every generator $g$ we have either $\sD_{\alpha}g\in K$ for all $\alpha\in\Fin$ or $\sD_{\alpha}g \not\in K$ for all $\alpha\in\Fin$.
In the latter case we obtain
\[
0 = \wIPlim_{\alpha,\beta} \sD_{\beta}g(\alpha) H
= \IPlim_{\alpha,\beta} g(\alpha)\inv g(\alpha\cup\beta) g(\beta)\inv H
= H,
\]
a contradiction.
Thus we may assume that all derivatives of the generators lie in $K$.
This extends to the whole group $K$ by \eqref{eq:sD-Leibniz} and invariance of $K$ under conjugation by constants.
\end{proof}

\subsection{Existence of primitive extensions}
Since our proof proceeds by induction over primitive extensions we need to know that such extensions exist.
First, we need a tool to locate non-trivial kernels.
\begin{lemma}[{\cite[Lemma 3.12]{MR1692634}}]
Let $X\to Y$ be an extension.
Suppose that $0\neq H \in L^{2}(X\times_{Y}X)$ satisfies $\E(H|Y)=0$ and that there exists $g\in F$ such that $\IPlim_{\alpha}g(\alpha)H=H$.

Then there exists a non-trivial self-adjoint non-negative definite $X|Y$-kernel $H'$ such that $\IPlim_{\alpha}g(\alpha)H'=H'$.
\end{lemma}
Second, we have to make sure that we cannot accidentally trivialize them.
\begin{lemma}[{\cite[Lemma 3.14]{MR1692634}}]
Let $Z\to Y$ be a $K$-compact extension.
Suppose that for some $g\in K$ and self-adjoint non-negative definite $Z|Y$-kernel $H$ we have
\[
\IPlim_{\alpha} \int (g(\alpha)H)(f' \otimes \bar f') \dif\tilde\gamma = 0
\]
for all $f'\in L^{\infty}(Z)$.
Then $H=0$.
\end{lemma}

The next theorem that provides existence of primitive extensions can be proved in the same way as in the commutative case \cite[Theorem 3.15]{MR1692634}.
The only change is that Theorem~\ref{thm:fvip-proj} is used instead of \cite[Theorem 2.17]{MR1692634} (note that $F$ is Noetherian, since it is a finitely generated nilpotent group).
\begin{theorem}
\label{thm:primitive-extension}
Let $X\to Y$ be a proper factor.
Then there exists a subgroup $K\leq F$ and a factor $X\to Z\to Y$ such that the extension $Z\to Y$ is proper and wlog $K$-primitive.
\end{theorem}

\subsection{Almost periodic functions}
For the rest of Section~\ref{sec:primitive} we fix a good group $\FE\leq\PE{\VIP}{\omega}$.
We will describe what we mean by ``good'' in Definition~\ref{def:good}, for the moment it suffices to say that $\FE$ is countable.
\begin{definition}[{\cite[Definition 3.1]{MR1692634}}]
Suppose that $(Z,\mathcal{C},\gamma)\to (Y,\mathcal{B},\nu)$ is a factor and $K\leq F$ a subgroup.
A function $f\in L^{2}(Z)$ is called \emph{$K$-almost periodic} if for every $\epsilon>0$ there exist $g_{1},\dots,g_{l}\in L^{2}(Z)$ and $D\in\mathcal{B}$ with $\nu(D)<\epsilon$ such that for every $\delta>0$ and $R\in\PE{K}{\omega}\cap\FE$ there exists $\alpha_{0}$ such that for every $\alpha_{0}<\vec\alpha\in \Fin^{\omega}_{<}$ there exists a set $E=E(\vec\alpha)\in\mathcal{B}$ with $\nu(E)<\delta$ such that for all $y\in Y\setminus(D\cup E)$ there exists $1\leq j\leq l$ such that
\[
\| R(\alpha)f - g_{j} \|_{y} < \epsilon.
\]
\index{almost periodic function}
The set of $K$-almost periodic functions is denoted by $\AP(Z,Y,K)$.
\end{definition}
The next lemma says that a characteristic function that can be approximated by almost periodic functions can be replaced by an almost periodic function right away.
\begin{lemma}[{\cite[Theorem 3.3]{MR1692634}}]
\label{lem:indicator-ap}
Let $A\in\mathcal{C}$ be such that $1_{A} \in\overline{\AP(Z,Y,K)}$ and $\delta>0$.
Then there exists a set $A'\subset A$ such that $\gamma(A\setminus A')<\delta$ and $1_{A'}\in \AP(Z,Y,K)$.
\end{lemma}
In the following lemma we have to restrict ourselves to $\PE{K}{\omega}\cap\FE$ since $\PE{K}{\omega}$ need not be countable.
\begin{lemma}[{\cite[Proposition 3.9]{MR1692634}}]
Let $X\to Y$ be an extension, $K\leq F$ a subgroup and $H$ a $X|Y$-kernel that is $K$-invariant.
Then wlog for all $R\in\PE{K}{\omega}\cap\FE$ and $\epsilon>0$ there exists $\alpha_{0}$ such that for all $\alpha_{0}<\vec\alpha$ we have
\[
\| R(\vec\alpha)H - H \| < \epsilon.
\]
\end{lemma}

With help of the above lemma we can show that in fact wlog every characteristic function can be approximated by almost periodic functions.
In view of Lemma~\ref{lem:indicator-ap} this allows us to reduce the question of multiple recurrence in a primitive extension to multiple recurrence for (relatively) almost periodic functions.
\begin{lemma}[{\cite[Theorem 3.11]{MR1692634}}]
\label{lem:ap-dense}
Let $Z\to Y$ be a $K$-compact extension.
Then wlog $\AP(Z,Y,K)$ is dense in $L^{2}(Z)$.
\end{lemma}

\subsection{Multiple mixing}
Under sufficiently strong relative mixing assumptions the limit behavior of a multicorrelation sequence $\prod_{i}S_{i}(\alpha)f_{i}$ on a primitive extension only depends on the expectations of the functions on the base space.
The appropriate conditions on the set $\{S_{i}\}_{i}$ are as follows.
\begin{definition}
\label{def:K-mixing}
Let $K\leq F$ be a subgroup.
A subset $A \subset F$ is called \emph{$K$-mixing} if $1_{G}\in A$ and $g\inv h\in F\setminus K$ whenever $g\neq h\in A$.
\index{mixing set!of polynomials}
\end{definition}
The requirement $1_{G}\in A$ is not essential, but it is convenient in inductive arguments.
In order to apply PET induction we will need the next lemma.

We say that a subgroup $K\leq F$ is \emph{invariant under equality of tails} if whenever $S\in K$ and $T\in F$ are such that there exists $\beta\in\Fin$ with $S_{\alpha}=T_{\alpha}$ for all $\alpha>\beta$ we have $T\in K$.
Every group $K\leq F$ that is the compact part of some primitive extension has this property.
\begin{lemma}
\label{lem:notinK}
Let $K\leq F$ be a subgroup that is invariant under equality of tails.
Let $S,T\in F$ be such that $S\inv T \not\in K$.
Then wlog
\[
(S \sD_{\beta}S)\inv (T\sD_{\beta}T) \not\in K
\quad\text{and}\quad
S\inv (T\sD_{\beta} T) \not\in K
\]
for all $\beta\in\Fin$.
\end{lemma}
\begin{proof}
If the first conclusion fails then by Hindman's theorem~\ref{thm:hindman} wlog
\[
h(\alpha) := (S \sD_{\alpha}S)\inv (T\sD_{\alpha}T) \in K \text{ for all }\alpha\in\Fin
\]
and $h(\emptyset)\not\in K$.
Proceed as in the proof of Lemma~\ref{lem:deri-poly}.
Analogously for the second conclusion.
\end{proof}

The next lemma is a manifestation of the principle that compact orbits can be thought of as being constant.
\begin{lemma}[{\cite[Proposition 4.2]{MR1692634}}]
\label{lem:wm-comp}
Let $Z\to Y$ be a $K$-primitive extension, $R^{\beta}\in K$ for each $\beta\in\Fin$ and $W\in F\setminus K$.
Let also $f,f'\in L^{\infty}(Z)$ be such that either $\E(f|Y)=0$ or $\E(f'|Y)=0$.
Then wlog
\[
\IPlim_{\beta,\alpha} \|\E(R^{\beta}(\alpha)f W(\beta)f'|Y) \| = 0.
\]
\end{lemma}

We come to the central result on multiple mixing.
\begin{theorem}[cf.\ {\cite[Theorem 4.10]{MR1692634}}]
\label{thm:wm}
Let $K\leq F$ be a subgroup.
For every $K$-mixing set $\{S_{0}\equiv 1_{G},S_{1},\dots,S_{t}\}\subset F$ the following statements hold.
\begin{enumerate}
\item For every $K$-primitive extension $Z\to Y$ and any $f_{0},\dots,f_{t}\in L^{\infty}(Z)$ we have wlog
\[
\wIPlim_{\alpha} \prod_{i=1}^{t} S_{i}(\alpha)f_{i} - \prod_{i=1}^{t} S_{i}(\alpha)\E(f_{i}|Y) = 0.
\]
\item
For every $K$-primitive extension $Z\to Y$ and any $f_{0},\dots,f_{t}\in L^{\infty}(Z)$ we have wlog
\[
\IPlim_{\alpha} \Big\| \E \big(\prod_{i=0}^{t} S_{i}(\alpha)f_{i} \big| Y \big) - \prod_{i=0}^{t} S_{i}(\alpha)\E(f_{i}|Y) \Big\| = 0.
\]
\item
For every $K$-primitive extension $Z\to Y$, any $U_{i,j}\in K$, and any $f_{i,j}\in L^{\infty}(Z)$ we have wlog
\[
\IPlim_{\alpha} \Big\| \E \big( \prod_{i=0}^{t} S_{i}(\alpha) \big( \prod_{j=0}^{s}U_{i,j}(\alpha)f_{i,j} \big) \big| Y \big)- \prod_{i=0}^{t} S_{i}(\alpha) \E \big( \prod_{j=0}^{s}U_{i,j}(\alpha)f_{i,j} \big| Y \big) \Big\| = 0.
\]
\end{enumerate}
\end{theorem}
We point out that the main induction loop is on the mixing set.
It is essential that, given $K\leq F$, all statements are proved simultaneously for all $K$-compact extensions since the step from weak convergence to strong convergence involves a fiber product via Lemma~\ref{lem:fiber-prod-prim}.
\begin{proof}
The proof is by PET-induction on the mixing set.
We only prove that the last statement for mixing sets with lower weight vector implies the first, the proofs of other implications are the same as in the commutative case.

By the telescope identity it suffices to consider the case $\E(f_{i_{0}}|Y)=0$ for some $i_{0}$.
By the van der Corput Lemma~\ref{lem:vdC} it suffices to show that wlog
\[
\IPlim_{\beta,\alpha} \int_{Z} \prod_{i=1}^{t} S_{i}(\alpha)f_{i} \prod_{i=1}^{t} S_{i}(\alpha\cup\beta)\bar f_{i} = 0.
\]
This limit can be written as
\[
\IPlim_{\beta,\alpha} \int_{Z} \prod_{i=1}^{t} S_{i}(\alpha)f_{i} \prod_{i=1}^{t} \underbrace{S_{i}(\alpha) \sD_{\beta}S_{i}(\alpha)}_{=:T_{i,\beta}(\alpha)} (S_{i}(\beta)\bar f_{i}).
\]
By Lemma~\ref{lem:notinK} we may wlog assume that $T_{i,\beta}\inv T_{j,\beta}$ and $S_{i}\inv T_{j,\beta}$ are mixing for all $\beta\in\Fin$ provided that $i\neq j$.
Re-indexing if necessary and using Hindman's theorem~\ref{thm:hindman} we may wlog assume $S_{i}\inv T_{i,\beta} \in K$ for all $\beta\in\Fin$ and $i\leq w$ and $S_{i}\inv T_{i,\beta} \not\in K$ for all $\beta\in\Fin$ and $i>w$ for some $w=0,\dots,t$.
Thus
\begin{equation}
\label{eq:larger-mixing-set}
S_{0},S_{1},\dots,S_{t},T_{w+1,\beta},\dots,T_{t,\beta}
\end{equation}
is a $K$-mixing set for every $\beta\neq\emptyset$.
Moreover it has the same weight vector as $\{S_{1},\dots,S_{t}\}$ since $T_{i,\beta} \sim S_{i}$.
Assume that $S_{j}$, $j\neq 0$, has the maximal level in \eqref{eq:larger-mixing-set}.
We have to show
\begin{multline*}
\IPlim_{\beta,\alpha} \int_{Z} \prod_{i=1}^{w} S_{j}\inv(\alpha)S_{i}(\alpha)(f_{i} \sD_{\beta}S_{i}(\alpha) (S_{i}(\beta)\bar f_{i}))\\
\cdot \prod_{i=w+1}^{t} S_{j}\inv(\alpha)S_{i}(\alpha)f_{i} S_{j}\inv(\alpha)T_{i,\beta}(\alpha) (S_{i}(\beta)\bar f_{i}) = 0.
\end{multline*}
For each fixed $\beta\in\Fin$ the limit along $\alpha$ comes from the $K$-mixing set
\[
S_{j}\inv S_{1},\dots,S_{j}\inv S_{t},S_{j}\inv T_{w+1,\beta},\dots,S_{j}\inv T_{t,\beta}
\]
that has lower weight vector.
Hence we can apply the induction hypothesis, thereby obtaining that the limit equals
\begin{multline*}
\IPlim_{\beta,\alpha} \int_{Z} \prod_{i=1}^{w} S_{j}\inv(\alpha)S_{i}(\alpha)\E(f_{i} \sD_{\beta}S_{i}(\alpha) (S_{i}(\beta)\bar f_{i}) |Y)\\
\cdot\prod_{i=w+1}^{t} S_{j}\inv(\alpha)S_{i}(\alpha)\E(f_{i}|Y) S_{j}\inv(\alpha)T_{i,\beta}(\alpha) \E(S_{i}(\beta)\bar f_{i}|Y)
\end{multline*}
This clearly vanishes if $i_{0}>w$, otherwise use Lemma~\ref{lem:wm-comp}.
\end{proof}

\subsection{Multiparameter multiple mixing}
In fact we need some information about relative polynomial mixing in several variables.
First we need to say what we understand under a mixing system of polynomial expressions.
Recall that by definition each $S\in\PE{F}{m}$ can be written in the form
\begin{equation}
\label{eq:PE-dec}
S(\alpha_{1},\dots,\alpha_{m}) = W^{(\alpha_{1},\dots,\alpha_{m-1})}(\alpha_{m}) \dots W^{\alpha_{1}}(\alpha_{2}) W(\alpha_{1}).
\end{equation}
\begin{definition}
\label{def:K-mixing-PE}
Let $K\leq F$ be a subgroup and $m\in\N$.
A set $\{S_{i}\}_{i=0}^{t} \subset\PE{F}{m}$ is called \emph{$K$-mixing} if $S_{0}\equiv 1_{G}$, the polynomial expressions $\{S_{i}\}$ are pairwise distinct, and for all $r$ and $i\neq j$ we have either $\forall\vec\alpha\in\Fin^{r}_{<}$ $W_{i}^{\vec\alpha}=W_{j}^{\vec\alpha}$ or $\forall\vec\alpha\in\Fin^{r}_{<}$ $(W_{i}^{\vec\alpha})\inv W_{j}^{\vec\alpha} \not\in K$.
\index{mixing set!of polynomial expressions}
\end{definition}
For $m=1$ this coincides with Definition~\ref{def:K-mixing}.
However, in general, this definition requires more than $\{W_{i}^{\vec\alpha}\}_{i}$ being (up to multiplicity) a $K$-mixing set in the sense of Definition~\ref{def:K-mixing} for every $\vec\alpha$.

\begin{theorem}[cf.\ {\cite[Theorem 4.12]{MR1692634}}]
\label{thm:mwm}
Let $Z\to Y$ be a $K$-primitive extension.
Then for every $m\geq 1$, every $K$-mixing set $\{S_{0},\dots,S_{t}\}\subset\PE{F}{m}$ and any $f_{0},\dots,f_{t}\in L^{\infty}(Z)$ we have wlog
\[
\IPlim_{\alpha_{1},\dots,\alpha_{m}} \Big\| \E \big( \prod_{i=0}^{t} S_{i}(\alpha_{1},\dots,\alpha_{m})f_{i} \big| Y \big) - \prod_{i=0}^{t} S_{i}(\alpha_{1},\dots,\alpha_{m})\E(f_{i}|Y) \Big\| = 0.
\]
\end{theorem}
\begin{proof}
We use induction on $m$.
The case $m=0$ is trivial since the product then consists only of one term.
Assume that the conclusion holds for $m$ and consider a $K$-mixing set of polynomial expressions in $m+1$ variables.
For brevity we write $\vec\alpha=(\alpha_{1},\dots,\alpha_{m})$ and $\alpha=\alpha_{m+1}$.
We may assume that $\|f_{i}\|_{\infty}\leq 1$ for all $i$ and $\E(f_{i_{0}}|Y)=0$ for some $i_{0}$.

By Definition~\ref{def:K-mixing-PE} and with notation from \eqref{eq:PE-dec}, for every $\vec\alpha$ there exists a $K$-mixing set $\{V_{j}^{\vec\alpha}\} \subset F$ such that $W_{i}^{\vec\alpha}=V_{j_{i}}^{\vec\alpha}$, where the assignment $i\to j_{i}$ does not depend on $\vec\alpha$.
Let also
\[
A_{j} = \{ S(\cdot)=S_{i}(\cdot,\emptyset) : j_{i}=j\}.
\]
In view of the Milliken-Taylor theorem~\ref{thm:milliken-taylor} and by a diagonal argument, cf.\ \cite[Lemma 1.4]{MR833409}, it suffices to show that for every $\delta>0$ there exist $\vec\alpha<\alpha$ such that
\[
\Big\| \E\big( \prod_{j}V_{j}^{\vec\alpha}(\alpha) (\prod_{S\in A_{j}}S(\vec\alpha)f_{S,j}) \big| Y \big) \Big\| \leq \delta
\]
provided that $\E(f_{S_{0},j_{0}}|Y)=0$ for some $j_{0},S_{0}$.
By the induction hypothesis there exists $\vec\alpha$ such that
\[
\Big\| \E\big( \prod_{S\in A_{j_{0}}}S(\vec\alpha)f_{S,j_{0}} \big| Y \big) \Big\|
< \delta,
\]
since $A_{j_{0}}$ is a $K$-mixing set.
This implies
\[
\Big\| \prod_{j}V_{j}^{\vec\alpha}(\alpha) \E\big(\prod_{S\in A_{j}}S(\vec\alpha)f_{S,j} \big| Y \big) \Big\| < \delta
\]
for all $\alpha>\vec\alpha$.
Since $\{V_{j}^{\vec\alpha}\}_{j}$ is a $K$-mixing set, Theorem~\ref{thm:wm} implies
\[
\IPlim_{\alpha} \Big\| \E\big( \prod_{j}V_{j}^{\vec\alpha}(\alpha) \prod_{S\in A_{j}}S(\vec\alpha)f_{S,j} \big| Y \big) \Big\| \leq \delta.
\qedhere
\]
\end{proof}

\subsection{Lifting multiple recurrence to a primitive extension}
We are nearing our main result, a multiple recurrence theorem for polynomial expressions.
In order to guarantee the existence of the limits that we will encounter during its proof we have to restrict ourselves to a certain good subgroup of the group of polynomial expressions.
It will be shown later that this restriction can be removed, cf. Corollary~\ref{cor:finite-good}.
\begin{definition}
\label{def:good}
We call a group $\FE\leq\PE{\VIP}{\omega}$ \emph{good} if it has the following properties.
\index{good group}
\begin{enumerate}
\item (Cardinality) $\FE$ is countable.
\item (Substitution) If $m\in\N$, $g\in \PE{\VIP}{m}\cap\FE$, and $\vec\beta\in\Fin^{m}_{<}$, then $g[\vec\beta]\in\FE$.
\item (Decomposition) If $K\leq F$ is a subgroup invariant under conjugation by constants and $\{S_{i}\}_{i=0}^{t} \subset\PE{F}{m}\cap\FE$ is a finite set with $S_{0}\equiv 1_{G}$,
then we have finite sets $\{T_{k}\}_{k=0}^{v-1}\subset \PE{F}{m}\cap\FE$ and $\{R_{i}\}_{i=0}^{t}\subset \PE{K}{m}\cap\FE$ with $R_{0}=T_{0}\equiv 1_{G}$ such that $S_{i}=T_{k_{i}}R_{i}$ and for every sub-IP-ring $\Fin'\subset\Fin$ the set $\{T_{k}\}$ is wlog $K$-mixing.
\end{enumerate}
\end{definition}
The property of being good is hereditary in the sense that a group that is good with respect to some IP-ring is also good with respect to any sub-IP-ring.

Let $(X,\mathcal{A},\mu)$ be a regular measure space with a right action of $G$ by measure-preserving transformations.
Let also $\FE\leq\PE{F}{\omega}$ be a good group.
By Hindman's theorem~\ref{thm:hindman} we may wlog assume that
\[
\wIPlim_{\alpha} g(\alpha)f
\]
exists for every $g\in F$ and $f\in L^{2}(X)$.
By the Milliken-Taylor theorem~\ref{thm:milliken-taylor} we may wlog assume that the limit
\[
a(A,m,\{S_{i}\}_{i}):= \IPlim_{\vec\alpha \in \Fin^{m}_{<}} \mu\left( \cap_{i=0}^{t} A S_{i}(\vec\alpha)\inv \right)
\]
exists for every $m\in\N$, every $A\in\mathcal{A}$, and every finite set $\{S_{0},\dots,S_{t}\}\subset \PE{F}{m}\cap\FE$.

The central result of this chapter is that this limit is in fact positive provided $\mu(A)>0$.
Since it will be proved by induction on a tower of factors, we formulate it in terms of factors.
\begin{definition}
\index{SZ property}
A factor $(X,\mathcal{A},\mu)\to (Y,\mathcal{B},\nu)$ is said to have the \emph{SZ (Szemer\'{e}di) property} if for every $B\in\mathcal{B}$ with $\nu(B)>0$ and every set of polynomial expressions $\{S_{0}\equiv 1_{G},S_{1},\dots,S_{t}\} \subset\PE{F}{m}\cap\FE$ one has
\[
a(B,m,\{S_{i}\}_{i})  > 0.
\]
\end{definition}
The result then reads as follows.
\begin{theorem}
\label{thm:SZ}
The identity factor $X\to X$ has the SZ property.
\end{theorem}
This generalizes \cite[Theorem 1.3]{MR1692634}.
Note that our lower bounds depend on the polynomial expressions involved and not only on their number.
We cannot obtain more uniform results in spirit of \cite[Definition 5.1]{MR1692634} due to the lack of control on the number $w$ provided by Corollary~\ref{cor:pair-color}.

It is relatively easy to show that the class of factors that satisfy the SZ property is closed under inverse limits, so there is a maximal such factor.
\begin{lemma}[{\cite[Proposition 5.2]{MR1692634}}]
\label{lem:maximal-SZ-factor}
For every separable regular measure-preserving system $X$ there exists a maximal factor that has the SZ property.
\end{lemma}

Hence it remains to show that the SZ property passes to primitive extensions.
\begin{proof}[Proof of Theorem~\ref{thm:SZ}]
By Lemma~\ref{lem:maximal-SZ-factor} there exists a maximal factor $X\to Y$ with the SZ property.
Assume that $X\neq Y$, then by Theorem~\ref{thm:primitive-extension} wlog there exists a subgroup $K\leq F$ and a factor $X\to Z$ such that $(Z,\mathcal{C},\lambda)\to (Y,\mathcal{B},\nu)$ is a proper $K$-primitive extension.
We will show that $Z$ also has the SZ property, thereby contradicting maximality of $Y$.

Let $A\in\mathcal{C}$ with $\lambda(A)>0$ and $\{S_{i}\}_{i=0}^{t} \subset\PE{F}{m}\cap\FE$ be a finite set with $S_{0}\equiv 1_{G}$.
We have to show
\begin{equation}
\label{eq:SZ-desired}
\IPlim_{\vec\alpha \in \Fin^{m}_{<}} \mu\left( \cap_{i=0}^{t} A S_{i}(\vec\alpha)\inv \right) > 0.
\end{equation}
By Lemma~\ref{lem:K-conj-inv} we may wlog assume that $K$ is an FVIP group and by Lemma~\ref{lem:ap-dense} that $\AP$ is dense in $L^{2}(Z)$.
Note that $\FE$ is still good with respect to the new IP-ring implied in the ``wlog'' notation.
Thus wlog we have a $K$-mixing set $\{W_{k}\}_{k=0}^{v-1}$ and polynomial expressions $R_{i}\in\PE{K}{m}\cap\FE$ with $R_{0}\equiv 1_{G}$ such that $S_{i}=W_{k_{i}}R_{i}$.

By Lemma~\ref{lem:indicator-ap} we may replace $A$ by a subset that has at least one half of its measure such that $1_{A}\in \AP$.
There exist $c=c(\lambda(A))>0$ and a set $B\in\mathcal{B}$ such that $\nu(B)>c$ and $\lambda_{y}(A)>c$ for every $y\in B$.
Pick $0 < \epsilon < \min(c/2,c^{v}/(4(t+1)))$.

By Corollary~\ref{cor:pair-color} there exist $N,w\in\N$ and
\[
\{L_{i},M_{i}\}_{i=1}^{w}\subset(\PE{K}{N}\cap\FE)\times(\PE{F}{N}\cap\FE)
\]
such that for every $l$-coloring of $\{L_{i},M_{i}\}$ there exists a number $a$ and sets $\beta_{1}<\dots<\beta_{m}\subset N$ such that the set $\{L_{a}R_{i}[\vec\beta],M_{a}W_{k}[\vec\beta]L_{a}\inv\}_{0\leq i\leq t,0\leq k< v}$ is monochrome (and in particular contained in the set $\{L_{i},M_{i}\}$).

Since $f=1_{A}\in \AP$ there exist functions $g_{1},\dots,g_{l}\in L^{2}(Z)$ and a set $D\in\mathcal{B}$ such that $\nu(D)<\epsilon$ and for every $\delta>0$ and $T\in\PE{K}{N}\cap\FE$ there exists $\alpha_{0}$ such that for every $\alpha_{0}<\vec\alpha\in\Fin^{N}_{<}$ there exists a set $E=E(\vec\alpha)\in\mathcal{B}$ with $\nu(E)<\delta$ such that for every $y\in (D\cup E)^{\complement}$ there exists $j$ such that $\|T(\vec\alpha)f-g_{j}\|_{y}<\epsilon$.
Let $B'=B\cap D^{\complement}$, so that $\nu(B')>c/2$.

Let $Q=|\Fin(N)^{m}_{<}|$ be the number of possible choices of $\vec\beta\in\Fin(N)^{m}_{<}$ and
\[
a_{1}:=a(B',N,\{1_{G}\}\cup \{M_{i}W_{k}[\vec\beta]\}_{1\leq i\leq w, k<v, \vec\beta\in\Fin^{m}_{<}(N)})>0.
\]

Using this with $\delta=a_{1}/2w^{2}$ and $T=L_{1},\dots,L_{w}$ we obtain wlog for every $\vec\alpha\in\Fin^{N}_{<}$ a set $E=E(\vec\alpha)\in\mathcal{B}$ with $\nu(E)<a_{1}/2w$ such that for every $y\in (D\cup E)^{\complement}$ and every $i=1,\dots,w$ there exists $j=j(y,i)$ such that
\begin{equation}
\label{eq:ap}
\| L_{i}(\vec\alpha)f-g_{j} \|_{y} < \epsilon
\text{ for every } 1\leq i\leq w.
\end{equation}
By Theorem~\ref{thm:mwm} we may also wlog assume that for every $\vec\alpha\in\Fin^{N}_{<}$ we have
\begin{equation}
\label{eq:assumption-mixing}
\big\| \E(\prod_{k<v}W_{k}(\vec\alpha)f|Y) - \prod_{k<v}W_{k}(\vec\alpha)\E(f|Y) \big\| < c^{v} (a_{1}/2wQ)^{1/2}/4.
\end{equation}
Recall that we have to show \eqref{eq:SZ-desired}.
To this end it suffices to find $a(A,m,\{S_{i}\}_{0\leq i\leq t})$ such that for an arbitrary sub-IP-ring there exists $\vec\gamma\in\Fin^{m}_{<}$ with
\[
\mu\left( \cap_{i=0}^{t} A S_{i}(\vec\gamma)\inv \right)
> a(A,m,\{S_{i}\}_{0\leq i\leq t}) > 0,
\]
so fix a sub-IP-ring $\Fin$.
By definition of $a_{1}$ there exists a tuple $\vec\alpha\in\Fin^{N}_{<}$ (that will remain fixed) such that
\[
\nu\left( C_{0} \right) > a_{1}, \text{ where } C_{0}:=\cap_{i=1,\dots,w,k<v,\vec\beta} B' W_{k}[\vec\beta](\vec\alpha)\inv M_{i}(\vec\alpha)\inv.
\]
Let
\[
C := C_{0} \setminus \cup_{i=1}^{w} E(\vec\alpha) M_{i}(\vec\alpha)\inv,
\]
so that $\nu(C)>a_{1}/2$.
For every $y\in C$ consider an $l$-coloring of $\{L_{i},M_{i}\}_{i}$ given by $i\in [1,w] \mapsto j(yM_{i}(\vec\alpha),i)$ determined by \eqref{eq:ap}.
By the assumptions on $\{L_{i},M_{i}\}$ there exist $j(y)$, $a\in [1,w]$ and $\beta_{1}<\dots<\beta_{m}\subset N$ such that
\[
\| L_{a}(\vec\alpha)R_{i}[\vec\beta](\vec\alpha)f-g_{j(y)} \|_{yM_{a}(\vec\alpha)W_{k}(\vec\beta)L_{a}(\vec\alpha)\inv} < \epsilon
\text{ for every } 0\leq i\leq t,0\leq k<v.
\]
This can also be written as
\begin{multline*}
\| W_{k}[\vec\beta](\vec\alpha)R_{i}[\vec\beta](\vec\alpha)f - W_{k}[\vec\beta](\vec\alpha)L_{a}(\vec\alpha)\inv g_{j(y)} \|_{yM_{a}(\vec\alpha)} < \epsilon\\
\text{ for every } 0\leq i\leq t,0\leq k<v.
\end{multline*}
Since this holds for every $i,k$ and we have $R_{0}\equiv 1_{G}$, this implies
\[
\| (W_{k}R_{i})[\vec\beta](\vec\alpha)f-W_{k}[\vec\beta](\vec\alpha)f \|_{y M_{a}(\vec\alpha)} < 2\epsilon.
\]
Passing to a subset $C'\subset C$ with measure at least $a_{1}/2wQ$, we may assume that $a$ and $\vec\beta$ do not depend on $y$.
Thus we obtain a set $B'':=C' M_{a}(\vec\alpha)$ of measure at least $a_{1}/2wQ$ and a tuple $(\gamma_{j}=\cup_{i\in\beta_{j}}\alpha_{i})_{j=1}^{m}$ such that
\[
\| W_{k}R_{i}(\vec\gamma)f-W_{k}(\vec\gamma)f \|_{y} < 2\epsilon
\]
for every $y\in B''$, $i$ and $k$.
Recall that $f$ is $\{0,1\}$-valued, so that
\[
\big\| \prod_{i=0}^{t}S_{i}(\vec\gamma)f - \prod_{k<v}W_{k}(\vec\gamma)f \big\|_{y}\\
=
\big\| \prod_{i=0}^{t}W_{k_{i}}R_{i}(\vec\gamma)f - \prod_{i=0}^{t}W_{k_{i}}(\vec\gamma)f \big\|_{y}\\
< 2(t+1)\epsilon
\]
for all $y\in B''$.
Moreover, since $B'' \subset \cap_{j} B' W_{j}(\vec\gamma)\inv$, one has
\[
\big| \prod_{k<v} W_{k}(\vec\gamma) \E(f|Y)(y) \big| \geq c^{v}
\]
for every $y\in B''$.
Therefore and by \eqref{eq:assumption-mixing} we obtain
\begin{align*}
\big\| \prod_{i=0}^{t}S_{i}(\vec\gamma)f \big\|
&\geq \big\| \prod_{i=0}^{t}S_{i}(\vec\gamma)f \big\|_{L^{2}(B'')}\\
&> \big\| \prod_{k<v}W_{k}(\vec\gamma)f \big\|_{L^{2}(B'')} - 2(t+1)\epsilon\nu(B'')^{1/2}\\
&\geq \big\| \E(\prod_{k<v}W_{k}(\vec\gamma)f |Y) \big\|_{L^{2}(B'')} - 2(t+1)\epsilon\nu(B'')^{1/2}\\
&\geq
\big\| \prod_{k<v}W_{k}(\vec\gamma)\E(f|Y) \big\|_{L^{2}(B'')}\\
&\qquad-
\big\| \E(\prod_{k<v}W_{k}(\vec\gamma)f|Y) - \prod_{k<v}W_{k}(\vec\gamma)\E(f|Y) \big\|\\
&\qquad - 2(t+1)\epsilon\nu(B'')^{1/2}\\
&> c^{v} (a_{1}/2wQ)^{1/2}/4 =: a(A,m,\{S_{i}\}_{i})^{1/2}.
\qedhere
\end{align*}
\end{proof}

\subsection{Good groups of polynomial expressions}
As we have already mentioned, good groups are just technical vehicles.
The point is that we can perform all operations that we are interested in within a countable set of polynomial expressions, so that we can wlog assume the existence of all IP-limits that we encounter.

The only non-trivial property of good groups is the decomposition property.
However, the following lemma essentially shows that it is always satisfied.
\begin{proposition}
\label{prop:decomposition-mixing-compact}
Let $K\leq F$ be a subgroup that is invariant under conjugation by constants, $m\in\N$ and $\{S_{i}\}_{i=0}^{t} \subset\PE{F}{m}$ be any finite set with $S_{0}\equiv 1_{G}$.
Then there exists a set $\{T_{k}\}_{k=0}^{v}\subset\PE{F}{m}$ that is wlog $K$-mixing and decompositions $S_{i}=T_{k_{i}}R_{i}$ such that $R_{i}\in\PE{K}{m}$.
\end{proposition}
\begin{proof}
We argue by induction on $m$.
The claim is trivial for $m=0$.
Assume that it holds for $m$, we show its validity for $m+1$.
For brevity we write $\vec\alpha=(\alpha_{1},\dots,\alpha_{m})$, $\alpha=\alpha_{m+1}$.

Consider the maps $\tilde S_{i}(\vec\alpha) = S_{i}(\vec\alpha,\emptyset)$.
By the induction hypothesis there exists a set $\{\tilde T_{k}\}_{k=0}^{\tilde v}\subset\PE{F}{m}$ that is wlog $K$-mixing and decompositions $\tilde S_{i}=\tilde T_{k_{i}} \tilde R_{i}$ such that $\tilde R_{i}\in\PE{K}{m}$.
Then $S_{i}(\vec\alpha,\alpha)=W_{i}^{\vec\alpha}(\alpha) \tilde T_{k_{i}}(\vec\alpha) \tilde R_{i}(\vec\alpha)$.

Let $i<j$.
By the Milliken-Taylor Theorem~\ref{thm:milliken-taylor} we may wlog assume that either $(W_{i}^{\vec\alpha})\inv W_{j}^{\vec\alpha} \not\in K$ for all $\vec\alpha\in\Fin^{m}_{<}$ (in which case we do nothing) or $(W_{i}^{\vec\alpha})\inv W_{j}^{\vec\alpha} \in K$ for all $\vec\alpha\in\Fin^{m}_{<}$.
In the latter case we have $W_{j}^{\vec\alpha} = W_{i}^{\vec\alpha} R^{\vec\alpha}$ with some $R^{\vec\alpha} \in K$ and we can write
\[
S_{j}(\vec\alpha,\alpha) = W_{i}^{\vec\alpha}(\alpha) \tilde T_{k_{j}}(\vec\alpha)
\underbrace{(\tilde T_{k_{j}}(\vec\alpha)\inv R^{\vec\alpha} \tilde T_{k_{j}}(\vec\alpha))}_{\in K}(\alpha) \tilde R_{j}(\vec\alpha),
\quad
\vec\alpha\in\Fin^{m}_{<}.
\]
Doing this for all pairs $i<j$ we obtain the requested decomposition with the set $\{T_{k}\}$ consisting of all products $W_{i}^{\vec\alpha}\tilde T_{k_{j}}(\vec\alpha)$ that occur above.
\end{proof}

\begin{corollary}
\label{cor:finite-good}
Every finite subset of $\PE{F}{\omega}$ is wlog contained in a good subgroup of $\PE{\VIP}{\omega}$.
\end{corollary}
\begin{proof}
Since $F$ is a countable Noetherian group, it has at most countably many subgroups.
Moreover, each $\Fin^{m}_{<}$ is countable, and there are only countably many finite tuples in any countable set.
Hence we can use Proposition~\ref{prop:decomposition-mixing-compact} to obtain a countable descending chain of sub-IP-rings such that the decomposition property holds for each tuple for one of these sub-IP-rings.
The required sub-IP-ring is then obtained by a diagonal procedure, cf.\ \cite[Lemma 1.4]{MR833409}.
\end{proof}
Thus the good group is not really relevant for our multiple recurrence theorem, which we can now formulate as follows.
\begin{theorem}
\label{thm:SZ-general}
Let $G$ be a nilpotent group and $F\leq\VIP$ an FVIP group.
Consider a right measure-preserving action of $G$ on an arbitrary (not necessarily regular) probability space $(X,\mathcal{A},\mu)$.
Let $S_{0},\dots,S_{t}\in\PE{F}{m}$ be arbitrary polynomial expressions and $A\in\mathcal{A}$ with $\mu(A)>0$.
Then there exists a sub-IP-ring $\Fin' \subset \Fin$ such that
\[
\IPlim_{\vec\alpha \in (\Fin')^{m}_{<}} \mu\left( \cap_{i=0}^{t} A S_{i}(\vec\alpha)\inv \right) > 0.
\]
\end{theorem}
\begin{proof}
We can assume $S_{0}\equiv 1_{G}$.
By Corollary~\ref{cor:finite-good} we may assume that $S_{0},\dots,S_{t}\in\FE$ for some good subgroup $\FE\leq\PE{\VIP}{m}$.
Then we can replace $G$ by a countable group that is generated by the union of ranges of elements of $\FE$.
Next, we can replace $\mathcal{A}$ by a separable $G$-invariant $\sigma$-algebra generated by $A$.
Finally, we can assume that $X$ is regular and apply Theorem~\ref{thm:SZ}.
\end{proof}
Theorem~\ref{thm:SZ-intro} follows from Theorem~\ref{thm:SZ-general} and Lemma~\ref{lem:poly-fvip} with the filtration \eqref{eq:scalar-poly-filtration-2}, $d$ being the maximal degree of the generalized polynomials $p_{i,j}$.

Observe that in Theorem~\ref{thm:gen-poly-FVIP} for (not necessarily admissible) generalized polynomials we can choose $n$ from a finite set that only depends on the generalized polynomial.
In view of this fact we have the following variant of Corollary~\ref{cor:SZ-combinatorial} for generalized polynomials.
\begin{corollary}
\label{cor:SZ-combinatorial-non-admissible}
Let $G$ be a finitely generated nilpotent group, $T_{1},\dots,T_{t}\in G$ and $p_{i,j}:\Z^{m}\to\Z$, $i=1,\dots,t$, $j=1,\dots,s$, be generalized polynomials.
Then there exist finite sets $S_{j}$, $j=1,\dots,s$, such that for every subset $E\subset G$ with positive upper Banach density the set
\[
\Big\{ \vec n\in\Z^{m} : \exists g\in G, \exists s_{j}\in S_{j} : gs_{j}\prod_{i=1}^{t}T_{i}^{p_{i,j}(\vec n)}\in E, j=1,\dots,s \Big\}
\]
is FVIP* in $\Z^{m}$.
\end{corollary}
Since every member set of an idempotent ultrafilter contains an IP set this implies a multidimensional version of \cite[Theorem 1.23]{MR2747062} that holds for every idempotent ultrafilter, see \cite[Remark 3.42]{MR2747062}.

%%% Local Variables: 
%%% mode: latex
%%% TeX-master: "phd-thesis.tex"
%%% End: 

\chapter{Higher order Fourier analysis}
\label{chap:fourier}
\renewcommand*{\PolyDomain}{\Z}
Through the work of Host and Kra \cite{MR2150389} and Ziegler \cite{MR2257397} on characterisitc factors for multiple term ergodic averages, \emph{nilmanifolds} became a central object of study in this area.
More recently, it became apparent that there are some advantages to studying \emph{polynomial}, rather than linear, structures on nilmanifolds, be it in form of dynamical parallelepipeds \cite{MR2600993}, cube spaces \cite{2010arXiv1009.3825A}, or polynomial sequences \cite{MR2877065}.
In this chapter we take the latter viewpoint, but put emphasis on qualitative ($N\to\infty$) rather than quanitative ($N$ large but fixed) phenomena.
A large part of this chapter is dedicated to the Green--Tao quantitative proof of Leibman's equidistribution results for polynomials on nilmanifolds, some parts of which are reused in our uniform Wiener-Wintner theorem for nilsequences (this is joint work with the author's advisor T.~Eisner \cite{arxiv:1208.3977}).

\section{Nilmanifolds and nilsequences}
Let us introduce the basic objects, and also fix the notation that will be used for them throughout this chapter.
By $G$ we denote a ($k$-step) nilpotent Lie group with a discrete cocompact subgroup $\Gamma\leq G$.
The compact manifold $G/\Gamma$ is called a \emph{($k$-step) nilmanifold}.
\index{nilmanifold}
It admits a unique left $G$-invariant Borel probability measure, called the \emph{Haar measure}, and integrals over $G/\Gamma$ are taken with respect to this measure unless stated otherwise.
Using the universal covering, we may and will assume that the connected component of the identity $G^{o}$ is simply connected.
We will also assume that $\Gamma$ is finitely generated.
We denote a filtration on $G$ by $\Gb$ and assume that every group in the filtration is a Lie subgroup of $G$.
The dimensions of these groups are denoted by $d:=\dim G$ and $d_{i}:=\dim G_{i}$.
More in general, ``Lie group'' stands for a nilpotent Lie group whose connected component of the identity is simply connected, and we only consider (pre-)filtrations in the category of Lie groups (nilpotent, with simply connected identity component).

The standard example that the reader should keep in mind is the Heisenberg group with the (lower central series) filtration
\[
\begin{pmatrix}
1 & \R & \R\\ 0 & 1 & \R\\ 0 & 0 & 1
\end{pmatrix}
=
\begin{pmatrix}
1 & \R & \R\\ 0 & 1 & \R\\ 0 & 0 & 1
\end{pmatrix}
\geq
\begin{pmatrix}
1 & 0 & \R\\ 0 & 1 & 0\\ 0 & 0 & 1
\end{pmatrix}
\geq
\begin{pmatrix}
1 & 0 & 0\\ 0 & 1 & 0\\ 0 & 0 & 1
\end{pmatrix}
\]
and the discrete Heisenberg group
$\left(\begin{smallmatrix}
1 & \Z & \Z\\ 0 & 1 & \Z\\ 0 & 0 & 1
\end{smallmatrix}\right)$
as a cocompact lattice.
The polynomial sequences with respect to this filtration are precisely the sequences of the form
$\left(\begin{smallmatrix}
1 & p_{1}(n) & p_{2}(n)\\ 0 & 1 & p_{3}(n)\\ 0 & 0 & 1
\end{smallmatrix}\right)$,
where $p_{1}$ and $p_{3}$ are linear real polynomials and $p_{2}$ is a quadratic real polynomial.

\subsection{Rationality and Mal'cev bases}
\begin{definition}[Rational subgroup]
\label{def:rational-filtration}
\index{rational!subgroup}
A subgroup $H\leq G$ is called \emph{$\Gamma$-rational} if $\Gamma\cap H \leq H$ is a cocompact subgroup.
\index{rational!filtration}
A filtration $\Gb$ on $G$ is called \emph{$\Gamma$-rational} if it consists of $\Gamma$-rational subgroups.
%for every $i=1,\dots,l$ the subgroup $\Gamma_{i}:=\Gamma\cap G_{i}$ is cocompact in $G_{i}$ and there exists a (fixed) Mal'cev basis for $G^{o}/\Gamma^{o}$ adapted to $\Gb^{o}$, where $G^{o}$ denotes the connected component of the identity of a group $G$ and $\Gamma^{o}:=\Gamma\cap G^{o}$.
\end{definition}
Let $\lag$ be the Lie algebra of a connected Lie group $G$.
Then $\exp:\lag\to G$ is a diffeomorphism; call its inverse $\log : G\to\lag$.
Let $X_{1},\dots,X_{d}$ be a basis for $\lag$.
\index{coordinates!of the first kind (exponential)}
An element $g\in G$ is said to have \emph{coordinates of the first kind} (or \emph{exponential coordinates}) $(t_{1},\dots,t_{d})$ if
\[
g = \exp(t_{1}X_{1}+\dots+t_{d}X_{d})
\]
\index{coordinates!of the second kind (Mal'cev)}
and \emph{coordinates of the second kind} $(u_{1},\dots,u_{d})$ if
\[
g = \exp(u_{1}X_{1})\dots \exp(u_{d}X_{d}).
\]
\begin{definition}[Mal'cev basis]
\label{def:malcev-basis}
\index{Mal'cev basis}
Assume that $G$ is connected.
An ordered basis $\{X_{1},\dots,X_{d}\}$ for the Lie algebra $\lag$ of $G$ is called a \emph{Mal'cev basis for $G/\Gamma$} if the following conditions are satisfied.
\begin{enumerate}
\item For each $i=1,\dots,d$ the subspace spanned by $X_{i},\dots,X_{d}$ is a Lie algebra ideal of $\lag$.
\item For each $g\in G$ there exist unique coordinates of the second kind $t_{1},\dots,t_{d_{1}}\in\R$, called \emph{Mal'cev coordinates} of $g$, such that $g=\exp(t_{1}X_{1})\dots\exp(t_{d}X_{d})$.
\item The lattice $\Gamma$ consists precisely of the elements with integer Mal'cev coordinates.
\end{enumerate}
Let $\Gb$ be a filtration of length $l$ on $G$ that consists of connected, simply connected Lie groups.
The Mal'cev basis $\{X_{1},\dots,X_{d}\}$ is said to be \emph{adapted to $\Gb$} if the following additional condition is satisfied.
\begin{enumerate}[resume]
\item For each $i=1,\dots,l$ the Lie algebra of $G_{i}$ coincides with $\<X_{d-d_{i}+1},\dots,X_{d}\>$.
\end{enumerate}
For not necessarily connected $G$ and $G_{i}$ a Mal'cev basis for $G/\Gamma$ (adapted to $\Gb$) is a Mal'cev basis for $G^{o}/\Gamma^{o}$ (adapted to $\Gb^{o}$).
\end{definition}
By a result of Mal'cev \cite{MR0028842} there always exists a Mal'cev basis adapted to the lower central series.
Using this fact we can explain the name ``rational subgroup'' as follows.
\begin{lemma}
\label{lem:rational-subgroup}
Assume that $G$ is connected and let $H\leq G$ be a connected Lie subgroup with Lie algebra $\mathfrak{h}$.
Then the following statements are equivalent.
\begin{enumerate}
\item\label{lem:rational-subgroup:cocompact} The subgroup $H$ is $\Gamma$-rational.
\item\label{lem:rational-subgroup:lattice} $\log\Gamma \cap \mathfrak{h}$ is a lattice in the Lie algebra $\mathfrak{h}$.
\item\label{lem:rational-subgroup:rational} The Lie algebra $\mathfrak{h}$ is spanned by rational combinations of vectors in a Mal'cev basis for $G/\Gamma$.
\end{enumerate}
\end{lemma}
\begin{proof}
Assume that \eqref{lem:rational-subgroup:cocompact} holds.
Mal'cev's result then implies existence of a Mal'cev basis for $H/(H\cap\Gamma)$, which in particular implies \eqref{lem:rational-subgroup:lattice}.
It is clear that \eqref{lem:rational-subgroup:lattice} implies \eqref{lem:rational-subgroup:rational}.
Finally, one may assume that the rational linear combintations in \eqref{lem:rational-subgroup:rational} are in fact integer linear combinations, and using properties of a Mal'cev basis this can be used to find a relatively compact fundamental domain for $\exp|_{\mathfrak{h}} \mod\Gamma$, proving \eqref{lem:rational-subgroup:cocompact}.
\end{proof}

%Although prefiltrations behave well in algebraic constructions, in our analytic arguments we will have to work with filtrations.
%Note that in a prefiltration $\Gb$ of length $l$, the subgroup $G_{l}$ need not be central in $G_{0}$.

\subsection{Commensurable lattices}
\begin{lemma}
\label{lem:comm-lattice}
Let $G/\Gamma$ be a nilmanifold and $\tilde\Gamma\leq G$ be a group that is commensurable with $\Gamma$.
Then the following assertions hold.
\begin{enumerate}
\item\label{lem:comm-lattice:lattice} $\tilde\Gamma$ is also a discrete cocompact subgroup.
\item\label{lem:comm-lattice:rational} Every $\Gamma$-rational subgroup $G'\leq G$ is also $\tilde\Gamma$-rational.
\end{enumerate}
\end{lemma}
\begin{proof}
To see \eqref{lem:comm-lattice:lattice} note that if $\tilde\Gamma\leq\Gamma$, then the natural map $G/\tilde\Gamma \to G/\Gamma$ is a covering map with finitely many sheets, and it follows that $G/\tilde\Gamma$ is compact.
If $\Gamma\leq\tilde\Gamma$, then $G/\tilde\Gamma$ is a quotient space of $G/\Gamma$, so it is clearly compact.
From this it follows that $\tilde\Gamma$ is cocompact in general.
Also, it is clear that $\tilde\Gamma$ is discrete if and only if $\Gamma$ is discrete.

The assertion \eqref{lem:comm-lattice:rational} follows since the groups $\Gamma\cap G'$ and $\tilde\Gamma\cap G'$ are commensurable whenever $\Gamma$ and $\tilde\Gamma$ are commensurable.
\end{proof}
An important class of examples of commensuarble lattices arises when one needs to replace a nilmanifold by a connected one.
\begin{lemma}
\label{lem:finer-lattice}
Let $G/\Gamma$ be a nilmanifold and $\Gb$ a $\Gamma$-rational filtration.
Then there exists a lattice $\Gamma\leq\tilde\Gamma\leq G$ such that $\Gamma$ has finite index in $\tilde\Gamma$ and $G_{i}/\tilde\Gamma_{i}$ is connected for every $i$.
\end{lemma}
\begin{proof}
We use induction on the length of the filtration.
If $\Gb$ is trivial, then there is nothing to show, so suppose that the conclusion holds for filtrations of length $d-1$ and consider a $\Gamma$-rational filtration $\Gb$ of length $d$.

By the rationality assumption we can write $G_{d}=G_{d}^{o}\oplus A$ in such a way that $\Gamma\cap A\leq A$ is a finite index subgroup.
Since $A$ is central in $G$, this implies that $\Gamma$ has finite index in $\Gamma A$.
Replacing $\Gamma$ by $\Gamma A$ if necessary, we may assume that $\Gamma G_{d} = \Gamma G_{d}^{o}$.

By the inductive assumption $\Gamma G_{d}/G_{d}$ is a finite index subgroup of a lattice $\tilde\Gamma_{/d}$ such that $(G_{i}/G_{d})/\tilde\Gamma_{/d}$ is connected for every $i$.
Let $\{\tilde\gamma_{j}\} \subset G/G_{d}$ be a finite set that together with $\Gamma G_{d}/G_{d}$ generates $\tilde\Gamma_{/d}$.
We can write $\tilde\gamma_{j}=g_{j}G_{d}$, and we have $g_{j}^{r}\in \Gamma G_{d}$ for some $r$ and all $j$.
Now recall that $\Gamma G_{d}=\Gamma G_{d}^{o}$ and that in the connected commutative Lie group $G_{d}^{o}$ arbitrary roots exist.
Hence, multiplying $g_{j}$ by an element of $G_{d}^{o}$ if necessary, we may assume that $g_{j}^{r}\in\Gamma$.

By Corollary~\ref{cor:finite-ext} $\Gamma$ has finite index in the group generated by $\Gamma$ and the elements $g_{j}$.
It remains to show that $G_{i}/\tilde\Gamma_{i}$ is connected for every $i$.
Recall that by the inductive assumption the quotient $G_{i}/\tilde\Gamma_{i} G_{d} = G_{i}/\tilde\Gamma_{i} G_{d}^{o}$ is connected, hence path connected.
Since the quotient of $G_{i}/\tilde\Gamma_{i}$ by continuous action of the path connected group $G_{d}^{o}$ is path connected, $G_{i}/\tilde\Gamma_{i}$ is connected.
\end{proof}
Since $\roots\Gamma \leq \comm_{G}(\Gamma)$, Lemma~\ref{lem:comm-lattice} has the following consequence.
\begin{corollary}
\label{cor:rational-subgroup-conjugate}
Let $G/\Gamma$ be a nilmanifold and $G'\leq G$ a $\Gamma$-rational subgroup.
Then for every $\gamma\in\roots{\Gamma}$ the subgroup $\gamma\inv G' \gamma$ is $\Gamma$-rational.
\end{corollary}
It is also useful to know what the conjugation map looks like in coordinates.
\begin{lemma}
\label{lem:conjugation}
Let $G/\Gamma$ be a nilmanifold with a Mal'cev basis adapted to $\Gb$ and $\gamma\in\roots{\Gamma}$.
Then the conjugation map $g\mapsto \gamma\inv g\gamma$ is linear, unipotent, and upper triangular with rational coefficients in coordinates of the first kind on $G^{o}$, and it is polynomial with rational coefficients in coordiantes of the second kind.

If in addition $G^{o}$ is commutative and $\gamma\in\Gamma$, then the conjugation map is linear and unipotent with integer coefficients in coordinates of the first and the second kind.
\end{lemma}
\begin{proof}
We have $\gamma^{r}\in\Gamma$ for some $r\in\N_{>0}$.
The conjugation map is conjugated to $\Ad(\gamma)$ by the exponential map, so it is linear in coordinates of the first kind.
It is unipotent and upper triangular in coordinates of the first kind since $G$ is nilpotent.

Suppose that $\Ad(\gamma)$ has an irrational coefficient in coordinates of the first kind.
Then $\Ad(\gamma)^{r} = \Ad(\gamma^{r})$ also has an irrational coefficient.
In view of the upper triangular form of the coordinate change maps between coordinates of the first and the second kind \cite[(A.2)]{MR2877065}, this implies that the $r$-th power of the conjugation map, written in coordinates of the second kind, maps some point of $\Z^{\dim G}$ to a point with an irrational coordinate.
This is a contradiction, since conjugation by an element of $\Gamma$ preserves $\Gamma$.

By \cite[Lemma A.2]{MR2877065} this implies that the conjugation map is polynomial with rational coefficients in coordinates of the second kind.

If $G^{o}$ is commutative, then coordinates of the first and the second kind coincide, so in the case $\gamma\in\Gamma$ the above argument shows that the conjugation map has integer coefficients in coordinates of both kinds.
\end{proof}

For completeness we also specialize this result to nilmanifolds whose structure group's connected component of the identity is commutative.
This is most useful in conjunction with Corollary~\ref{thm:leibman-equidistribution-criterion}.
\begin{lemma}[{\cite[Proposition 3.1]{MR2191231}}]
\label{lem:affine-unipotent}
Let $X=G/\Gamma$ be a connected nilmanifold and suppose that $G^{o}$ is commutative.
Then there is a homomorphism $X\cong\T^{d}$ such that for every $a\in G$ the map $x\mapsto ax$ is conjugated to a unipotent affine transformation on $\T^{d}$, that is, there exists a nilpotent integer matrix $N$ and a constant $b\in\T^{d}$ such that, with the above identification, $ax = x+Nx+b$.
\end{lemma}
\begin{proof}
Let $d=\dim G$.
Since $G^{o}$ is commutative, we have $G^{o}\cong\R^{d}$, the Lie group isomorphism being given by coordinates of the first or second kind (which coincide).
With this identification we have $(\Gamma\cap G^{o})\cong\Z^{d}$.

Since $X$ is connected, every element $a\in G$ can be written as $a=g\gamma$ with $g\in G^{o}$, $\gamma\in\Gamma$.
In particular, $X\cong G^{o}/\Gamma$.
For every $h\in G^{o}$ we have
\[
ah\Gamma = g(\gamma h \gamma\inv) \Gamma.
\]
The conjugation map by $\gamma$ is unipotent with integer coefficients in coordinates of the second kind by Lemma~\ref{lem:conjugation}, and multiplication by $g$ is a translation in coordinates of the second kind.
\end{proof}

\subsection{Cube construction}
We outline a special case of the cube construction of Green, Tao, and Ziegler \cite[Definition B.2]{MR2950773} using notation of Green and Tao \cite[Proposition 7.2]{MR2877065}.
We will only have to perform it on filtrations, but even in this case the result is in general only a prefiltration.
\begin{definition}[Cube filtration]
\index{cube filtration $\Gb^{\square}$}
Given a prefiltration $\Gb$ we define the prefiltration $\Gb^{\square}$ by
\[
G^{\square}_{i} := G_{i} \times_{G_{i+1}} G_{i} = \<G_{i}^{\triangle},G_{i+1}\times G_{i+1}\> = \{(g_{0},g_{1})\in G_{i}\times G_{i} : g_{0}\inv g_{1}\in G_{i+1}\},
\]
where $G^{\triangle}=\{(g_{0},g_{1})\in G^{2} : g_{0}=g_{1}\}$ is the diagonal group corresponding to $G$.
By an abuse of notation we refer to the filtration obtained from $\Gb^{\square}$ by replacing $G_{0}^{\square}$ with $G_{1}^{\square}$ as the ``filtration $\Gb^{\square}$''.
\end{definition}
To see that this indeed defines a prefiltration let $x\in G_{i}$, $y\in G_{i+1}$, $u\in G_{j}$, $v\in G_{j+1}$, so that $(x,xy)\in G_{i}^{\square}$ and $(u,uv)\in G_{j}^{\square}$.
Then $[(x,xy),(u,uv)]=([x,u],[xy,uv]) \in G_{i+j}^{\square}$ by \eqref{eq:ab-cd} (or see \cite[Proposition 7.2]{MR2877065}).

For induction purposes it is important to know that $\Gb^{\square}$ is rational provided that $\Gb$ is.
This follows from the next lemma.
\begin{lemma}[Rationality of the cube filtration]
\label{lem:rational-cube}
Let $\Gb$ be a $\Gamma$-rational filtration.
Then the filtration
\[
G_{0}^{2} = G_{1}^{2} \geq G_{1}^{\square} \geq G_{2}^{2} \geq G_{2}^{\square} \geq
\dots \geq G_{l}^{2} \geq G_{l}^{\square} \geq G_{l+1}^{2} = \{1_{G\times G}\}.
\]
is $\Gamma^{2}$-rational.
In particular, $\Gamma^{\square}=\Gamma^{2} \cap G_{1}^{\square}$ is a cocompact lattice in $G_{1}^{\square}$ and the filtration $\Gb^{\square}$ is $\Gamma^{\square}$-rational.
\end{lemma}
\begin{proof}
Observe first that $(\Gb^{\square})^{o} = (\Gb^{o})^{\square}$, since both these prefiltrations consist of closed connected subgroups of $G^{2}$ whose Lie algebras coincide.
The existence of the required Mal'cev basis follows from a result of Green and Tao \cite[Lemma 7.4]{MR2877065}.
Clearly, $\Gamma_{i}^{2}$ is cocompact in $G_{i}^{2}$ for every $i=1,\dots,l$.

It remains to show that $\Gamma_{i}^{\square} = \Gamma^{2}\cap G_{i}^{\square}$ is cocompact in $G_{i}^{\square}$ for every $i=1,\dots,l$.
The existence of an adapted Mal'cev basis implies that $\Gamma^{2} \cap (G_{i}^{\square})^{o}$ is cocompact in $(G_{i}^{\square})^{o}$.
Let $\tilde\Gamma \geq \Gamma$ be the finite index surgroup provided by Lemma~\ref{lem:finer-lattice}.
Writing
\[
\tilde\Gamma^{\square}_{i} = \< (\tilde\Gamma_{i})^{\triangle}, (\tilde\Gamma_{i+1})^{2} \>
\]
we see that $\tilde\Gamma^{\square}_{i}$ is a finitely generated subgroup of $\roots{\Gamma^{\square}_{i}}$, so it is a finite index surgroup of $\Gamma^{\square}_{i}$ by Corollary~\ref{cor:finite-ext}.
On the other hand,
\[
G_{i}^{\square}
= \< (G_{i}^{o}\tilde\Gamma_{i})^{\Delta}, (G_{i+1}^{o}\tilde\Gamma_{i+1})^{2}\>
= (G_{i}^{\square})^{o} \tilde\Gamma_{i}^{\square},
\]
so that $G_{i}^{\square}/\tilde\Gamma_{i}^{\square}$ is connected.
By the above it is compact, and in view of Lemma~\ref{lem:comm-lattice} this implies that $G_{i}^{\square}/\Gamma^{\square}$ is compact.
\end{proof}

\begin{lemma}
\label{lem:poly-cube}
Let $g\in\poly$.
Then for every $k\in\Z$ the map
\[
\cuben{g}{k}(n) := (g(n+k),g(n))
\]
is $\Gb^{\square}$-polynomial.
\end{lemma}
\begin{proof}
We use induction on the length $l$ of the prefiltration $\Gb$.
Indeed, for $l=-\infty$ there is nothing to show.
If $l\geq 0$, then $\cuben{g}{k}$ takes values in $G^{\square}_{0}$ since $g(n)\inv g(n+k)=D_{k}g(n)\in G_{1}$ by definition of a polynomial.
Moreover $D_{k'}(\cuben{g}{k})=\cuben{(D_{k'}g)}{k}(n)$, so that $D_{k'}(\cuben{g}{k})$ is $\Gb[+1]^{\square}$-polynomial by the induction hypothesis.
\end{proof}

\subsection{Vertical characters}\label{sec:vertical}
Let $G/\Gamma$ be a nilmanifold and $\Gb$ a $\Gamma$-rational filtration of length $l$.
Then $G/\Gamma$ is a smooth principal bundle with the compact commutative Lie structure group $G_{l}/\Gamma_{l}$.
The fibers of this bundle are called ``vertical'' tori (as opposed to the ``horizontal'' torus $G/\Gamma G_{2}$) and everything related to Fourier analysis on them is called ``vertical''.
\begin{definition}[Vertical character]
\index{vertical character}
Let $G/\Gamma$ be a nilmanifold and $\Gb$ a $\Gamma$-rational filtration on $G$.
A measurable function $F$ on $G/\Gamma$ is called a \emph{vertical character} if there exists a character $\chi\in \widehat{G_{l}/\Gamma_{l}}$ such that
for every $g_{l}\in G_{l}$ and a.e.\ $y\in G/\Gamma$
we have $F(g_{l}y)=\chi(g_{l}\Gamma_{l})F(y)$.
\end{definition}

\begin{definition}[Vertical Fourier series]
\label{def:vertical-Fourier-series}
\index{vertical Fourier series}
Let $G/\Gamma$ be a nilmanifold and $\Gb$ be a $\Gamma$-rational filtration on $G$.
For every $F\in L^{2}(G/\Gamma)$ and $\chi\in\widehat{G_{l}/\Gamma_{l}}$ let
\begin{equation}
\label{eq:vert-char-repr}
F_{\chi}(y):=\int_{G_{l}/\Gamma_{l}}F(g_{l}y) \ol{\chi}(g_{l}) \dif g_{l}.
\end{equation}
\end{definition}
With this definition $F_{\chi}$ is
defined almost everywhere
and is a vertical character as witnessed by the character $\chi$.
The usual Fourier inversion formula implies that $F=\sum_{\chi\in\widehat{G_{l}/\Gamma_{l}}}F_{\chi}$ in $L^{2}(G/\Gamma)$.
We further need the following variant of Bessel's inequality.
\begin{lemma}[Bessel-type inequality for vertical Fourier series]\label{lem:bessel}
\index{vertical Fourier series!Bessel-type inequality}
Let $p\in[2,\infty)$ and $F\in L^{p}(G/\Gamma)$.
Then
\[
\sum_\chi \|F_\chi\|_{L^p(G/\Gamma)}^p\leq \|F\|_{L^p(G/\Gamma)}^p.
\]
\end{lemma}
Note that the analogue for $p=\infty$ follows immediately from \eqref{eq:vert-char-repr}.
\begin{proof}
Since vertical characters have constant absolute value on $G_l/\Gamma_{l}$-fibers, we have by \eqref{eq:vert-char-repr} and the Cauchy-Schwarz inequality
\begin{align*}
\|F_\chi\|_{L^p(G/\Gamma)}^p
&=
\int_{G/\Gamma} \int_{G_l/\Gamma_{l}} |F_\chi(h h_l)|^2 \dif h_l \cdot |F_\chi(h)|^{p-2}\, \dif h\\
&\leq \int_{G/\Gamma} \int_{G_l/\Gamma_{l}}   |F_\chi(hh_l)|^2 \dif h_l  \Big(\int_{G_l/\Gamma_{l}} |F(hh_l)|^2 \dif h_l\Big)^{p/2 -1} \dif h
\end{align*}
for every $\chi$.
By the Plancherel identity and H\"older's inequality this implies
\begin{align*}
\sum_{\chi} \|F_{\chi}\|_{L^p(G/\Gamma)}^p
&\leq \int_{G/\Gamma} \Big( \int_{G_l/\Gamma_{l}} |F(hh_l)|^2 \dif h_l\Big)^{p/2} \dif h \\
&\leq \int_{G/\Gamma} \int_{G_l/\Gamma_{l}} |F(hh_l)|^p \dif h_l \dif h
= \|F\|_{L^p(G/\Gamma)}^p,
\end{align*}
finishing the proof.
\end{proof}
The correct analog of the Plancherel identity for vertical Fourier series reads
\[
\sum_\chi \|F_\chi\|_{U^{l}(G/\Gamma)}^{2^{l}}=\|F\|_{U^{l}(G/\Gamma)}^{2^{l}},
\]
where $U^{l}$ stands for appropriate Gowers-Host-Kra seminorms, see \textcite[Lemma 10.2]{MR2944094} for the case $l=3$.

\begin{definition}[Sobolev space]
\index{Sobolev space}
Let $G/\Gamma$ be a nilmanifold with a $\Gamma$-rational filtration, so in particular we have a Mal'cev basis $\{X_{1},\dots,X_{d}\}$ for the Lie algebra of $G$.
We identify the vectors $X_{i}$ with their extensions to right invariant vector fields on $G/\Gamma$.
The Sobolev space $W^{j,p}(G/\Gamma)$, $j\in\N$, $1\leq p<\infty$, is defined by the norm
\[
\| F \|_{W^{j,p}(G/\Gamma)}^{p}
=
\sum_{a=0}^{j} \sum_{b_{1},\dots,b_{a}=1}^{d} \| X_{b_{1}}\dots X_{b_{a}} F \|_{L^{p}(G/\Gamma)}^{p}.
\]
\end{definition}
We will write $A\lesssim_{D} B$ if $A$ and $B$ satisfy the inequality $A\leq CB$ with some constant $C$ that depends on some auxiliary constant(s) $D$ and some geometric data.

\begin{lemma}[Control on Sobolev norms in a vertical Fourier series]
\label{lem:estimate-vertical-fourier-series}
\index{vertical Fourier series!Bessel-type inequality for Sobolev norms}
Let $j\in\N$ and $p\in[2,\infty)$.
For every smooth function $F$ on $G/\Gamma$ we have
\[
\sum_\chi \|F_\chi\|_{W^{j,p}(G/\Gamma)}\lesssim_{j,p} \|F\|_{W^{j+d_{l},p}(G/\Gamma)}.
\]
\end{lemma}
\begin{proof}
The compact abelian Lie group $G_{l}/\Gamma_{l}$ is isomorphic to a product of a torus and a finite group.
In order to keep notation simple we will consider the case $G_{l}/\Gamma_{l} \cong \T^{d_{l}}$, the conclusion for disconnected $G_{l}/\Gamma_{l}$ follows easily from the connected case.
We rescale the last $d_{l}$ elements of the Mal'cev basis in such a way that they correspond to the unit tangential vectors at the origin of the torus $\T^{d_{l}}$.
The characters on $G_{l}/\Gamma_{l}$ are then given by $\chi_{\m}(z_{1},\dots,z_{m})=z_{1}^{m_{1}}\cdot\dots\cdot z_{d_{l}}^{m_{d_{l}}}$ with $\m=(m_{1},\dots,m_{d_{l}})\in\Z^{d_{l}}$.
Observe that by (\ref{eq:vert-char-repr}) and the commutativity of $G_l$ we have $(\partial_{i} F)_\m=\partial_{i}(F_\m)=m_iF_\m$ for every $i$ and $\m$, where $\partial_{i}$ denotes the derivative along the $i$-th coordinate in $\T^{d_{l}}$.
Therefore, by H\"older's inequality and Lemma~\ref{lem:bessel}
\begin{multline*}
\Big(\sum_{m_1, \ldots, m_{d_l}\neq 0}\|F_{\chi_{\m}} \|_{L^{p}}\Big)^{p}
= \Big(\sum_{m_1, \ldots, m_{d_l}\neq 0} \frac{1}{|m_1\cdots m_{d_l}|}\|m_1\cdots m_{d_l} F_{\chi_{\m}} \|_{L^{p}}\Big)^{p}\\
\leq \Big(\sum_{m_1, \ldots, m_{d_l}\neq 0} \Big|\frac{1}{m_1\cdots m_{d_l}}\Big|^{\frac{p}{p-1}} \Big)^{p-1} \sum_\m \|m_1\cdots m_{d_l}  F_{\chi_{\m}} \|_{L^{p}}^{p} \\
\lesssim \sum_{\m} \|\partial_{1}\ldots \partial_{d_l} F_{\chi_{\m}}\|_{L^{p}}^{p}
\leq \|\partial_{1}\ldots \partial_{d_l} F\|_{L^{p}}^{p}
\leq \|F\|_{W^{d_l,p}}^{p}.
\end{multline*}
By the centrality of $G_{l}$ the operations of taking derivatives along elements of the Mal'cev basis and taking the $\chi$-th vertical character \eqref{eq:vert-char-repr} commute, so we have
\[
\sum_{m_1, \ldots, m_{d_l}\neq 0}\|F_{\chi_{\m}}\|_{W^{j,p}}
\lesssim \|F\|_{W^{j+d_l,p}}
\]
for every $j\in \N$.
The same argument works if some of the indices $(m_1,\ldots,m_{d_l})$ vanish, in which case a smaller number of derivatives is added to $j$, and thus altogether
\[
\sum_\m\|F_{\chi_{\m}}\|_{W^{j,p}}
\lesssim \|F\|_{W^{j+d_l,p}}.
\qedhere
\]
\end{proof}

We will need an estimate on the $L^{\infty}$ norm of a vertical character in terms of a Sobolev norm with minimal smoothness requirements.
To this end we would like to use a Sobolev embedding theorem on $G/\Gamma G_{l}$, since this manifold has lower dimension than $G/\Gamma$.
Morally, a vertical character is a function on the base space $G/\Gamma G_{l}$ that is extended to the principal $G_{l}/\Gamma_{l}$-bundle $G/\Gamma$ in a multiplicative fashion.
However, in general this bundle lacks a global cross-section, so we are forced to work locally.
\begin{lemma}[Sobolev embedding]\label{lem:sobolev-embedding}
\index{theorem!Sobolev embedding}
Let $G/\Gamma$ be a nilmanifold and $\Gb$ a $\Gamma$-rational filtration of length $l$ on $G$.
Then for every $1\leq p\leq\infty$ and every vertical character $F\in W^{d-d_{l},p}(G/\Gamma)$ we have
\[
\|F\|_{\infty} \lesssim_{p} \|F\|_{W^{d-d_{l},p}},
\]
where the implied constant does not depend on $F$.
\end{lemma}
\begin{proof}
The case $p=\infty$ is clear, so we may assume $p<\infty$.

Since $\Gamma$ is discrete there exists a neighborhood $U \subset G$ of the identity such that the quotient map $U\to G/\Gamma$ is a diffeomorphism onto its image.
Let $M\subset G$ be a $(d-d_{l})$-dimensional submanifold that intersects $G_{l}$ in $e_{G}$ transversely.
By joint continuity of multiplication in $G$ we may find neighborhoods of identity $V\subset G_{l}$ and $W\subset M$ such that $VW\subset U$.
By transversality the differential of the map $\psi:V\times W \to G$, $(v,w)\mapsto vw$ is invertible at $(e_{G},e_{G})$, so by the inverse function theorem and shrinking $V,W$ if necessary we may assume that $\psi$ is a diffeomorphism onto its image.
We may also assume that $V,W$ are connected, simply connected and have smooth boundaries.
Recalling that the quotient map $U\to G/\Gamma$ is a diffeomorphism, we obtain a chart $\Psi:V\times W\to G/\Gamma$ for a neighborhood of $e_{G}\Gamma$ that has the additional property that $\Psi(g_{l}v,w)=g_{l}\Psi(v,w)$ whenever $v,g_{l}v\in V$.
Shrinking $V$ and $W$ further if necessary we may assume that the differential of $\Psi$ and its inverse are uniformly bounded.
By homogeneity we obtain similar charts for some neighborhoods of all points of $G/\Gamma$.
By compactness $G/\Gamma$ can be covered by finitely many such charts, so it suffices to estimate $\|F\|_{L^{\infty}(\im\Psi)}$ in terms of $\|F\|_{W^{d-d_{l},p}(\im\Psi)}$.

By definition of Sobolev norms we have
\[
\int_{v\in V} \|F\circ\Psi\|_{W^{d-d_{l},p}(\{v\}\times W)}^{p} \dif v \lesssim \|F\circ\Psi\|_{W^{d-d_{l},p}(V\times W)}^{p} \lesssim \|F\|_{W^{d-d_{l},p}(\im\Psi)}^{p}.
\]
Since $F$ is a vertical character and by multiplicativity of $\Psi$ in the first argument, the integrand on the left-hand side is constant, so that
\[
\|F\circ\Psi\|_{W^{d-d_{l},p}(\{v\}\times W)} \lesssim \|F\|_{W^{d-d_{l},p}(\im\Psi)} \text{ for all } v\in V,
\]
the bound being independent of $v$.
Now, $W$ is a $d-d_{l}$ dimensional manifold, so the usual Sobolev embedding theorem \cite[Theorem 4.12 Part I Case A]{MR2424078} applies and we obtain
\[
\|F\circ\Psi\|_{L^{\infty}(\{v\}\times W)} \lesssim \|F\circ\Psi\|_{W^{d-d_{l},p}(\{v\}\times W)} \lesssim \|F\|_{W^{d-d_{l},p}(\im\Psi)}^{p}.
\]
By the above discussion this implies the desired estimate.
\end{proof}

\section{Leibman's orbit closure theorem}
\subsection{Nilsequences}
With the advent of Host-Kra-Ziegler structure theory, \emph{nilsequences} came to be seen as the basic structure block of measure-preserving dynamical systems.
\begin{definition}
\label{def:nilsequence}
\index{nilsequence}
Let $G/\Gamma$ be a nilmanifold.
Let further $\Gb$ be a $\Gamma$-rational filtration of length $l$ on $G$.
Then for every polynomial $g\in\poly$ and $F\in C(G/\Gamma)$ we call the sequence $(F(g(n)\Gamma))_{n}$ a \emph{basic $l$-step nilsequence}.
An \emph{$l$-step nilsequence} is a uniform limit of basic $l$-step nilsequences (which are allowed to come from different nilmanifolds and filtrations).
\end{definition}
Note that the groups in the filtration $\Gb$ are not assumed to be connected.
In fact, by the remark following \cite[Theorem 3]{MR2445824}, not every nilsequence arises from nilmanifolds associated to connected Lie groups.
Nilsequences appear naturally in connection with norm convergence of multiple ergodic averages \cite{MR2150389}.
The $1$-step nilsequences are exactly the almost periodic sequences.
For examples and a complete description of $2$-step nilsequences see \textcite{MR2445824}.
For a characterization of nilsequences of arbitrary step in terms of their local properties see \cite[Theorem 1.1]{MR2600993}.

Although it is possible to express basic nilsequences as basic nilsequences of the same step associated to ``linear'' sequences of the form $(g^{n})_{n}$ (this is essentially due to Leibman \cite{MR2122919}, see e.g.\ \textcite[Proposition 2.1]{MR2465660} or \textcite[Proposition C.2]{MR2950773} in the setting of connected Lie groups),
 ``polynomial'' nilsequences, in addition to being formally more general, seem to be better suited for inductive purposes.
 This has been observed recently and utilized in connection with additive number theory, see e.g.~\textcite{MR2950773} and \textcite{MR2680398}.
 
Clearly, one can replace $\poly$ by $\polyn$ in Definition~\ref{def:nilsequence}.
Indeed, if $g\in\poly$ and $F\in C(G/\Gamma)$, then
\[
F(g(n)) = F_{g(0)}(g(0)\inv g(n)),
\]
where $F_{a}(x) := F(ax)$ is another continuous function on $G/\Gamma$.
Now the argument is a polynomial sequence that vanishes at zero.

In this construction we have $\|F_{a}\|_{\infty} \leq \|F_{a}\|_{\infty}$.
Unfortunately, one cannot in general estimate the norm of $F_{a}$ in a function space (such as Sobolev space or the space of Lipschitz functions) by the norm of $F$ in the same space.
A remedy consists in restricting $a$ to a relatively compact subset of $G$.
\begin{lemma}[Fundamental domain]
\label{lem:fundamental-domain}
\index{fundamental domain}
Let $\Gamma\leq G$ be a cocompact lattice.
Then there exists a relatively compact set $K\subset G$ and a map $G\to K$, $g\mapsto \{g\}$ such that $g\Gamma = \{g\}\Gamma$ and $\{\{g\}\}=\{g\}$ for each $g\in G$.
\end{lemma}
This follows readily from local homeomorphy of $G$ and $G/\Gamma$, from local compactness of $G$ and from compactness of $G/\Gamma$.
For example, for $G=\R$ and $\Gamma=\Z$ the fundamental domain $K$ can be taken to be the interval $[0,1)$ with the usual fractional part map $\{\cdot\}$.
In case of a general connected Lie group the fundamental domain can be taken to be $[0,1)^{d}$ in Mal'cev coordinates \cite[Lemma A.14]{MR2877065}, but we do not need this information.
For each nilmanifold that we consider we fix some map $\{\cdot\}$ as above.

Using the fractional part map we can rewrite a nilsequence associated to $g\in\poly$ and $F\in C(G/\Gamma)$ as
\[
F(g(n)) = F_{\{g(0)\}}(\{g(0)\}\inv g(n) g(0)\inv \{g(0)\}).
\]
This is made possible by the fact that $g(0)\inv \{g(0)\} \in\Gamma$.
Note that $F_{\{g(0)\}}$ belongs to a compact subset of $C(G/\Gamma)$ that does not depend on $g$.

Henceforth we will mostly consider nilsequences associated to polynomial sequences that vanish at zero, keeping at mind that the general case can be treated by the above trick.

 A key tool for many inductive proofs is the following modification of a construction due to Green and Tao, see e.g.\ \cite[Lemma 1.6.13]{MR2931680} and
\cite[\textsection 7]{MR2877065}, which shows that discrete derivatives of nilsequences associated to vertical characters are nilsequences of lower step.
Let $\Gb$ be a filtration of length $l$, $g\in\polyn$, and $F\in C(G/\Gamma)$ be a vertical character.
Then
\[
F(g(n+k)) \ol{F(g(n))}
= \cube{F}{g(k)}(\cuben{g}{k}(0)\inv \cuben{g}{k}),
\]
where $\cube{F}{a}(x,y) := F(ax)\ol{F(y)}$.
To see that this is a nilsequence of step $l-1$, note that the sequence $\cuben{g}{k}$ is $\Gb^{\square}$-polynomial by Lemma~\ref{lem:poly-cube}.
Moreover, the function $\cube{F}{a}$ is $G_{l}^{\square}$-invariant since $F$ is a vertical character, so we may factor by $G_{l}^{\square}$, thereby reducing the length of the filtration.

This construction suffers from the deficiency outlined above, namely that there is in general no control on $\cube{F}{a}$ in terms of $F$.
This can be resolved in the same way as before, considering the $G_{1}^{\square}$-valued $\Gb^{\square}$-polynomial sequence
\begin{equation}
\label{eq:g-tilde}
\cube{g}{k} := (\{g(k)\}\inv g(n+k) g(k)\inv \{g(k)\}, g(n)).
\end{equation}
Then we obtain
\[
F(g(n+k)) \ol{F(g(n))}
= \cube{F}{\{g(k)\}}(\cube{g}{k}(n)).
\]
We will sometimes abuse the notation and write $\cube{F}{k}$ instead of $\cube{F}{\{g(k)\}}$.
\begin{lemma}[Control on Sobolev norms in the cube construction]
\label{lem:sobolev-norm-of-tensor-product}
With the above notation we have
\begin{equation}
\label{eq:estimate-discrete-derivative}
\|\cube{F}{k}\|_{W^{j,p}(\tilde{G}/\tilde{\Gamma})}\lesssim_{j,p} \|F\|^{2}_{W^{j,2p}(G/\Gamma)} \text{ for any } j\in\N, p\in [1,\infty),
\end{equation}
where the implied constant does not depend on $k$ and $F$.
\end{lemma}
\begin{proof}
For the Mal'cev basis on $\tilde G/\tilde\Gamma$ that is induced by the Mal'cev basis on $G_{1}^{\square}/\Gamma^{\square}$ we have
\[
\|\cube{F}{k}\|_{W^{j,p}(\tilde G/\tilde\Gamma)} = \|F_{\{g(k)\}}\otimes\ol{F}\|_{W^{j,p}(G_{1}^{\square}/\Gamma^{\square})},
\]
so it suffices to estimate the latter quantity.

To this end observe that the Haar measure on $G_{1}^{\square}/\Gamma^{\square}$ is a self-joining of the Haar measure on $G/\Gamma$ under the canonical projections to the coordinates.
Therefore and by the Cauchy-Schwarz inequality we have
\begin{align*}
\|F_{0}\otimes F_{1}\|_{L^p(G_{1}^{\square}/\Gamma^{\square})}^{2p}
&= \Big(\int_{G_{1}^{\square}/\Gamma^{\square}}|F_{0}(y_{0}) F_{1}(y_{1})|^{p} \dif\mu_{G_{1}^{\square}/\Gamma^{\square}}(y_{0},y_{1})\Big)^2\\
&\leq \int_{G_{1}^{\square}/\Gamma^{\square}} |F_{0}(y_{0})|^{2p} \dif\mu_{G_{1}^{\square}/\Gamma^{\square}}(y_{0},y_{1})\\
&\qquad \cdot \int_{G_{1}^{\square}/\Gamma^{\square}} |F_{1}(y_{1})|^{2p} \dif\mu_{G_{1}^{\square}/\Gamma^{\square}}(y_{0},y_{1})\\
&= \int_{G/\Gamma} |F_{0}|^{2p} \dif\mu_{G/\Gamma}  \int_{G/\Gamma} |F_{1}|^{2p} \dif\mu_{G/\Gamma}\\
&=\|F_{0}\|_{L^{2p}(G/\Gamma)}^{2p} \|F_{1}\|_{L^{2p}(G/\Gamma)}^{2p}
\end{align*}
for any smooth functions $F_{0},F_{1}$ on $G/\Gamma$.
Now recall that $\{g(k)\}\in K$ for some fixed compact set $K\subset G_{1}$, so that by smoothness of the group operation $\|F_{\{g(k)\}}\|_{L^{2p}(G/\Gamma)} \lesssim \|F\|_{L^{2p}(G/\Gamma)}$.
Similar calculations for the derivatives lead to the bound
\[
\|F_{\{g(k)\}}\otimes\ol{F}\|_{W^{j,p}(G_{1}^{\square}/\Gamma^{\square})} \lesssim_{j,p} \|F\|_{W^{j,2p}(G/\Gamma)}^{2}.
\qedhere
\]
\end{proof}

\subsection{Reduction of polynomials to connected Lie groups}
In the context of nilsequences it will sometimes be useful to replace $\Gb$-polynomial sequences by $\Gb^{o}$-polynomial sequences.
As remarked earlier, this is not possible in general.
Here we show that this becomes possible upon passing to an appropriate subsequence.

Given a prefiltration $\Gb$ we define a prefiltration $\Gb^{o}$ by $G_{i}^{o}=G_{i}^{o}$.
\begin{lemma}
\label{lem:split-connected-lattice}
Let $X=G/\Gamma$ be a nilmanifold and $\Gb$ a $\Gamma$-rational filtration.
Assume that $G_{i}/\Gamma_{i}$ is connected for each $i$.
Then every $\Gb$-polynomial sequence $g(n)$ can be written in the form
\[
g(n) = g^{o}(n) \gamma(n),
\]
where $g^{o}$ is a $\Gb^{o}$-polynomial sequence, and $\gamma$ is a $\Gamma_{\bullet}$-polynomial sequence.
\end{lemma}
\begin{proof}
We use induction on the length of the filtration $\Gb$.
If $\Gb$ has length $0$, then $g\equiv\id$, so we can take $g^{o}=\gamma\equiv\id$.
Suppose therefore that the conclusion is known for filtrations of length $d-1$ and consider a filtration $\Gb$ of length $d$.

By the induction hypothesis applied to the filtration $\Gb[/d]$ we can write
\[
g(n)G_{d} = g_{/d}^{o}(n) \gamma_{/d}(n),
\]
where $g_{/d}^{o}$ is a $(\Gb[/d])^{o}$-polynomial sequence and $\gamma_{/d}$ is a $\Gamma_{\bullet}/G_{d}$-polynomial sequence.
Since $(G_{i}/G_{d})^{o}$ is covered by $G_{i}^{o}$ for every $i$, we can lift $g_{/d}^{o}$ to a $\Gb^{o}$-polynomial sequence $g^{o}$ (here ``lift'' means that $g^{o}_{/d}=g^{o}G_{d}$).
Also, we can clearly lift $\gamma_{/d}$ to a $\Gamma_{\bullet}$-polynomial sequence $\gamma$.

It follows that $h=g (g^{o}\gamma)\inv$ is a $\Gb$-polynomial sequence with values in $G_{d}$.
By the rationality and connectedness assumption we can write $G_{d} = G_{d}^{o} \oplus A$ with $A \leq \Gamma$.
Splitting $h=h^{o}h^{\Gamma}$ accordingly and replacing $g^{o}$ and $\gamma$ by $g^{o}h^{o}$ and $\gamma h^{\Gamma}$, respectively, we obtain the claim.
\end{proof}

\begin{lemma}
\label{lem:periodic}
Let $G$ be a nilpotent group with a filtration $\Gb$ and let $\Gamma\leq G$ be a finite index subgroup.
Then for every $\Gb$-polynomial sequence $g(n)$ the sequence $g(n)\Gamma$ is periodic.
\end{lemma}
\begin{proof}
Replacing $\Gamma$ by a finite index subgroup that is normal in $G$ and working modulo $\Gamma$, we may assume that $G$ is finite and $\Gamma$ is trivial.

We use induction on length $d$ of $\Gb$.
If $d=0$, then the conclusion holds trivially.
If $d>0$, then by the induction hypotesis the discrete derivative $D_{1}g$ is periodic, and the conclusion follows.
\end{proof}

\begin{corollary}
\label{cor:splitting}
Let $G/\Gamma$ be a nilmanifold with a $\Gamma$-rational filtration $\Gb$.
Then there exists a lattice $\tilde\Gamma$ such that $\Gamma$ is a finite index subgroup of $\tilde\Gamma$ and
every $\Gb$-polynomial sequence $g$ can be written in the form $g=g^{o}\gamma$, where $g^{o}$ is $\Gb^{o}$-polynomial and $\gamma$ is $\tilde\Gamma_{\bullet}$-polynomial.
In particular, $\gamma\Gamma$ is periodic.
\end{corollary}
\begin{proof}
Consider $\tilde\Gamma$ given by Lemma~\ref{lem:finer-lattice}.
The required splitting is provided by Lemma~\ref{lem:split-connected-lattice}.
The claimed periodicity follows from Lemma~\ref{lem:periodic}.
\end{proof}

\subsection{Equidistribution criterion}
Recall that a sequence $(x_{n})$ in a regular measure space $(X,\mu)$ is called
\begin{enumerate}
\item \emph{equidistributed on $X$} if for every $f\in C(X)$ we have
\[
\lim_{N\to\infty} \aveN f(x_{n}) = \int f \dif\mu,
\]
\index{sequence!equidistributed}
\item \emph{well-distributed on $X$} if for every F\o{}lner sequence $(\Fo_{N})$ in $\Z$ and every $f\in C(X)$ we have
\index{sequence!well-distributed}
\[
\lim_{N\to\infty} \aveFN f(x_{n}) = \int f \dif\mu,
\]
\item \emph{totally equidistributed on $X$} if its restriction to every arithmetic progression $a\Z+b$, $0\leq b< a$, in $\Z$ is equidistributed on $X$, and
\index{sequence!totally equidistributed}
\item \emph{totally well-distributed on $X$} if its restriction to every arithmetic progression in $\Z$ is well-distributed on $X$.
\index{sequence!totally well-distributed}
\end{enumerate}
Leibman's equidistribution criterion tells that the only obstruction to total well-distribution of $\Gb$-polynomial sequences on a connected nilmanifold $G/\Gamma$ are \emph{horizontal characters}, that is, continuous homomorphisms $\eta:G\to\R$ such that $\eta(\Gamma)\subset\Z$ (see Theorem~\ref{thm:equid-crit-connected} for the precise formulation).
\index{horizontal character}
We will give a qualitative version of the proof that is due to Green and Tao \cite{MR2877065}.
For reader's convenience we will keep the notation as close to \cite{MR2877065} as possible.

The proof proceeds by induction on the length of the filtration.
In each step one performs the cube construction and factors out the diagonal central subgroup.
The induction hypothesis gives some information about horizontal characters on $G^{\square}$.
The next lemma describes how such horizontal characters induce horizontal characters on $G$.
\begin{lemma}
\label{lem:char-cube}
Let $G/\Gamma$ be a nilmanifold with a $\Gamma$-rational filtration $\Gb$.
Let $\eta:G^{\square} \to \R$ be a horizontal character.
Then the map
\[
\eta_{1}:G\to\R,
\quad g\mapsto \eta(g,g)
\]
is a horizontal character on $G/\Gamma$, the map
\[
\eta_{2}:G_{2}\to\R,
\quad g\mapsto\eta(g,\id)
\]
is a horizontal character on $G_{2}/\Gamma_{2}$ such that $\eta_{2}([G,G_{2}]) = \{0\}$ and the map $(x,y) \mapsto \eta_{2}([x,y])$ is a bihomomorphism (that is, a group homomorphism in each variable when the other variable is fixed).

In particular, the map $g\mapsto \eta_{2}([g,\gamma])$ is a horizontal character on $G/\Gamma$ for every $\gamma\in\Gamma$.
\end{lemma}
\begin{proof}
It is clear that $\eta_{1}$ and $\eta_{2}$ are horizontal characters.
For any $g\in G, h\in G_{2}$ we have
\begin{multline*}
\eta_{2}([g,h])
=\eta((g\inv h\inv gh,\id))
=\eta((g,g)\inv (h,\id)\inv (g,g) (h,\id))\\
=-\eta(g,g) - \eta(h,\id) + \eta(g,g) + \eta(h,\id)
= 0,
\end{multline*}
hence $\eta_{2}([G,G_{2}]) = \{0\}$ (note that the restriction $h\in G_{2}$ in the above calculation is necessary because otherwise $(h,\id) \not\in G^{\square}$).
This, together with the commutator identity \eqref{eq:ab-cd}, also shows that  the map $(x,y) \mapsto \eta_{2}([x,y])$ is a bihomomorphism.
\end{proof}

% \begin{corollary}
% \label{cor:conjugate-hor-char}
% Let $G/\Gamma$ be a nilmanifold and $\gamma\in\roots{\Gamma}$.
% Then there exists $R\in\N_{>0}$ such that for every horizontal character $\eta:G^{o}\to\R$ the map $g\mapsto R\eta(\gamma\inv g \gamma)$ is a horizontal character.
% \end{corollary}

The main step in the proof of the equidistribution criterion is the following trichotomy that allows one to transfer information from the cube spaces to the original nilmanifold.
We have nothing to add to the proof in \cite[\textsection 7]{MR2877065}.
\newcommand{\uBd}{upper Banach density}
\begin{proposition}
\label{prop:equid-trichotomy}
Let $G/\Gamma$ be a nilmanifold and $\Gb$ a $\Gamma$-rational filtration on $G$ consisting of connected groups.
Let $g\in\polyn$ be such that $g(1)=\{g(1)\}$ and suppose that there is a set of $h$ of \uBd\ at least $\delta$ and a non-trivial horizontal character $\eta:G^{\square} \to \R$ such that $\eta(\cube{g}{h}(\Z))\subset\Z$.
Then at least one of the following statements holds.
\begin{enumerate}
\item\label{equid-trichotomy:eta1} The map $\eta_{1}$ is a non-trivial horizontal character and $\eta_{1}(g(\Z))\subset\Z$.
\item\label{equid-trichotomy:eta2} There exists $i=1,\dots,d$ such that the map $\tilde\eta_{i}:G\to\R$, $g\mapsto \eta_{2}([g,\exp X_{i}])$, is a non-trivial horizontal character and $q\tilde\eta_{i}(g(\Z))\subset\Z$ for some natural number $q$ that is bounded in terms of $\delta$.
\item\label{equid-trichotomy:q} The map $\eta_{2}:G_{2}\to\R$ is a non-trivial horizontal character and $q\eta_{2}(g_{2}(\Z))\subset\Z$ for some natural number $q$ that is bounded in terms of $\delta$, where $g_{2}$ is defined by $g(n)=g(1)^{n}g_{2}(n)$.
\end{enumerate}
\end{proposition}
The next result shows what happens if a polynomial sequence fails to be well-distributed.
This is a qualitative version of the main result from \cite{MR2877065}, but we note that not all quantitativity has been removed.
In fact, it is essential for inductive purposes to have some uniformity over all polynomials.
\begin{theorem}
\label{thm:equid-crit-ind}
Let $G/\Gamma$ be a nilmanifold associated to a connected group $G$ and $\Gb$ a $\Gamma$-rational filtration on $G$.
Let a F\o{}lner sequence $(\Fo_{N})$ in $\Z$, a function $F\in C(G/\Gamma)$ with $\int F=0$, $\delta>0$, $s\in\N$, and $0\leq r<s$ be given.
Then there exists a finite set of non-trivial horizontal characters such that for every $g\in\polyn$ with $\limsup_{N} |\aveFN F(g(sn+r)\Gamma)|>\delta$ there exists a horizontal character $\eta$ on this list such that $\eta(g(\Z))\subset\Z$.
\end{theorem}
\begin{proof}
We use induction on the length $l$ of the filtration and the dimension $d_{2}$ of the group $G_{2}$.

First we reduce to the case that $\Gb$ consists of connected groups.
To this end we split $g=g^{o}\gamma$ as in Corollary~\ref{cor:splitting}, where $g^{o}$ is $\Gb^{o}$-polynomial and $\gamma$ is $\tilde\Gamma_{\bullet}$-polynomial for some finite index surgroup $\tilde\Gamma \geq \Gamma$ that does not depend on $g$.
In particular, $\gamma\Gamma$ is periodic, and the period $s'$ does not depend on $g$.
By the pigeonhole principle there exists $0\leq r'<s'$ such that
\[
\limsup_{N} | \ave{n}{(\Fo_{N}-r')/s'} F(g(s(s'n+r')+r)) | > \delta,
\]
and we can apply the connected case of the theorem.
Thus we may assume that the filtration $\Gb$ consists of connected groups.

Next we show that we may assume $s=1$, $r=0$.
In general our assumption can be rewritten as
\[
\limsup_{N} | \aveFN F_{\{g(r)\}}( \{g(r)\}\inv g(sn+r) g(r)\inv \{g(r)\}) | > \delta.
\]
Since $\{g(r)\}$ lies in a fixed compact set, the set of functions $F_{\{g(r)\}}$ that may appear above is compact, so it can be covered by finitely many balls of radius $\delta/2$, the covering being independent of $g$.
Hence there is a finite set of continuous functions on $G/\Gamma$ such that we have
\[
\limsup_{N} | \aveFN \tilde F( \{g(r)\}\inv g(sn+r) g(r)\inv \{g(r)\}) | > \delta/2
\]
for one of the functions $\tilde F$ in this set.
We can now apply the $s=1$, $r=0$ case of the theorem to the polynomial sequence $(\{g(r)\}\inv g(sn+r) g(r)\inv \{g(r)\})_{n}$ and the function $\tilde F$.
This provides us with a finite set of horizontal characters, for one of which we have
\[
\eta(\{g(r)\}\inv g(s\Z+r) g(r)\inv \{g(r)\})\subset\Z.
\]
This immediately implies $\eta(g(s\Z+r) g(r)\inv)\subset\Z$.
Now the sequence $(\eta(g(n)g(r)\inv))_{n}$ is a polynomial of degree at most $l$ that takes integer values on the arithmetic progression $s\Z+r$.
Hence, multiplying $\eta$ by a natural number that does not depend on $g$ if necessary, we may assume $\eta(g(\Z)g(r)\inv) \subset\Z$.
In view of $g(0)=\id$ this implies $\eta(g(\Z))\subset\Z$ as required.
Hence we may assume that $s=1$, $r=0$.

It remains to prove the conclusion under the additional assumptions that $\Gb$ consists of connected groups, $s=1$, and $r=0$.
Replacing $g$ by the sequence
\[
g(n) (\{g(1)\}\inv g(1))^{-n}
\]
we may also assume that $g(1) = \{g(1)\}$.
By uniform approximation we may assume that $F$ is smooth (this can be achieved for example using a smooth partition of identity and working locally).

If $l=1$, then $G/\Gamma$ is a torus.
Smoothness implies that the Fourier series $F = \sum_{\chi} F_{\chi}$ converges absolutely, so we may truncate it to a finite number of summands.
Given a polynomial $g$ as in the hypothesis, by the pigeonhole principle we see that
\[
\limsup_{N} |\aveFN F_{\chi}(g(n)\Gamma) | > 0
\]
for one of the (finitely many) Fourier components $F_{\chi}$.
We may assume $|F_{\chi}| \equiv 1$.
Then we have $F_{\chi}(g(n)\Gamma) = F_{\chi}(g(1)\Gamma)^{n}$, and the Kronecker equidistribution criterion implies that $(F_{\chi}\circ g\Gamma)\equiv 1$.
The character $F_{\chi}$ lifts to a horizontal character on $G$, and we obtain the claim.

Suppose now that $l\geq 2$.
Analogously to the commutative case, smoothness implies that the vertical Fourier series $F=\sum_{\chi}F_{\chi}$ (Definition~\ref{def:vertical-Fourier-series}) converges absolutely, so, decreasing $\delta$ if necessary, we can assume that $F$ has a vertical frequency $\chi$.
If this frequency vanishes, then we can factor out $G_{l}$ and use induction on the length of filtration.

Assume now that the vertical frequency $\chi$ is non-trivial.
By the van der Corput lemma (Lemma~\ref{VdC}) the set of $h$ such that
\[
\limsup_{N} | \aveFN F(g(h+n))\ol{F(g(n))} | > \delta
\]
has positive \uBd.
Recall that the above can be written as
\[
\limsup_{N} | \aveFN \cube{F}{h}(\cube{g}{h}(n)) | > \delta.
\]
Since the fractional part function $\{\cdot\}$ has relatively compact range, the set of functions $\cube{F}{h}$ is relatively compact.
Choosing a sufficienttly fine finite covering of this set, pigeonholing and decreasing $\delta$ if necessary we obtain \emph{one} function $\cube{F}{a}$ such that
\[
\limsup_{N} | \aveFN \cube{F}{a}(\cube{g}{h}(n)) | > \delta
\]
for a set of $h$ of positive \uBd.
Note that $\cube{F}{a}$ has a non-trivial vertical frequency with respect to $G_{l}^{2}$ and is $G_{l}^{\Delta}$-invariant.
Hence, factoring out $G_{l}^{\Delta}$, we see that $\cube{g}{h}$ is polynomial with respect to the filtration $\Gb^{\square}/G_{l}^{\Delta}$ that has length $l-1$ and $\cube{F}{a}$ has zero integral.

By the induction hypothesis we obtain a finite list of horizontal characters $\eta : G^{\square}/G_{l}^{\Delta} \to \R$ such that for each $h$ in out positive \uBd\ set there exists a character on this list with $\eta(\cube{g}{h}(\Z))\subset\Z$.
By the pigeonhole principle we may assume that the character $\eta$ does not depend on $h$.

We are now in position to apply Proposition~\ref{prop:equid-trichotomy}.
If the first or the second alternative from that proposition holds, then we are done, since the horizontal characters provided by that alternatives only depend on $\eta$.
It remains to consider the case that the last alternative from that proposition holds.
In this case $g$ is $\Gb'$-polynomial, where the filtration $\Gb'$ is defined by $G_{1}'=G_{1}$, $G_{i}' = G_{i} \cap \eta_{2}\inv(\frac1q \Z)$ for $i\geq 2$ (the fact that this is a filtration follows from Lemma~\ref{lem:char-cube}).
Note that $\dim G_{2}' < \dim G_{2}$ since $\eta_{2}$ is a non-trivial horizontal character on $G_{2}$.
\end{proof}
Now we bootstrap the last result to total well-distribution.
\begin{theorem}[Leibman's equidistribution criterion, connected case]
\label{thm:equid-crit-connected}
\index{equidistribution criterion}
Let $G/\Gamma$ be a nilmanifold associated to a connected group $G$ and $\Gb$ a $\Gamma$-rational filtration on $G$.
Then for every $g\in\polyn$ exactly one of the following alternatives holds.
\begin{enumerate}
\item For every subgroup $\tilde\Gamma\leq G$ that is commensurable with $\Gamma$ the sequence $g(n)\tilde\Gamma$ is totally well-distributed on $G/\tilde\Gamma$ or
\item there exists a non-trivial horizontal character $\eta : G \to \R$ such that $\eta(g(\Z))\subset\Z$.
\end{enumerate}
\end{theorem}
\begin{proof}
It is clear that the two statements are mutually exclusive, so it suffices to show that at least one of them holds.
Suppose that the first statement fails, that is, there exists a subgroup $\tilde\Gamma\leq G$ that is commensurable with $G$, an arithmetic progression $s\Z+r$, a F\o{}lner sequence $(\Fo_{N})$, and a function $F\in C(G/\tilde\Gamma)$ such that
\[
\aveFN F(g(sn+r) \tilde\Gamma) \not\to \int_{G/\tilde\Gamma} F
\quad\text{as } N\to\infty.
\]
Without loss of generality we may assume $\int F=0$.
By Theorem~\ref{thm:equid-crit-ind} we obtain a non-trivial horizontal character $\tilde\eta$ on $G/\tilde\Gamma$ such that $\tilde\eta(g(\Z)) \subset\Z$.
Note that $\tilde\eta$ takes integer values on the finite index subgroup $\tilde\Gamma\cap\Gamma \leq \Gamma$, from which it follows that $\eta(\Gamma)\subset \frac1R \Z$ for some $R$.
Hence $\eta = R\tilde\eta$ is a non-trivial horizontal character on $G/\Gamma$ such that $\eta(g(\Z)) \subset\Z$.
\end{proof}

\subsection{Leibman's orbit closure theorem}
In order to describe the orbit closure of a polynomial in a nilmanifold we need one more decomposition result for polynomials.
\begin{lemma}
\label{lem:splitting-equid}
Let $G/\Gamma$ be a nilmanifold with a $\Gamma$-rational filtration $\Gb$ and $H\leq G$ a rational subgroup.
Then for every $h\in\polyn[\Hb]$ there exists a closed connected rational subgroup $\tilde H\leq H$ such that $h$ can be written in the form $h=h^{o}\gamma$, where $\gamma\in\polyn[\roots\Gamma_{\bullet}]$, $h^{o} \in\polyn[\tHb^{o}]$, and for every subgroup $\tilde\Gamma\leq G$ that is commensurable with $\Gamma$ the sequence $h^{o} \tilde\Gamma$ is totally well-distributed on $\tilde H/\tilde\Gamma$.
\end{lemma}
It clearly suffices to obtain the conclusion for $H=G$, the other cases are only needed for the induction process.
\begin{proof}
We use induction on the dimension of $H$.
If $\dim H=0$, then $H\leq\roots\Gamma$, and we can set $h^{o}\equiv\id$, $\gamma=h$.
Suppose now that the conclusion is known for rational subgroups of dimension $<\dim H$.

Consider the splitting $h=h^{o}\gamma$ provided by Corollary~\ref{cor:splitting} applied to the nilmanifold $H/(H\cap\Gamma)$.
Replacing $h^{o}$ by $h^{o}h^{o}(0)\inv$ and $\gamma$ by $h^{o}(0) \gamma$ we may assume $h^{o}(0)=\gamma(0)=\id$.
Suppose that the conclusion of the lemma does not hold with $\tilde H = H^{o}$.
Then Theorem~\ref{thm:equid-crit-connected} shows that $h^{o}$ takes values in a proper rational subgroup $\tilde H\leq H^{o}$, namely the inverse image of $\Z$ under the horizontal character figuring in the second alternative in that theorem.
In this case we can conclude by the induction hypothesis.
\end{proof}
Thus we have split an arbitrary polynomial into a ``totally equidistributed'' and a ``rational'' part.
Further analysis of the rational part now yields Leibman's orbit closure theorem.
\begin{theorem}[Leibman's orbit closure theorem, cf.\ {\cite[Theorem B]{MR2122919}}]
\label{thm:Leibman-orbit-closure}
\index{theorem!Leibman orbit closure}
Let $G/\Gamma$ be a nilmanifold with a $\Gamma$-rational filtration $\Gb$.
Then for every $\Gb$-polynomial sequence $g$ there exists a closed connected $\Gamma$-rational subgroup $H$ such that $\Z$ can be partitioned into progressions on each of which $g(n)\Gamma$ is totally well-distributed on $g(0)Hc/\Gamma$ for some $c\in\roots\Gamma$.
\end{theorem}
In order to obtain the precise statement of \cite[Theorem B]{MR2122919} one could consider $g(0)Hg(0)\inv$ instead of $H$.
Note that this subgroup is in general not $\Gamma$-rational.
\begin{proof}
We may assume $g(0)=\id$.
Consider the group $H$ and the splitting $g=h^{o}\gamma$ provided by Lemma~\ref{lem:splitting-equid}.
Since $\gamma$ takes values in a finitely generated subgroup of $\roots\Gamma$ and by Lemma~\ref{lem:periodic}, we can split $\Z$ into arithmetic progressions $Z_{i}$ such that $\gamma(n)\Gamma = c_{i}\Gamma$ for some $c_{i}\in\roots\Gamma$ and all $n\in Z_{i}$.

By Lemma~\ref{lem:splitting-equid} the sequence $h^{o} c_{i}\Gamma c_{i}\inv|_{Z_{i}}$ is totally well-distributed on $H /(c_{i}\Gamma c_{i}\inv)$.
By conjugation and translation invariance this implies that the sequence
\[
g\Gamma|_{Z_{i}}
= h^{o}c_{i}\Gamma|_{Z_{i}}
= c_{i} c_{i}\inv (h^{o}c_{i}\Gamma c_{i}\inv) c_{i}|_{Z_{i}}
\]
is totally well-distributed on $H c_{i}/\Gamma \subset G/\Gamma$.
\end{proof}
One immediate consequence is the pointwise ergodic theorem for polynomials in nilmanifolds.
\begin{corollary}[{\cite[Theorem A]{MR2122919}}]
\label{cor:Leibman-pointwise}
Let $G/\Gamma$ be a nilmanifold, $g$ a polynomial sequence, and $f\in C(G/\Gamma)$. Then the limit $\operatorname{UC-lim}_{n} f(g(n)\Gamma)$ exists.
\end{corollary}
The other consequence is a well-distribution criterion.
It shows that for many purposes it suffices to consider nilmanifolds whose structure groups' connected components of identity are commutative.
\begin{corollary}[{\cite[Theorem C]{MR2122919}}]
\label{thm:leibman-equidistribution-criterion}
Let $G/\Gamma$ be a connected nilmanifold with a $\Gamma$-rational filtration $\Gb$.
Then for a $\Gb$-polynomial sequence $g(n)$ the following statements are equivalent.
\begin{enumerate}
\item\label{equid-crit:equid} $g(n)\Gamma$ is totally well-distributed on $G/\Gamma$,
\item\label{equid-crit:dense-horiz-torus} $g(n)\Gamma [G^{o},G^{o}]$ is dense in $G/\Gamma [G^{o},G^{o}]$.
\end{enumerate}
\end{corollary}
The connectedness requirement cannot be removed as the example $G=\Z/2\Z$, $g=(\dots,1,0,0,0,1,0,0,0,1,\dots)$, shows.
This sequence is polynomial of degree $3$ and its image is all of $G$, but it is not equidistributed.
\begin{proof}
Both statements are invariant under multiplication by constants on the left, so we may assume $g(0)=\id$.

(\ref{equid-crit:equid}) clearly implies (\ref{equid-crit:dense-horiz-torus}).
For the converse consider the subgroup $H\leq G$ provided by Theorem~\ref{thm:Leibman-orbit-closure}.
If $H\leq G^{o}$ is a proper subgroup, then it has lower dimension than $G^{o}$, so that $g(n)\Gamma$ takes values in a finite union of submanifolds of $G/\Gamma$ of strictly lower dimension, contradicting density.
Hence $\Z$ splits into finitely many arithmetic progressions, and the restriction of $g(n)\Gamma$ to each of these progressions is totally well-distributed on $G/\Gamma$.
This implies the claim.
\end{proof}

\subsection{Nilsystems}
\index{nilsystem}
A \emph{($k$-step) nilsystem} is a measure-preserving system of the form $(X,T)$, where $X=G/\Gamma$ is a ($k$-step) nilmanifold and $Tg\Gamma = ag\Gamma$ for some $a\in G$ and all $g\Gamma\in G/\Gamma$.
\index{pro-nilsystem}
A \emph{$k$-step pro-nilsystem} is an inverse limit of $k$-step nilsystems in the category of measure-preserving systems (equivalently, in the category of topological dynamical systems with an invariant Borel probability measure \cite[Theorem A.1]{MR2600993}).
\index{nilfactor}
\index{pro-nilfactor}
A \emph{(pro-)nilfactor} of a measure-preserving dynamical system is a factor that is also a (pro-)nilsystem.

We will now state and prove an important characterization of ergodic nilsystems.
For other proofs see \cite{MR0167569}, \cite{MR0267558}, or \cite[\textsection 2.17--2.20]{MR2122919}.
\begin{lemma}
\label{lem:nilsystem-ergodic}
\index{nilsystem!ergodic}
Let $X=(G/\Gamma,a)$ be a nilsystem.
Then the following statements are equivalent.
\begin{enumerate}
\item\label{lem:nilsystem-ergodic:top-trans} $X$ is topologically transitive as a topological dynamical system.
\item\label{lem:nilsystem-ergodic:erg} $X$ is ergodic with respect to the Haar measure.
\item\label{lem:nilsystem-ergodic:uniq-erg} $X$ is uniquely ergodic.
\end{enumerate}
\end{lemma}
\begin{proof}
(\ref{lem:nilsystem-ergodic:uniq-erg}) clearly implies (\ref{lem:nilsystem-ergodic:erg}) since the Haar measure is invariant.
(\ref{lem:nilsystem-ergodic:erg}) implies (\ref{lem:nilsystem-ergodic:top-trans}) since the Haar measure has full support.

Suppose now that (\ref{lem:nilsystem-ergodic:top-trans}) holds.
It is easy to see that $X$ is distal, cf.\ \cite[Theorem 2.14]{MR2122919}.
Hence any two orbit closures in $X$ either coincide or are disjoint.
By topological transitivity at least one orbit is dense in $X$, so that every orbit is dense in $X$.
Fix an orbit $(a^{n}x)$.
It follows from Theorem~\ref{thm:Leibman-orbit-closure} that $\Z$ splits into a finite union of arithmetic progressions in such a way that the restriction of the orbit to each of these progressions is totally well-distributed on a connected component of $X$.
By topological transitivity we know that $a$ permutes the connected components cyclically.
It follows that $(a^{n}x)$ is well-distributed with respect to the Haar measure on $G/\Gamma$.
\end{proof}

It is a classical fact that the Kronecker factor of an ergodic nilsystem $(G/\Gamma,T)$ is the canonical map $G/\Gamma \to G/\Gamma G_{2}$, where $G_{2}=[G,G]$.
The nilmanifold $G/\Gamma G_{2}$ is a compact homogeneous space of the abelian Lie group $G/G_{2}$, hence a disjoint union of finitely many tori.
The fibers of the projection $G/\Gamma \to G/\Gamma G_{2}$ are isomorphic to the homogeneous space $G_{2}/\Gamma_{2}$, where $\Gamma_{2}=\Gamma\cap G_{2}$.
By a result of Mal'cev $\Gamma_{2}$ is a cocompact subgroup of $G_{2}$ \cite{MR0028842}, so each such fiber is also a nilmanifold.

\section{Background from ergodic theory}
In this section we will state several results about the pointwise ergodic theorem, measure disintegration and Host-Kra-Ziegler factors.
Not all of them are needed in the proof of our Wiener-Wintner theorem, but they will come in handy in the next chapter when we will be dealing with the return times theorem.
\subsection{F\o{}lner sequences}
\begin{definition}
\label{def:folner}
Let $\AG$ be a locally compact second countable group with left Haar measure $|\cdot|$.
A sequence of sets $\Fo_{n} \subset \AG$ is called
\begin{enumerate}
\item\index{F\o{}lner sequence!weak}
a \emph{(weak) F\o{}lner sequence} if for every compact set $K\subset \AG$ one has \[
|\Fo_{n} \Delta K\Fo_{n}| / |\Fo_{n}| \to 0 \text{ as } n\to\infty,
\]
\item\index{F\o{}lner sequence!strong}
a \emph{strong F\o{}lner sequence} if for every compact set $K\subset \AG$ one has \[
|\partial_{K}(\Fo_{n})|/|\Fo_{n}| \to 0 \text{ as } n\to\infty,
\]
where $\partial_{K}(\Fo)= K\inv \Fo \cap K\inv \Fo^{\complement}$ is the \emph{$K$-boundary} of $\Fo$, and
\item\index{F\o{}lner sequence!tempered}
\emph{($C$-)tempered} if there exists a constant $C$ such that
\[
|\cup_{i<j} \Fo_{i}\inv \Fo_{j}| < C |\Fo_{j}|
\quad\text{for every }j.
\]
\end{enumerate}
\end{definition}
Note that any of the above conditions implies that $(\Fo_{n})$ is a Følner net in the sense of Definition~\ref{def:folner-net}.

Every strong F\o{}lner sequence is also a weak F\o{}lner sequence.
In countable groups the converse is also true, but already in $\R$ this is no longer the case: let for example $(\Fo_{n})$ be a sequence of nowhere dense sets $\Fo_{n}\subset [0,n]$ of Lebesgue measure $n-1/n$, say.
This is a weak but not a strong F\o{}lner sequence (in fact, $\partial_{[-1,0]} \Fo_{n}$ is basically $[0,n+1]$).
However, a weak F\o{}lner sequence can be used to construct a strong F\o{}lner sequence.
\begin{lemma}
\label{lem:strong-Folner-exist}
Assume that a locally compact second countable group $\AG$ admits a weak F\o{}lner sequence.
Then $\AG$ also admits a strong F\o{}lner sequence.
\end{lemma}
\begin{proof}
We follow the argument in \cite[Lemma 2.6]{2012arXiv1205.3649P}.
Let $(V_{j})$ be a countable basis for the topology of $\AG$ that consists of relatively compact sets.
Let $K\subset\AG$ be a compact set, then it is covered by a finite union of $V_{j}$'s.
Hence we obtain a countable ascending chain of compact subsets of $\AG$ such that every compact subset is contained in one of the sets in this collection, namely the collection of $K_{N}:=\cup_{j=1}^{N}\ol{V_{j}}$.

Let $N$ be arbitrary and set $K:=K_{N}$, $\epsilon:=\frac1N$.
It suffices to find a compact set $\Fo$ with $|\partial_{K}(\Fo)|/|\Fo|<\epsilon$.
Let $(\Fo_{n})$ be a weak F\o{}lner sequence, then there exists $n$ such that $|K\inv K \Fo_{n} \Delta \Fo_{n}| < \epsilon |\Fo_{n}|$.
Set $\Fo=K\Fo_{n}$, then
\[
\partial_{K}\Fo
= K\inv K\Fo_{n} \cap K\inv (K\Fo_{n})^{\complement}
\subset K\inv K\Fo_{n}\cap \Fo_{n}^{\complement},
\]
and this has measure less than $\epsilon |\Fo_{n}| \leq \epsilon |\Fo|$.
\end{proof}
Since every weak (hence also every strong) F\o{}lner sequence has a tempered subsequence \cite[Proposition 1.4]{MR1865397}, this implies that every lcsc amenable group admits a tempered strong F\o{}lner sequence.

\subsection{Lindenstrauss covering lemma}
Given a collection of intervals, the classical Vitali covering lemma allows one to select a disjoint subcollection that covers a fixed fraction of the union of the full collection.
The appropriate substitute in the setting of tempered F\o{}lner sequences is the Lindenstrauss random covering lemma.
It allows one to select a \emph{random} subcollection that is \emph{expected} to cover a fixed fraction of the union and to be \emph{almost} disjoint.
The almost disjointness means that the expectation of the counting function of the subcollection is uniformly bounded by a constant.
As such, the Vitali lemma is stronger whenever it applies, and the reader who is only interested in the standard F\o{}lner sequence in $\Z$ can skip this subsection.

We use two features of Lindenstrauss' proof of the random covering lemma that we emphasize in its formulation below.
The first feature is that the second moment (and in fact all moments) of the counting function is also uniformly bounded (this follows from the bound for the moments of a Poisson distribution).
The second feature is that the random covering depends measurably on the data.
We choose to include the explicit construction of the covering in the statement of the lemma instead of formalizing this measurability statement.
To free up symbols for subsequent use we replace the auxiliary parameter $\delta$ in Lindenstrauss' statement of the lemma by $C\inv$ and expand the definition of $\gamma$.

For completeness we recall that a \emph{Poisson point process} with intensity $\alpha$ on a measure space $(X,\mu)$ is a counting (i.e.\ atomic, with at most countably many atoms and masses of atoms in $\N$) measure-valued map $\Upsilon:\Omega \to M(X)$ such that for every finite measure set $A\subset X$ the random variable $\omega\mapsto \Upsilon(\omega)(A)$ is Poisson with mean $\alpha\mu(A)$ and for any disjoint sets $A_{i}$ the random variables $\omega\mapsto \Upsilon(\omega)|_{A_{i}}$ are jointly independent (here and later $\Upsilon|_{A}$ is the measure $\Upsilon|_{A}(B)=\Upsilon(A\cap B)$).
It is well-known that on every $\sigma$-finite measure space there exists a Poisson process.
\begin{lemma}[{\cite[Lemma 2.1]{MR1865397}}]
\label{lem:lindenstrauss-covering}
\index{Lindenstrauss covering lemma}
Let $\AG$ be a lcsc group with left Haar measure $|\cdot|$.
Let $(\Fo_{N})_{N=L}^{R}$ be a $C$-tempered sequence.
Let $\Upsilon_{N}:\Omega_{N}\to M(\AG)$ be independent Poisson point processes with intensity $\alpha_{N}=\delta/|\Fo_{N}|$ w.r.t.\ the right Haar measure $\rho$ on $\AG$ and let $\Omega:=\prod_{N}\Omega_{N}$.

Let $A_{N|R+1}\subset \AG$, $N=L,\dots,R$, be sets of finite measure.
Define (dependent!) counting measure-valued random variables $\Sigma_{N} : \Omega \to M(\AG)$ in descending order for $N=R,\dots,L$ by
\begin{enumerate}
\item $\Sigma_{N} := \Upsilon_{N}|_{A_{N|N+1}}$,
\item $A_{i|N} := A_{i|N+1}\setminus \Fo_{i}\inv \Fo_{N}\Sigma_{N} = \{ a\in A_{i|N+1} : \Fo_{i}a\cap \Fo_{N}\Sigma_{N}=\emptyset\}$ for $i<N$.
\end{enumerate}
Then for the counting function
\[
\Lambda = \sum_{N}\Lambda_{N},
\quad
\Lambda_{N}(g)(\omega) = \sum_{a\in\Sigma_{N}(\omega)}1_{\Fo_{N}a}(g)
\]
the following holds.
\begin{enumerate}
\item $\Lambda$ is a measurable, a.s.\ finite function on $\Omega\times \AG$,
\item $\E(\Lambda(g)|\Lambda(g)\geq 1) \leq 1+C\inv$ for every $g\in \AG$,
\item $\E(\Lambda^{2}(g)|\Lambda(g)\geq 1) \leq (1+C\inv)^{2}$ for every $g\in \AG$,
\item $\E(\int\Lambda) \geq (2C)\inv |\cup_{N=L}^{R}A_{N}|$.
\end{enumerate}
\end{lemma}
Recall that the \emph{maximal function} is defined by
\index{maximal function}
\[
Mf(x) := \sup_{N} \Big| \aveFn f(T^{n}x) \Big|
\text{ for } f\in L^{1}(X).
\]
The \emph{Lindenstrauss maximal inequality} \cite[Theorem 3.2]{MR1865397} asserts that for every $f\in L^{1}(X)$ and every $\lambda>0$ we have
\index{Lindenstrauss maximal inequality}
\begin{equation}
\label{eq:maximal-inequality}
\mu\{Mf>\lambda\} \lesssim \lambda\inv \|f\|_{1},
\end{equation}
where the implied constant depends only on the constant in the temperedness condition.
This implies the following pointwise ergodic theorem.
\begin{theorem}[{\cite[Theorem 1.2]{MR1865397}}]
\label{thm:Lindenstrauss-pointwise}
\index{theorem!pointwise ergodic}
Let $\AG$ be a locally compact second countable amenable group with a tempered F\o{}lner sequence $(\Fo_{N})$.
Suppose that $\AG$ measurably acts on a probability space $(X,\mu)$ by measure-preserving transformations.
Then for every $f\in L^{1}(X,\mu)$ there exists a full measure subset $X'\subset X$ such that for every $x\in X'$ the limit
\[
\lim_{N\to\infty} \frac1{|\Fo_{N}|} \int_{g\in \Fo_{N}} f(gx)
\]
exists.
If the action is ergodic, then the limit equals $\int_{X} f \dif\mu$ a.e.
\end{theorem}
The temperedness assumption cannot be dropped even for sequences of intervals with growing length in $\Z$, see \textcite{MR553340} and \textcite{MR1182661}.

\subsection{Fully generic points}
Let $(X,\AG)$ be an ergodic measure-preserving system and $f\in L^{1}(X)$.
Recall that a point $x\in X$ is called \emph{generic} for $f$ if
\index{generic point}
\[
\lim_{n\to\infty} \E_{g\in \Fo_{n}} f(gx) = \int_{X}f.
\]
In the context of countable group actions fully generic points for $f\in L^{\infty}(X)$ are usually defined as points that are generic for every function in the closed $\AG$-invariant algebra spanned by $f$.
For uncountable groups this is not a good definition, since this algebra need not be separable.
The natural substitute for shifts of a function $f\in L^{\infty}(X)$ is provided by convolutions
\[
c*f(x) = \int_{\AG} c(g\inv) f(gx) \dif g,
\quad
c\in L^{1}(\AG).
\]
Since $L^{1}(\AG)$ is separable and convolution is continuous as an operator $L^{1}(\AG)\times L^{\infty}(X) \to L^{\infty}(X)$, the closed convolution-invariant algebra generated by $f$ is separable.

We call a point $x\in X$ \emph{fully generic} for $f$ if it is generic for every function in this algebra.
\index{fully generic point}
In view of the Lindenstrauss pointwise ergodic theorem (Theorem~\ref{thm:Lindenstrauss-pointwise}), if $(\Fo_{n})$ is tempered, then for every $f\in L^{1}(X)$ a.e.\ $x\in X$ is generic.
Consequently, for every $f\in L^{\infty}(X)$ a.e.\ $x\in X$ is fully generic.

\subsection{Ergodic decomposition}
A measure-preserving system $(X,\mu,T)$ is called \emph{regular} if $X$ is a compact metric space, $\mu$ is a Borel probability measure and $T$ is continuous.
Every measure-preserving system is measurably isomorphic to a regular measure-preserving system upon restriction to a separable $T$-invariant sub-$\sigma$-algebra \cite[\textsection 5.2]{MR603625}.

\index{ergodic decomposition}
The \emph{ergodic decomposition} of the measure on a regular measure-preserving system $(X,\mu,T)$ is a measurable map $x\mapsto\mu_{x}$ from $X$ to the space of $T$-invariant ergodic Borel probability measures on $X$, unique up to equality $\mu$-a.e., such that $\mu$-a.e.\ $x\in X$ is generic for every $f\in C(X)$ w.r.t.\ $\mu_{x}$ and $\mu=\int\mu_{x}\dif\mu(x)$ \cite[\textsection 5.4]{MR603625}.
Moreover, for every $f\in L^{1}(\mu)$, for $\mu$-a.e.\ $x\in X$  we have that $f\in L^{1}(\mu_{x})$ and $x$ is generic for $f$ w.r.t.\ $\mu_{x}$.

In connection with the multiple term return times theorem we find it illuminating to think of the ergodic decomposition in a particular way (that will be generalized in \textsection\ref{sec:RTT-cube}).
Let $(X,\mu,T)$ be a regular measure-preserving system.
By the Lindenstrauss pointwise ergodic theorem (Theorem~\ref{thm:Lindenstrauss-pointwise}) a.e.\ $x\in X$ is generic for some $T$-invariant Borel probability measure $\m_{x}$ on $X$, i.e.\ $\aveFN f(T^{n}x) \to \int f \dif \m_{x}$ for every $f\in C(X)$.
It follows easily that the function $x\mapsto\m_{x}$ is measurable and
\begin{equation}
\label{eq:m-disint}
\mu = \int\m_{x} \dif\mu(x)
\end{equation}
In particular, for $\mu$-a.e.\ $x$ the measure $\m_{y}$ is defined for $\m_{x}$-a.e.\ $y$.
To see that $\m_{x}$ is ergodic for $\mu$-a.e.\ $x$ it suffices to verify that
\begin{equation}
\label{eq:my-mx}
\int \int \Big| \int f \dif\m_{y} - \int f \dif\m_{x} \Big|^{2} \dif\m_{x}(y) \dif\mu(x)
= 0 \text{ for every } f\in C(X),
\end{equation}
since this says precisely that the ergodic averages of $f$ converge pointwise $\m_{x}$-a.e.\ to an $\m_{x}$-essentially constant function for $\mu$-a.e.\ $x$, and the latter full measure set can be chosen independently from $f$ since $C(X)$ is separable.
By definition of $\m_{x},\m_{y}$, the dominated convergence theorem and \eqref{eq:m-disint} we can rewrite the integral in \eqref{eq:my-mx} as
\begin{multline*}
2 \lim_{N} \int (\aveFN T^{n}f)^{2}(x) \dif\mu(x)\\
- 2 \lim_{N} \int (\aveFN T^{n}f)(x) \int (\aveFN T^{n}f)(y) \dif\m_{x}(y) \dif\mu(x)\\
=
2 \lim_{N} \int (\aveFN T^{n}f)^{2}(x) \dif\mu(x)\\
- 2 \lim_{N} \lim_{M} \int (\aveFN T^{n}f)(x) (\aveFMn T^{m}f)(x) \dif\mu(x),
\end{multline*}
and this vanishes by the Lindenstrauss pointwise ergodic theorem (Theorem~\ref{thm:Lindenstrauss-pointwise}) and the dominated convergence theorem.
\subsection{Host-Kra cube spaces}
We recall the basic definitions and main results surrounding the uniformity seminorms.
Let $(X,\mu,T)$ be a regular, not necesserily ergodic, measure-preserving system.
\index{cube measure}
The \emph{cube measures} $\mu^{[l]}$ on $X^{[l]}:=X^{2^{l}}$ are defined inductively starting with $\mu^{[0]}:=\mu$.
In the inductive step, given $\mu^{[l]}$, fix an ergodic decomposition
\[
\mu^{[l]}=\int_{X^{[l]}} \m_{x} \dif\mu^{[l]}(x)
\]
as in \eqref{eq:m-disint}.
The space on which $\m_{x}$ is defined can be inferred from the subscript $x$.
Define
\begin{equation}
\label{eq:mul+1}
\mu^{[l+1]}:=\int_{X^{[l]}} \delta_{x}\otimes\m_{x} \dif\mu^{[l]}(x).
\end{equation}
Using \eqref{eq:m-disint} and \eqref{eq:my-mx}, we can write the above integral as
\begin{multline}
\label{eq:mul+1-conventional}
\mu^{[l+1]}
=
\int\int \delta_{y}\otimes\m_{y} \dif\m_{x}(y)\dif\mu^{[l]}(x)\\
=
\int\int \delta_{y}\otimes\m_{x} \dif\m_{x}(y)\dif\mu^{[l]}(x)
=
\int \m_{x}\otimes\m_{x} \dif\mu^{[l]}(x),
\end{multline}
which is the usual definition of the cube measures.
\begin{definition}[Gowers-Host-Kra seminorms {\cite[\textsection 3.5]{MR2150389}}]
\index{Gowers-Host-Kra seminorms}
\index{uniformity seminorms}
The \emph{Gowers-Host-Kra seminorms}, or \emph{uniformity seminorms}, are defined by
\begin{equation}
\label{eq:uniformity-seminorm-integral}
\| f \|_{U^{l+1}(X,\mu,T)}^{2^{l+1}} := \int \otimes_{\epsilon\in \{0,1\}^{l+1}} f \dif\mu^{[l+1]} = \int \E\big( \otimes_{\epsilon\in \{0,1\}^{l}} f | \mathcal{I}^{[l]} \big)^{2} \dif\mu^{[l]},
\end{equation}
where $\mathcal{I}^{[l]}$ is the $T^{[l]}$-invariant sub-$\sigma$-algebra on $X^{[l]}$.
\end{definition}
We will write $U^{l}$ or $U^{l}(X)$ instead of $U^{l}(X,\mu,T)$ if no confusion is possible.
In a special case these seminorms have been introduced by Bergelson \cite{MR1774423}.

If $\mu = \int \mu_{x} \dif\mu(x)$ is the ergodic decomposition, then
\[
\|f\|_{U^{l}(X,\mu)}^{2^{l}}=\int \|f\|_{U^{l}(X,\mu_{x})}^{2^{l}} \dif\mu(x)
\text{ for all } f\in L^{\infty}(\mu).
\]
It follows from the mean ergodic theorem that the uniformity seminorms can be recursively computed by the following folmulas.
\[
\|f\|_{U^0(X,\mu)}=\int_X f\dif\mu,\quad
\|f\|_{U^{l+1}(X,\mu)}^{2^{l+1}}=\lim_{N\to\infty}\frac{1}{N}\sum_{n=1}^N\|T^n f \bar{f}\|_{U^{l}(X,\mu)}^{2^{l}}.
\]
For $f_{\epsilon} \in L^{\infty}(X)$, $\epsilon\in\{0,1\}^{l}$, we will abbreviate $f^{[l]}:=\otimes_{\epsilon\in\{0,1\}^{l}}f_{\epsilon}$.
It follows by induction on $l\in\N$ that
\begin{equation}
\label{eq:est-U-by-L}
\|\cdot\|_{U^{l+1}(X)}\leq \|\cdot\|_{L^{2^{l}}(X)},
\end{equation}
see \cite{MR2944094} for subtler analysis.
The uniformity seminorms satisfy the \emph{Cauchy-Schwarz-Gowers inequality} \cite[Lemma 3.9.(1)]{MR2150389}
\index{Cauchy-Schwarz-Gowers inequality}
\begin{equation}
\label{eq:CSG}
\Big| \int f^{[l]} \dif\mu^{[l]} \Big| \leq \prod_{\epsilon\in\{0,1\}^{l}} \|f_{\epsilon}\|_{U^{l}}.
\end{equation}

For every $l$ the uniformity seminorm $U^{l+1}$ determines a factor $\HKZ_l(X)$ of $(X,\mu,T)$, called the Host-Kra factor of order $l$, that is characterized by the relation
\[
\|f\|_{U^{l+1}(X)}=0 \iff \E(f|\HKZ_l(X))=0
\]
that holds for all $f\in L^\infty(X)$.
The structure of the factors $\HKZ_{l}$ is captured by the following result of Host and Kra.
\begin{theorem}[\cite{MR2150389}]
\label{thm:HK-structure}
\index{theorem!Host-Kra structure}
Suppose that $(X,\mu,T)$ is ergodic.
Then $\HKZ_{l}(X)$ is measurably isomorphic to a pro-nilsystem of step $l$.
\end{theorem}
We should like to mention that these factors have been also independently constructed by Ziegler \cite{MR2257397}.

In the non-ergodic case one could use this result on every ergodic component, but it is not clear in which sense the resulting pro-nilsystems vary measurably with the ergodic component (some work on this problem has been done by Austin \cite{MR2725895}).
At any rate, the following decomposition result of Chu, Frantzikinakis, and Host suffices for our purposes.
\begin{theorem}[{\cite[Proposition 3.1]{MR2795725}}]
\label{thm:CFH-decomposition}
Suppose that $f\in L^{\infty}(\HKZ_{l})$ for some $l$.
Then for every $\epsilon>0$ there exists a function $f_{s}\in L^{\infty}(\HKZ_{l}(X),\mu)$ such that $\|f_{s}\|_{\infty} \leq \|f\|_{\infty}$ and the following statements hold.
\begin{enumerate}
\item $\|f-f_{s}\|_{1} < \epsilon$ and
\item for every $x\in X$ the sequence $(f_{s}(T^{n}x))_{n}$ is an $l$-step nilsequence.
\end{enumerate}
\end{theorem}

Since the uniformity seminorms are bounded by the supremum norm and invariant under $T$ and complex conjugation, they can also be calculated using smoothed averages
\begin{equation}
\label{eq:uniformity-seminorms-smoothed}
\|f\|_{U^{l+1}(X)}^{2^{l+1}}=\lim_{K\to\infty}\frac{1}{K^{2}}\sum_{k=-K}^K (K-|k|)\|T^k f \bar{f}\|_{U^{l}(X)}^{2^{l}}.
\end{equation}
This will allow us to use the following quantitative version of the classical van der Corput estimate (the proof is included for completeness).
Here $o_{K}(1)$ stands for a quantity that goes to zero for each fixed $K$ as $N\to\infty$.
\begin{lemma}[Van der Corput]\label{VdC}
\index{van der Corput lemma}
Let $(\Fo_{N})_{N}$ be a F\o{}lner sequence in $\Z$ and $(u_n)_{n\in\Z}$ be a sequence in a Hilbert space with norm bounded by $C$.
Then for every $K>0$ we have
\[
\Big\| \aveFn u_{n} \Big\|^{2} \leq \Big| \frac{2}{K^{2}} \sum_{k=-K}^{K}(K-|k|) \aveFn \langle u_{n},u_{n+k} \rangle \Big| + C^{2} o_{K}(1).
\]
\end{lemma}
\begin{proof}
Let $K>0$ be given.
By the definition of a F\o{}lner sequence we have
\[
\aveFn u_{n}
=
\aveFn \frac{1}{K} \sum_{k=1}^K u_{k+n} + C o_{K}(1).
\]
By Hölder's inequality
\begin{multline*}
\Big\| \aveFn \frac{1}{K} \sum_{k=1}^K u_{k+n} \Big\|^{2}
\leq
\aveFn \Big\|\frac{1}{K} \sum_{k=1}^K u_{k+n}\Big\|^{2}\\
=
\frac{1}{K^{2}} \sum_{k=-K}^{K}(K-|k|) \aveFn \langle u_{n},u_{n+k} \rangle + C^{2} o_{K}(1),
\end{multline*}
and the claim follows using the
estimate $(a+b)^{2}\leq 2a^{2} + 2b^{2}$.
\end{proof}

\section{Wiener-Wintner theorem for nilsequences}
\index{Wiener-Wintner theorem}
The classical Wiener-Wintner theorem \cite{MR0004098} says that for the standard F\o{}lner sequence $\Fo_{N} = [1,N]$ on the amenable group $\Z$, every invertible ergodic measure-preserving transformation $T:X\to X$, and every $f\in L^1(X,\mu)$ there exists a subset $X'\subset X$ with full measure such that the weighted averages
\begin{equation}\label{ave-ww}
\aveFN f(T^nx) \lambda^n
\end{equation}
converge as $N\to\infty$ for every $x\in X'$ and every $\lambda$ in the unit circle $\T$.

A result of Lesigne \cites{MR1074316,MR1257033} shows that the weights $(\lambda^n)$ above can be replaced by polynomial sequences of the form $(\lambda_1^{p_1(n)}\cdots \lambda_k^{p_k(n)})$, $\lambda_j\in \T$, $p_j\in \Z[X]$ (or, equivalently, $(e^{2\pi i p(n)})$, $p\in\R[X]$).
More recently, Host and Kra \cite[Theorem 2.22]{MR2544760} showed that this can be enlarged to the class of nilsequences.

In a different direction, Bourgain's uniform Wiener-Wintner theorem \cite{MR1037434} asserts convergence of the averages \eqref{ave-ww} to zero for $f$ orthogonal to the Kronecker factor \emph{uniformly in $\lambda$}, cf.\ \textcite{MR1995517}.
A joint extension of this result and Lesigne's polynomial Wiener-Wintner theorem has been obtained by Frantzikinakis \cite{MR2246591}.
In the same spirit, we prove a uniform version of the Wiener-Wintner theorem for nilsequences.
Our result applies to arbitrary tempered F\o{}lner sequences $(\Fo_{N})$ in $\Z$.
\begin{theorem}[Uniform Wiener-Wintner for nilsequences]
\label{thm:uniform-convergence-to-zero}
\index{Wiener-Wintner theorem!for nilsequences}
Assume that $(X,\mu,T)$ is ergodic and let $f\in L^{1}(X)$ be such that $\E(f|\HKZ_{l}(X))=0$.
Let further $G/\Gamma$ be a nilmanifold with a $\Gamma$-rational filtration $\Gb$ on $G$ of length $l$.
Then for a.e. $x\in X$ we have
\begin{equation}
\label{eq:ave-uniform}
\lim_{N\to\infty} \sup_{g\in\poly, F\in W^{k,2^{l}}(G/\Gamma)}
\|F\|_{W^{k,2^{l}}(G/\Gamma)}\inv
\Big| \aveFn f(T^{n}x)F(g(n)\Gamma) \Big| = 0,
\end{equation}
where $k = \sum_{r=1}^{l}(d_{r}-d_{r+1})\binom{l}{r-1}$ with $d_{i}=\dim G_{i}$.

If in addition $(X,T)$ is a uniquely ergodic topological dynamical system and $f\in C(X)\cap \HKZ_l(X)^\bot$, then we have
\begin{equation}
\label{eq:ave-uniform-in-X}
\lim_{N\to\infty} \sup_{g\in\poly, F\in W^{k,2^{l}}(G/\Gamma), x\in X}
\|F\|_{W^{k,2^{l}}(G/\Gamma)}\inv
\Big| \aveFn f(T^{n}x) F(g(n)\Gamma) \Big| = 0.
\end{equation}
\end{theorem}
In view of a counterexample in Section~\ref{sec:counterexample} the Sobolev norm cannot be replaced by the $L^{\infty}$ norm.
On the other hand, we have not investigated whether the above order $k$ is optimal and believe that it is not.

The conclusion \eqref{eq:ave-uniform} differs from the uniform polynomial Wiener-Wintner theorem of Frantzikinakis \cite{MR2246591} in several aspects.
First, our class of weights is considerably more general, comprising all nilsequences rather than polynomial phases (a polynomial phase $f(p(n)\Z)$, $f\in C(\R/\Z)$, $p\in\R[X]$ is also a nilsequence of step $\deg p$ with the filtration $\R = \dots = \R \geq \{0\}$ of length $\deg p$ and cocompact lattice $\Z$).
Also, our result does not require total ergodicity, an assumption that cannot be omitted in the result of Frantzikinakis.
The price for these improvements is that we have to assume the function to be orthogonal to the Host-Kra factor and not only to the Abramov factor of order $l$ (i.e.\ the factor generated by the generalized eigenfunctions of order $\leq l$).

The conclusion \eqref{eq:ave-uniform-in-X} generalizes a result of Assani \cite[Theorem 2.10]{MR1995517}, which corresponds essentially to the case $l=1$.
Note that without the orthogonality assumption on the function, everywhere convergence can fail even for averages \eqref{ave-ww} for some $\lambda\in \T$.
For more information on this phenomenon we refer to \textcite{MR1271545}, \textcite{MR1995517}, and \textcite{MR2480747}.

\subsection{The uniformity seminorm estimate}\label{sec:estimate}
The general strategy of estimation of averages in \eqref{eq:ave-uniform} is to induct on the filtration length $l$.
In the induction step we decompose $F$ into a vertical Fourier series and use the quantitative van der Corput estimate.
The resulting terms involve nilsequences of lower step that fall under the induction hypothesis.

For inductive purposes it will be convenient to work with the following version of Theorem~\ref{thm:uniform-convergence-to-zero}.
\begin{theorem}[Uniformity seminorms control averages uniformly]\label{thm:WW-unif}
Assume that $(X,\mu,T)$ is ergodic.
Then for every $f\in L^\infty(X)$ and every point $x$ that is fully generic for $f$ with respect to $(\Fo_N)$ the following holds.
For every $l\in\N$ and $\epsilon>0$ there exists $N_{0}$ such that for every nilmanifold $G/\Gamma$ with a $\Gamma$-rational filtration $\Gb$ on $G$ of length $l$, every smooth function $F$ on $G/\Gamma$, and every $g\in\poly$ we have
\begin{equation}\label{eq:control-by-uniformity-norms}
\forall N\geq N_{0}
\quad
\Big| \aveFn f(T^nx) F(g(n)\Gamma) \Big|
\lesssim \|F\|_{W^{k,2^l}(G/\Gamma)} (\|f\|_{U^{l+1}(X)} + \epsilon),
\end{equation}
where $k=\sum_{r=1}^{l}(d_{r}-d_{r+1})\binom{l}{r-1}$ and the implied constant depends only on the nilmanifold $G/\Gamma$, filtration $\Gb$ and the Mal'cev basis that is implicit in the definition of $\Gamma$-rationality.

If in addition $(X,T)$ is uniquely ergodic and $f\in C(X)$, then the conclusion holds for every $x\in X$, and $N_{0}$ can be chosen independently of $x$.
\end{theorem}
Note that the full measure set in this theorem is explicitely identified as the set of fully generic points for $f$.

Example~\ref{ex:assani} below shows that there is in general no constant $C$ such that the estimate
\begin{equation}\label{eq:assani}
\limsup_{N\to\infty}\Big|\aveN f(T^nx)F(S^ny)\Big|\leq C \|F\|_\infty \|f\|_{U^2(X)}
\end{equation}
holds for every $1$-step basic nilsequence $F(S^ny)$, even without uniformity.
Thus one cannot expect to replace the Sobolev norm by $\|F\|_\infty$ in Theorem~\ref{thm:WW-unif}.

\begin{remark}\label{remark-hk-ww}
Quantifying the proof of Host and Kra \cite[Proposition 5.6]{MR2150389} using standard Fourier analysis on $\T^{d(2^l-1)}$, one obtains the non-uniform upper bound
\[
\limsup_{N} \Big| \aveFn f(T^{n}x) F(g(n)\Gamma) \Big| \lesssim \|F\|_{W^{d(2^l-1),2}(G/\Gamma)} \|f\|_{U^{l+1}(X)}
\]
for ``linear'' sequences $g(n)=h^{n}h'$, where the implied constant depends on geometric data like the choice of a decomposition of identity on the pointed cube space $(G/\Gamma)^{[k]}_*=(G/\Gamma)^{2^l-1}$.
Note also that Host and Kra worked with intervals with growing length instead of tempered F\o{}lner sequences in $\Z$.
\end{remark}

\begin{proof}[Proof of Theorem~\ref{thm:WW-unif}]
We argue by induction on $l$.
In the case $l=0$ the group $G$ is trivial, so $\|F\|_{\infty} = \|F\|_{W^{0,1}(G/\Gamma)}$ and the claim follows by the definition of generic points.
We now assume that the claim holds for $l-1$ and show that it holds for $l$.
Write $a_n:=F(g(n)\Gamma)$.

Assume first that $F$ is a vertical character and recall the notation from Section~\ref{sec:vertical}.
Let $\delta>0$ be chosen later.
For the dimensions $(\tilde d_{i})$ of the groups in the filtration $\tilde\Gb$
we have the relations
$\tilde d_{i}-\tilde d_{i+1}=(d_{i}-d_{i+1})+(d_{i+1}-d_{i+2})$, $i=1,\dots,l-1$.
By the induction hypothesis applied to $\tilde{G}/\tilde{\Gamma}$ with the induced $\tilde\Gamma$-rational filtration and Lemma~\ref{lem:sobolev-norm-of-tensor-product} we have
\begin{align*}
\Big| \aveFn (T^kf \bar{f})(T^nx) a_{n+k}{\ol{a_n}}\Big|
&\lesssim \|\tilde{F}_{k}\|_{W^{\tilde k,2^{l-1}}} (\|T^kf \bar{f}\|_{U^{l}(X)} + \delta)\\
&\lesssim \|F\|_{W^{\tilde k,2^{l}}}^2 (\|T^kf \bar{f}\|_{U^{l}(X)} + \delta)
\end{align*}
with $\tilde k = \sum_{r=1}^{l-1}(\tilde d_{r}-\tilde d_{r+1}) \binom{l-1}{r-1} = \sum_{r=1}^{l}(d_{r}-d_{r+1}) \binom{l}{r-1} - d_{l}$
for any integer $k$ provided that $N$ is large enough depending on $l$, $k$, $\delta$ and $x$.
Let $K$ be chosen later.
The van der Corput Lemma~\ref{VdC} implies
\begin{align*}
\Big| \aveFn f(T^nx) a_n\Big|^2
&\leq
\frac2{K^{2}} \sum_{k=-K}^{K}(K-|k|)
\Big| \aveFn (T^kf \bar{f})(T^nx) a_{n+k}{\ol{a_n}}\Big|\\
&\qquad +\|F\|_{\infty}^{2}\|f\|_{\infty}^{2}o_{K}(1)\\
&\lesssim
\frac1{K^{2}} \sum_{k=-K}^{K}(K-|k|) \|F\|_{W^{\tilde k,2^{l}}}^2 (\|T^kf \bar{f}\|_{U^{l}(X)} + \delta)\\
&\qquad +\|F\|_{\infty}^{2} o_{K}(1)
\end{align*}
provided that $N$ is large enough depending on $l$, $K$, $\delta$ and $x$.
By Lemma~\ref{lem:sobolev-embedding} this is dominated by
\[
\|F\|_{W^{\tilde k,2^{l}}}^2 \left(\frac1{K^{2}} \sum_{k=-K}^{K}(K-|k|) \|T^kf \bar{f}\|_{U^{l}(X)} + \delta + o_{K}(1)\right).
\]
By the Cauchy-Schwarz inequality this is dominated by
\[
\|F\|_{W^{\tilde k,2^{l}}}^2 \Big( \Big(\frac1{K^{2}} \sum_{k=-K}^{K}(K-|k|) \|T^kf \bar{f}\|_{U^{l}(X)}^{2^{l}}\Big)^{1/2^{l}} + \delta +o_{K}(1) \Big)=:I.
\]
By \eqref{eq:uniformity-seminorms-smoothed} for sufficiently large $K = K(f,\delta)$ the above average over $k$ approximates $\|f\|_{U^{l+1}(X)}^{2}$ to within $\delta$, so we have
\begin{align*}
I &\lesssim
\|F\|_{W^{\tilde k,2^{l}}}^2 (\|f\|_{U^{l+1}(X)}^{2} + 2\delta +o_{K}(1)).
\end{align*}
Taking $\delta=\delta(\epsilon)$ sufficiently small and $N\geq N_{0}(l,f,\epsilon,x)$ sufficiently large we obtain
\[
\Big| \aveFn f(T^nx) a_n\Big|
\lesssim
\|F\|_{W^{\tilde k,2^{l}}} (\|f\|_{U^{l+1}(X)} + \epsilon).
\]
Note that $N_{0}$ does not depend on $F$.

Let now $(a_n)=(F(g(n)\Gamma))$ be an arbitrary $l$-step basic nilsequence on $G/\Gamma$.
Let $F=\sum_\chi F_\chi$ be the vertical Fourier series.
By the above investigation of the vertical character case, since the vertical Fourier series of $F$ converges absolutely and by Lemma~\ref{lem:estimate-vertical-fourier-series} we get
\begin{align*}
\Big| \aveFn f(T^n x) F(g(n)\Gamma) \Big|
&\lesssim
\sum_\chi \|F_\chi\|_{W^{\tilde k,2^l}} (\|f\|_{U^{l+1}(X)} + \epsilon)\\
&\lesssim
\|F\|_{W^{\tilde k+d_{l},2^{l}}} (\|f\|_{U^{l+1}(X)} + \epsilon)
\end{align*}
for $N\geq N_{0}$ as required.

Under the additional assumptions that $(X,T)$ is uniquely ergodic and $f\in C(X)$ we obtain the additional conclusion that the estimate is uniform in $x\in X$ for $l=0$ from uniform convergence of ergodic averages $\aveFn T^{n}f$, see e.g.\ \cite[Theorem 6.19]{MR648108}.
For general $l$ it suffices to observe that in the above proof the dependence of $N_{0}$ on $x$ comes in only through the inductive hypothesis.
Also, there is no need for temperedness of $(\Fo_{N})$ in this case.
\end{proof}

\begin{proof}[Proof of Theorem~\ref{thm:uniform-convergence-to-zero}]
Let $f\in L^{1}(X)$ with $\E(f|\HKZ_{l}(X))=0$ be given.
By truncation we can approximate it by a sequence of bounded functions $(f_{j})\subset L^{\infty}(X)$ such that $f_{j}\to f$ in $L^{1}$.
Replacing each $f_{j}$ by $f_{j}-\E(f_{j}|\HKZ_{l}(X))$ we may assume that $\E(f_{j}|\HKZ_{l}(X))=0$ for every $j$.

By Theorem~\ref{thm:WW-unif} we have
\[
\lim_{N\to\infty} \sup_{g\in\poly, F\in W^{k,2^{l}}(G/\Gamma)}
\|F\|_{W^{k,2^{l}}(G/\Gamma)}\inv
\Big| \aveFn f_{j}(T^{n}x)F(g(n)\Gamma) \Big| = 0
\]
for $x$ in a set of full measure and every $j$.
By the Sobolev embedding theorem \cite[Theorem 4.12 Part I Case A]{MR2424078} we have $\|F\|_{\infty} \lesssim \|F\|_{W^{k,2^{l}}(G/\Gamma)}$ for $F\in W^{k,2^{l}}(G/\Gamma)$.
This shows that
\begin{multline*}
\sup_{g\in\poly, F\in W^{k,2^{l}}(G/\Gamma)}
\|F\|_{W^{k,2^{l}}(G/\Gamma)}\inv
\Big| \aveFn f(T^{n}x)F(g(n)\Gamma) \Big|\\
\lesssim
\aveFn |f-f_{j}|(T^{n}x)\\
+ \sup_{g\in\poly, F\in W^{k,2^{l}}(G/\Gamma)} \|F\|_{W^{k,2^{l}}(G/\Gamma)}\inv \Big| \aveFn f_{j}(T^{n}x)F(g(n)\Gamma) \Big|.
\end{multline*}
Fixing a $j$, restricting to the set of points that are generic for $|f-f_{j}|$ with respect to~$\{\Fo_N\}$ and letting $N\to\infty$ we can estimate the limit by $\|f-f_{j}\|_{1}$ pointwise on a set of full measure.
Hence the limit vanishes a.e.

Under the additional assumptions that $(X,T)$ is uniquely ergodic and $f$ is continuous the uniform convergence \eqref{eq:ave-uniform-in-X} follows directly from Theorem~\ref{thm:WW-unif}.
\end{proof}

\subsection{A counterexample}\label{sec:counterexample}
The following example shows that there is no constant $C$ such that the estimate (\ref{eq:assani})
holds for every $1$-step basic nilsequence $F(S^ny)$. Thus one cannot replace the Sobolev norm by $\|F\|_\infty$ in Theorem \ref{thm:WW-unif} even without uniformity in $F$ and $g$.
\begin{example}[I.~Assani]\label{ex:assani}
We begin as in \textcite{MR2901351} and consider an irrational rotation system $(\T,\mu,T)$ on the unit circle, $f\in C(\T)$, $x\in \T$ and define $S:=T$, $y:=x$ and $F:=\bar{f}$.
We have
\[
\limsup_{N\to\infty}\Big|\aveN f(T^nx)\bar{f}(T^nx) \Big|
= \sum_{k=-\infty}^\infty |\hat{f}(k)|^2=\|f\|_2^2.
\]
By $\|f\|_{U^2(\T)}^4=\sum_{k=-\infty}^\infty |\hat{f}(k)|^4$, the inequality \eqref{eq:assani} takes the form
\begin{equation}\label{eq:assani-2}
\|f\|_2^2
\leq C\|f\|_\infty \Big(\sum_{k=-\infty}^\infty |\hat{f}(k)|^4 \Big)^{1/4}.
\end{equation}
Let now $\{a_n\}_{n=1}^\infty\subset \R$ and consider random polynomials
\[
P_N(t,\omega) :=\sum_{n=1}^N r_n(\omega) a_n \cos(nt),
\]
where $r_n$ are the Rademacher functions taking the values $1$ and $-1$ with equal probability. By \textcite[pp.\ 67--69]{MR833073}, there is an absolute constant $D$ such that for every $N$
\[
\mathbb{P}\Big\{\omega:\, \|P_N(\cdot,\omega)\|_\infty \geq D\Big(\sum_{n=1}^N a_n^2 \log N \Big)^{1/2}\Big\}\leq \frac{1}{N^2}.
\]
Therefore for every $N\in\N$ there is $\omega$ (or a choice of signs $+$ or $-$) so that
\[
\|P_N(\cdot,\omega)\|_\infty\leq D\Big(\sum_{n=1}^N a_n^2 \log N \Big)^{1/2}.
\]

Assume now that inequality \eqref{eq:assani-2} holds for some constant $C$ and every $f\in C(\T)$. Then by the above for $f=P_N(\cdot,\omega)$ we have
\[
\sum_{n=1}^N a_n^2
\leq CD(\log N)^{1/2}
\Big(\sum_{n=1}^N a_n^2\Big)^{1/2} \Big(\sum_{n=1}^N a_n^4\Big)^{1/4}
\]
and hence
\[
\sum_{n=1}^N a_n^2 \leq (CD)^2 \|(a_n)\|_{l^4} \log N.
\]
Taking $a_n=\sqrt{\log n/n}$
implies $\sum_{n=1}^N \log n/n\leq \tilde{C} \log N$ for some $\tilde{C}$ and all $N$, a contradiction.
\end{example}
We also refer to \textcite{MR2753294} and \textcite{MR2901351} for related issues.

\subsection{Wiener-Wintner theorem for generalized nilsequences}\label{sec:WW-gen-nilseq}
Let $\Gb$ be a $\Gamma$-rational filtration on $G$ and $g\in\poly$ be a polynomial sequence.
By Leibman' orbit closure theorem (Theorem~\ref{thm:Leibman-orbit-closure}), the sequence $g(n)\Gamma$ is contained and equidistributed in a finite union $\tilde Y$ of sub-nilmanifolds of $G/\Gamma$.
\index{nilsequence!generalized}
For a Riemann integrable function $F:\tilde Y \to \C$ we call the bounded sequence $(F(g(n)\Gamma))_{n}$ a \emph{basic generalized $l$-step nilsequence} (one obtains the same notion upon replacing the polynomial $g(n)$ by a ``linear'' polynomial $(g^{n})_{n}$).
A \emph{generalized $l$-step nilsequence} is a uniform limit of basic generalized $l$-step nilsequences.

A concrete example of a generalized nilsequence is $(e^{i[n\alpha]n\beta})$ for $\alpha, \beta\in \R$ or, more generally, bounded sequences of the form $(p(n))$ and $(e^{ip(n)})$ for a generalized polynomial $p$, i.e., a function obtained from conventional polynomials using addition, multiplication, and taking the integer part, see \textcite{MR2318563}.

We also obtain an extension of the Wiener-Wintner theorem for nilsequences due to Host and Kra \cite[Corollary 2.23]{MR2544760} to non-ergodic systems.
\begin{theorem}[Wiener-Wintner for generalized nilsequences]
\label{thm:WW-gen-nilseq}
\index{Wiener-Wintner theorem!for generalized nilsequences}
For every $f\in L^1(X,\mu)$ there exists a set $X'\subset X$ of full measure such that for every $x\in X'$ the averages
\begin{equation}\label{ave}
\aveFn a_n f(T^n x)
\end{equation}
converge for every generalized nilsequence $(a_n)$.

If in addition $(X,T)$ is a uniquely ergodic topological dynamical system, $f\in C(X)$ and the projection $\pi:X\to\HKZ_l(X)$ is continuous for some $l$ then the averages \eqref{ave} converge for \emph{every} $x\in X$ and every $l$-step generalized nilsequence $(a_n)$.
\end{theorem}
See \textcite[remarks following Theorem 3.5]{2012arXiv1203.3778H} for examples of systems for which the additional hypothesis is satisfied.

A consequence of this result concerning norm convergence of weighted polynomial multiple ergodic averages due to Chu \cite{MR2465660}, cf.\  \textcite{MR2544760} for the linear case, is discussed in Section~\ref{sec:multiple}.

In view of Theorem \ref{thm:WW-unif} the Wiener-Wintner theorem for generalized nilsequences (Theorem~\ref{thm:WW-gen-nilseq}) follows by a limiting argument from the decomposition theorem theorem for functions on non-ergodic measure preserving systems.
\begin{proof}[Proof of Theorem~\ref{thm:WW-gen-nilseq}]
Restricting to the separable $T$-invariant $\sigma$-algebra generated by $f$ 
we may assume that $(X,\mu,T)$ is regular.
Let $\mu = \int \mu_{x} \dif\mu(x)$ be the ergodic decomposition.

Consider first a function $0\leq f\leq 1$ and let $\tilde f:=\E(f|\HKZ_{l}(X))$.
By Theorem~\ref{thm:CFH-decomposition} we obtain a sequence of functions $(f_{j}) \subset L^{\infty}(X)$ such that the following holds.
\begin{enumerate}
\item We have $\|f_{j}\|_{L^{\infty}(X,\mu)} \leq 1$ and $\|\tilde f-f_{j}\|_{L^{1}(X,\mu)} \to 0$ as $j\to\infty$.
\item For every $j$ and $\mu$-a.e.\ $x\in X$ the sequence $(f_{j}(T^{n}x))_{n}$ is an $l$-step nilsequence.
\end{enumerate}
Using the first condition we can pass to a subsequence such that $\|\tilde f-f_{j}\|_{L^{2^{l-1}}(X,\mu_{x})} \to 0$ for a.e.\ $x\in X$.
Thus we obtain a full measure subset $X'\subset X$ such that the following holds for every $x\in X'$:
\begin{enumerate}
\item for every $j$ the sequence $(f_{j}(T^{n}x))_{n}$ is an $l$-step nilsequence,
\item for every $j$ the point $x$ is fully generic for $f-f_{j}$ with respect to an ergodic measure $\mu_{x}$ and
\item $\|f-f_{j}\|_{U^{l}(X,\mu_{x})} \to 0$ as $j\to\infty$ (this follows from the basic inequality \eqref{eq:est-U-by-L}).
\end{enumerate}
Let $x\in X'$ and $(a_{n})$ be a basic $l$-step nilsequence of the form $a_{n}=F(g(n)\Gamma)$ with smooth $F$.
Since the product of two nilsequences is again a nilsequence, by Corollary~\ref{cor:Leibman-pointwise} the limit
\[
\lim_{N\to\infty} \aveFn f_j(T^n x) F(g(n)\Gamma)
\]
exists for every $j\in\N$.
By Theorem~\ref{thm:WW-unif} we have
\[
\limsup_{N\to\infty} \Big| \aveFn (f-f_{j})(T^{n}x)F(g(n)\Gamma) \Big| \lesssim \|f-f_{j}\|_{U^{l}(X,\mu_{x})}
\]
for every $j$, where the constant does not depend on $j$, and this implies the existence of the limit \eqref{ave}.

Let now $x\in X'$ and $(a_{n})$ be a basic generalized nilsequence of the form $a_n=F(g(n)\Gamma)$ with a real valued Riemann integrable function $F$.
Let $\veps>0$.
Since $F$ is Riemann integrable on $\tilde Y=\ol{\{g(n)\Gamma : n\in\Z\}}$ (which is a finite union of sub-nilmanifolds with the weighted Haar measure $\nu$) and by the Tietze extension theorem, there exist continuous functions $F_\veps$ and $H_\veps$ on $G/\Gamma$ with $F_\veps\leq F\leq H_\veps$ such that $\int (H_\veps - F_\veps) \dif\nu < \veps$.
By mollification we may assume that $H_\veps$ and $F_\veps$ are smooth.
By the above the limits $\lim_{N} \aveFn f(T^{n}x) H_{\veps}(g(n)\Gamma)$ and $\lim_{N} \aveFn f(T^{n}x) F_{\veps}(g(n)\Gamma)$ exist.
By continuity of $F_\veps$ and $H_\veps$ we have for every $x\in X'$
\begin{multline*}
\left(\limsup_{N\to\infty}- \liminf_{N\to\infty}\right)\aveFn f(T^nx) F(g(n)\Gamma)\\
\leq \lim_{N\to\infty} \aveFn f(T^nx) (H_\veps - F_\veps)(g(n)\Gamma)\\
\leq \lim_{N\to\infty} \aveFn (H_\veps - F_\veps)(g(n)\Gamma)
=  \int_{\tilde Y} (H_\veps - F_\veps) \dif\nu <  \veps,
\end{multline*}
and since $\veps>0$ was arbitrary this proves the existence of the limit \eqref{ave}.

A limiting argument allows one to replace the basic generalized nilsequence by a generalized nilsequence.
By linearity we obtain the conclusion for $f\in L^{\infty}(X)$.
The general case $f\in L^{1}(X)$ follows from the maximal inequality \eqref{eq:maximal-inequality}.

Under the additional assumptions of unique ergodicity of $(X,T)$ and continuity of the projection $\pi:X\to\HKZ_{l}(X)$ we find that the functions $f_{j}$ can be chosen to be continuous on $X$ by \cite[Theorem A]{MR2600993} and every point is fully generic for $f-f_{j}$, allowing us to replace the set of full measure $X'$ in the above argument by $X$.
\end{proof}

\subsection{$L^{2}$ convergence of weighted multiple averages}\label{sec:multiple}
The Wiener-Wintner theorem (Theorem~\ref{thm:WW-gen-nilseq} for linear nilsequences) has been used by Host and Kra \cite[Theorem 2.25]{MR2544760} to show that the values of a bounded measurable function along almost every orbit of an ergodic transformation are good weights for $L^{2}$ convergence of linear multiple ergodic averages.
A polynomial extension of this result was proved by Chu \cite[Theorem 1.1]{MR2465660}.
Since our Theorem~\ref{thm:WW-gen-nilseq} is stated for ``polynomial'' nilsequences, we can slightly shorten the proof of her result, that we formulate for $L^{1}$ functions and tempered F\o{}lner sequences.

\begin{corollary}[Convergence of weighted multiple ergodic averages]
\label{cor:weighted-multiple-convergence}
Let $(\Fo_N)$ be as above and let $\phi \in L^1(X)$.
Then there is a set $X'\subset X$ of full measure such that for every $x\in X'$ the sequence $\phi(T^n x)$ is a \emph{good weight for polynomial multiple ergodic averages along $(\Fo_N)$}, i.e., for every measure-preserving system $(Y,\nu, S)$, integer polynomials $p_{1},\dots,p_{k}$ and functions $f_1,\ldots,f_k\in L^\infty(Y,\nu)$ the averages
\begin{equation}
\label{eq:weighted-multiple}
\aveFn \phi(T^nx) S^{p_{1}(n)} f_1\cdots S^{p_{k}(n)} f_k
\end{equation}
converge in $L^2(Y,\nu)$ as $N\to \infty$.
\end{corollary}
In order to reduce to an appropriate nilfactor we need the following variant of \cite[Theorem 2.2]{MR2465660}.
Recall that two polynomials are called \emph{essentially distinct} if their difference is not constant.
\begin{lemma}
\label{lem:weighted-multiple-zero}
Let $(\Fo_{N})_{N}$ be an arbitrary F\o{}lner sequence in $\Z$.
For every $r,d\in\N$ there exists $k\in\N$
such that for every ergodic system $(X,\mu,T)$, any functions $f_{1},\dots,f_{r} \in L^{\infty}(X)$ with $\|f_{1}\|_{U^{k}(X)}=0$,
any non-constant pairwise essentially distinct integer polynomials $p_{1},\dots,p_{r}$ of degree at most $d$ and any bounded sequence of complex numbers $(a_{n})_{n}$ we have
\[
\limsup_{N\to\infty} \Big\| \aveFn a_{n} T^{p_{1}(n)}f_{1} \cdots T^{p_{r}(n)}f_{r} \Big\|_{L^{2}(X)} = 0.
\]
\end{lemma}
\begin{proof}
We may assume that $(a_{n})$ is bounded by $1$.
By a variant of the van der Corput lemma \cite[Lemma 4]{MR2151605} there exists a F\o{}lner sequence $(\Theta_{M})$ in $\Z^{3}$ such that the square of the left-hand side is bounded by
\begin{multline*}
\limsup_{M} \frac{1}{|\Theta_{M}|} \Big| \sum_{(n,v,w)\in\Theta_{M}} a_{n+v}\ol{a_{n+w}} \int_{X} \prod_{i=1}^{r} T^{p_{i}(n+v)}f_{i} T^{p_{i}(n+w)}\ol{f_{i}} \Big|\\
\leq
\limsup_{M} \frac{1}{|\Theta_{M}|} \sum_{(n,v,w)\in\Theta_{M}} \Big| \int_{X} \prod_{i=1}^{r} T^{p_{i}(n+v)}f_{i} T^{p_{i}(n+w)}\ol{f_{i}} \Big|.
\end{multline*}
By the Cauchy-Schwarz inequality the square of this expression is bounded by
\begin{multline*}
\limsup_{M} \frac{1}{|\Theta_{M}|} \sum_{(n,v,w)\in\Theta_{M}} \Big| \int_{X} \prod_{i=1}^{r} T^{p_{i}(n+v)}f_{i} T^{p_{i}(n+w)}\ol{f_{i}} \Big|^{2}\\
=
\limsup_{M} \frac{1}{|\Theta_{M}|} \sum_{(n,v,w)\in\Theta_{M}} \int_{X\times X} \prod_{i=1}^{r} (T\times T)^{p_{i}(n+v)}(f_{i}\otimes \ol{f_{i}}) (T\times T)^{p_{i}(n+w)}(\ol{f_{i}}\otimes f_{i}).
\end{multline*}
Let $\mu\times\mu = \int_{s\in Z} (\mu\times\mu)_{s} \dif s$ be the ergodic
decomposition of $\mu\times\mu$.
By Fatou's lemma the above expression is bounded by
\begin{multline*}
\int_{s\in Z}\limsup_{M} \frac{1}{|\Theta_{M}|} \sum_{(n,v,w)\in\Theta_{M}}\\
\int_{X\times X} \prod_{i=1}^{r} (T\times T)^{p_{i}(n+v)}(f_{i}\otimes \ol{f_{i}}) (T\times T)^{p_{i}(n+w)}(\ol{f_{i}}\otimes f_{i}) \dif(\mu\times\mu)_{s}\, \dif s\\
\leq
\int_{s\in Z}\limsup_{M} \Big\| \frac{1}{|\Theta_{M}|} \sum_{(n,v,w)\in\Theta_{M}}\\
\prod_{i=1}^{r} (T\times T)^{p_{i}(n+v)}(f_{i}\otimes \ol{f_{i}}) (T\times T)^{p_{i}(n+w)}(\ol{f_{i}}\otimes f_{i}) \Big\|_{L^{1}(X\times X,(\mu\times\mu)_{s})}\, \dif s.
\end{multline*}
Convergence to zero of the integrand follows from \textcite[Theorem 3]{MR2151605} provided that $\|f_{1} \otimes \ol{f_{1}}\|_{U^{k-1}(X\times X,(\mu\times\mu)_{s})}=0$ for some sufficiently large $k$.
It follows from \textcite[Lemma 3.1]{MR2150389} and the original definition of the uniformity seminorms in \cite[\textsection 3.5]{MR2150389} that
\[
\|f_{1}\|_{U^{k}(X)}^{2^{k}} = \int_{s\in Z} \|f_{1} \otimes \ol{f_{1}}\|_{U^{k-1}(X\times X,(\mu\times\mu)_{s})}^{2^{k-1}}\, \dif s.
\]
Thus the hypothesis ensures convergence to zero of the integrand in the previous display for a.e.\ $s$ provided that $k$ is large enough.
\end{proof}
\begin{proof}[Proof of Corollary~\ref{cor:weighted-multiple-convergence}]
By ergodic decomposition it suffices to consider ergodic systems $(Y,\nu,S)$.

Assume first that $\phi\in L^{\infty}(X)$.
By Lemma~\ref{lem:weighted-multiple-zero} we may assume that each $f_{i}$ is measurable with respect to some Host-Kra factor $\HKZ_{l}(Y)$.

By density we may further assume that each $f_{i}$ is a continuous function on a nilsystem factor of $Y$.
In this case the sequence $S^{p_{i}(n)}f_{i}(y)$ is a basic nilsequence of step at most $l \deg p_{i}$ for each $y\in Y$, and the product $\prod_{i} S^{p_{i}(n)}f_{i}(y)$ is also a basic nilsequence of step at most $l \max_{i}\deg p_{i}$.
Therefore the averages \eqref{eq:weighted-multiple} converge pointwise on $Y$ for a.e.\ $x\in X$ by Theorem~\ref{thm:WW-gen-nilseq}, and by the Dominated Convergence Theorem they converge in $L^{2}(Y)$.

We can finally pass to $\phi\in L^{1}(X)$ using the maximal inequality \eqref{eq:maximal-inequality}.
\end{proof}
%%% Local Variables: 
%%% mode: latex
%%% TeX-master: "phd-thesis.tex"
%%% End: 

\chapter{Return times theorems}
\label{chap:RTT}
\index{universally good weight}
We call a sequence $(a_{n})$ a \emph{universally good weight} (for pointwise convergence of ergodic averages along a tempered F\o{}lner sequence $(\Fo_{N})$ in $\Z$) if, for every measure-preserving system $(Y,S)$ and every $g\in L^{\infty}(Y)$, the averages
\[
\aveFN a_{n} g(S^{n}y)
\]
converge as $N\to\infty$ for almost every (a.e.)\ $y\in Y$.
In the last chapter we have seen that nilsequences are universally good weights.

It turns out that universally good weights are fairly ubiquitous.
\index{Return times theorem}
In fact, Bourgain's return times theorem \cite{MR1557098} asserts that, given any ergodic measure-preserving system $(X,T)$, for every $f\in L^{\infty}(X)$ and a.e.\ $x\in X$ the sequence of weights $a_{n}=f(T^{n}x)$ is universally good along the standard F\o{}lner sequence $\Fo_{N}=[1,N]$.
The name ``return times theorem'' comes from the case of a characteristic function $f=1_{A}$, $A\subset X$.
Then the theorem can be equivalently formulated by saying that, for a.e.\ $x\in X$, the pointwise ergodic theorem on any system $Y$ holds along the sequence of return times of $x$ to $A$.

A particularly illustrative case is that of a shift system on $X=\Omega^{\Z}$ with a bounded function $f$ that depends only on the zeroth coordinate.
In this case the return times theorem asserts that if the weights $(a_{n})$ are chosen according to independent random variables with the same distribution as $f$, then the resulting sequence is almost surely a universally good weight.

We will consider two generalizations of the return times theorem: to arbitrary amenable groups and to multiple term averages.

\section{Return times theorem for amenable groups}
Bourgain's return times theorem has been extended to discrete countable amenable groups for which an analog of the Vitali covering lemma holds by Ornstein and Weiss \cite[\textsection 3]{MR1195256}.
We extend this result to general, not necessarily discrete, locally compact second countable amenable groups.
It has been observed by Lindenstrauss \cite{MR1865397} that this is possible in the discrete case.
In the non-discrete case we have to restrict ourselves to the class of \emph{strong} F\o{}lner sequences (see Definition~\ref{def:folner}).
This is not a serious restriction in the sense that every lcsc amenable group admits such a sequence by Lemma~\ref{lem:strong-Folner-exist}.

A secondary goal of this section is to formulate and prove the Bourgain-Furstenberg-Katznelson-Ornstein (BFKO) orthogonality criterion \cite{MR1557098} at an appropriate level of generality.
This criterion provides a sufficient condition for the values of a function along an orbit of an ergodic measure-preserving transformation to be good weights for convergence to zero in the pointwise ergodic theorem.

Its original formulation is slightly artificial, since it assumes something about the whole measure-preserving system but concludes something that only involves a single orbit.
A more conceptual approach is to find a condition that identifies good weights and to prove that it is satisfied along almost all orbits of a measure-preserving system in a separate step.
For $\Z$-actions this seems to have been first explicitly mentioned by \textcite[\textsection 4]{MR1286798}.
In order to state the appropriate condition for general lcsc amenable groups we need some notation.

Throughout this section, $\AG$ denotes a lcsc amenable group with left Haar measure $|\cdot|$ and $(\Fo_{N})$ a F\o{}lner sequence in $\AG$.
The \emph{lower density} of a subset $S\subset \AG$ is defined by $\ld(S) := \liminf_{N} |S\cap \Fo_{N}| / |\Fo_{N}|$ and the \emph{upper density} is defined accordingly as $\ud(S) := \limsup_{N} |S\cap \Fo_{N}| / |\Fo_{N}|$.
\index{density!lower}
\index{density!upper}
All functions on $\AG$ that we consider are real-valued and bounded by $1$.
We denote averages by $\ave{g}{\Fo_{n}}:=\frac{1}{|\Fo_{n}|} \int_{g\in \Fo_{n}}$.
For $c\in L^{\infty}(\AG)$ we let
\[
S_{\delta,L,R} :=
\{a : \forall L\leq n\leq R\,
| \ave{g}{\Fo_{n}}c(g)c(ga) | < \delta\}.
\]
Our orthogonality  condition on the map $c$ is then the following.
\begin{equation}
\tag{$\perp$}
\label{eq:cond}
\forall\delta>0\,
\exists N_{\delta}\in\N\,
\forall N_{\delta} \leq L \leq R\quad
\ld(S_{\delta,L,R})
> 1-\delta.
\end{equation}
Very roughly speaking, this tells that there is little correlation between $c$ and its translates.
The condition \eqref{eq:cond} is an analytic counterpart of being orthogonal to the Kronecker factor, as the next result shows (see \textsection\ref{subsec:cond} for the proof).
\begin{lemma}
\label{lem:ae-perp}
Let $(X,\mu,\AG)$ be an ergodic measure-preserving system and $f\in L^{\infty}(X)$ be orthogonal to the Kronecker factor.
Then for a.e.\ $x\in X$ the map $g\mapsto f(gx)$ satisfies \eqref{eq:cond}.
\end{lemma}
The main result of this section is that the orthogonality condition is sufficient for the map to be a universally good weight for convergence to zero.
\begin{theorem}
\label{thm:good-weight-ptw-conv-to-zero}
Assume that $(\Fo_{N})$ is a tempered strong F\o{}lner sequence and $c\in L^{\infty}(\AG)$ satisfies the condition \eqref{eq:cond}.
Then for every ergodic measure-preserving system $(X,\AG)$ and $f\in L^{\infty}(X)$ we have
\[
\lim_{N\to\infty} \ave{g}{\Fo_{N}} c(g) f(gx) = 0
\quad \text{for a.e.\ } x\in X.
\]
\end{theorem}
This, together with a Wiener-Wintner type result, leads to the following return times theorem.
\begin{theorem}
\label{thm:rtt-amenable}
\index{Return times theorem!for amenable groups}
Let $\AG$ be a lcsc group with a tempered strong F\o{}lner sequence $(\Fo_{n})$.
Then for every ergodic measure-preserving system $(X,\AG)$ and every $f\in L^{\infty}(X)$ there exists a full measure set $\tilde X\subset X$ such that for every $x\in\tilde X$ the map $g\mapsto f(gx)$ is a good weight for the pointwise ergodic theorem along $(\Fo_{n})$.
\end{theorem}
The material in this section first appeared in \cite{arxiv:1301.1884}.

\subsection{The orthogonality condition}
\label{subsec:cond}
Now we verify that the BFKO condition implies \eqref{eq:cond}.
\begin{lemma}
\label{lem:to-zero-in-density}
Let $\AG$ be a lcsc group with a tempered F\o{}lner sequence $(\Fo_{n})$.
Let $(X,\AG)$ be an ergodic measure-preserving system and $f\in L^{\infty}(X)$ be bounded by $1$.
Let $x\in X$ be a fully generic point for $f$ such that
\begin{equation}
\label{eq:BFKO-cond}
\lim_{n}\ave{g}{\Fo_{n}} f(gx)f(g\xi) = 0
\quad\text{for a.e. }\xi\in X.
\end{equation}
Then the map $g\mapsto f(gx)$ satisfies \eqref{eq:cond}.
\end{lemma}
\begin{proof}
Let $\delta>0$ be arbitrary.
By Egorov's theorem there exists an $N_{\delta}\in\N$ and a set $\Xi\subset X$ of measure $>1-\delta$ such that for every $n\geq N_{\delta}$ and $\xi\in\Xi$ the average in \eqref{eq:BFKO-cond} is bounded by $\delta/2$.

Let $N_{\delta} \leq L \leq R$ be arbitrary and choose a continuous function $\eta : \R^{[L,R]} \to [0,1]$ that is $1$ when all its arguments are less than $\delta/2$ and $0$ when one of its arguments is greater than $\delta$ (here and later $[L,R]=\{L,L+1,\dots,R\}$).
Then by the Stone-Weierstrass theorem the function
\[
h(\xi):=\eta(|\ave{g}{\Fo_{L}}f(gx)f(g\xi)|,\dots,|\ave{g}{\Fo_{R}}f(gx)f(g\xi)|)
\]
lies in the closed convolution-invariant subalgebra of $L^{\infty}(X)$ spanned by $f$.

By the assumption $x$ is generic for $h$.
Since $h|_{\Xi} \equiv 1$, we have $\int_{X} h > 1-\delta$.
Hence the set of $a$ such that $h(ax)>0$ has lower density $>1-\delta$.

For every such $a$ we have
\[
| \ave{g}{\Fo_{n}} f(gax)f(gx) | < \delta,
\quad L\leq n\leq R.
\qedhere
\]
\end{proof}
\begin{proof}[Proof of Lemma~\ref{lem:ae-perp}]
Let $f\in L^{\infty}(X)$ be orthogonal to the Kronecker factor.
By \cite[Theorem 1]{MR0174705} this implies that the ergodic averages of $f\otimes f$ converge to $0$ in $L^{2}(X\times X)$.
By the Lindenstrauss pointwise ergodic theorem (Theorem~\ref{thm:Lindenstrauss-pointwise}) this implies \eqref{eq:BFKO-cond} for a.e.\ $x\in X$.
Since a.e.\ $x\in X$ is also fully generic for $f$, the conclusion follows from Lemma~\ref{lem:to-zero-in-density}.
\end{proof}

\subsection{Self-orthogonality implies orthogonality}
In our view, the BFKO orthogonality criterion is a statement about bounded measurable functions on $\AG$.
We encapsulate it in the following lemma.
\begin{lemma}
\label{lem:orth}
Let $(\Fo_{N})$ be a $C$-tempered strong F\o{}lner sequence.

Let $\epsilon>0$, $K\in\N$ and $\delta>0$ be sufficiently small depending on $\epsilon,K$.
Let $c\in L^{\infty}(\AG)$ be bounded by $1$ and $[L_{1},R_{1}],\dots,[L_{K},R_{K}]$ be a sequence of increasing intervals of natural numbers such that the following holds for any $j<k$ and any $N\in [L_{k},R_{k}]$.
\begin{enumerate}
\item $|\partial_{\Fo_{(j)}}\Fo_{N}|<\delta |\Fo_{N}|$, where $\Fo_{(j)}=\cup_{N=L_{j}}^{R_{j}} \Fo_{N}$
\item $S_{\delta,L_{j},R_{j}}$ has density at least $1-\delta$ in $\Fo_{N}$.
\end{enumerate}
Let $f\in L^{\infty}(\AG)$ be bounded by $1$ and consider the sets
\[
A_{N}:=
\{a : |\ave{g}{\Fo_{N}}c(g)f(ga)| \geq \epsilon\},
\quad
A_{(j)} := \cup_{N=L_{j}}^{R_{j}}A_{N}.
\]
Then for every compact set $I\subset \AG$ with $|I\cap \Fo_{(j)}\inv I^{\complement}|<\delta |I|$ for every $j$ we have
\[
\frac1K \sum_{j=1}^{K} d_{I}(A_{(j)}) < \frac{5C}{\epsilon \sqrt{K}}.
\]
\end{lemma}
Under the assumption \eqref{eq:cond} a sequence $[L_{1},R_{1}],\dots,[L_{K},R_{K}]$ with the requested properties can be constructed for any $K$.
\begin{proof}
For $1\leq k\leq K$, $L_{k}\leq N \leq R_{k}$ let $\Upsilon_{N}:\Omega_{N}\to M(\AG)$ be independent Poisson point processes of intensity $\alpha_{N}=\delta |\Fo_{N}|\inv$ w.r.t.\ the \emph{right} Haar measure.

Let $\Omega = \prod_{k=1}^{K}\prod_{N=L_{k}}^{R_{k}}\Omega_{N}$.
We construct random variables $\Sigma_{N} : \Omega \to M(A_{N})$ that are in turn used to define functions
\[
c^{(k)} := \sum_{N=L_{k}}^{R_{k}}\sum_{a\in\Sigma_{N}} \pm c|_{\Fo_{N}}(\cdot a\inv),
\quad k=1,\dots,K,
\]
where the sign is chosen according to as to whether $\ave{g}{\Fo_{N}}c(g)f(ga)$ is positive or negative.
These functions will be mutually nearly orthogonal on $I$ and correlate with $f$, from where the estimate will follow by a standard Hilbert space argument.

We construct the random variables in reverse order, beginning with $k=K$.
Let the set of ``admissible origins'' be
\[
O^{(j)}:=A_{(j)}\cap
\Big(\big(I
\setminus \Fo_{(j)}\inv I^{\complement}
\big)
\setminus \cup_{k=j+1}^{K}\cup_{N=L_{k}}^{R_{k}}\cup_{a\in\Sigma_{N}} (\partial_{\Fo_{(j)}}(\Fo_{N}) \cup (S_{\delta,L_{j},R_{j}}^{\complement} \cap \Fo_{N}))a
\Big).
\]
This set consists of places where we could put copies of initial segments of $c$ in such a way that they would correlate with $f$ and would not correlate with the copies that were already used in the functions $c^{(k)}$ for $k>j$.

Let $A_{N|R_{j}+1} := O^{(j)}\cap A_{N}$ and construct random coverings $\Sigma_{N}$, $N=L_{j},\dots,R_{j}$ as in Lemma~\ref{lem:lindenstrauss-covering} (if the Vitali lemma is available, then one can use deterministic coverings that it provides instead).
By Lemma~\ref{lem:lindenstrauss-covering} the counting function
\[
\Lambda^{(j)} = \sum_{N=L_{j}}^{R_{j}}\Lambda_{N},
\quad
\Lambda_{N}(g)(\omega) = \sum_{a\in\Sigma_{N}(\omega)}1_{\Fo_{N}a}(g)
\]
satisfies
\begin{enumerate}
\item $\E(\Lambda^{(j)}(g)) \leq (1+C\inv)$ for every $g\in \AG$
\item $\E(\Lambda^{(j)}(g)^{2}) \leq (1+C\inv)^{2}$ for every $g\in \AG$
\item $\E(\int\Lambda^{(j)}) \geq (2C)\inv |O^{(j)}|$.
\end{enumerate}
In particular, the last condition implies that
\[
\E\int_{I} c^{(j)}f > \epsilon(2C)\inv |O^{(j)}|,
\]
while the second shows that $\|c^{(j)}\|_{L^{2}(\Omega\times I)}\leq (1+C\inv)|I|^{1/2}$.
Moreover, it follows from the definition of $O^{(j)}$ that
\[
|\E\int_{I} c^{(j)} c^{(k)} | \leq |I|\delta(1+C\inv)
\]
whenever $j<k$.
Using the fact that $|c^{(j)}|\leq \Lambda^{(j)}$ and the Hölder inequality we obtain
\begin{multline*}
\sum_{j=1}^{K} \epsilon(2C)\inv \E|O^{(j)}|
< \E\int_{I} \sum_{j=1}^{K} c^{(j)}f
\leq \big( \E\int_{I} \big( \sum_{j=1}^{K} c^{(j)} \big)^{2} \big)^{1/2} |I|^{1/2}\\
< |I| \sqrt{K(1+C\inv)^{2} + K^{2}\delta(1+C\inv)}.
\end{multline*}
This can be written as
\[
\frac1K \sum_{j=1}^{K} \E|O^{(j)}|
< \frac{|I|}{\epsilon\gamma} \sqrt{(1+C\inv)^{2}/K+\delta(1+C\inv)}.
\]
Finally, the set $O^{(j)}$ has measure at least
\begin{multline*}
|I|(d_{I}(A_{(j)})-\delta)
-\sum_{k=j+1}^{K}\sum_{N=L_{k}}^{R_{k}}\sum_{a\in\Sigma_{N}} (|\partial_{\Fo_{(j)}}(\Fo_{N}a)|+|(S_{\delta,L_{j},R_{j}}^{\complement} \cap \Fo_{N})a|)\\
\geq
|I| (d_{I}(A_{(j)})-\delta)
-2\delta \sum_{k=j+1}^{K}\sum_{N=L_{k}}^{R_{k}}\sum_{a\in\Sigma_{N}} |\Fo_{N}a|
\end{multline*}
(here we have used the largeness assumptions on $L_{k}$), so
\[
\E |O^{(j)}|
\geq
|I| (d_{I}(A_{(j)})-\delta)
-2\delta (K-j) |I| (1+C\inv)
>
|I| (d_{I}(A_{(j)})-4\delta K)
\]
and the conclusion follows provided that $\delta$ is sufficiently small.
\end{proof}

The BFKO criterion for measure-preserving systems follows by a transference argument.
\begin{proof}[Proof of Theorem~\ref{thm:good-weight-ptw-conv-to-zero}]
Assume that the conclusion fails for some measure-preserving system $(X,\AG)$ and $f\in L^{\infty}$.
Then we obtain some $\epsilon>0$ and a set of positive measure $\Xi\subset X$ such that
\[
\limsup_{N\to\infty} | \ave{g}{\Fo_{N}}c(g)f(gx) | > 2\epsilon
\quad\text{for all } x\in\Xi.
\]
We may assume $\mu(\Xi)>\epsilon$.
Shrinking $\Xi$ slightly (so that $\mu(\Xi)>\epsilon$ still holds) we may assume that for every $\ul{N}\in\N$ there exists $F(\ul{N})\in\N$ (independent of $x$) such that for every $x\in\Xi$ there exists $\ul{N}\leq N\leq F(\ul{N})$ such that the above average is bounded below by $\epsilon$.

Let $K > 25 C^{2} \epsilon^{-4}$ and $[L_{1},R_{1}],\dots,[L_{K},R_{K}]$ be as in Lemma~\ref{lem:orth} with $R_{j}=F(L_{j})$.
In this case that lemma says that at least one of the sets $\cup_{N=L_{j}}^{R_{j}}A_{N}$ has upper density less than $\epsilon$.

Choose continuous functions $\eta_{j} : \R^{[L_{j},R_{j}]}\to [0,1]$ that are $1$ when at least one of their arguments is greater than $2\epsilon$ and $0$ if all their arguments are less than $\epsilon$.
Let
\[
h(x) := \prod_{j=1}^{K} \eta_{j}(|\ave{g}{\Fo_{L_{j}}}c(g)f(gx)|,\dots,|\ave{g}{\Fo_{R_{j}}}c(g)f(gx)|).
\]
By construction of $F$ we know that $h|_{\Xi}\equiv 1$, so that $\int_{X} h > \epsilon$.
Let $x_{0}$ be a generic point for $h$ (e.g.\ any fully generic point for $f$), then $\ld\{a : h(ax_{0})>0\} > \epsilon$.
In other words,
\[
\ld\{ a : \forall j\leq K\,
\exists N \in [L_{j},R_{j}]\,
| \ave{g}{\Fo_{N}} c(g)f(gax_{0}) | > \epsilon
\} > \epsilon.
\]
This contradicts Lemma~\ref{lem:orth} with $f(g)=f(gx_{0})$.
\end{proof}

For translations on compact groups we obtain the same conclusion everywhere.
It is not clear to us whether an analogous statement holds for general uniquely ergodic systems.
\begin{corollary}
\label{cor:good-weight-everywhere-conv-to-zero}
Let $\AG$ be a lcsc group with a $C$-tempered strong F\o{}lner sequence $(\Fo_{n})$.
Let $c\in L^{\infty}(\AG)$ be a function bounded by $1$ that satisfies the condition \eqref{eq:cond}.
Let also $\Omega$ be a compact group and $\chi : \AG\to\Omega$ a continuous homomorphism.
Then for every $\phi\in C(\Omega)$ we have
\[
\lim_{N\to\infty} \ave{g}{\Fo_{N}} c(g) \phi(\chi(g)\omega) = 0
\quad \text{for every } \omega\in\Omega.
\]
\end{corollary}
\begin{proof}
We may assume that $\chi$ has dense image, so that the translation action by $\chi$ becomes ergodic.
By Theorem~\ref{thm:good-weight-ptw-conv-to-zero} we obtain the conclusion a.e., and the claim follows by uniform continuity of $\phi$.
\end{proof}

For $\Z$-actions Lemma~\ref{lem:to-zero-in-density} and Theorem~\ref{thm:good-weight-ptw-conv-to-zero} imply the following orthogonality criterion which is due to Bourgain, Furstenberg, Katznelson, and Ornstein in the case of the standard Ces\`aro averages \cite[Proposition]{MR1557098}.
\begin{proposition}
\label{prop:BFKO}
Let $(X,T)$ be an ergodic measure-preserving system, $(\Fo_{N})$ a tempered F\o{}lner sequence in $\Z$, and $f\in L^{\infty}(X)$.
Assume that $x\in X$ is fully generic for $f$ and
\[
\aveFN f(T^{n}x) f(T^{n}\xi)\to 0
\quad\text{ as } N\to\infty
\,\text{ for a.e.\ }\xi\in X.
\]
Then for every measure-preserving system $(Y,S)$ and $g\in L^{\infty}(Y)$ we have
\[
\aveFN f(T^{n}x)g(S^{n}y)\to 0
\quad\text{ as } N\to\infty
\,\text{ for a.e.\ }y\in Y.
\]
\end{proposition}

\subsection{Return times theorem for amenable groups}
We turn to the deduction of the return times theorem (Theorem~\ref{thm:rtt-amenable}).
This will require two distinct applications of Theorem~\ref{thm:good-weight-ptw-conv-to-zero}.
We begin with a Wiener-Wintner type result.

Recall that the Kronecker factor of a measure-preserving dynamical system corresponds to the reversible part of the Jacobs-de Leeuw-Glicksberg decomposition of the associated Koopman representation.
In particular, it is spanned by the finite-dimensional $\AG$-invariant subspaces of $L^{2}(X)$.
We refer to \cite{isem-book} for a treatment of the JdLG decomposition.

Let $F \subset L^{2}(X)$ be a $d$-dimensional $\AG$-invariant subspace and $f\in F$.
We will show that for a.e.\ $x\in X$ we have $f(gx) = \phi(\chi(g)u)$ for some $\phi\in C(U(d))$, continuous representation $\chi : \AG \to U(d)$, and a.e.\ $g\in \AG$.
To this end choose an orthonormal basis $(f_{i})_{i=1,\dots,d}$ of $F$.
Then by the invariance assumption we have $f_{i}(g\cdot) = \sum_{j} c_{i,j} f_{j}(\cdot)$, and the matrix $(c_{i,j})$ is unitary since the $\AG$-action on $X$ is measure-preserving.
This gives us a measurable representation $\chi$ that is automatically continuous \cite[Theorem 22.18]{0837.43002}.
The point $u=(u_{i})$ is given by the coordinate representation $f=\sum u_{i}f_{i}$.
Thus we have $f(g\cdot) = \sum_{i}(\chi(g)u)_{i} f_{i}(\cdot)$ in $L^{2}(X)$ and hence, fixing some measurable representatives for $f_{i}$'s, a.e.\ on $X$.
By Fubini's theorem we obtain a full measure subset of $X$ such that the above identity holds for a.e.\ $g\in \AG$.
For every $x$ from this set we obtain the claim with the continuous function $\phi(U) = \sum_{i} (Uu)_{i}f_{i}(x)$.

\begin{corollary}[Wiener-Wintner-type theorem]
\label{cor:wiener-wintner}
\index{Wiener-Wintner theorem!for amenable groups}
Let $\AG$ be a lcsc group with a tempered strong F\o{}lner sequence $(\Fo_{n})$.
Then for every ergodic measure-preserving system $(X,\AG)$ and every $f\in L^{\infty}(X)$ there exists a full measure set $\tilde X \subset X$ such that the following holds.
Let $\Omega$ be a compact group and $\chi : \AG\to\Omega$ a continuous homomorphism.
Then for every $\phi\in C(\Omega)$, every $\omega\in\Omega$ and every $x\in \tilde X$ the limit
\[
\lim_{N\to\infty} \ave{g}{\Fo_{N}} f(gx) \phi(\chi(g)\omega)
\]
exists.
\end{corollary}
\begin{proof}
By Lemma~\ref{lem:ae-perp} and Corollary~\ref{cor:good-weight-everywhere-conv-to-zero} we obtain the conclusion for $f$ orthogonal to the Kronecker factor.

By linearity and in view of the Lindenstrauss maximal inequality \eqref{eq:maximal-inequality} it remains to consider $f$ in a finite-dimensional invariant subspace of $L^{2}(X)$.
In this case, for a.e.\ $x\in X$ we have $f(gx)=\phi'(\chi'(g)u_{0})$ for some finite-dimensional representation $\chi':\AG\to U(d)$, some $u_{0}\in U(d)$, some $\phi'\in C(U(d))$ and a.e.\ $g\in \AG$.
The result now follows from uniqueness of the Haar measure on the closure of $\chi\times\chi'(\AG)$.
\end{proof}
A different proof using unique ergodicity of an ergodic group extension of a uniquely ergodic system can be found in \cite{MR1195256}.

Finally, the return times theorem follows from a juxtaposition of previous results.
\begin{proof}[Proof of Theorem~\ref{thm:rtt-amenable}]
By Lemma~\ref{lem:ae-perp} and Theorem~\ref{thm:good-weight-ptw-conv-to-zero} the conclusion holds for $f\in L^{\infty}(X)$ orthogonal to the Kronecker factor.

By linearity and in view of the Lindenstrauss maximal inequality \eqref{eq:maximal-inequality} it remains to consider $f$ in a finite-dimensional invariant subspace of $L^{2}(X)$.
In this case, for a.e.\ $x\in X$ we have $f(gx)=\phi(\chi(g)u_{0})$ for some finite-dimensional representation $\chi:\AG\to U(d)$, some $u_{0}\in U(d)$, some $\phi\in C(U(d))$ and a.e.\ $g\in \AG$.
The conclusion now follows from Corollary~\ref{cor:wiener-wintner}.
\end{proof}

\section{Multiple term return times theorem}
An extension of the return times theorem to averages involving multiple terms has been obtained by Rudolph \cite{MR1489899}.
The precise statement of this result is fairly long, so we begin by introducing the appropriate notation and concepts.
For the whole section we fix a tempered F\o{}lner sequence $(\Fo_{N})$ in $\Z$.

\subsection{Conventions about cube measures}
\begin{definition}
\label{def:D}
A \emph{system} is a regular ergodic measure-preserving system $(X,\mu,T)$ with a distinguished set $D\subset L^{\infty}(X)$ that satisfies the following conditions.
\begin{enumerate}
\item (Cardinality) $D$ is countable.
\item (Density) $D$ contains an $L^{\infty}$-dense subset of $C(X)$.
\item (Algebra) $D$ is a $T$-invariant $\Q$-algebra (i.e., closed under translation by $T$, pointwise product, and $\Q$-linear combinations) and is closed under absolute value.
\item (Decomposition) For every $f\in D$ and $l\in\N$ there exist $l$-step nilfactors $Z_{j}$, $j\in\N$, of $X$ and decompositions
\begin{equation}
% Tag equation with "Dec(l)", \ref generates tag "Dec"
\makeatletter
\def\tagform@#1{\maketag@@@{\ignorespaces#1\unskip\@@italiccorr} (l)}
\makeatother
\tag{Dec}
\label{eq:dec}
f=f_{\perp}+f_{\HKZ,j}+f_{err,j},
\quad j\in\N,
\end{equation}
such that $f_{\perp},f_{\HKZ,j},f_{err,j}\in D$, $f_{\perp}\perp\HKZ_{l}(X)$, $f_{\HKZ,j} \in C(Z_{j})$, $\|f_{err,j}\|_{L^{\infty}(\mu)}$ is uniformly bounded in $j$ and $\|f_{err,j}\|_{L^{1}(\mu)}\to 0$ as $j\to\infty$.
\end{enumerate}
\end{definition}
For any regular ergodic measure-preserving system $(X,\mu,T)$, any countable subset of $L^{\infty}(X)$ is contained in a set $D$ that satisfies the above conditions.
Indeed, by the Host--Kra structure theorem (Theorem~\ref{thm:HK-structure}) \emph{every} bounded function on $X$ has a decomposition of the form \ref{eq:dec}$(l)$ for every $l\in\N$.

Our multiple term return times theorem will be formulated on cube spaces.
As a first preparatory step we fix well-behaved full measure subsets of the cube spaces associated to the individual systems.
\begin{lemma}
\label{lem:Yl}
Let $(X,\mu,T,D)$ be a system.
Then there exist measurable subsets $Y_{l}\subset X^{[l]}$ such that for every $l\in\N$ the following statements hold.
\begin{enumerate}
\item $\mu^{[l]}(Y_{l}) = 1$ and for every $y\in Y_{l}$ we have $m_{y}(Y_{l})=1$.
\item\label{it:disintegration} For every $y\in Y_{l}$ the measure $\m_{y}$ is ergodic and one has
\begin{equation}
\label{eq:motimesm}
\m_{y}\otimes\m_{y} = \int_{Y_{l+1}} \m_{x} \dif(\m_{y}\otimes\m_{y})(x).
\end{equation}
\item\label{it:generic}
$Y_{l} \subset (\tilde X)^{[l]}$, where $\tilde X\subset X$ is the set of points that are generic for each $f\in D$ w.r.t.\ $\mu$.
\item\label{it:orth} For every $y\in Y_{l}$, every $k\in\N$, and any functions $f_{\epsilon}\in D$, $\epsilon\in\{0,1\}^{l}$, such that $f_{\epsilon}\perp\HKZ_{k+l}(X)$ for some $\epsilon$ we have $f^{[l]} \perp \HKZ_{k}(X^{[l]},\m_{y})$.
\end{enumerate}
\end{lemma}
\begin{proof}
The fact that (\ref{it:orth}) holds for full measure subsets of $X^{[l]}$ follows from the Cauchy-Schwarz-Gowers inequality~\eqref{eq:CSG}.
The sets $(\tilde X)^{[l]}\subset X^{[l]}$ have full measure by the pointwise ergodic theorem and the definition \eqref{eq:mul+1} of cube measures, taking care of (\ref{it:generic}).
Also, the measure $\m_{y}$ is ergodic for $\mu^{[l]}$-a.e.\ $y\in X^{[l]}$, taking care of the first part of (\ref{it:disintegration}).

The only delicate point is \eqref{eq:motimesm}.
By \eqref{eq:m-disint} and \eqref{eq:mul+1-conventional}, for a fixed full measure domain of integration this disintegration identity holds for $\mu^{[l]}$-a.e.\ $y\in X^{[l]}$.
However, the domain of integration is yet to be determined.
This is done by a fixed-point procedure: choose tentative sets $Y_{l} \subset X^{[l]}$ that satisfy all conditions but \eqref{eq:motimesm} for every $l$.
For every $l$ this gives a $\mu^{[l]}$-full measure subset of $y\in X^{[l]}$ for which \eqref{eq:motimesm} holds.
The intersection of this set with $Y_{l}$ gives a new tentative set $Y_{l}$.
This way for each $l$ we obtain a decreasing sequence of tentative full measure subsets of $X^{[l]}$ whose intersection still has full measure and satisfies all requested properties.
\end{proof}

We are now in position to define what we mean by universal full measure sets.
Recall that we write $f_{i}^{[l]}=\otimes_{\epsilon\in\{0,1\}^{l}}f_{i,\epsilon}$, where $f_{i,\epsilon} \in L^{\infty}(X_{i})$.
\begin{definition}
\label{def:luae}
Let $P$ be a statement about ergodic regular measure-preserving systems $(X_{i},\mu_{i},T_{i})$, functions $f_{i}^{[l]}$ and points $x_{i}\in X_{i}^{[l]}$, $i=0,\dots,k$.
We say that $P$ holds for \emph{$[l]$-universally almost every ($[l]$-u.a.e.)}\ tuple $x_{0},\dots,x_{k}$ if
\begin{itemize}
\item[$(0)$] For every system $(X_{0},\mu_{0},T_{0},D_{0})$ there exists a measurable set $\tilde X_{0}^{[l]}\subset X_{0}^{[l]}$ such that for every $y_{0}\in Y_{0,l}$ we have $\m_{y_{0}}(\tilde X_{0}^{[l]})=1$ and
\item[$(1)$] for every system $(X_{1},\mu_{1},T_{1},D_{1})$ there exists a measurable set $\tilde X_{1}^{[l]}\subset X_{0}^{[l]}\times X_{1}^{[l]}$ such that for every $\vec x_{0}\in\tilde X_{0}^{[l]}$ and every $y_{1}\in Y_{1,l}$ we have $\m_{y_{1}}\{x_{1}: (\vec x_{0},x_{1})\in\tilde X_{1}\}=1$ and
\item[]\begin{center}$\vdots$\end{center}
\item[$(k)$] for every system $(X_{k},\mu_{k},T_{k},D_{k})$ there exists a measurable set $\tilde X_{k}^{[l]}\subset X_{0}^{[l]}\times\dots\times X_{k}^{[l]}$ such that for every $\vec x_{k-1}\in\tilde X_{k-1}^{[l]}$ and every $y_{k}\in Y_{k,l}$ we have $\m_{y_{k}}\{x_{k}: (\vec x_{k-1},x_{k})\in\tilde X_{k}\}=1$ and
\end{itemize}
we have $P(f_{0}^{[l]},\dots,f_{k}^{[l]},\vec x_{k})$ for every $\vec x_{k}\in\tilde X_{k}$ and any $f_{i,\epsilon}\in D_{i}$, $0\leq i\leq k$, $\epsilon\in\{0,1\}^{l}$.
\end{definition}

\subsection{Return times theorem on cube spaces}
\label{sec:RTT-cube}
Our multiple term return times theorem states that certain pro-nilfactors are characteristic for return time averages on cube spaces.
\begin{theorem}
\label{thm:RTT-cube}
\index{Return times theorem!multiple term}
For any $k,l\in\N$, the limit
\begin{equation}
\label{eq:av}
\lim_{N\to\infty} \aveFN \prod_{i=0}^{k}f_{i}^{[l]}(T_{i}^{[l]}x_{i})
\end{equation}
exists for $[l]$-u.a.e.\ $x=(x_{1},\dots,x_{k})$.
If in addition
\begin{equation}
\tag{CF}
\label{eq:cf}
\begin{aligned}
f_{0,\epsilon}&\perp\HKZ_{k+l}(X_{0})
\text{ for some }
\epsilon\in\{0,1\}^{l}\\
\text{or }
f_{i,\epsilon}&\perp\HKZ_{k+l+1-i}(X_{i})
\text{ for some }
\epsilon\in\{0,1\}^{l}
\text{ and }
1\leq i\leq k,
\end{aligned}
\end{equation}
then the limit vanishes $[l]$-u.a.e.
\end{theorem}
Note carefully that, unlike in the nilsequence Wiener-Wintner theorem, we have to consider \emph{ergodic} measure-preserving systems here.
This is due to the fact that in the ergodic case Structure Theorem~\ref{thm:HK-structure} implies a decomposition result that is stronger than Decomposition Theorem~\ref{thm:CFH-decomposition}, namely, one can then assume that the structured function $f_{s}$ is continuous on the pro-nilsystem given by the structure theorem.
We use this feature of the structured function in the proof of Lemma~\ref{lem:m-tilde-m}.
A possible way to handle the non-ergodic case would be to identify an orthogonality condition in the spirit of \eqref{eq:cond} that would guarantee convergence of weighted multiple averages to zero u.a.e.
We will not attempt this here.

We refer to the statement of Theorem~\ref{thm:RTT-cube} with fixed $k,l$ as \RTT{k,l}, with fixed $k$ and arbitrary $l$ as \RTT{k,\cdot}, and to the condition \eqref{eq:cf} for fixed $k,l$ as \eqref{eq:cf}$(k,l)$ (``CF'' stands for ``characteristic factors'').
Birkhoff's pointwise ergodic theorem \cite{0003.25602} is \RTT{0,0}, Bourgain's return times theorem \cite{MR1557098} is \RTT{1,0}, and Rudolph's multiple term return times theorem \cite{MR1489899} is \RTT{k,0} for arbitrary $k\in\N$ (with the standard F\o{}lner sequence $\Fo_{N}=[1,N]$).
More about the history of these and related results can be found in a recent survey by Assani and Presser \cite{2012arXiv1209.0856A}.

The fact that the Host-Kra-Ziegler pro-nilfactor $\HKZ_{k}(X_{0})$ is characteristic for the first term in \RTT{k,0} in the sense that if $f_{0}\perp\HKZ_{k}(X_{0})$, then the averages \eqref{eq:av} converge to zero $[0]$-u.a.e., is due to Assani and Presser \cite[Theorem 4]{MR2901351}.
However, their proof depends on the convergence result \RTT{k,0}.
Moreover, \ref{eq:cf}$(k,0)$ also identifies characteristic factors for the other terms.

We prove both results, \RTT{k,\cdot} and characteristicity, simultaneously by induction on $k$ using Host-Kra structure theory.
This proof first appeared in \cite{arxiv:1210.5202}.

The base case $k=0$ follows by definition of $Y_{0,l}$ and the pointwise ergodic theorem.
For the remaining part of this section we assume \RTT{k,\cdot} for some fixed $k$ and prove \RTT{k+1,\cdot}.
If $k>0$, then we also assume all other results of this section for $k-1$ in place of $k$ (thus, strictly speaking, it is the conjunction of all results in this section that is proved by induction).

In order to prove \RTT{k+1,l} for a given $l$ we write
\begin{equation}
\label{eq:X}
X_{0}^{[l+1]} \times\dots\times X_{k}^{[l+1]}
=
(X_{0}^{[l]} \times\dots\times X_{k}^{[l]})^{2}=:X^{2}.
\end{equation}
From \RTT{k,l+1} we know that the appropriate ergodic averages converge $[l+1]$-u.a.e.\ on $X^{2}$.
We would like to apply Proposition~\ref{prop:BFKO} with this $X$ and $Y=X_{k+1}^{[l]}$.
This will necessitate the dependence of the universal sets in Definition~\ref{def:luae} on preceding systems.
The remaining part of this section is dedicated to reformulating \RTT{k,l+1} in such a way that it can be plugged into Proposition~\ref{prop:BFKO}.

This involves the following steps.
First we use \RTT{k,\cdot} to construct a certain universal measure disintegration with built-in genericity properties on a product of ergodic systems (Theorem~\ref{thm:return-times-disintegration}).
We use characteristic factors for \RTT{k,\cdot} to represent measures in this disintegration in a different way.
Finally, we verify a certain instance of \RTT{k+1,\cdot} (Lemma~\ref{lem:cube-meas}).

\subsection{A measure-theoretic lemma}
We will need the classical fact that the Kronecker factor is characteristic for $L^{2}$ convergence of ergodic averages with arbitrary bounded scalar weights, see e.g.\ \cite[Corollary 7.3]{MR2544760} for a more general version.
\begin{lemma}
\label{lem:Z1-char-weight}
Let $(X,T)$ be an ergodic measure-preserving system and $f\in L^{2}(X)$ be orthogonal to $\HKZ_{1}(X)$.
Then for any bounded sequence $(a_{n})_{n}$ one has
\[
\lim_{N} \aveFN a_{n}T^{n}f = 0
\quad\text{in } L^{2}(X).
\]
\end{lemma}
The next lemma is our main tool for dealing with cube measures.
Informally, it shows that a certain kind of universality for $\mu^{[1]}\otimes\nu^{[1]}$ implies some universality for $(\mu\times\nu)^{[1]}$.

Recall that, for ergodic measure-preserving systems $(X,\mu),(Y,\nu)$, the projection onto the invariant factor of $X\times Y$ has the form $\phi(x,y)=\psi(\pi_{1}(x),\pi_{1}(y))$, where $\pi_{1}$ are projections onto the Kronecker factors and $\psi$ is the quotient map of $\HKZ_{1}(X)\times\HKZ_{1}(Y)$ by the orbit closure of the identity.
To see this, recall that by Lemma~\ref{lem:Z1-char-weight} the function $f\otimes g$, $f\in L^{\infty}(X)$, $g\in L^{\infty}(Y)$, is orthogonal to the invariant factor of $X\times Y$ whenever $f\perp\HKZ_{1}(X)$ or $g\perp\HKZ_{1}(Y)$.
Thus the invariant sub-$\sigma$-algebra on $X\times Y$ is contained in $\HKZ_{1}(X)\times\HKZ_{1}(Y)$, i.e.\ it is (isomorphic to) the invariant sub-$\sigma$-algebra of a product of two compact group rotations (cf.\ e.g.\ \cite[Theorem 1.9]{MR1325712}).
In particular, for an ergodic system $Y$ the invariant factor of $Y\times Y$ is isomorphic to $\HKZ_{1}(Y)$.
\begin{lemma}
\label{lem:cube-change-order}
Let $(X,\mu),(Y,\nu)$ be ergodic measure-preserving systems and fix measure disintegrations
\[
\mu = \int_{\kappa\in\HKZ_{1}(X)} \mu_{\kappa} \dif\kappa,
\quad
\nu = \int_{\lambda\in\HKZ_{1}(Y)} \mu_{\lambda} \dif\lambda.
\]
This induces an ergodic decomposition
\[
\nu\otimes\nu = \int_{\lambda\in\HKZ_{1}(Y)} (\nu\otimes\nu)_{\lambda} \dif\lambda,
\quad
(\nu\otimes\nu)_{\lambda} = \int_{\lambda'\in\HKZ_{1}(Y)} \nu_{\lambda'}\otimes\nu_{\lambda' \lambda\inv} \dif\lambda'.
\]
Let $x\in X$ and $\Lambda\subset\HKZ_{1}(Y)$ be a full measure set.
Assume that for $\mu$-a.e.\ $\xi$ and every $\lambda\in\Lambda$, for $(\nu\otimes\nu)_{\lambda}$-a.e.\ $(\eta,\eta')$ some statement $P(x,\xi,\eta,\eta')$ holds.
Then $P(x,\xi,y,\eta)$ also holds for $\nu$-a.e.\ $y$ and $\tilde\m_{x,y}$-a.e.\ $(\xi,\eta)$, where
\[
\tilde\m_{x,y} =
\int_{\kappa\in\HKZ_{1}(X),\lambda\in\HKZ_{1}(Y):\psi(\pi_{1}(x),\pi_{1}(y))=\psi(\kappa,\lambda)} \mu_{\kappa}\otimes\nu_{\lambda} \dif(\kappa,\lambda),
\]
the homomorphism $\psi$ is as above and the integral is taken over an affine subgroup (i.e.\ a coset of a closed subgroup) with respect to its Haar measure.
\end{lemma}
\begin{proof}
Recall that $\ker\psi$ has full projections on both coordinates.
Therefore, for \emph{every} $x$ there is a full measure set of $\xi$ such that the set $\Lambda$ has full measure in $\{\lambda : \psi(\pi_{1}(x)\pi_{1}(\xi)\inv,\lambda)=\id\}$ (note that this is a closed affine subgroup of $\HKZ_{1}(Y)$ that therefore has a Haar measure).

In particular, for a full measure set of $\xi$ (that depends on $Y$) the hypothesis holds for a.e.\ $\lambda$ with $\psi(\pi_{1}(x)\pi_{1}(\xi)\inv,\lambda)=\id$, i.e.\ we have $P(x,\cdot)$ for a set of full measure w.r.t.\ the measure
\begin{multline*}
\int_{\xi\in X} \delta_{\xi}\otimes \int_{{\lambda\in\HKZ_{1}(Y) : \psi(\pi_{1}(x)\pi_{1}(\xi)\inv,\lambda)=\id}}
(\nu\otimes\nu)_{\lambda} \dif\lambda \dif\mu(\xi)\\
=
\int_{\kappa\in\HKZ_{1}(X)} \int_{\lambda\in\HKZ_{1}(Y):\psi(\pi_{1}(x)\kappa\inv,\lambda)=\id} \mu_{\kappa}\otimes
(\nu\otimes\nu)_{\lambda} \dif\lambda \dif\kappa\\
=
\int_{\kappa\in\HKZ_{1}(X)} \int_{\lambda\in\HKZ_{1}(Y):\psi(\pi_{1}(x)\kappa\inv,\lambda)=\id} \mu_{\kappa}\otimes
\int_{\lambda'\in\HKZ_{1}(Y)} \nu_{\lambda'}\otimes\nu_{\lambda'\lambda\inv} \dif\lambda' \dif\lambda \dif\kappa\\
=
\int_{\kappa\in\HKZ_{1}(X)} \int_{\lambda\in\HKZ_{1}(Y):\psi(\pi_{1}(x)\kappa\inv,\lambda)=\id} \mu_{\kappa}\otimes
\int_{y\in Y} \delta_{y}\otimes\nu_{\pi(y)\lambda\inv} \dif\nu(y) \dif\lambda \dif\kappa\\
=
\int_{y\in Y}
\int_{\kappa\in\HKZ_{1}(X)} \int_{\lambda\in\HKZ_{1}(Y):\psi(\pi_{1}(x)\kappa\inv,\lambda)=\id} \mu_{\kappa}\otimes
\delta_{y}\otimes\nu_{\pi(y)\lambda\inv} \dif\lambda \dif\kappa \dif\nu(y)\\
=
\int_{y\in Y} \int_{\kappa\in\HKZ_{1}(X),\lambda\in\HKZ_{1}(Y):\psi(\pi_{1}(x),\pi_{1}(y))=\psi(\kappa,\lambda)} \mu_{\kappa}
\otimes \delta_{y}\otimes\nu_{\lambda} \dif(\kappa,\lambda) \dif\nu(y)\\
=
\int_{y\in Y} \delta_{y}\otimes \tilde\m_{x,y} \dif\nu(y).
\end{multline*}
This gives $P(x,\xi,y,\eta)$ for $\nu$-a.e.\ $y$ and $\tilde\m_{x,y}$-a.e.\ pair $(\xi,\eta)$ as required.
\end{proof}
The next lemma provides us with means for using the measure $\tilde\m_{x,y}$ in a higher step setting.
\begin{lemma}
\label{lem:prod-fiber-erg}
Let $(Z,g),(Z',g')$ be ergodic nilsystems and $\psi:\HKZ_{1}(Z)\times\HKZ_{1}(Z') \to H$ the factor map modulo the orbit closure of $(\pi_{1}(g),\pi_{1}(g'))$.
Then for every $\lambda\in\HKZ_{1}(Z)$ and a.e.\ $\lambda'\in\HKZ_{1}(Z')$ the rotation by $(g,g')$ on the nilmanifold
\[
N_{\lambda,\lambda'} = \{(z,z')\in Z\times Z' : \psi(\pi_{1}(z),\pi_{1}(z')) = \psi(\lambda,\lambda')\}
\]
is uniquely ergodic.
\end{lemma}
\begin{proof}
By Lemma~\ref{lem:nilsystem-ergodic} it suffices to prove ergodicity to obtain unique ergodicity.

Since $N_{\lambda,\lambda'}$ only depends on $\psi(\lambda,\lambda')$ and $\ker\psi$ has full projection on $\HKZ_{1}(Z)$ it suffices to verify the conclusion for a full measure set of $(\lambda,\lambda')$.
To this end it suffices to check that for any $f\in C(Z),f'\in C(Z')$ the limit of the ergodic averages of $f\otimes f'$ is essentially constant on $N_{\lambda,\lambda'}$.
We decompose $f=f_{\perp}+f_{\HKZ}$ with $f_{\perp}\perp\HKZ_{1}(Z)$ and $f_{\HKZ}\in L^{\infty}(\HKZ_{1}(Z))$, and analogously for $f'$.
For $f_{\HKZ}\otimes f'_{\HKZ}$ the limit is essentially constant on $N_{\lambda,\lambda'}$ for any $(\lambda,\lambda')$ since the rotation is ergodic on $(\pi_{1}\times\pi_{1})(N_{\lambda,\lambda'})$.

On the other hand, the limit of the ergodic averages of tensor products involving $f_{\perp}$ vanishes on $Z\times Z'$ a.e.\ by Lemma~\ref{lem:Z1-char-weight}, hence also a.e.\ on a.e.\ fiber $N_{\lambda,\lambda'}$.
\end{proof}

\subsection{Universal disintegration of product measures}
The return times theorem can be seen as a statement about measure disintegration, cf.\ \textcite[Theorem 4]{MR1357765} for the case $k=1$.
\begin{theorem}
\label{thm:return-times-disintegration}
\index{Return times theorem!measure disintegration form}
Let $(X_{i},\mu_{i},T_{i},D_{i})$, $i=0,\dots,k$, be systems.
Then $[l]$-u.a.e.\ $x_{0},\dots,x_{k}$ is generic for some measure $\m_{x_{0},\dots,x_{k}}$ on $X_{0}^{[l]}\times\dots\times X_{k}^{[l]}$ and every function $\otimes_{i=0}^{k}f_{i}^{[l]}$, $f_{i,\epsilon}\in D_{i}$.

Moreover, for $[l]$-u.a.e.\ $x_{0},\dots,x_{k-1}$ and every $y_{k}\in Y_{l,k}$ one has
\begin{equation}
\label{eq:m-disintegration}
\m_{x_{0},\dots,x_{k-1}}\otimes\m_{y_{k}} = \int \m_{x_{0},\dots,x_{k}} \dif\m_{y_{k}}(x_{k}).
\end{equation}
\end{theorem}
\begin{proof}
By Theorem~\ref{thm:RTT-cube} we obtain convergence of the averages
\[
\aveFN \prod_{i=0}^{k}f_{i}^{[l]}(T_{i}^{n}x_{i})
\]
for $[l]$-u.a.e.\ $x_{0},\dots,x_{k}$ and any $f_{i,\epsilon}\in D_{i}$.
For continuous functions $f_{i,\epsilon}\in D_{i}$ we define $\m_{x_{0},\dots,x_{k}}(\otimes_{i=0}^{k}f_{i}^{[l]})$ as the limit of these averages.
By the Stone-Weierstraß theorem these tensor products span a dense subspace $C(X_{0}^{[l]}\times\dots\times X_{k}^{[l]})$, so by density the above (bounded) linear form admits a unique continuous extension.

In order to obtain \eqref{eq:m-disintegration} it suffices to verify that the integrals of functions of the form $\otimes_{i=0}^{k} f_{i}^{[l]}$, $f_{i,\epsilon}\in D_{i}$, with respect to both measures coincide.
By genericity and the dominated convergence theorem we have for $[l]$-u.a.e.\ $x_{0},\dots,x_{k-1}$ that
\begin{multline*}
\int \int \otimes_{i<k}f_{i}^{[l]}\otimes f_{k}^{[l]} \dif\m_{x_{0},\dots,x_{k}} \dif\m_{y_{k}}(x_{k})\\
=
\int \lim_{N} \aveFN \prod_{i<k}f_{i}^{[l]}(T_{i}^{n}x_{i})\cdot f_{k}^{[l]}(T_{k}^{n}x_{k}) \dif\m_{y_{k}}(x_{k})\\
=
\lim_{N} \aveFN \prod_{i<k}f_{i}^{[l]}(T_{i}^{n}x_{i})\cdot \int f_{k}^{[l]}(T_{k}^{n}x_{k}) \dif\m_{y_{k}}(x_{k})\\
=
\int \otimes_{i<k}f_{i}^{[l]} \dif m_{x_{0},\dots,x_{k-1}} \int f_{k}^{[l]} \dif\m_{y_{k}}
\end{multline*}
as required.
\end{proof}
We will now represent the measure $\m_{x_{0},\dots,x_{k}}$ for $[l]$-u.a.e.\ $x_{0},\dots,x_{k}$ in the form $\tilde\m_{x,y}$ in the notation of Lemma~\ref{lem:cube-change-order}.
At this step we have to use the information about characteristic factors.
We begin with a preliminary observation.
\begin{lemma}
\label{lem:uae-fiber-ae}
If some property P holds for $[l]$-u.a.e.\ $x_{0},\dots,x_{k}$ then, for $[l]$-u.a.e.\ $x_{0},\dots,x_{k}$, P holds $\m_{x_{0},\dots,x_{k}}$-a.e.
\end{lemma}
\begin{proof}
For $k=0$ this follows from \eqref{eq:m-disintegration}.
Assume that the conclusion is known for $k-1$ and show it for $k$.

By the induction hypothesis, for $[l]$-u.a.e.\ $x_{0},\dots,x_{k-1}$, $\m_{x_{0},\dots,x_{k-1}}$-a.e., for every $y_{k}\in Y_{k,l}$, P holds $\m_{y_{k}}$-a.e.\ in $x_{k}$.
The conclusion follows from \eqref{eq:m-disintegration}.
\end{proof}
\begin{lemma}
\label{lem:m-tilde-m}
For $[l]$-u.a.e.\ $x_{0},\dots,x_{k}$ we have
\[
\m_{x_{0},\dots,x_{k}}=\tilde\m_{x,y},
\]
where we use the notation of Lemma~\ref{lem:cube-change-order} with
\[
(X,\mu)=(X_{0}^{[l]}\times\dots\times X_{k-1}^{[l]},\m_{x_{0},\dots,x_{k-1}}),
\]
$(Y,\nu)=(X_{k}^{[l]},\m_{x_{k}})$, $x=(x_{0},\dots,x_{k-1})$ and $y=x_{k}$.
\end{lemma}
\begin{proof}
To verify that the measures coincide it suffices to check that the integrals of functions of the form $\otimes_{i=0}^{k}f_{i}^{[l]}$, $f_{i,\epsilon}\in D_{i}$ coincide.
To this end consider the splittings $f_{i,\epsilon}=f_{i,\epsilon,\perp}+f_{i,\epsilon,\HKZ,j}+f_{i,\epsilon,err,j}$, $j\in\N$, given by \ref{eq:dec}$(k+l+1-i)$.

Projections of tensor products that involve $f_{i,\epsilon,\perp}$ on one of the Kronecker factors vanish a.e.\ for $[l]$-u.a.e.\ $x,y$ by Corollary~\ref{cor:orth} for $k-1$ that is part of the induction hypothesis for this section.
Since $\ker\psi$ has full projections on both coordinates the corresponding integrals w.r.t. $\tilde\m_{x,y}$ also vanish.
The integrals w.r.t. $\m_{x,y}$ vanish for $[l]$-u.a.e.\ $x,y$ by Theorem~\ref{thm:RTT-cube}.

For the main terms we have
\begin{multline}
\label{eq:m-erg-main-term}
\int \otimes_{i=0}^{k} f_{i,\HKZ,j}^{[l]} \dif\tilde\m_{x,y}\\
=
\int\limits_{\mathclap{\kappa\in\HKZ_{1}(X),\lambda\in\HKZ_{1}(Y):\psi(\pi_{1}(x),\pi_{1}(y))=\psi(\kappa,\lambda)}}
\E(\otimes_{i=0}^{k-1} f_{i,\HKZ,j}^{[l]}|\HKZ_{1}(X))(\kappa) \E(f_{k,\HKZ,j}^{[l]}|\HKZ_{1}(Y))(\lambda) \dif(\kappa,\lambda).
\end{multline}
Since the underlying nilmanifold of a nilsystem is a bundle of nilmanifolds over its Kronecker factor, the conditional expectation above is just integration in the fibers, and by uniqueness of the Haar measure the whole integral equals
\[
\int_{\kappa\in Z_{j},\lambda\in Z_{j}':\psi(\pi_{1}(x),\pi_{1}(y))=\psi(\pi_{1}(\kappa),\pi_{1}(\lambda))} \otimes_{i=0}^{k-1} f_{i,\HKZ,j}^{[l]}(\kappa) f_{k,\HKZ,j}^{[l]}(\lambda) \dif(\kappa,\lambda),
\]
where $Z_{j}$ is the orbit closure of $x$ in $\prod_{i=0}^{k-1}Z_{i,j}^{[l]}$ and $Z_{j}'$ is the orbit closure of $y$ in $Z_{k,j}^{[l]}$.
By Lemma~\ref{lem:prod-fiber-erg}, the above fibers of $Z_{j}\times Z_{j}'$ are uniquely ergodic for every $x$ and a.e.\ $y$, and the integral then equals
\[
\lim_{N} \aveFN \otimes_{i=0}^{k-1} f_{i,\HKZ,j}^{[l]}(T^{n}x) f_{k,\HKZ,j}^{[l]}(S^{n}y)
=
\int \otimes_{i=0}^{k}f_{i,\HKZ,j}^{[l]} \dif\m_{x,y}.
\]
It remains to treat the error terms, i.e. the case $f_{i',\epsilon'}=f_{i',\epsilon',err,j}$ for some $i',\epsilon'$.
By Lemma~\ref{lem:Yl}(\ref{it:generic}), for $[l]$-u.a.e.\ $x,y$ we have
\[
\int \otimes_{i=0}^{k}f_{i}^{[l]} \dif\m_{x,y} \lesssim \|f_{i',\epsilon'}\|_{L^{1}(\mu_{i})}
\to 0 \quad \text{as} \quad j\to\infty.
\]
Similarly, we have $\int | \otimes_{i=0}^{k-1}f_{i}^{[l]} | \dif\m_{x} \lesssim \|f_{i',\epsilon'}\|_{L^{1}(\mu_{i})}$ if $i'<k$ and $\int | f_{k}^{[l]} | \dif\m_{y_{k}} \lesssim \|f_{k,\epsilon}\|_{L^{1}(\mu_{i})}$ if $i'=k$ for $[l]$-u.a.e.\ $x,y$.
This implies that either $\E(\otimes_{i=0}^{k-1}f_{i}^{[l]}|\HKZ_{1}(X))$ or $\E(f_{k}^{[l]}|\HKZ_{1}(Y))$ converges to zero in probability for $[l]$-u.a.e.\ $x,y$, so
\[
\int \otimes_{i=0}^{k}f_{i}^{[l]} \dif \tilde\m_{x,y}
\to 0 \quad \text{as} \quad j\to\infty
\]
for $[l]$-u.a.e.\ $x_{0},\dots,x_{k}$ since $\ker\psi$ has full projections on coordinates.
\end{proof}
\begin{corollary}
\label{cor:m-ergodic}
For $[l]$-u.a.e.\ $x_{0},\dots,x_{k}$ the measure $\m_{x_{0},\dots,x_{k}}=\tilde\m_{x,y}$ is ergodic.
\end{corollary}
Note that even for a non-ergodic invariant measure on a regular system there may exist generic points, so the mere fact that $\vec x$ is generic for $\m_{\vec x}$ does not suffice.
\begin{proof}
In order to see that $\m_{x_{0},\dots,x_{k}}$ is ergodic it suffices to verify that for any continuous functions $f_{i,\epsilon}\in C(X_{i})$ we have
\begin{equation}
\label{eq:criterion-ergodicity-of-m}
\lim_{N}\aveFN \otimes_{i=0}^{k}f_{i}^{[l]}(T^{n}\vec\xi) = \int \otimes_{i=0}^{k}f_{i}^{[l]} \dif\m_{x_{0},\dots,x_{k}}
\quad \text{for } \m_{x_{0},\dots,x_{k}}\text{-a.e.\ }\vec\xi.
\end{equation}
Recall that for $[l]$-u.a.e.\ $x_{0},\dots,x_{k}$ the limit on the left-hand side of \eqref{eq:criterion-ergodicity-of-m} exists for $\m_{x_{0},\dots,x_{k}}$-a.e.\ $\vec\xi$ by Lemma~\ref{lem:uae-fiber-ae} and equals $\int \otimes_{i=0}^{k}f_{i}^{[l]} \dif\m_{\vec\xi}$.
Splitting the $f_{i,\epsilon}$'s as before it suffices to verify \eqref{eq:criterion-ergodicity-of-m} for the main terms, and this follows directly from \eqref{eq:m-erg-main-term}.
\end{proof}

\subsection{The sufficient special case of convergence to zero}
The last hypothesis of Proposition~\ref{prop:BFKO} is a certain special case of its conclusion.
Recall that we already have u.a.e.\ convergence to zero on $X^{2}$ (as defined in \eqref{eq:X}), but not yet in the required sense.
This is now corrected using Lemma~\ref{lem:cube-change-order}.
\begin{lemma}[Change of order in the cube construction]
\label{lem:cube-order}
Let $l\in\N$ and $P$ be a statement about points of $\prod_{i=0}^{k}X_{i}^{[l+1]}$.
Assume that for $[l+1]$-u.a.e.\ $x_{0},\dots,x_{k}$
we have $P(x_{0},\dots,x_{k})$.

Then for $[l]$-u.a.e.\ $x_{0},\dots,x_{k}$,
for $\m_{x_{0},\dots,x_{k}}$-a.e.\ $x'$,
we have $P(x_{0},\dots,x_{k},x')$.
\end{lemma}
Strictly speaking, the coordinates of $x'$ in $(x_{0},\dots,x_{k},x')$ should be attached to $x_{0},\dots,x_{k}$ but we do not want to introduce additional notation at this point.
\begin{proof}
The base case $k=0$ follows directly from \eqref{eq:motimesm}.

Assume now that $k>0$.
By the inductive hypothesis of this section the conclusion holds for $k-1$, so
for $[l]$-u.a.e.\ $x_{0},\dots,x_{k-1}$,
\emph{for $\m_{x_{0},\dots,x_{k-1}}$-a.e.\ $x'$,
for every $y_{k}\in Y_{k,l+1}$ and $\m_{y_{k}}$-a.e.\ $x_{k}$,}
we have $P(x_{0},\dots,x_{k-1},x',x_{k})$.

Using \eqref{eq:motimesm} we can rewrite the emphasized part of the statement as
``for $\m_{x_{0},\dots,x_{k-1}}$-a.e.\ $x'$,
for every $\tilde y_{k}\in Y_{k,l}$,
for every ergodic component $\mu_{e}$ of $(\m_{\tilde y_{k}})^{2}$ from a fixed full measure set, for $\mu_{e}$-a.e.\ $x_{k}$''
The conclusion follows by Lemma~\ref{lem:cube-change-order} and Lemma~\ref{lem:m-tilde-m}.
\end{proof}
\begin{lemma}
\label{lem:cube-meas}
Let $l,l'\in\N$ and assume \ref{eq:cf}$(k,l+l')$.
Then for $[l]$-u.a.e.\ $\vec x_{0}=(x_{0},\dots,x_{k})$,
for $\m_{\vec x_{0}}$-a.e.\ $\vec x_{1}$, \ldots, for $\m_{\vec x_{0},\dots,\vec x_{l'-1}}$-a.e.\ $\vec x_{l'}$
the ergodic averages of the function $\otimes_{i=0}^{k}f_{i}^{[l+l']}$ converge to zero at $(\vec x_{0},\dots,\vec x_{l'})$.
\end{lemma}
Again, the tensor product $\otimes_{i=0}^{k}f_{i}^{[l+l']}$ should be arranged in a different order, but in our opinion the above notation makes our goal more clear: it is not the function but the order in which we build the product space that changes.
\begin{proof}
We use induction on $l'$.
The case $l'=0$ is precisely Theorem~\ref{thm:RTT-cube}.
Assume that the conclusion is known for $l+1$ and $l'-1$.
The claim for $l$ and $l'$ follows by Lemma~\ref{lem:cube-order}.
\end{proof}
\begin{corollary}
\label{cor:orth}
Let $l,l'\in\N$ and assume \ref{eq:cf}$(k,l+l')$.
Then for $[l]$-u.a.e.\ $x_{0},\dots,x_{k}$
we have
$f_{0}^{[l]}\otimes\dots\otimes f_{k}^{[l]} \perp\HKZ_{l'}(\m_{x_{0},\dots,x_{k}})$.
\end{corollary}
\begin{proof}
This follows from Lemma~\ref{lem:cube-meas} by Lemma~\ref{lem:uae-fiber-ae}, the definition of cube measures \eqref{eq:mul+1},
the characterization of uniformity seminorms \eqref{eq:uniformity-seminorm-integral} and the ergodic theorem.
\end{proof}
\begin{proof}[Proof of Theorem~\ref{thm:RTT-cube} for $k+1$]
Let $k,l\in\N$ be fixed, our objective is to prove \RTT{k+1,l}.
Assume first \ref{eq:cf}$(k,l+1)$.
Then Lemma~\ref{lem:cube-meas} with $l'=1$ states that
for $[l]$-u.a.e.\ $x=(x_{0},\dots,x_{k})$,
for $\m_{x_{0},\dots,x_{k}}$-a.e.\ $x'$,
for any $f_{i,\epsilon}\in D_{i}$
we have
\[
\lim_{N}\aveFN \otimes_{i=0}^{k}f_{i}^{[l]}((\otimes_{i=0}^{k}T_{i}^{[l]})^{n}x)
\cdot \otimes_{i=0}^{k}f_{i}^{[l]}((\otimes_{i=0}^{k}T_{i}^{[l]})^{n}x')
= 0.
\]
For $[l]$-u.a.e.\ $x_{0},\dots,x_{k}$ we obtain genericity w.r.t.\ $\m_{x_{0},\dots,x_{k}}$ by Theorem~\ref{thm:return-times-disintegration}, ergodicity of $\m_{x_{0},\dots,x_{k}}$ by Corollary~\ref{cor:m-ergodic} and orthogonality of $\otimes_{i=0}^{k} f_{i}^{[l]}$ to the Kronecker factor of $\m_{x_{0},\dots,x_{k}}$ by Corollary~\ref{cor:orth}, so Proposition~\ref{prop:BFKO} with $X=(X_{0}^{[l]}\times\dots\times X_{k}^{[l]},\m_{x_{0},\dots,x_{k}})$ and $Y=(X_{k+1}^{[l]},\m_{y_{k+1}})$ implies the claimed convergence to zero $[l]$-u.a.e.

This takes care of the terms $f_{i,\epsilon,\perp}$ in the splittings $f_{i,\epsilon} = f_{i,\epsilon,\perp}+f_{i,\epsilon,\HKZ,j}+f_{i,\epsilon,err,j}$ given by \ref{eq:dec}$(k+l+1-i)$ (respectively, \ref{eq:dec}$(k+l)$ for $i=0$).
By an approximation argument like in the proof of Lemma~\ref{lem:m-tilde-m} it suffices to consider the main terms, so we may assume that $\prod_{i=0}^{k} f_{i}^{[l]}((T_{i}^{[l]})^{n}x)$ is a nilsequence.
The claimed convergence a.e.\ in $x_{k+1}$ then follows from Theorem~\ref{thm:WW-gen-nilseq}.

Finally, assume \ref{eq:cf}$(k+1,l)$.
This means that we have either \ref{eq:cf}$(k,l+1)$ or $f_{k+1,\epsilon}\perp\HKZ_{l+1}(X_{k+1})$ for some $\epsilon$.
In the former case the limit is zero $[l]$-u.a.e.\ by the above argument and in the latter case by definition of $Y_{k+1,l}$ and Lemma~\ref{lem:Z1-char-weight}.
\end{proof}

\section{Wiener-Wintner return times theorem for nilsequences}
\label{sec:ww}
We also obtain the following joint extension of the multiple term return times theorem and the Wiener-Wintner theorem for nilsequences, thereby generalizing \cite[Theorem 1]{MR1357765}.
\begin{theorem}[Wiener-Wintner return times theorem for nilsequences]
\label{thm:WWRTT}
\index{Return times theorem!Wiener-Wintner-type}
Let $k,l\in\N$ and $f_{i}\in L^{\infty}(X_{i})$, $i=0,\dots,k$.
Then for u.a.e.\ $x_{0},\dots,x_{k}$ and every $l$-step nilsequence $(a_{n})_{n}$ the averages
\[
\aveFN a_{n} \prod_{i=0}^{k}f_{i}(T_{i}^{n}x_{i})
\]
converge as $N\to\infty$ (to zero if in addition $f_{0}\perp\HKZ_{k+l}(X_{0})$ or $f_{i}\perp\HKZ_{k+l+1-i}(X_{i})$ for some $i=1,\dots,k$).
\end{theorem}
The first step in the proof is the identification of characteristic factors in the spirit of \textcite[\textsection 4]{MR1357765}.
\begin{lemma}
\label{lem:HKZ-char-for-WWRTT}
Let $f_{i}\in L^{\infty}(X_{i})$, $i=0,\dots,k$, and assume \ref{eq:cf}$(k,l)$.
With the notation of Theorem~\ref{thm:uniform-convergence-to-zero}, for u.a.e.\ $x_{0},\dots,x_{k}$ we have
\[
\lim_{N\to\infty} \sup_{g\in\poly, F\in W^{\tilde k,2^{l}}(G/\Gamma)}
\|F\|_{W^{\tilde k,2^{l}}(G/\Gamma)}\inv
\Big| \aveFN F(g(n)\Gamma) \prod_{i=0}^{k}f_{i}(T_{i}^{n}x_{i}) \Big| = 0,
\]
where $\tilde k = \sum_{r=1}^{l}(d_{r}-d_{r+1})\binom{l}{r-1}$.
\end{lemma}
\begin{proof}
By Corollary~\ref{cor:orth} we have $\otimes_{i=0}^{k}f_{i} \perp \HKZ_{l}(\m_{\vec x})$ for u.a.e.\ $\vec x\in X_{0}\times\dots\times X_{k}$ and by Theorem~\ref{thm:return-times-disintegration} u.a.e.\ $\vec x$ is fully generic for $\otimes_{i} f_{i}$ w.r.t.\ $\m_{\vec x}$.
The claim follows by Theorem~\ref{thm:WW-unif}.
\end{proof}
Theorem~\ref{thm:WWRTT} now follows from equidistribution results on nilmanifolds.
\begin{proof}[Proof of Theorem~\ref{thm:WWRTT}]
Fix $k,l\in\N$.
By Lemma~\ref{lem:HKZ-char-for-WWRTT} it suffices to consider $f_{i}\in L^{\infty}(\HKZ_{l+k+1-i}(X_{i}))$.
By the pointwise ergodic theorem we can assume that each $f_{i}$ is a continuous function on a nilfactor of $X_{i}$.
The conclusion follows since any product of nilsequences is a nilsequence and every nilsequence converges in the uniform Ces\`aro sense by Corollary~\ref{cor:Leibman-pointwise}.
\end{proof}
%%% Local Variables: 
%%% mode: latex
%%% TeX-master: "phd-thesis.tex"
%%% End: 

\backmatter
\pagestyle{plain}
\cleardoublepage
\phantomsection
\addcontentsline{toc}{chapter}{Bibliography}
\printbibliography
\cleardoublepage
\phantomsection
\addcontentsline{toc}{chapter}{Index}
\printindex
\end{document}